\PassOptionsToPackage{top=2.5cm, bottom=2cm, left=2cm, right=2cm}{geometry}
\documentclass[final,3p,times, leqno]{elsarticle}

\makeatletter
\def\ps@pprintTitle{%
   \let\@oddhead\@empty
   \let\@evenhead\@empty
   \def\@oddfoot{\reset@font\hfil\thepage\hfil}
   \let\@evenfoot\@oddfoot
}
\makeatother

\usepackage[x11names]{xcolor}
\usepackage[
  colorlinks,
]{hyperref}

\newcommand\rurl[1]{%
  \href{http://#1}{\nolinkurl{#1}}%
}

\AtBeginDocument{\hypersetup{citecolor=blue, urlcolor=magenta, linkcolor=magenta}}

\usepackage{amsthm, amssymb, amsfonts, mathrsfs, amsmath}
\usepackage{lmodern}
\usepackage{hyperref}
\usepackage[T5]{fontenc}
\usepackage{amscd}
\usepackage{graphicx}

\usepackage{pb-diagram}

\theoremstyle{plain}
\newtheorem{thm}{Theorem}[section]

\newtheorem{corl}[thm]{Corollary}

\renewcommand{\theequation}{\arabic{section}.\arabic{equation}}

\theoremstyle{plain}
\newtheorem{therm}{Theorem}[subsection]
\newtheorem{propo}[therm]{Proposition}
\newtheorem{lema}[therm]{Lemma}
\newtheorem{corls}[therm]{Corollary}

\theoremstyle{definition}

\newtheorem{nx}[therm]{Remark}

\theoremstyle{definition}
\newtheorem{dn}{Definition}[subsection]

\newtheorem{rema}[dn]{Remark}

\newtheorem*{acknow}{Acknowledgments}

\theoremstyle{plain}

\newtheorem{theo}[dn]{Theorem}
\newtheorem{mde}[dn]{Proposition}
\newtheorem{conje}[dn]{Conjecture}

\def\DD{D\kern-.7em\raise0.4ex\hbox{\char '55}\kern.33em}

\renewcommand{\baselinestretch}{1.2}

\usepackage{sectsty}
\subsectionfont{\fontsize{11.5}{11.5}\selectfont}

\begin{document}
\fontsize{11.5pt}{11.5}\selectfont

\begin{frontmatter}

\title{On Peterson's open problem\\ and representations of the general linear groups}


\author{\DD\d{\u a}ng V\~o Ph\'uc}
\address{{\fontsize{10pt}{10}\selectfont Faculty of Education Studies, University of Khanh Hoa,\\ 01 Nguyen Chanh, Nha Trang, Khanh Hoa, Viet Nam\\[1mm]  \textit{(Dedicated to Professor Frank Williams)}}}
\ead{dangvophuc@ukh.edu.vn; dangphuc150488@gmail.com}

\begin{abstract}
Fix $\mathbb Z/2$ is the prime field of two elements and write $\mathcal A_2$ for the mod $2$ Steenrod algebra. Denote by $GL_d:= GL(d, \mathbb Z/2)$ the general linear group of rank $d$ over $\mathbb Z/2$ and by $\mathscr P_d$ the polynomial algebra $\mathbb Z/2[x_1, x_2, \ldots, x_d]$ as a connected unstable $\mathcal A_2$-module on $d$ generators of degree one. We study the {\it Peterson "hit problem"} of finding the minimal set of $\mathcal A_2$-generators for $\mathscr P_d.$ It is equivalent to determining a $\mathbb Z/2$-basis for the space of "cohits" $Q\mathscr P_d := \mathbb Z/2\otimes_{\mathcal A_2} \mathscr P_d \cong \mathscr P_d/\mathcal A_2^+\mathscr P_d.$ This $Q\mathscr P_d$ is also a representation of $GL_d$ over $\mathbb Z/2.$ The problem for $d= 5$ is not yet completely solved,  and unknown in general. In this work, we give an explicit solution to the hit problem of five variables in the generic degree $n = r(2^t -1) + 2^ts$ with $r = d = 5,\ s  =8$ and $t$ an arbitrary non-negative integer. An application of this study to the cases $t = 0$ and $t  = 1$ shows that the Singer algebraic transfer of rank $5$ is an isomorphism in the bidegrees $(5, 5+(13.2^{0} - 5))$ and $(5, 5+(13.2^{1} - 5)).$ Moreover, the result when $t\geq 2$ was also discussed. Here, the Singer transfer of rank $d$ is a $\mathbb Z/2$-algebra homomorphism from $GL_d$-coinvariants of certain subspaces of $Q\mathscr P_d$ to the cohomology groups of the Steenrod algebra, ${\rm Ext}_{\mathcal A_2}^{d, d+*}(\mathbb Z/2, \mathbb Z/2).$ It is one of the useful tools for studying mysterious Ext groups and the Kervaire invariant one problem.

\end{abstract}

\end{frontmatter}

\tableofcontents

\section{Introduction and statement of results}\label{s1}
Throughout this article, we shall work only at the prime $2.$ Let $Sq^k: H^{*}(\mathbb X)\to H^{k+*}(\mathbb X)$ be the stable cohomology operation of degree $k\geq 0$, which is introduced by Steenrod in $1947$ (see \cite{S.E}). Here $H^{*}(\mathbb X)$ is the singular cohomology group of the topological space $\mathbb X$ with coefficients in $\mathbb Z/2.$ The $\mathbb Z/2$-graded algebra $\mathcal A_2$ generated by the operations $Sq^k$ is called {\it the mod $2$ Steenrod algebra} and acts in a natural way on the cohomology of any space $\mathbb X$. For $d$ a natural number, we denote by $B(\mathbb Z/2)^{\times d}$ the classifying space of elementary abelian $2$-group $(\mathbb Z/2)^{\times d}$  of rank $d$ and by $\mathscr P_d = \mathbb Z/2[x_1, x_2, \ldots, x_d]$ the polynomial algebra on $d$ variables of degree $1$. Of course, $B(\mathbb Z/2)^{\times d}$ is homotopy equivalent to $(\mathbb RP(\infty))^{\times d},$ where $\mathbb RP(\infty)$ denotes the infinite real projective space. Since $\mathscr P_d$ is isomorphic to the cohomology with $\mathbb Z/2$-coefficients of $B(\mathbb Z/2)^{\times d},$ it has a connected unstable left $\mathcal A_2$-module structure.  The left action of $\mathcal {A}_2$ on $\mathscr P_d$ is determined by the unstable condition $Sq^{1}(x_i) = x_i^{2}, \ Sq^{k}(x_i) = 0$ for $k > 1$ and Cartan's formula (see \cite{S.E}).

The investigation of the homotopy classification of topological spaces leads us to the study of the cohomology groups of the Steenrod algebra, ${\rm Ext}_{\mathcal A_2}^{d,*}(\mathbb Z/2, \mathbb Z/2).$ It has been thoroughly studied for homological degrees $d\leq 5$ (see Adams \cite{J.A2}, Adem \cite{Adem}, Wall \cite{Wall}, Wang \cite{Wang}, Tangora \cite{M.T}, Lin \cite{W.L}, Chen \cite{T.C}). However, for $d$ higher, the calculations seem to be difficult. Moreover, it has a deep connection with the "hit problem" of our interest in determining the minimal set of $\mathcal A_2$-generators for $\mathscr P_d.$ Equivalently, we need to find the dimension of the quotient space
$$Q\mathscr P_d := \mathbb Z/2\otimes_{\mathcal A_2} \mathscr P_d \cong \mathscr P_d/\mathcal A_2^+\mathscr P_d$$
in each $d$ and degree $n\geq 1.$ Here $\mathcal A_2^+$ denotes the {\it augmentation ideal} of $\mathcal A_2$ and $\mathbb Z/2$ is viewed as a right $\mathcal A_2$-module concentrated in grading $0.$ This $Q\mathscr P_d$ will also be called the space of "cohits". The hit problem was posed by Peterson \cite{F.P1} in 1987. However, it remains open for $d\geq 5.$

As well known, the general linear group $GL_d := GL(d, \mathbb Z/2)$ acts regularly on $\mathscr P_d$ by matrix substitution. Further, the two actions of $\mathcal A_2$ and $GL_d$ upon $\mathscr P_d$ commute with each other; hence there is an inherited action of $GL_d$ on $Q\mathscr P_d.$ From this event, one of the applications of the hit problem of Peterson is to study the representations of the general linear groups over $\mathbb Z/2.$ Therefrom the hit problem has attracted great interest of many algebraic topologists (see Crabb and Hubbuck \cite{C.H}, Crossley \cite{M.C1}, Kameko \cite{M.K}, Mothebe and his collaborators \cite{M.M, M.M2, M.K.R}, Nam \cite{T.N1}, Pengelley and Williams \cite{P.W1, P.W2}, Priddy \cite{S.P}, Silverman and Singer \cite{S.S}, Singer \cite{W.S2}, Peterson \cite{F.P2}, the present author \cite{P.S1, P.S2, D.P2, D.P3, D.P3-1, D.P9, D.P10, D.P11, D.P12}, Sum \cite{N.S1, N.S3, N.S, N.S5, N.S6, N.S7, N.S4}, Walker and Wood \cite{W.W1, W.W2, W.W3}, Wood \cite{R.W1, R.W2} and others). 

Several other aspects of the hit problems were then studied by many authors. For instance, the hit problem for the symmetric of polynomials $\mathscr P_d^{S_d}$ as an $\mathcal A_2$-submodule of $\mathscr P_d,$ has been of interest in \cite{J.W}, where $S_d$ is the symmetric group on $d$ letters acting on the right of $\mathscr P_d,$ and $\mathscr P_d^{S_d}$ is isomorphic to the cohomology algebra $H^{*}(BO(d))$ of the Grassmannian of $d$-dimensional vector subspaces of $\mathbb RP(\infty).$ The space $BO(d)$ is the classifying space of the orthogonal group $O(d),$ Note also that the symmetric polynomials in $d$ variables divisible by all of them can be identified with the cohomology algebra of the Thom space $MO(d)$ of the standard $d$-plane bundle over $BO(d)$ associated with the bordism theory of closed smooth manifolds.  In \cite{H.P, H.N}, H\uhorn ng and his collaborators have studied the hit problem for the rank $d$ Dickson algebra,  $\mathscr P_d^{GL_d}$ (the algebra of $GL_d$-invariants). The Dickson algebra is also an unstable $\mathcal A_2$-module and is dual to the coalgebra of Dyer-Lashof operations of the length $d$ (see Madsen \cite{Madsen}). The relationship between Kudo-Araki-May algebra and the hit problem has been investigated by Pengelley and Williams \cite{P.W5, P.W3, P.W4}, and by Singer \cite{W.S3}. In \cite{A.S}, Ault and Singer have examined the dual problem of the Peterson hit problem, which is to determine the set of $\mathcal A_2^+$-annihilated elements in the homology of $B(\mathbb Z/2)^{\times d}.$ Recently, Zare \cite{Zare} has used geometric methods to study the hit problem for $H_*(B(\mathbb Z/2)^{\times d})$ (the dual of the hit problem of Peterson) as well as the hit problem for $H_*(BO(d))$ (the dual of the symmetric hit problem of Janfada and Wood). His main idea is based on the relation between the Dyer-Lashof algebra and these dual hit problems. 
Let $P_{\mathcal A_2}H_{*}(B(\mathbb Z/2)^{\times d})$ be the subspace of $H_{*}(B(\mathbb Z/2)^{\times d})$ consisting of all elements that are $\mathcal A_2^+$-annihilated. With the idea of describing the cohomology groups of the Steenrod algebra by means of modular representations of the general linear groups, William Singer \cite{W.S1} constructed a transfer homomorphism of rank $d$ from $GL_d$-coinvariants of the $\mathcal A_2^+$-annihilated elements of the dual of $\mathscr P_d$ to the cohomology of the Steenrod algebra:
$$ Tr_d: \mathbb Z/2 \otimes_{GL_d}P_{\mathcal A_2}H_{*}(B(\mathbb Z/2)^{\times d})\to {\rm Ext}_{\mathcal {A}_2}^{d,d+*}(\mathbb Z/2,\mathbb Z/2),$$
which is related to the geometrical transfer $tr_d: \pi_*^{\mathbb S}(B(\mathbb Z/2)^{\times d}_+)\to \pi_{*}^{\mathbb S}(\mathbb S^0)$ of the stable homotopy of spheres. More explicitly, $tr_d$ induces $Tr_d$ at the $E^{2}$-term of the Adams spectral sequence \cite{J.A}.  These transfers play a key role in the study of the Kervaire invariant one problem, which is one of the oldest unresolved issues in Differential and Algebraic topology. This problem was first introduced by Browder \cite{Browder} where he indicated that smooth framed manifolds of Kervaire invariant one exist only in dimensions of the form $2^{i+1}-2,$ and that  a manifold exists in that dimension if and only if the class $h_i^{2}\in {\rm Ext}_{\mathcal A_2}^{2, 2^{i+1}}(\mathbb Z/2, \mathbb Z/2)$ in the $E^{2}$-term of the classical Adams spectral represents an element $\theta_i: \mathbb S^{2^{i+1}-2}\to \mathbb S^{0}$ in the stable homotopy groups of spheres $\pi^{\mathbb S}_{2^{i+1}-2}.$ These elements $\theta_i$ for $0\leq i\leq 5$ are known to exist (see also Lin-Mahowald \cite{Lin-Mahowald}), but they do not exist when $i\geq 7$ (see the work of Hill, Hopkins, Ravenel \cite{Hill} and the discussion therein). So far the case $i = 6$ is not yet known. 

Singer \cite{W.S1} points out the non-trivial value of $Tr_d$ by proving that it is an isomorphism for $d\leq 2.$ In 1993, by using a basis consisting of the all the classes represented by certain polynomials in $\mathscr P_3$, Boardman indicated in \cite{J.B} that $Tr_3$ is also an isomorphism. Through these events, the Singer cohomological transfer can be viewed as an useful tool in the study of the $d$-th cohomology groups of the Steenrod algebra. Many mathematicians then investigated this transfer map (see Bruner et. al. \cite{B.H.H}, Ch\ohorn n and H\`a \cite{C.Ha}, Crossley \cite{M.C2}, H\`a \cite{L.H}, H\uhorn ng \cite{V.H}, Minami \cite{N.M}, Nam \cite{T.N2}, the present author \cite{D.P2, D.P3, D.P3-1, D.P5, D.P6, D.P7, D.P8, D.P9, D.P9-1, D.P10, D.P11, D.P12}, Sum \cite{N.S, N.S5, N.S6, N.S4} and others). In \cite{W.S1}, using the invariant theory, Singer claims that $Tr_d$ is an isomorphism for $d = 4$ in a range of internal degrees, but $Tr_5$ is not an epimorphism. Afterwards, he gave a hypothesis that {\it the transfer $Tr_d$ is a monomorphism for any positive integer $d$}. This prediction is currently still open for arbitrary $d > 4.$ 

For a non-negative integer $n,$ let $(\mathscr P_d)_n$ be the subspace of $\mathscr P_d$ consisting of all the homogeneous polynomials of degree $n$ in $\mathscr P_d.$  Denote by $(Q\mathscr P_d)_n$ the subspace of $Q\mathscr P_d$ consisting of all the classes represented by the homogeneous polynomials in $(\mathscr P_d)_n.$ One of the extremely useful tools for computing the hit problem and studying Singer's transfer is the Kameko squaring operation \cite{M.K} $(\widetilde {Sq^0_*})_{(d, 2n+d)}: (Q\mathscr P_d)_{2n+d} \to (Q\mathscr P_d)_n,$ which is an epimorphism of  $\mathbb Z/2(GL_d)$-modules. We refer to Sect.\ref{s2} for its precise meaning. Let $\mu(n)$ denote the smallest number $u$ such that $\alpha(n + u)\leq u,$ where $\alpha(k)$ is the number of $1$'s in the dyadic expansion of the positive integer $k.$ By Kameko \cite{M.K}, if $\mu(2n+d) = d,$ then $(\widetilde {Sq^0_*})_{(d, 2n+d)}$ is an $\mathbb Z/2(GL_d)$-module isomorphism.

\newpage
Recall that to solve the hit problem of Peterson, we will determine $Q\mathscr P_d$ in each degree $n\geq 1.$ However, as explicitly shown in \cite{D.P2}, it is enough to consider this space in the following "generic degree":
\begin{equation} \label{ct}n =  r(2^t-1) + 2^ts,
\end{equation}
whenever $r,\ t, \ s$ are non-negative integers such that $0\leq \mu(s)  <  r \leq d.$ Hit problems have been completely solved in \cite{M.K, F.P1, N.S1, N.S3} for $d\leq 4.$ For $r = d - 1$ and $s > 0,$ it was investigated by Crabb-Hubbuck~\cite{C.H}, Nam \cite{T.N1}, Repka-Selick \cite{R.S}, and Sum \cite{N.S3}. For $r = d-1$ and $s = 0,$ it is partially studied by Mothebe \cite{M.M} and by us \cite{P.S1, P.S2}. The case $r = d - 2 =3$ was probed by the present author \cite{D.P2} for $s = 1,$ and by Sum \cite{N.S7} for $s = 2^{m+u} + 2^{m} - 2,\, m\geq 0,\ u > 0,\ t\geq 6.$ The recent results when $r = d = 5$ were explicitly determined in \cite{N.S6, N.S4, N.T1} for $s\in \{2, 3, 5, 7, 10\}$, and by the present author \cite{D.P3} for $s = 6.$ The case $r = d = 5,\ s = 26$ and $t=0$ was studied by Walker-Wood \cite{W.W1}. The authors indicated in \cite{W.W1} that in any minimal generating set for the $\mathcal A_2$-module $\mathscr P_d,$ there are $2^{\binom{d}{2}}$ elements in degree $2^{d} - d - 1.$ For $d = 5,$ we see that $2^{5} - 6 = 26 = 5(2^{0} - 1) + 26.2^{0}$ and $\dim((Q\mathscr P_5)_{5(2^{0} - 1) + 26.2^{0}}) = 2^{\binom{5}{2}} = 1024.$ More generally, in generic degree of form \eqref{ct} for $d = r = 5,\ s = 26$ and $t\geq 0,$ we have $\mu(5(2^{t}-1) + 26.2^{t}) = 5$ for all $t > 0.$ This implies that the iterated Kameko map $((\widetilde {Sq^0_*})_{(5, 5(2^{t} - 1) + 26.2^{t})})^{t}:  (Q\mathscr P_5)_{5(2^{t} - 1) + 26.2^{t}} \to (Q\mathscr P_5)_{26}$ is an isomorphism for all $t\geq 0.$ So, $Q\mathscr P_5$ has dimension $1024$ in degree $5(2^{t}-1) + 26.2^{t}$ for any $t\geq 0.$ This event and the result for the case $d = 6,\ t\geq 5$ have also been studied in \cite{N.T2}. For $d = r = 5,\ s = 42,\ t\geq 0,$ and $d = 6,\ r = 5,\ s = 42,\ t\geq 5,$ we remark that since $\mu(5(2^{t} - 1) + 42.2^{t}) = 5$ for $t > 0,$ the iterated homomorphism $((\widetilde {Sq^0_*})_{(5, 5(2^{t} - 1) + 42.2^{t})})^{t}:  (Q\mathscr P_5)_{5(2^{t} - 1) + 42.2^{t}} \to (Q\mathscr P_5)_{42}$ is an isomorphism for arbitrary $t \geq 0.$ By using a computer program of Robert R. Bruner written in MAGMA,  we get $\dim(Q\mathscr P_5)_{5(2^{t} - 1) + 42.2^{t}} = 2520$ for all $t \geq 0.$  The direct proofs of this result will be published in detail elsewhere. On the other hand,  notice that $\mu(42) = 4$ and $\alpha(42 + \mu(42)) = 4 = \mu(42),$ by Sum \cite[Theorem 1.3]{N.S3}, it may be concluded that
$$ \dim((Q\mathscr P_6)_{5(2^{t} - 1) + 42.2^{t}}) = (2^{6}-1)\dim((Q\mathscr P_{5})_{42}) = 158760\, \mbox{ for all }\, t\geq 5.$$
It is currently difficult to solve hit problems in the general case.

In the present work, based on our works in \cite{P.S1, P.S2, D.P2, D.P3}, we continue our study of the hit problem of five variables in generic degree of \eqref{ct}. At the same time, by using these computational techniques, we examine Singer's algebraic transfer of rank $5$ in some internal degrees.  More precisely, we explicitly determine a basis of $Q\mathscr P_d$ for $d  = 5$ and generic degree of \eqref{ct} with $r = d = 5,\ s  = 8$ and $t$ an arbitrary non-negative integer. (A basis of this space is a set consisting of all the classes represent by the \textit{admissible monomials} of degree $n$ in $\mathscr P_d.$ We refer to Sect.\ref{s2} for the concept of the admissible monomial.) Using this result combining with the computations of  ${\rm Ext}_{\mathcal{A}_2}^{5,13.2^t}(\mathbb{Z}/2, \mathbb{Z}/2)$ (see Tangora \cite{M.T}, Chen \cite{T.C}, Lin \cite{W.L}), and a direct computation using a result in \cite{C.Ha} on the representation in the $\mathbb Z/2$-lambda algebra $\Lambda$ of the transfer homomorphism of rank $5,$ we show that $Tr_5$ is an isomorphism when acting on $\mathbb Z/2 \otimes_{GL_5}P_{\mathcal A_2}H_{13.2^t-5}(B(\mathbb Z/2)^{\times 5})$ for $t \in \{0,\, 1\}.$ (The information on the algebra $\Lambda$ can be found below in this section.) This gives a quite efficient method to study the isomorphism of the fifth transfer in some internal degrees of \eqref{ct}. Furthermore, our approach is different from that of Boardman \cite{J.B}, and of Singer \cite{W.S1}. The following is our first main result.

\begin{thm}\label{dl1} 
Let us consider the generic degree $13.2^t-5$ with $t$ an arbitrary non-negative integer. Then, we have
$$  \dim(Q\mathscr P_5)_{13.2^t-5} = \left\{\begin{array}{ll}
 174 &\mbox{if $t = 0,$}\\
 840&\mbox{if $t = 1,$}\\
 1894 &\mbox{if $t \geq 2.$}
\end{array}\right.
$$
\end{thm}

Note that $13.2^t-5 = 5(2^{t}-1) + 8.2^{t}.$ The theorem will be proved by determining explicitly an admissible monomial basis for $Q\mathscr P_5$ in generic degree $13.2^t-5.$ The first idea for the proof of the theorem is from Kameko's squaring operation. More clearly, since $$5(2^{t}-1) + 8.2^{t} = 2^{t+3} + 2^{t+2} + 2^{t-1} + 2^{t-2} + 2^{t-2}  - 5,$$ $\mu(13.2^t-5) = 5$ for all $t > 2,$ which leads to the iterated linear transformation $$(\widetilde {Sq^0_*})^{t-2}_{(5, 13.2^t-5)}: (Q\mathscr P_5)_{13.2^t-5} \to (Q\mathscr P_5)_{13.2^2-5}$$ being an isomorphism of $\mathbb Z/2(GL_5)$-modules for any $t \geq 2$ and therefore  we need only to  study the structure of $(Q\mathscr P_5)_{13.2^t-5}$ for $0\leq t \leq 2.$ The result when $t = 0$ was computed by T\'in \cite{N.T2}. We remark that for $t\in \{1, 2\},$ since $(\widetilde {Sq^0_*})_{(5, 13.2^t-5)}: (Q\mathscr P_5)_{13.2^t-5} \to (Q\mathscr P_5)_{13.2^{t-1}-5}$ is an epimorphism, we need only to determine the kernel of $(\widetilde {Sq^0_*})_{(5, 13.2^t-5)}.$ To study this space, we combine our recent results in \cite{P.S1} with previous results by Kameko \cite{M.K}, Mothebe \cite{M.M2, M.K.R}, Singer \cite{W.S2}, Sum \cite{N.S3}, and T\'in \cite{N.T2}. 

Recently, Sum \cite{N.S6} has proved some properties of $\mathcal A_2$-generators for $\mathscr P_d.$ Then, he made a conjecture on the relation between the admissible monomials for the polynomial algebras (see Sect.\ref{s3}). The conjecture helps us to reduce remarkably in computing the hit problem. From the results of Peterson \cite{F.P1}, Kameko \cite{M.K} and Sum \cite{N.S3}, this conjecture holds true for $d\leq 4.$ Sum proved in \cite{N.S6} that the conjecture is true in the case $d = 5$ and the degree $n$ of the form \eqref{ct} for $(r; s) = (5; 10)$ and $t\geq 0.$ Based upon the proof of Theorem \ref{dl1}, and previous results of the present author and Sum (see \cite{P.S1, P.S2, D.P2, D.P3, N.S7}), the conjecture also satisfies for $d = 5$ and in degrees of the form \eqref{ct} for $(r; s) = (4; 0),\ (5; 6),\ (5; 8)$ and $(3; s),$ for $s = 1,\ t \geq 0,$ and $s = 2^{m+u} + 2^{m} - 2,\, m\geq 0,\ u > 0,\ t\geq 6.$ 

It is well known that $\mathbb Z/2 \otimes_{GL_d}P_{\mathcal A_2}H_{n}(B(\mathbb Z/2)^{\times d})$ is dual to $(Q \mathscr P_d)^{GL_d}_n,$ the subspace of $Q \mathscr P_d$ generated all $GL_d$-invariants of degree $n.$ Computation of the $GL_d$-invariants is very difficult, particularly for values of $d$ as large as $d = 5.$ The understanding of special cases should be a helpful step toward the solution of the general problem. Now,  applying Theorem \ref{dl1} for $t =1,$ we get the following, which is our second main result.

\begin{thm}\label{dl2} 
There exists uniquely a non-zero class in $(Q\mathscr P_5)_{13.2^{1} - 5}$ invariant under the usual action of $GL_5.$ This implies that $(Q\mathscr P_5)_{13.2^{1} - 5}^{GL_5}$ is one-dimensional.
\end{thm}

One should note that Kameko's map $(\widetilde {Sq^0_*})_{(5, 13.2^{1} - 5)}:  (Q\mathscr P_5)_{13.2^{1} - 5} \to (Q\mathscr P_5)_{13.2^{0} - 5}$ is an epimorphism of $\mathbb Z/2(GL_5)$-modules. So, in order to prove Theorem \ref{dl2}, we describe the $\mathbb Z/2$-vector space structure of $(Q\mathscr P_5)^{GL_5}_{13.2^{0} - 5}.$ Using this and a monomial basis of $(QP_5)_{21}$ given in the proof of Theorem \ref{dl1}, we explicitly compute all $GL_5$-invariants of $Q\mathscr P_5$ in degree $13.2^{1} - 5.$ 

In addition to the Singer transfer mentioned above, the (mod 2) lambda algebra $\Lambda$ (see Bousfield et al. \cite{Bousfield}) is also one of  the important tools to study mod-2 cohomology groups ${\rm Ext}_{\mathcal{A}_2}^{d,d+*}(\mathbb{Z}/2, \mathbb{Z}/2).$ $\Lambda$ is defined as a differential, bigraded, associative algebra with unit over $\mathbb{Z}/2,$ is generated by $\lambda_i\in\Lambda^{1,i},$ satisfying the Adem relations 
\begin{equation}\label{ct2}
\lambda_i\lambda_{2i+d+1} = \sum_{j\geq 0}\binom{d-j-1}{j}\lambda_{i+d-j}\lambda_{2i+1+j}\ \ (i\geq 0,\ d\geq 0)
\end{equation}
and the differential 
\begin{equation}\label{ct3}
\partial (\lambda_{d-1}) = \sum_{j\geq 1}\binom{d-j-1}{j}\lambda_{d-j-1}\lambda_{j-1}\ \ (d\geq 1),
\end{equation}
where $\binom{d-j-1}{j}$ denotes the binomial coefficient reduced modulo $2.$ Furthermore, we have $$H^{d, *}(\Lambda, \partial) = {\rm Ext}_{\mathcal{A}_2}^{d, d+*}(\mathbb{Z}/2,\mathbb{Z}/2).$$ For non-negative integers $j_1, \ldots, j_d,$ a monomial  $\lambda_{j_1}\ldots\lambda_{j_d}\in \Lambda$ is called \textit{the monomial of the length $d$}. We shall write $\lambda_J,\ J = (j_1, \ldots, j_d)$ for $\prod_{1\leq k\leq d}\lambda_{j_k}$ and refer to $\ell(J) = d$ as the length of $J.$ It should be noted that the algebra $\Lambda$ is not commutative and that the bigrading of a monomial indexed by $J$ may be written $(d, n),$ where the homological degree $d,$ as above, is the length of $J,$ and $n = \sum_{1\leq k\leq d}j_k.$ A monomial $\lambda_J$ is called \textit{admissible} if $j_k\leq 2j_{k+1}$ for all $1\leq k \leq d-1.$ By the relations \eqref{ct2}, the $\mathbb Z/2$-vector subspace $$\Lambda^{d, *} = \langle \{\lambda_J| J = (j_1, \ldots, j_d), \ j_k\geq 0, 1\leq k\leq d\}\rangle$$ of $\Lambda$ has an additive basis consisting of all admissible monomials of the length $d.$ Recall that the dual of $\mathscr P_d$ is isomorphic to $\Gamma(a_1,\ldots,a_d),$ the divided power algebra generated by $a_1, \ldots, a_d,$ where $a_t = a_t^{(1)}$ is dual to $x_t$ with respect to the basis of $\mathscr P_d$ consisting of all monomials in $x_1, \ldots, x_d.$ In other words, $H_*(B(\mathbb{Z}/2)^{\times d}) = H_*((\mathbb{Z}/2)^{\times d})\cong \Gamma(a_1,\ldots,a_d)$. We note that the algebra $H_*(B(\mathbb{Z}/2)^{\times d})$ has a right $\mathcal A_2$-module structure. The right action of $\mathcal A_2$ on this algebra is given by $(a^{(j)}_t)Sq^{k} = \binom{j-k}{k}a^{(j-k)}_t = Sq_*^{k}(a^{(j)}_t)$ and Cartan's formula. (Note that $Sq_*^{k}$ denotes the dual of $Sq^{k}$.) In \cite{C.Ha}, Ch\ohorn n and H\`a have established a homomorphism $\psi_d: H_*(B(\mathbb{Z}/2)^{\times d})\longrightarrow \Lambda^{d, *},$ which is considered as a representation in the $\mathbb Z/2$-lambda algebra of Singer's transfer of rank $d$ and determined by the following inductive formula:
$$ \psi_d(a^{J}) = \left\{ \begin{array}{ll}
\lambda_{j_1}&\text{if}\; \ell(J) = 1, \\
\sum_{t\geq j_d}\psi_{d-1}(\prod_{1\leq k\leq d-1}a_{k}^{(j_{k})}Sq^{t-j_d})\lambda_t &\text{if}\; \ell(J) > 1,
\end{array} \right.$$
for any $a^{J} := \prod_{1\leq k\leq d}a_{k}^{(j_{k})}\in H_*(B(\mathbb{Z}/2)^{\times d})$ and $J := (j_1, j_2, \ldots, j_{d}).$ Note that $\psi_d$ is not an $\mathcal A_2$-homomorphism. The authors showed in \cite{C.Ha} that if $Z\in P_{\mathcal A_2}H_*(B(\mathbb{Z}/2)^{\times d}),$ then $\psi_d(Z)$ is a cycle in $\Lambda^{d, *}$ and $Tr_d([Z]) = [\psi_d(Z)]$. Applying this event and Theorem \ref{dl2} into the investigation of the Singer transfer of rank $5,$ we obtain the following theorem, which is our third main result.

\begin{thm}\label{dl3} 
The cohomological transfer 
$$ Tr_5: \mathbb Z/2 \otimes_{GL_5}P_{\mathcal A_2}H_{13.2^{1} - 5}(B(\mathbb Z/2)^{\times 5})\to {\rm Ext}_{\mathcal {A}_2}^{5,5+(13.2^{1} - 5)}(\mathbb Z/2,\mathbb Z/2)$$
is an isomorphism.
\end{thm}

As it is known, there exists an endomorphism $Sq^0$ of  the lambda algebra $\Lambda,$ determined by $Sq^0(\lambda_J = \prod_{1\leq k\leq d}\lambda_{j_k}) = \prod_{1\leq k\leq d}\lambda_{2j_k+1},$ where $\lambda_J$ is not necessarily admissible. It respects the relations in \eqref{ct2} and commutes with the differential $\partial$ in \eqref{ct3}. Then, $Sq^0$ induces the classical squaring operation in the Ext groups
$$Sq^0: H^{d, *}(\Lambda, \partial) = {\rm Ext}^{d, d+*}_{\mathcal {A}_2}(\mathbb Z/2, \mathbb Z/2) \to H^{d, d+2*}(\Lambda, \partial) = {\rm Ext}^{d, 2(d+*)}_{\mathcal {A}_2}(\mathbb Z/2, \mathbb Z/2).$$ This $Sq^0$ is not the identity map (see \cite{A.L}). 
As above mentioned, the structure of the groups ${\rm Ext}_{\mathcal A_2}^{d, d+*}(\mathbb Z/2, \mathbb Z/2)$ has been intensively studied by many authors, but remains very mysterious in general. In what follows, $(Sq^{0})^{t}: {\rm Ext}^{*, *}_{\mathcal {A}_2}(\mathbb Z/2, \mathbb Z/2) \to {\rm Ext}^{*, *}_{\mathcal {A}_2}(\mathbb Z/2, \mathbb Z/2)$ denotes the composite $Sq^{0}\ldots Sq^{0}$ ($t$ times of $Sq^{0}$)  if $t > 1,$ is $Sq^{0}$ if $t = 1,$ and  is the identity map if $t = 0.$ A family $\{a_t:\, t\geq 0\} \subset {\rm Ext}_{\mathcal A_2}^{d, d+*}(\mathbb Z/2, \mathbb Z/2)$ is called a $Sq^0$-family if $a_t = (Sq^0)^t(a_0)$ for $t\geq 0.$  We now return to the internal degree $13.2^{t}-5$ in Theorem \ref{dl1}. It has been shown (see Tangora \cite{M.T}, Lin \cite{W.L}, Chen \cite{T.C}) that
$${\rm Ext}_{\mathcal {A}_2}^{5, 5 + (13.2^t - 5)}(\mathbb Z/2,\mathbb Z/2)   = \left\{\begin{array}{ll}
0 &\mbox{if $t = 0$},\\ 
\langle h_{t+1}f_{t-1}\rangle &\mbox{if $t\geq 1,$}
 \end{array}\right.
$$ 
and that $h_{t+1}f_{t-1} = h_{t}g_{t}\neq 0,$  where $h_{t} = (Sq^0)^{t}(h_0)$ is the Adams element in ${\rm Ext}_{\mathcal A_2}^{1, 2^{t}}(\mathbb Z/2, \mathbb Z/2),$ $g_t = (Sq^0)^{t-1}(g_1)$ and $f_{t-1} = (Sq^0)^{t-1}(f_0)$ are the elements non-zero in ${\rm Ext}_{\mathcal {A}_2}^{4, 12.2^t}(\mathbb Z/2, \mathbb Z/2)$ and ${\rm Ext}_{\mathcal {A}_2}^{4, 11.2^t}(\mathbb Z/2, \mathbb Z/2),$ respectively, for any $t\geq 1.$ (Note that by Lin \cite{W.L}, the groups ${\rm Ext}_{\mathcal {A}_2}^{4, 4 + *}(\mathbb Z/2,\mathbb Z/2)  $ contains seven $Sq^{0}$-families of indecomposable elements, namely $$\{d_t\},\ \{e_t\},\ \{f_t\},\ \{g_{t+1}\},\ \{p_t\},\ \{D_3(t)\},\ \{p'_t\}\ (t\geq 0)).$$ As well known, Singer \cite{W.S1} showed that the transfer $Tr_1$ detects the family $\{h_t|\, t\geq 1\}$ and that $\bigoplus_{d\geq 0}Tr_d$ is an algebra homomorphism. Following Nam \cite{T.N2}, the family $\{f_{t-1}|\, t\geq 1\}$ was detected by $Tr_4.$ These data imply that $h_{t+1}f_{t-1}$ is in the image of $Tr_5$ for all $t\geq 1.$ In Sect.\ref{s5}, we give another direct proof of this event for the case $t=1.$ More specifically, we proved $h_2f_0\in {\rm Im}(Tr_5)$ by using Theorem \ref{dl2} and a representation of $Tr_5$ over the lambda algebra.

As above shown, to prove Theorem \ref{dl2}, we need to determine all $GL_5$-invariants of $(Q\mathscr P_5)_{13.2^{0} - 5}.$ Applying Theorem \ref{dl1} for $t = 0$ with a basis of  $Q\mathscr P_5$ in degree $13.2^{0} - 5$ (see \cite{N.T2}), we showed that $(Q\mathscr P_5)_{13.2^{0} - 5}^{GL_5}$ is zero (see Theorem \ref{bb8} in Sect.\ref{s4}). This result together with a fact of the fifth cohomology group ${\rm Ext}_{\mathcal {A}_2}^{5, 5 + (13.2^0 - 5)}(\mathbb Z/2,\mathbb Z/2),$ it may be concluded that $Tr_5$ is a trivial isomorphism when acting on the space $\mathbb Z/2 \otimes_{GL_5}P_{\mathcal A_2}H_{13.2^{0} - 5}(B(\mathbb Z/2)^{\times 5}).$ As an immediate consequence from this and Theorem \ref{dl3}, we get

\begin{corl}
Singer's conjecture for $Tr_5$ holds in the bidegrees $(5,5+8)$ and $(5, 5+21).$ 
\end{corl}

To end this introduction, we will discuss whether $Tr_5$ is an isomorphism or not in the bidegree $(5, 5 + (13.2^t - 5))$ for $t\geq 2.$ Since the iterated Kameko homomorphism $$((\widetilde {Sq^0_*})_{(5, 13.2^t-5)})^{t-2}:  (Q\mathscr P_5)_{13.2^t-5} \to (Q\mathscr P_5)_{13.2^2-5}$$ is an $GL_5$-module isomorphism for all $t\geq 2,$  to examine Singer's conjecture for $Tr_5$ in the above bidegree, we need only to determine all $GL_5$-invariants of $(Q\mathscr P_5)_{13.2^t-5}$ for $t = 2.$ Recall that Kameko's map $$(\widetilde {Sq^0_*})_{(5, 13.2^2-5)}:  (Q\mathscr P_5)_{13.2^2-5} \to (Q\mathscr P_5)_{13.2^1-5}$$ is an epimorphism of $GL_5$-modules and that the element $h_3f_1\in {\rm Ext}_{\mathcal {A}_2}^{5, 5 + (13.2^2-5)}(\mathbb Z/2,\mathbb Z/2) $ is in the image of $Tr_5.$ So, by Theorem \ref{dl2}, we deduce $$ 1\leq \dim\big(\mathbb Z/2 \otimes_{GL_5}P_{\mathcal A_2}H_{13.2^2-5}(B(\mathbb Z/2)^{\times 5})\big) \leq \dim ({\rm Ker}(\widetilde {Sq^0_*})_{(5,13.2^2-5)})^{GL_5} +  1.$$ Furthermore, all elements of $\mathbb Z/2 \otimes_{GL_5}P_{\mathcal A_2}H_{13.2^2-5}(B(\mathbb Z/2)^{\times 5})$ are of the form $$(\gamma[\varphi(u_0)] + [v])^*,$$ where $\gamma\in \mathbb Z/2,$ and the mapping $\varphi: \mathscr P_5\to \mathscr P_5$ determined by setting $\varphi(u) = x_1x_2x_3x_4x_5u^2$ for any $u\in \mathscr P_5, \ v \in (\mathscr P_5)_{13.2^2-5}$ such that $[v]$ belongs to $\text{Ker}(\widetilde {Sq^0_*})_{(5, 13.2^2-5)},$ and $u_0 \in (\mathscr P_5)_{13.2^1-5}.$ Based on Theorem \ref{dl2}, $[u_0]$ is the only non-zero element in $(Q\mathscr P_5)_{13.2^1-5}^{GL_5}.$ Direct calculating the elements $(\gamma[\varphi(u_0)] + [v])^*$ is a hard work. By using techniques of the hit problem of five variables, we will describe explicitly all these elements in the near future. From these data with the fact that $h_{t+1}f_{t-1}\in {\rm Im}(Tr_5)$ for $t\geq 1,$ we conclude that if $({\rm Ker}(\widetilde {Sq^0_*})_{(5,13.2^2-5)})^{GL_5}$ is zero, then $Tr_5$ is an isomorphism in the bidegree $(5, 5 +  (13.2^t - 5))$ for every $t\geq 2.$ This means that Singer's conjecture for $Tr_5$ also satisfies in this bidegree. However, it will be much more interesting if $({\rm Ker}(\widetilde {Sq^0_*})_{(5,13.2^2-5)})^{GL_5}$ is non-trivial and $\dim(Q\mathscr P_5)_{13.2^2-5}^{GL_5} \neq 1.$

The structure of the paper is as follows. First, some background is reviewed in Sect.\ref{s2}. In the next section, we present Singer's criterion on $\mathcal A_2$-decomposable and Sum's conjecture related to the minimal set of generators for $\mathcal A_2$-modules $\mathscr P_d.$ Then, the $\mathcal A_2$-generators for $\mathscr P_5$ in degree $13.2^t-5$ are described explicitly by proving Theorem \ref{dl1}. In Sect.\ref{s4}, we prove Theorem \ref{dl2} by using the admissible monomial bases of $(Q\mathscr P_5)_{13.2^0-5}$ and $(Q\mathscr P_5)_{13.2^1-5}.$ Based upon Theorem \ref{dl2} and a representation in the lambda algebra of the fifth Singer transfer, the proof of Theorem \ref{dl3} is presented in Sect.\ref{s5}.   All the admissible monomials of degree $13.2^t-5$ in $\mathscr P_5$ are described in the Appendix.

\begin{acknow}
This research is supported financially by the National Foundation for Science and Technology Development (NAFOSTED) of Viet Nam under Grant No. 101.04-2017.05. 

\medskip
The author is very grateful to the anonymous referees for the careful reading of the manuscript and for their insightful comments and detailed suggestions, which have led me to improve considerably this work.

\medskip

I would also like to thank Robert Bruner, Phan Ho\`ang Ch\ohorn n, Mbakiso Fix Mothebe for illuminating conversations on the hit problem and the Ext groups, Hadi Zare for sending a copy of \cite{Zare}, and L\ecircumflex~Minh H\`a for pointing out an error in the original manuscript. 


\medskip

Finally, the author would like to give my deepest sincere thanks to Professor Nguy\~\ecircumflex n Sum for many helpful discussions. 
\end{acknow}

\section{Preliminaries}\label{s2}
This section starts with a recollection of the Kameko squaring operation and some information related to the Peterson hit problem.

\subsection{Kameko's squaring operation}
Recall that the polynomial algebra $\mathscr P_d = \mathbb Z/2[x_1, \ldots, x_d]$ is an unstable left module on the ring $\mathcal A_2.$  Let $GL_d:=GL(d, \mathbb Z/2)$ denote the general linear group of rank $d$ over the field $\mathbb Z/2.$ An usual right action of this group on $\mathscr P_d$ is given by $((f)w)(x_1,x_2,\ldots,x_d) = f((x_1)w, (x_2)w,\ldots,(x_d)w),$ where $w = (w_{ij})\in GL_d$ and $(x_j)w = \sum_{1\leq i\leq d}x_iw_{ij},\,\,1\leq j\leq d.$ Thus, $\mathscr P_d$ (resp. $(\mathscr P_d)^{*}$) has also a right (resp. left) $GL_d$-module structure. Furthermore, since the two actions of $\mathcal A_2$ and $GL_d$ upon $\mathscr P_d$ (resp. $(\mathscr P_d)^{*}$) commute with each other, there is an inherited action of $GL_d$ on $Q\mathscr P_d$ (resp. $(Q\mathscr P_d)^{*} = P_{\mathcal A_2}H_*(B(\mathbb Z/2)^{\times d})$).

We knew that the homological algebra $\{H_n(B(\mathbb{Z}/2)^{\times d})| n\geq 0\}$ is dual to $\mathscr P_d.$ Moreover, it is isomorphic to $\Gamma(a_1, \ldots, a_d),$ the divided power algebra generated by $a_1, \ldots, a_d,$ each of degree one, where $a_j = a_j^{(1)}$ is dual to $x_{j}.$ Here the duality is taken with respect to the basis of $\mathscr P_d$ consisting of all monomials in $x_1, \ldots, x_d.$ We now denote by $P_{\mathcal A_2}H_n(B(\mathbb Z/2)^{\times d})$  the primitive subspace consisting of all elements in $H_n(B(\mathbb Z/2)^{\times d})$, which are annihilated by every Steenrod's operation $Sq^{k},\ k>0.$ So, it is dual to $(Q\mathscr P_d)_n.$ By Kameko \cite{M.K}, we have the monomorphism
$$ \begin{array}{ll}
\overline{Sq}^0: P_{\mathcal A_2}H_{n}(B(\mathbb Z/2)^{\times d}) &\longrightarrow P_{\mathcal A_2}H_{d+2n}(B(\mathbb Z/2)^{\times d})\\
\ \ \ \ \ \ \ \ \ \ \ \prod_{1\leq t\leq d}a_t^{(s_t)}&\longmapsto \prod_{1\leq t\leq d}a_t^{(2s_t+1)}
\end{array}$$
where $\prod_{1\leq t\leq d}a_t^{(s_t)}$ is dual to $\prod_{1\leq t\leq d}x_t^{s_t}.$ Further, $Sq_*^{2k+1}\overline{Sq}^0 = 0,$ and $Sq_*^{2k}\overline{Sq}^0 = \overline{Sq}^0Sq_*^{k}$ for any $k \geq 0,$ where $Sq_*^{k}$ denotes the dual Steenrod operation. Note that $\overline{Sq}^0$ is also an $GL_d$-module homomorphism (see \cite{B.H.H}, \cite{V.H}). Then, $\overline{Sq}^0$ induces Kameko's squaring operation in the dual of the spaces $(Q\mathscr P_d)_*^{GL_d}$: 
$$ \widetilde {Sq^0} = id_{\mathbb Z/2}\otimes_{GL_d} \overline{Sq}^0: \mathbb Z/2 \otimes_{GL_d}P_{\mathcal A_2}H_{n}(B(\mathbb Z/2)^{\times d})\to \mathbb Z/2 \otimes_{GL_d}P_{\mathcal A_2}H_{d+2n}(B(\mathbb Z/2)^{\times d}).$$
This $\widetilde {Sq^0}$ commutes with the classical squaring operation $$Sq^0: {\rm Ext}^{d, d+n}_{\mathcal {A}_2}(\mathbb Z/2, \mathbb Z/2)\to {\rm Ext}^{d, 2d+2n}_{\mathcal {A}_2}(\mathbb Z/2, \mathbb Z/2)$$ through the $d$-th Singer transfer (see \cite{A.C.H}, \cite{J.B}, \cite{N.M}). In other words, the following diagram is commutative:
$$ \begin{diagram}
\node{\mathbb Z/2 \otimes_{GL_d}P_{\mathcal A_2}H_{n}(B(\mathbb Z/2)^{\times d})} \arrow{e,t}{Tr_d}\arrow{s,r}{\widetilde{Sq^0}} 
\node{{\rm Ext}_{\mathcal A_2}^{d, d+n}(\mathbb Z/2, \mathbb Z/2)} \arrow{s,r}{Sq^0}\\ 
\node{\mathbb Z/2 \otimes_{GL_d}P_{\mathcal A_2}H_{d+2n}(B(\mathbb Z/2)^{\times d})} \arrow{e,t}{Tr_d} \node{{\rm Ext}_{\mathcal A_2}^{d, 2d+2n}(\mathbb Z/2, \mathbb Z/2).}
\end{diagram}$$
The dual homomorphism $\widetilde {Sq^0_*}: (Q\mathscr P_d)^{GL_d}_{d+2n}\to (Q\mathscr P_d)^{GL_d}_{n}$ of $\widetilde {Sq^0}$ is induced by the homomorphism $(\widetilde {Sq^0_*})_{(d, d+2n)}: (Q\mathscr P_d)_{d+2n} \to (Q\mathscr P_d)_n.$ The latter is given by the $\mathbb Z/2$-linear map 
$$ \begin{array}{ll}
\delta: (\mathscr P_d)_{d+2n}&\longrightarrow (\mathscr P_d)_n\\
x_1^{t_1}x_2^{t_2}\ldots x_d^{t_d}&\longmapsto \left\{\begin{array}{ll}
x_1^{\frac{t_1-1}{2}}x_2^{\frac{t_2-1}{2}}\ldots x_d^{\frac{t_d-1}{2}}& \text{if $t_1, \ldots, t_d$ odd},\\
0& \text{otherwise}.
\end{array}\right.
\end{array}$$ 
 
Denote by $\alpha(n)$ the number of $1$'s in dyadic expansion of $n.$ We consider the arithmetic function (see \cite{M.K}, \cite{N.S4}):
$$ \begin{array}{ll}
\mu(n) &= \mbox{min}\big\{u\in \mathbb N:\ \alpha(n + u)\leq u\big\}\\
& = {\rm min}\{u\in\mathbb N:\,\,n = \sum_{1\leqslant j\leqslant u}(2^{s_j}-1),\,\,s_j> 0,\,\,1\leqslant j\leqslant u\}.
\end{array}$$
From the above data, $(\widetilde {Sq^0_*})_{(d, d+2n)}$ is an $\mathbb Z/2(GL_d)$-module epimorphism. However, in particular, if $\mu(d+2n) = d$ then it is an isomorphism. According to H\uhorn ng \cite[Theorem 1.5]{V.H}, if Singer's conjecture for the $d$-th algebraic transfer is true, then $Tr_d$ does not detect the non-zero elements $u\in {\rm Ext}^{d, d+n}_{\mathcal {A}_2}(\mathbb Z/2, \mathbb Z/2)$ such that $Sq^0(u) = 0$ and $\mu(2n+d) = d.$ In this case, $u$ is called \textit{critical}. This leads us to the study of the kernel of $\widetilde {Sq^0}.$ Recall that $\overline{Sq}^0$ is a monomorphism, but the squaring operation $\widetilde {Sq^0}= id_{\mathbb Z/2}\otimes_{GL_d} \overline{Sq}^0$ is not a monomorphism in general. Indeed, by using a computer calculation, H\uhorn ng provided a counter-example in \cite{V.H} that $\widetilde {Sq^0}$ is not a monomorphism when acting on $\mathbb Z/2 \otimes_{GL_5}P_{\mathcal A_2}H_{15}(B(\mathbb Z/2)^{\times 5}).$ This was confirmed again by the works of Sum \cite{N.S, N.S4}. Thereafter, H\uhorn ng \cite{V.H} conjectured that \textit{$\widetilde {Sq^0}$ is a monomorphism if and only if $d\leq 4.$} By Boardman \cite{J.B}, and Singer \cite{W.S1}, the conjecture satisfies for $d\leq 3.$ We hope that it can be verified for $d = 4$ by using the dual of $\widetilde {Sq^0}$ and the results on the hit problem in \cite{N.S6}.

Thus, to verify Singer's conjecture, in addition to the techniques of the hit problem mentioned in this paper, we can use the relationship between the algebraic transfer and critical elements. However, finding critical elements is difficult.

\subsection{On the hit problem of Peterson}

To study the hit problem, we need some relevant notations and concepts. For a natural number $n,$ denote by $\alpha_t(n)$ the $t$-th coefficients in dyadic expansion of $n.$ This means $\alpha(n) = \sum_{t\geq 0}\alpha_t(n).$
Further, $n$ can be represented as follows: $n = \sum_{t\geq 0}\alpha_t(n)2^t,$ where $\alpha_t(n)\in \{0, 1\},\ t = 0, 1, \ldots.$ Consider the monomial $x = x_1^{u_1}x_2^{u_2}\ldots x_d^{u_d}\in \mathscr P_d,$ we define two sequences associated with $x$ by
$\omega(x) :=(\sum_{1\leq j\leq d}\alpha_{0}(u_j),\ \sum_{1\leq j\leq d}\alpha_{1}(u_j),\ldots,\ \sum_{1\leq j\leq d}\alpha_{t-1}(u_j),\ldots)$ and $(u_1, u_2, \ldots, u_d),$ which are called the {\it weight vector} and the \textit{exponent vector} of $x,$ respectively. From now on, we shall write $\omega_t(x)$ for $\sum_{1\leq j\leq d}\alpha_{t-1}(u_j),\ t = 1, 2, \ldots$

Let $\omega = (\omega_1, \omega_2, \ldots, \omega_t,\ldots)$ be a sequence of non-negative integers. Then,  the sequence $\omega$ are called \textit{the weight vector} if $\omega_t  = 0$ for $t\gg 0.$ We define $\deg(\omega) = \sum_{t\geq 1}2^{t-1}\omega_t.$ The sets of all the weight vectors and the exponent vectors are given the left lexicographical order.

Recall that a homogeneous element $f\in\mathscr P_d$ is called \textit{$\mathcal A_2$-decomposable} (or "\textit{hit}") if it is in the image of positive degree elements of $\mathcal A_2.$ This means that $f$ belongs to $\mathcal A_2^+\mathscr P_d.$ 

{\bf The equivalence relations on $\mathscr P_d$} (see \cite{M.K, D.P2}). For a weight vector $\omega,$ we denote two subspaces associated with $\omega$: $$\mathscr P_d(\omega) = \langle\{ x\in\mathscr P_d|\, \deg(x) = \deg(\omega),\  \omega(x)\leq \omega\}\rangle,$$ $$\mathscr P^{-}_d(\omega) = \langle \{ x\in\mathscr P_d|\, \deg(x) = \deg(\omega),\  \omega(x) < \omega\}\rangle.$$ Let us now consider the homogeneous polynomials $f,$ and $g$ in $\mathscr P_d$ with $\deg(f) = \deg(g).$ We define the following binary relations "$\equiv$" and "$\equiv_{\omega}$" on $\mathscr P_d$:    
\begin{enumerate}
\item [(i)]$f \equiv g $ if and only if $f = g\ {\rm modulo}\, (\mathcal{A}_2^+\mathscr P_d).$ Specifically, if $f \equiv 0$ then $f$ is $\mathcal A_2$-decomposable.
\item[(ii)] $f \equiv_{\omega} g$ if and only if $f, g\in \mathscr P_d(\omega)$ and $f  = g\ {\rm modulo}\, ((\mathcal{A}_2^+\mathscr P_d \cap \mathscr P_d(\omega))+ \mathscr P_d^-(\omega)).$
\end{enumerate}
It is easily seen that these binary relations are equivalence ones. Let $Q\mathscr P_d(\omega)$ denote the quotient of $\mathscr P_d(\omega)$ by the equivalence relation "$\equiv_\omega$". Then, we have the $\mathbb Z/2$-quotient space
$$Q\mathscr P_d(\omega) = \mathscr P_d(\omega)/((\mathcal{A}_2^+\mathscr P_d\cap \mathscr P_d(\omega)) + \mathscr P_d^-(\omega)).$$
By Sum \cite{N.S4}, $Q\mathscr P_d(\omega)$ is also an $GL_d$-module. The following events are shown in \cite{D.P2}. However, to make the paper self-contained, we will present again them in detail. 
$$ \dim((Q\mathscr P_d)_n) = \sum\limits_{{\rm deg}\,\omega = n}\dim(Q\mathscr P_d(\omega)),$$
$$  \dim((Q\mathscr P_d)^{GL_d}_n)\leq \sum_{\deg(\omega) = n}\dim(Q\mathscr P_d(\omega)^{GL_d}).$$
Indeed, by Walker and Wood \cite{W.W3}, we have a filtration of $Q\mathscr P_d$:
$$ \{0\}\subseteq \cdots \subseteq \mathscr P_d^{-}(\omega)/((\mathcal A_2^{+}\mathscr P_d) \cap \mathscr P_d^{-}(\omega)) \subseteq \mathscr P_d(\omega)/((\mathcal A_2^{+}\mathscr P_d) \cap \mathscr P_d(\omega)) \subseteq \cdots \subseteq \mathscr P_d/(\mathcal A_2^{+}\mathscr P_d) = Q\mathscr P_d.$$
This is not only a filtration of $Q\mathscr P_d$ as a vector space, but also as a $GL_d$-module. The inclusion of $\mathscr P_d^{-}(\omega)$ into $\mathscr P_d(\omega)$ induces the monomorphism $$\mathscr P_d^{-}(\omega)/((\mathcal A_2^{+}\mathscr P_d) \cap\mathscr P_d^{-}(\omega)) \to  \mathscr P_d(\omega)/((\mathcal A_2^{+}\mathscr P_d) \cap\mathscr P_d(\omega))$$ and the  following sequence is short exact:
$$ \begin{array}{ll}
0\to\mathscr  P_d^{-}(\omega)/((\mathcal A_2^{+}\mathscr P_d) \cap\mathscr P_d^{-}(\omega)) & \to\mathscr P_d(\omega)/((\mathcal A_2^{+}\mathscr P_d) \cap \mathscr P_d(\omega))\to \\
&\hspace{2cm}\to\mathscr P_d(\omega)/(((\mathcal{A}_2^+\mathscr P_d)\cap\mathscr P_d(\omega)) + \mathscr P_d^-(\omega))\to 0.
\end{array}$$
From this, we get  $$Q\mathscr P_d(\omega) \cong (\mathscr P_d(\omega)/((\mathcal A_2^{+}\mathscr P_d) \cap \mathscr P_d(\omega)))/(\mathscr P_d^{-}(\omega)/((\mathcal A_2^{+}\mathscr P_d) \cap\mathscr P_d^{-}(\omega))).$$
This isomorphism is also an isomorphism of $GL_d$-modules. Combining these with the filtration of $Q\mathscr P_d$, we have immediate the above claims.

{\bf The linear order on $\mathscr P_d$} (see \cite{M.K}).  Let $u = x_1^{a_1}x_2^{a_2}\ldots x_d^{a_d}$ and $v = x_1^{b_1}x_2^{b_2}\ldots x_d^{b_d}$ be monomials of the same degree in $\mathscr P_d.$ We write $a,\ b$ for the exponent vectors of $u, \ v,$ respectively. We say that $a < b$ if there is a positive integer $m$ such that $a_j = b_j$ for all $j < m$ and $a_m < b_m,$ and that $u  < v$ if and only if one of the following holds:
\begin{enumerate}
\item[(i)] $\omega(u) < \omega(v);$
\item[(ii)] $\omega(u) = \omega(v)$ and $a < b.$
\end{enumerate}

{\bf The inadmissible monomial} (see \cite{M.K}). We say that a monomial $u\in\mathscr P_d$ is {\it inadmissible}, if there exist monomials $x_1, x_2,\ldots, x_k$ such that $x_j < u$ for $1\leq j\leq k$ and $u \equiv  \sum_{1\leq i\leq k}x_i.$ Then, $u$ is said to be {\it admissible}, if it is not inadmissible.

Obviously, the set of all the admissible monomials of degree $n$ in $\mathscr P_d$ is \textit{a minimal set of $\mathcal {A}_2$-generators for $\mathscr P_d$ in degree $n.$} So, $(Q\mathscr P_d)_n$ is an $\mathbb Z/2$-vector space with a basis consisting of all the classes represent by the admissible monomials of degree $n$ in $\mathscr P_d.$

{\bf The strictly inadmissible monomial} (see \cite{M.K}). A monomial $u\in\mathscr P_d$ is said to be {\it strictly inadmissible} if and only if there exist monomials $x_1, x_2,\ldots, x_k$ such that $x_j < u$ for $1\leq j \leq k$ and $ u = \sum_{1\leq j\leq k}x_j + \sum_{1\leq  t\leq  2^s - 1}Sq^{t}(y_t),$ where $s = {\rm max}\{i\in\mathbb Z: \omega_i(u) > 0\}$ and suitable polynomials $y_t\in \mathscr P_d.$

Note that every the strictly inadmissible monomial is inadmissible but the converses not generally true (see a counter-example in \cite{D.P2}). The following result is used to study $Q\mathscr P_5$ in the next section.

\begin{theo}[see \cite{M.K}]\label{dlKS}
Let $x, y$ and $u$ be monomials in $\mathscr P_d$ such that $\omega_i(x) = 0$ for $i > r >0, \omega_t(u)\neq 0$ and $\omega_i(u)=0$ for $i  > t > 0.$ 
Then, if $u$ is inadmissible, then $xu^{2^r}$ is also inadmissible. Furthermore, if $u$ is strictly inadmissible, then $uy^{2^t}$ is also strictly inadmissible.
\end{theo}

Let $\mathscr P_d^0$ and $\mathscr P_d^+$ denote the $\mathcal A_2$-submodules of $\mathscr P_d$ spanned all the monomials $x_1^{t_1}x_2^{t_2}\ldots x_d^{t_d}$ such that $t_1t_2\ldots t_d = 0,$ and $t_1t_2\ldots t_d > 0,$ respectively. Denote by $Q\mathscr P_d^0:= \mathbb Z/2\otimes_{\mathcal A_2} \mathscr P_d^0,$ and by $Q \mathscr P_d^+:= \mathbb Z/2\otimes_{\mathcal A_2} \mathscr P_d^+.$  Then, we can see that $ Q\mathscr P_d = Q\mathscr P_d^0\,\bigoplus\, Q\mathscr P_d^+.$ We end this section by establishing a formula below on the dimension of $Q\mathscr P_d^0$ in degree $n$, which will be used in the next section. Note that this formula is similar to the one of \cite{M.K.R}. 

Let $\mathcal I= (j_1, j_2, \ldots, j_r),$ where $1 \leq j_1 < \ldots < j_r \leq d$,\, $1 \leq r \leq d-1,$ and let $r := \ell(\mathcal I)$ be the length of $\mathcal I.$ We denote $\mathscr P_{\mathcal I} = \langle \{x_{j_1}^{t_1}x_{j_2}^{t_2}\ldots x_{j_r}^{t_r}\;|\; t_s \in \mathbb N,\, s = 1,2, \ldots, r\}\rangle \subset \mathscr P_d$. Then, 
$\mathscr P_{\mathcal I}$ is $\mathcal {A}_2$-submodule of $\mathscr P_d.$ Moreover, it is isomorphic to $\mathscr P_r.$ Straightforward calculations indicate that
$$ Q\mathscr P_d^0  = \bigoplus_{1\leq r\leq d-1}\bigoplus_{\ell(\mathcal I) = r}Q\mathscr P_{\mathcal I}^+,$$
where $\mathscr P_{\mathcal I}^+=\langle\{x_{j_1}^{t_1}x_{j_2}^{t_2}\ldots x_{j_r}^{t_r}\in \mathscr P_{\mathcal I}\;|\; t_1t_2\ldots t_r > 0,\, 1\leq r\leq d-1\}\rangle.$ It is easily seen that $\dim (Q\mathscr P_{\mathcal I}^+)_n = \dim (Q\mathscr P_r^+)_n$ for all $n,$ and that $\binom dr$ is the number of the sequences $\mathcal I$ of length $r.$ Therefore, we get
$$ \dim(Q\mathscr P_d^0)_n = \sum_{1\leq  r\leq d-1}\binom{d}{r}\dim(Q\mathscr P_r^+)_n.$$  

\section{Generators of  the $\mathcal A_2$-module $\mathscr P_5$ in the generic degree $5(2^{t}-1) + 8.2^{t}$}\label{s3}

In this section we study the structure of $Q\mathscr P_5$ in degree $5(2^{t}-1) + 8.2^{t}$ for $t$ a positive integer. More explicitly, we will prove Theorem \ref{dl1} as given at the beginning. We first review some homomorphisms and related results, Sum's conjecture \cite{N.S6} and Singer's criterion on $\mathcal A_2$-decomposable \cite{W.S1}.

\subsection{Singer's criterion on $\mathcal A_2$-decomposable}\label{s3.3}

\begin{dn}\label{spi}
A monomial $z = x_1^{t_1}x_2^{t_2}\ldots x_d^{t_d}$ in $\mathscr P_d$ is called a {\it spike} if $t_j = 2^{a_j} - 1$ for $a_j$ a non-negative integer and $1\leq j \leq d.$ If $z$ is a spike with $a_1 > a_2 > \ldots > a_{r-1}\geq a_r > 0$ and $a_j = 0$ for $j  > r,$ then it is called a {\it minimal spike}.
\end{dn}

\begin{mde}[see \cite{P.S1, D.P2}]\label{PS}
All the spikes in $\mathscr P_d$ are admissible and their weight vectors are weakly decreasing. Furthermore, if a weight vector $\omega = (\omega_1, \omega_2, \ldots)$ is weakly decreasing and $\omega_1\leq d,$ then there is a spike $z$ in $\mathscr P_d$ such that $\omega(z) = \omega.$
\end{mde}

We refer the reader to \cite{D.P2} for the detailed proofs of the proposition. Singer showed in \cite{W.S2} that if $\mu(n)\leq d,$ then there exists uniquely a minimal spike of degree $n$ in $\mathscr P_d.$ Further,  we have the following, which is one of the important keys for examining the hit monomials in generic degrees.

\begin{theo}[Singer~\cite{W.S2}]\label{dlsig}
Suppose that $X\in \mathscr P_d$ is a monomial of degree $n,$ where $\mu(n)\leq d.$ Let $z$ be the minimal spike of degree $n$ in $\mathscr P_d.$ If $\omega(X) < \omega(z),$ then $X$ is $\mathcal A_2$-decomposable.
\end{theo}

\subsection{Some homomorphisms and Sum's conjecture}\label{s3.1}

For $1\leq k\leq d,$ we define the map $\rho_{(k,\,d)}: \mathscr P_{d-1}\to  \mathscr P_d$ of $\mathbb Z/2$-algebras by setting
$$ \rho_{(k,\,d)}(x_j) = \left\{ \begin{array}{ll}
{x_j}&\text{if }\;1\leq j < k, \\
x_{j+1}& \text{if}\; k\leq j < d.
\end{array} \right.$$

We consider the following set
$$ \mathcal N_d := \{(k; \mathscr K)\;|\; \mathscr K = (k_1,k_2,\ldots,k_r), 1\leqslant k < k_1< k_2 < \ldots < k_r\leq d, \ 0\leq r < d\},$$  
where by convention, $\mathscr K = \emptyset,$ if $r = 0.$ Denote by $r = \ell(\mathscr K)$ the length of $\mathscr K.$ For any $(k; \mathscr K)\in\mathcal{N}_d,$ we have the projection (see \cite{N.S3}) $\pi_{(k; \mathscr K)}: \mathscr P_d\to \mathscr P_{d-1},$ which is determined by 
$$\pi_{(k; \mathscr K)}(x_j) = \left\{ \begin{array}{ll}
{x_j}&\text{if }\;1\leq j < k, \\
\sum_{p\in \mathscr K}x_{p-1}& \text{if}\; j = k,\\
x_{j-1}&\text{if}\; k < j \leq d.
\end{array} \right.$$
Note that $\rho_{(k,\,d)}$ and $\pi_{(k; \mathscr K)}$ are also the homomorphisms of the $\mathcal {A}_2$-modules. In particular, we have $\pi_{(k; \emptyset)}(x_k) = 0$ for $1\leq k\leq d$ and $\pi_{(k; \mathscr K)}(\rho_{(k,\,d)}(u)) = u$ for any $u\in \mathscr P_{d-1}.$

\begin{mde}[see \cite{P.S1}]\label{mdPS}
If $x$ is a monomial in $\mathscr P_d,$ then $\pi_{(k; \mathscr K)}(x)\in \mathscr P_{d-1}(\omega(x)).$
\end{mde}
This result implies that if $\omega$ is a weight vector and $x\in \mathscr P_d(\omega(x)),$ then $\pi_{(k; \mathscr K)}(x)\in \mathscr P_{d-1}(\omega).$ Furthermore, $\pi_{(k; \mathscr K)}$ passes to a homomorphism from $Q\mathscr P_d(\omega)$ to $Q\mathscr P_{d-1}(\omega).$

Let  $(k; \mathscr K)\in\mathcal{N}_d,\ 0 < r < d,$ and let $ x_{(\mathscr K,\,u)} = x_{k_u}^{2^{r-1} + 2^{r-2} +\, \cdots\, + 2^{r-u}}\prod_{u < m\leq r}x_{k_m}^{2^{r-m}}$ for $1\leq u\leq r,\; x_{(\emptyset, 1)} = 1.$  In \cite{N.S3}, Sum has defined an $\mathbb Z/2$-linear transformation $\phi_{(k; \mathscr K)}: \mathscr P_{d-1}\to \mathscr P_d,$ which is determined by
 $$ \phi_{(k; \mathscr K)}(x) = \left\{ \begin{array}{ll}
\rho_{(k,\,d)}(x) & \text{if $\mathscr K = \emptyset$},\\
\dfrac{x_k^{2^{r} - 1}\rho_{(k,\,d)}(x)}{x_{(\mathscr K,\,u)}}&\text{if there exist $u$ such that:} \\
& t_{k_1 - 1} =t_{k_2 - 1} = \ldots = t_{k_{(u-1)} - 1} = 2^{r} - 1,\\
&t_{k_{u} - 1} > 2^{r} - 1,\\
&\alpha_{r-m}(t_{k_{u} - 1}) = 1,\;\forall m,\ 1\leq m\leq u,\ \mbox{and}\\
&\alpha_{r-m}(t_{k_{m}-1}) = 1,\;\forall m,\ u < m \leq r,\\
0&\text{otherwise},
\end{array} \right.$$
for any $x = x_1^{t_1}x_2^{t_2}\ldots x_{d-1}^{t_{d-1}}$ in $\mathscr P_{d-1}.$ Note that $\phi_{(k; \mathscr K)}$ is not an  $\mathcal A_2$-homomorphism in general. Moreover, for each $x\in \mathscr P_{d-1},$ if  $\phi_{(k; \mathscr K)}(x)\neq  0,$ then $\omega( \phi_{(k; \mathscr K)}(x)) = \omega(x).$ 

From now on, we adopt the following notations: For a natural number $d$, we consider $\Gamma_d = \{1,2,\ldots, d\},\ X_{(S,\, d)} = X_{(\{s_1,s_2,\ldots,s_r\},\,d)} = \prod_{s\in\Gamma_d\setminus S}x_s,$ where $S = \{s_1,s_2,\ldots,s_r\}\subseteq  \Gamma_d.$ In particular, $$X_{(\Gamma_d,\, d)} = 1,\ \ X_{({\emptyset},\, d)} = x_1x_2\ldots x_d,\ X_{(\{s\},\, d)} = x_1\ldots \hat{x}_s\ldots x_d,\,\, 1\leq s\leq d.$$ Now consider $X = x_1^{t_1}x_2^{t_2}\ldots x_d^{t_d}\in \mathscr P_d$ and let $S_j(X) = \{s\in \Gamma_d:\; \alpha_j(t_s) = 0\}$ for $j\geq 0.$ Then, by a simple computation, we get $X = \prod_{j\geq 0}X_{(S_{j}(X),\,d)}^{2^{j}}.$ 

The following examples on the map $ \phi_{(k; \mathscr K)}$ can be found in \cite{N.S3}. However, we present them in more detail.

(i) Let $\mathscr K = (j)$ and $1\leq k < j\leq d.$ Then, for any the monomial $x = x_1^{a_1}x_2^{a_2}\ldots x_{d-1}^{a_{d-1}}\in \mathscr P_{d-1}$ and $\alpha_0(a_{j-1}) = 1,$ we conclude $\phi_{(k; \mathscr K)}(x) = \dfrac{x_k\rho_{(k,\,d)}(x)}{x_j}.$

(ii) Let $m$ be a positive integer and let $x = Y^{2^{m}-1}y^{2^m}$ with $y = x_1^{b_1}x_2^{b_2}\ldots x_{d-1}^{b_{d-1}}$ and $Y  = X_{(\{d\}, d)}= x_1x_2\ldots x_{d-1}\in \mathscr P_{d-1}.$ Then if   $m > r = \ell(\mathscr K)$ and $u = 1$ then 
$$ \phi_{(k; \mathscr K)}(x) = \phi_{(k; \mathscr K)}(Y^{2^m-1})(\rho_{(k, d)}(y))^{2^m} = x_{k}^{2^r-1}\prod\limits_{1\leqslant t\leqslant r}x_{k_t}^{2^m-2^{r-t}-1}X^{2^m-1}_{(\{k, k_1, \ldots, k_r\}, d)}(\rho_{(k, d)}(y))^{2^m}.$$
Indeed, since $\rho_{(k, d)}$ is an $\mathbb Z/2$-algebras homomorphism, 
$$ \rho_{(k, d)}(x) = \rho_{(k, d)}(X^{2^{m}-1})(\rho_{(k, d)}(y))^{2^m}, \ 1\leq k\leq d.$$
Since $Y^{2^m-1} = x_1^{2^m-1}\ldots x_{d-1}^{2^m-1}$ and $2^d-1 > 2^r-1$, for each $(k; \mathscr K),\ \mathscr K = (k_1, k_2, \ldots, k_r)$ and $u = 1,$ we have
$$ \begin{array}{lll}
 \phi_{(k; \mathscr K)}(Y^{2^m-1}) &=& \dfrac{x_k^{2^r-1}x_2^{2^m-1}\ldots x_{k_1}^{2^m-1}\ldots x_{k_r}^{2^m-1}\ldots x_d^{2^m-1}}{x_{k_1}^{2^{r-1}}x_{k_2}^{2^{r-2}}\ldots x_{k_r}^{2^{r-r}}}\\
&=& x_k^{2^r-1}x_{k_1}^{2^m-2^{r-1}-1}\ldots x_{k_2}^{2^m- 2^{r-2}-1}\ldots x_{k_r}^{2^m- 2^{r-r}-1}X^{2^m-1}_{(\{k, k_1, \ldots, k_r\}, d)}\\
&=&x_k^{2^r-1}\prod\limits_{1\leqslant t\leqslant r}x_{k_t}^{2^m-2^{r-t}-1}X^{2^m-1}_{(\{k, k_1, \ldots, k_r\}, d)}.
\end{array}$$
Then, one gets
$$ \begin{array}{lll}
\phi_{(k; \mathscr K)}(x) &=& \phi_{(k; \mathscr K)}(Y^{2^m-1}y^{2^m})\\
&=& \dfrac{x_k^{2^r-1}\rho_{(k, d)}(X^{2^m-1}y^{2^m})}{x_{(\mathscr K; 1)}} = \left(\dfrac{x_k^{2^r-1}\rho_{(k, d)}(Y^{2^m-1})}{x_{(\mathscr K; 1)}}\right)(\rho_{(k, d)}(y))^{2^m}\\
&=&\phi_{(k; \mathscr K)}(X^{2^m-1})(\rho_{(k, d)}(y))^{2^m} \\
&=& x_k^{2^r-1}\prod\limits_{1\leqslant t\leqslant r}x_{k_t}^{2^m-2^{r-t}-1}X^{2^m-1}_{(\{k, k_1, \ldots, k_r\}, d)}(\rho_{(k, d)}(y))^{2^m}.
\end{array}$$
Now, if $m = r, \ b_{j-1} = 0,\ j = k_1, k_2, \ldots, k_{u-1}$ and $b_{k_u-1} > 0$ then for each $(k; \mathscr K)$ and $1\leq u\leq r=m,$ we have 
$$ \phi_{(k; \mathscr K)}(x) = \phi_{(k_u; \{k_{u+1}, \ldots, k_m\})}(Y^{2^m-1})(\rho_{(k, d)}(y))^{2^m}.$$ 
Indeed, we have
$$ \begin{array}{ll}
Y^{2^m-1}y^{2^m} &= (x_1\ldots x_{d-1})^{2^m-1}(x_1^{b_1}\ldots \hat{x}_{k_1}^{b_{k_1}}\ldots \hat{x}_{k_{(u-1)}-1}^{b_{k_{(u-1)}-1} }x_{k_u-1}^{b_{k_u-1}}\ldots x_{d-1}^{b_{d-1}})^{2^	m}\\
&=(x_1\ldots x_{k_1}\ldots x_{k_{u-1}-1}x_{k_u-1}\ldots x_{d-1})^{2^m-1}(x_1^{b_1}\ldots \hat{x}_{k_1}^{b_{k_1}}\ldots \hat{x}_{k_{(u-1)}-1}^{b_{k_{(u-1)}-1} }x_{k_u-1}^{b_{k_u-1}}\ldots x_{d-1}^{b_{d-1}})^{2^m}.
\end{array}$$
Then, we get
$$ \begin{array}{ll}
\phi_{(k; \mathscr K)}(x) &= \dfrac{x_k^{2^m-1}\rho_{(k, d)}(Y^{2^m-1})}{x_{k_u}^{2^{m-1}+2^{m-2}+\cdots+2^{m-u}}\prod\limits_{u< t \leqslant r}x_{k_t}^{2^{m-t}}}(\rho_{(k, d)}(y))^{2^m}.\\
&=\dfrac{x_k^{2^m-1}x_2^{2^m-1}\ldots x_{k_{u-1}}^{2^m-1} x_{k_u}^{2^m-1}\prod\limits_{u+1\leqslant t \leqslant m}x_{k_t}^{2^{m}-1}\ldots x_d^{2^m-1}}{x_{k_u}^{2^{m-1}+2^{m-2}+\cdots+2^{m-u}}\prod\limits_{u< t \leqslant m}x_{k_t}^{2^{m-t}}}(\rho_{(k, d)}(y))^{2^m}\\
&=x_k^{2^m-1}x_2^{2^m-1}\ldots x_{k_{u-1}}^{2^m-1} x_{k_u}^{2^{m-(u+1)} + \cdots + 2^{m-m}}\ldots x_d^{2^m-1}\\
&\quad \prod\limits_{u+1\leqslant t \leqslant m}x_{k_t}^{2^{m}-2^{m-t}-1}\ldots x_d^{2^m-1}(\rho_{(k, d)}(y))^{2^m}\\
&=\left(x_{k_u}^{2^{m-(u+1)} + \cdots + 2^{m-m}}\prod\limits_{u+1\leqslant t \leqslant m}x_{k_t}^{2^{m}-2^{m-t}-1}X^{2^m-1}_{(\{k_u, k_{u+1}, \ldots, k_m\}, d)}\right)(\rho_{(k, d)}(y))^{2^m}.\\
&= \phi_{(k_u; \{k_{u+1}, \ldots, k_m\})}(Y^{2^m-1})(\rho_{(k, d)}(y))^{2^m}.
\end{array}$$

We end this subsection by reviewing Sum's conjecture \cite{N.S6} on the relation between the admissible monomials for the polynomial algebras. 

For a subset $\mathcal U\subset \mathscr P_{d-1},$ we denote 
$$ \begin{array}{ll}
\overline{\Phi}^0(\mathcal  U) &= \bigcup\limits_{1\leq k \leq d}\phi_{(k; \emptyset)}(\mathcal  U) = \bigcup\limits_{1\leq k \leq d}\rho_{(k,\,d)}(\mathcal  U),\\
\overline{\Phi}^+(\mathcal  U) &= \bigcup\limits_{(k; \mathscr K)\in\mathcal{N}_d,\;0 < \ell(\mathscr K) < d}\phi_{(k; \mathscr K)}(\mathcal U)\setminus \mathscr P_d^0,\\
\overline{\Phi}(\mathcal  U) &= \overline{\Phi}^0(\mathcal  U) \bigcup \overline{\Phi}^+(\mathcal  U).
\end{array}$$
Since $\rho_{(k,\,d)}$ is a homomorphism of the $\mathcal A_2$-modules, if $\mathcal  U$ is a minimal set of generators for the $\mathcal A_2$-module $\mathscr P_{d-1}$ in degree $n$, then $\overline{\Phi}^0(\mathcal U) $ is also a minimal set of generators for the $\mathcal A_2$-module $\mathscr P_d^0$ in degree $n.$

Now, for a polynomial $f\in \mathscr P_d,$ we denote by $[f]$ the classes in $Q\mathscr P_d$ represented by $f.$ If $\omega$ is a weight vector and $f\in \mathscr P_d(\omega),$ then denote by $[f]_\omega$ the classes in $Q\mathscr P_d(\omega)$ represented by $f.$ For a subset $\mathscr{B}\subset \mathscr P_{d},$ we denote $[\mathscr B] = \{[f]\, :\, f\in \mathscr B\}.$ If $\mathscr B\subset  \mathscr P_d(\omega),$ then we set $[\mathscr B]_{\omega} = \{[f]_{\omega}\, :\, f\in \mathscr B\}.$

Denote by $\mathscr{B}_d(n)$ the set of all admissible monomials of degree $n$ in  $\mathscr P_d.$  Thus when we write $x\in \mathscr{B}_d(n)$ we mean that it is an admissible monomial of degree $n.$ We set
$$ \mathscr{B}_d^0(n) := \mathscr{B}_d(n)\cap (\mathscr P_d^0)_n,\,\,\mathscr{B}_d^+(n) := \mathscr{B}_d(n)\cap (\mathscr P_d^+)_n.$$
If $\omega$ is a weight vector of degree $n,$ we set 
$$ \mathscr{B}_d(\omega) := \mathscr{B}_d(n)\cap \mathscr P_d(\omega),\, \mathscr{B}_d^0(\omega) := \mathscr{B}_d(\omega)\cap (\mathscr P^0_d)_n,\, \mathscr{B}_d^+(\omega) := \mathscr{B}_d(\omega)\cap (\mathscr P^+_d)_n.$$
Then, $[\mathscr{B}_d(\omega)]_\omega,\, [\mathscr{B}^0_d(\omega)]_\omega$ and $[\mathscr{B}_d^+(\omega)]_\omega$ are respectively the bases of the $\mathbb Z/2$-vector spaces $$Q\mathscr P_d(\omega),\ Q\mathscr P_d^0(\omega):= Q\mathscr P_d(\omega)\cap (Q\mathscr P^0_d)_n\ \mbox{and }\, Q\mathscr P_d^+(\omega) := Q\mathscr P_d(\omega)\cap (Q\mathscr P^+_d)_n.$$

Throughout this paper, to prove a certain subset of  $Q\mathscr P_d$ is linearly independent, we use a result in Sum \cite{N.S3} combining with Theorem \ref{dlsig} (Singer's criterion on the $\mathcal A_2$-decomposable) and Proposition \ref{mdPS}. More precisely,  let $\mathscr B$ be a finite subset of $\mathscr P_d$ consisting of some monomials of degree $n.$ Denote by $|\mathscr B|$ the cardinal of $\mathscr B.$ To prove the set $[\mathscr B]$ is linearly independent in $(Q\mathscr P_d)_n,$ we denote the elements of $\mathscr B$ by $\mathcal Y_{n, \, i},\ 1\leq i \leq m = |\mathscr B|$ and assume that there is a linear relation
$$ \mathcal S = \sum_{1\leq i\leq m}\gamma_i\mathcal Y_{n,\, i} = 0\,\, {\rm modulo}(\mathcal A_2^+\mathscr P_d + \mathscr P_d^-(\omega)),$$
with $\gamma_i\in \mathbb Z/2$ for all $i,\, 1\leq i\leq m.$ For $(k; \mathscr K)\in \mathcal N_d,$ we explicitly compute $\pi_{(k; \mathscr K)}(\mathcal  S)$ in terms of the admissible monomials in $\mathscr P_{d-1}\, ({\rm modulo}(\mathcal A_2^+\mathscr P_{d-1} + \mathscr P_{d-1}^-(\omega))).$ Computing from some relations $\pi_{(k; \mathscr K)}(\mathcal  S) =  0\,\, {\rm modulo}(\mathcal A_2^+\mathscr P_{d-1} + \mathscr P_{d-1}^-(\omega))$ with $(k; \mathscr K)\in \mathcal N_d,$ we obtain $\gamma_i = 0$ for all $i.$

In \cite{N.S6}, Sum made the following conjecture, which plays an important role in studying the minimal set of $\mathcal A_2$-module $\mathscr P_d$ in certain generic degree. 

\begin{conje}[Sum \cite{N.S6}]\label{gtS}
If $\omega$ is a weight vector, then $\overline{\Phi}(\mathscr B_{d-1}(\omega))\subseteq \mathscr B_d(\omega).$
\end{conje}

Obviously, if this conjecture is true, then $\overline{\Phi}(\mathscr B_{d-1}(n))\subseteq \mathscr B_d(n)$  for any positive integer $n.$ In other words, if $x\in \mathscr B_{d-1}(n),$ then $ \phi_{(k; \mathscr K)}(x) \in \mathscr B_d(n).$ 
By previous results of Peterson \cite{F.P1}, Kameko \cite{M.K} and Sum \cite{N.S3}, the conjecture is true for $d\leq 4.$ In particular, we have the following remark.

\begin{rema}\label{nxax}
Consider the spike monomial $Y =X_{(\{d\}, d)} = x_1x_2\ldots x_{d-1}\in \mathscr P_{d-1}$. Let $m$ be a positive integer such that $m > r = \ell(\mathscr K).$ Then, from the above calculations, we have
$$ \phi_{(k; \mathscr K)}(Y^{2^m-1}) = x_{k}^{2^r-1}\prod_{1\leq t\leq r}x_{i_t}^{2^m-2^{r-t}-1}X^{2^m-1}_{(\{k, k_1, \ldots, k_r\}, d)}.$$ 
It is easy to see that $\omega(Y) = \underset{\mbox{{$m$ times of $(d-1)$}}}{\underbrace{(d-1, d-1, \ldots, d-1}}).$ Based on the results in \cite{P.S2, N.S7}, the set $\{\phi_{(k; \mathscr K)}(Y^{2^m-1}): (k; \mathscr K)\in \mathcal{N}_d\}$ is a basis of $Q\mathscr P_{d}(\omega(Y)).$ Note that this also holds true for $m\leq r$ (see Sum \cite{N.S7}). By Proposition \ref{PS}, $Y^{2^m-1}$ is admissible. Combining these data, Sum's conjecture is true for the weight vector $\omega(Y),$ where $d$ is an arbitrary positive integer.
\end{rema}

In \cite{N.S6}, Sum showed that Conjecture \ref{gtS} is true for $d = 5$ and any weight vector of generic degree of \eqref{ct} with $r = d = 5,$  $s = 10$ and $t\geq 0.$  In the next subsection, we will show that this conjecture is also satisfying for $d = 5$ and in generic degree of Theorem \ref{dl1}. 

\subsection{Proof of Theorem \ref{dl1}}\label{s3.1}
As shown in Sect.\ref{s1}, we have  $\mu(13.2^t-5) = 5$ for every $t > 2,$ hence the inverse function $ \widetilde{\phi}: (Q\mathscr P_5)_{13.2^{t-1}-5}\to (Q\mathscr P_5)_{13.2^t-5}$ of $(\widetilde {Sq^0_*})_{(5, 13.2^t-5)}$ defined by $\widetilde{\phi}([u]) = [X_{(\emptyset, 5)}u^{2}]$ for all $[u]\in (Q\mathscr P_5)_{13.2^{t-1}-5},\ t > 2.$ On the other hand, since the iterated Kameko squaring operation
$$ (\widetilde {Sq^0_*})^{t-2}_{(5, 13.2^t-5)}: (Q\mathscr P_5)_{13.2^t-5} \to (Q\mathscr P_5)_{13.2^2-5}$$
is an $\mathbb Z/2$-vector space isomorphism for every $t \geq 2,$ a basis of $Q\mathscr P_5$ in degree $13.2^t-5$ is the set $$[\mathscr B_5(13.2^t-5)] = \widetilde{\phi}^{t-2}([\mathscr B_5(13.2^2-5)])$$ for $t > 2.$ Thus, we need only to find the minimal set of $\mathbb Z/2$-generators for $(Q\mathscr P_5)_{13.2^t-5}$ with $t \in \{0,\,1,\, 2\}.$ It has been determined by T\'in \cite{N.T1} for $t = 0.$ Note that our methods of studying $Q\mathscr P_5$ in this paper are different from the ones of T\'in.

\subsubsection{The case  $t = 1$}\label{s3.2.1}
Consider Kameko's homomorphism $(\widetilde {Sq_*^0})_{(5,21)}: (Q\mathscr P_5)_{21}\longrightarrow (Q\mathscr P_5)_{8}.$ We know that it is an epimorphism of $\mathbb Z/2$-vector spaces, hence $(Q\mathscr P_5)_{21}\cong \mbox{Ker}(\widetilde {Sq_*^0})_{(5,21)}\bigoplus (Q\mathscr P_5)_{8}.$ Note that $\mbox{Ker}(\widetilde {Sq_*^0})_{(5,21)}$ is isomorphic to $(Q\mathscr P_5^0)_{21}\bigoplus (\mbox{Ker}(\widetilde {Sq_*^0})_{(5,21)}\cap (Q\mathscr P_5^+)_{21}).$ From the calculations of $Q\mathscr P_d$ in degree $21$ for $1\leq d\leq 4$ (see \cite{M.K}, \cite{F.P1}, \cite{N.S3}) and $Q\mathscr P_5$ in degree $8$ (see \cite{N.T2}), we have $$ |\mathscr B_1^+(21)| = 0,\ |\mathscr B_2^+(21)| = 0,\ |\mathscr B_3^+(21)| = 7,\  |\mathscr B_4^+(21)| = 66,\ |\mathscr B_5(8)| = 174.$$ We note that $(Q\mathscr P_3)_{21}\cong (Q\mathscr P_3)_3$ and $\mathscr B_3(21) = \mathscr B_3^+(21) = \widetilde{\varphi}^2(\mathscr B_3(3))$ with the $\mathbb Z/2$-linear map $\widetilde{\varphi}: \mathscr P_3\to \mathscr P_3,$ determined by $\widetilde{\varphi}(u) = X_{(\emptyset,\, 3)}u^2,\ \forall u\in \mathscr P_3.$
Since $(Q\mathscr P_5^0)_{21} = \bigoplus_{1\leq r\leq 4}\bigoplus_{1\leq j\leq \binom{5}{r}}(Q\mathscr P_r^+)_{21},$ we deduce $$ \dim(Q\mathscr P_5^0)_{21} = \binom 53.7 + \binom 54.66 = 400.$$ Moreover, a direct computation shows that $\mathscr B_5^0(21) = \overline{\Phi}^0(\mathscr B_4(21)) = \{\mathcal Y_{21,\, i}\, :\, 1\leq i \leq 400\},$ where the monomials $\mathcal Y_{21,\, i},\, 1\leq i\leq 400,$ are listed in Sect.\ref{s5.2} of  the Appendix.

\begin{propo}\label{md21.1}
The set $\{[\mathcal Y_{21, i}] : 401 \leq i \leq 666\}$ is the basis of the $\mathbb Z/2$-vector space $\mbox{Ker}(\widetilde {Sq_*^0})_{(5,21)}\cap (Q\mathscr P_5^+)_{21}$. Here the monomials $\mathcal Y_{i}:=\mathcal Y_{21, i},\ 401 \leq i \leq 666$, which are determined in Sect.\ref{s5.3}
\end{propo}
Combining Proposition \ref{md21.1} and the above data, we deduce that the $\mathbb Z/2$-vector space $(Q\mathscr P_5)_{21}$ is $840$-dimensional. This completes the proof of the theorem for the case $t = 1.$ 

We now need to some results for the proof of Proposition \ref{md21.1}. First, we have the following lemma.

\begin{lema}\label{bd21.1}
If $u\in \mathscr B_5(21)$ and $[u]\in \mbox{\rm Ker}(\widetilde {Sq_*^0})_{(5,21)},$ then the weight vector of $u$ is either $\omega(u) = (3,3,1,1)$ or $\omega(u) = (3,3,3).$  
\end{lema}

\begin{proof}
Note that $x_1^{15}x_2^3x_3^3\in (\mathscr P_5)_{21}$ is the minimal spike, and that by Proposition \ref{PS}, it is an admissible monomial. Moreover,  $\omega(x_1^{15}x_2^3x_3^3) = (3,3,1,1).$ Since $[u]\neq [0],$ by Theorem \ref{dlsig}, we get either $\omega_1(u) = 3$ or $\omega_1(u) = 5.$ If $\omega_1(u) = 5$ then $u = X_{(\emptyset,\, 5)}y^2$ with $y$ a monomial of degree $8$ in $\mathscr P_5.$ Since $u$ is admissible, by Theorem \ref{dlKS}, one gets $y\in\mathscr B_5(8).$ So $(\widetilde {Sq_*^0})_{(5,21)}([u]) = [y]\neq [0].$ This contradicts the fact that $[u]\in \mbox{Ker}(\widetilde {Sq_*^0})_{(5,21)}.$ Hence, $\omega_1(u) = 3.$ Then, we have $u = X_{(\{i, j\},\,5)}y_1^2$ with $1\leq i< j  \leq 5$ and $y_1\in \mathscr B_5(9)$. Since $y_1$ is admissible, according to a result in \cite{N.T1}, we have either $\omega(y_1) = (3,1,1)$ or $\omega(y_1) = (3,3)$. The lemma is proved.
\end{proof}

As an immediate consequence, we see that the dimension of ($\mbox{Ker}(\widetilde {Sq_*^0})_{(5,21)}\cap (Q\mathscr P_5^+)_{21}$) is equal to the sum of dimensions of $Q\mathscr P_5^+(3,3,1,1)$ and $Q\mathscr P_5^+(3,3,3).$ This leads us to determine the subspaces $Q\mathscr P_5^+(\omega),$ where the weight vectors $\omega$ are $(3,3,1,1)$ and $(3,3,3).$

The following lemma is an immediate corollary from a result in \cite{N.S3}.

\begin{lema}\label{bd21.2}
The following monomials are strictly inadmissible:
\begin{enumerate}
\item[(i)]  $x_i^2x_jx_k^3x_{\ell}^7,\, x_i^6x_jx_k^3x_{\ell}^3,\, x_i^2x_j^5x_k^3x_{\ell}^3,\, x_i^2x_jx_k^3x_{\ell}^3,\ 1\leq i < j\leq 5,\, 1\leq k,\,\ell\leq 5,\ k\neq \ell,\ k, \ell\neq i, j;$
\item[(ii)]  $x_i^3x_j^4x_k^3x_{\ell}^3,\, 1\leq i < j <k <\ell\leq 5;$
\item[(iii)] $\rho_{(k,\, 5)}(X),\ 1\leq k\leq 5,$  where $X$ is one of the following monomials:

\begin{tabular}{lllr}
$ x_1^{3}x_2^{4}x_3^{7}x_4^{7}$, & $ x_1^{3}x_2^{7}x_3^{4}x_4^{7}$, & $ x_1^{3}x_2^{7}x_3^{7}x_4^{4}$, & \multicolumn{1}{l}{$ x_1^{7}x_2^{3}x_3^{4}x_4^{7}$,} \\
$ x_1^{7}x_2^{3}x_3^{7}x_4^{4}$, & $ x_1^{7}x_2^{7}x_3^{3}x_4^{4}$, & $ x_1^{7}x_2^{8}x_3^{3}x_4^{3}$. &  
\end{tabular}
\end{enumerate}
\end{lema}

\newpage
\begin{lema}\label{bd21.3}
The following monomials are strictly inadmissible:
\begin{enumerate}
\item[(i)] $x_i^2x_j^2x_kx_{\ell}x_m^7,\, x_i^2x_j^6x_kx_{\ell}x_m^3,\, x_i^6x_j^2x_kx_{\ell}x_m^3,\, x_i^2x_j^2x_k^5x_{\ell}x_m^3,\, x_i^2x_j^4x_kx_{\ell}^3x_m^3,\\ x_i^4x_j^2x_kx_{\ell}^3x_m^3, x_i^2x_jx_k^4x_{\ell}^3x_m^3,$ where $(i, j, k, \ell, m)$ is a permutation of $(1, 2, 3, 4, 5);$  
\item[(ii)] 
 \begin{tabular}[t]{lllr}
$ x_1^{3}x_2^{2}x_3x_4x_5^{6}$ & \multicolumn{1}{l}{$ x_1^{3}x_2^{6}x_3x_4x_5^{2}$,} & \multicolumn{1}{l}{$ x_1^{3}x_2^{2}x_3^{5}x_4^{2}x_5$,} & \multicolumn{1}{l}{$ x_1^{3}x_2^{2}x_3x_4^{5}x_5^{2}$,} \\
$ x_1^{3}x_2^{2}x_3^{5}x_4x_5^{2}$, & \multicolumn{1}{l}{$ x_1^{3}x_2^{5}x_3^{2}x_4^{2}x_5$,} & \multicolumn{1}{l}{$ x_1^{3}x_2^{2}x_3^{2}x_4x_5$,} & \multicolumn{1}{l}{$ x_1^{3}x_2^{2}x_3x_4^{2}x_5$,} \\
$ x_1^{3}x_2^{2}x_3x_4x_5^{2}$. &       &       &  
\end{tabular}
\end{enumerate}
\end{lema}

\begin{proof}
Consider the monomials $X = x_i^2x_j^2x_kx_{\ell}x_m^7$ and $Y = x_1^{3}x_2^{2}x_3x_4x_5^{6}.$ We prove that these monomials are strictly inadmissible. The others can be proved by a similar computation. Obviously, $\omega(X) = \omega(Y) = (3,3,1).$ By a direct computation using the Cartan formula, we obtain
$$ \begin{array}{ll}
\medskip
X & = Sq^2(x_ix_jx_kx_{\ell}x_m^7) + Sq^4(x_ix_jx_kx_{\ell}x_m^5) + x_ix_jx_k^2x_{\ell}^2x_m^7\\
\medskip
&\quad + x_ix_j^2x_kx_{\ell}^2x_m^7 + x_ix_j^2x_k^2x_{\ell}x_m^7 \ {\rm modulo}\,(\mathscr P_5^-(3,3,1));\\
Y &= Sq^1(x_1^{3}x_2x_3x_4x_5^{6}) + x_1^{3}x_2x_3^2x_4x_5^{6} + x_1^{3}x_2x_3x_4^2x_5^{6} \ {\rm modulo}\,(\mathscr P_5^-(3,3,1)).
\end{array}$$
These equalities show that $X$ and $Y$ are strictly inadmissible. The lemma follows.
\end{proof}

\begin{lema}\label{bd21.4}
The following monomials are strictly inadmissible:

\begin{center}
\begin{tabular}{lrrr}
$X_{1}=x_1x_2^{6}x_3^{8}x_4^{3}x_5^{3}$, & \multicolumn{1}{l}{$X_{2}= x_1x_2^{3}x_3^{6}x_4^{6}x_5^{5}$,} & \multicolumn{1}{l}{$X_{3}=x_1^3x_2^{5}x_3^{2}x_4^6x_5^{5}$,} & \multicolumn{1}{l}{$X_{4}=x_1x_2^{6}x_3^{7}x_4^{6}x_5$,} \\
$X_{5}=x_1x_2^{7}x_3^{6}x_4^{6}x_5$, & \multicolumn{1}{l}{$X_{6}=x_1^{7}x_2x_3^{6}x_4^{6}x_5$,} & \multicolumn{1}{l}{$X_{7}=x_1x_2^{2}x_3^{6}x_4^{5}x_5^{7}$,} & \multicolumn{1}{l}{$X_{8}=x_1x_2^{2}x_3^{6}x_4^{7}x_5^{5}$,} \\
$X_{9}=x_1x_2^{6}x_3^{2}x_4^{5}x_5^{7}$, & \multicolumn{1}{l}{$X_{10}=x_1x_2^{6}x_3^{2}x_4^{7}x_5^{5}$,} & \multicolumn{1}{l}{$X_{11}=x_1x_2^{2}x_3^{7}x_4^{6}x_5^{5}$,} & \multicolumn{1}{l}{$X_{12}=x_1x_2^{6}x_3^{7}x_4^{2}x_5^{5}$,} \\
$X_{13}=x_1x_2^{7}x_3^{2}x_4^{6}x_5^{5}$, & \multicolumn{1}{l}{$X_{14}=x_1x_2^{7}x_3^{6}x_4^{2}x_5^{5}$,} & \multicolumn{1}{l}{$X_{15}=x_1^{7}x_2x_3^{2}x_4^{6}x_5^{5}$,} & \multicolumn{1}{l}{$X_{16}=x_1^{7}x_2x_3^{6}x_4^{2}x_5^{5}$,} \\
$X_{17}=x_1x_2^{6}x_3^{3}x_4^{4}x_5^{7}$, & \multicolumn{1}{l}{$X_{18}=x_1x_2^{6}x_3^{3}x_4^{7}x_5^{4}$,} & \multicolumn{1}{l}{$X_{19}=x_1x_2^{6}x_3^{7}x_4^{3}x_5^{4}$,} & \multicolumn{1}{l}{$X_{20}=x_1x_2^{7}x_3^{6}x_4^{3}x_5^{4}$,} \\
$X_{21}=x_1^{3}x_2^{4}x_3x_4^{6}x_5^{7}$, & \multicolumn{1}{l}{$X_{22}=x_1^{3}x_2^{4}x_3x_4^{7}x_5^{6}$,} & \multicolumn{1}{l}{$X_{23}=x_1^{3}x_2^{4}x_3^{7}x_4x_5^{6}$,} & \multicolumn{1}{l}{$X_{24}=x_1^{3}x_2^{7}x_3^{4}x_4x_5^{6}$,} \\
$X_{25}=x_1^{7}x_2x_3^{6}x_4^{3}x_5^{4}$, & \multicolumn{1}{l}{$X_{26}=x_1^{7}x_2^{3}x_3^{4}x_4x_5^{6}$,} & \multicolumn{1}{l}{$X_{27}=x_1x_2^{6}x_3^{6}x_4^{3}x_5^{5}$,} & \multicolumn{1}{l}{$X_{28}=x_1x_2^{6}x_3^{3}x_4^{6}x_5^{5}$,} \\
$X_{29}=x_1x_2^{6}x_3^{6}x_4^{7}x_5$, & \multicolumn{1}{l}{$X_{30}=x_1^{3}x_2x_3^{6}x_4^{6}x_5^{5}$,} & \multicolumn{1}{l}{$X_{31}=x_1^{3}x_2^{5}x_3^{6}x_4^{6}x_5$,} & \multicolumn{1}{l}{$X_{32}=x_1^{3}x_2^{4}x_3^{5}x_4^{2}x_5^{7}$,} \\
$X_{33}=x_1^{3}x_2^{4}x_3^{5}x_4^{7}x_5^{2}$, & \multicolumn{1}{l}{$X_{34}=x_1^{3}x_2^{4}x_3^{7}x_4^{5}x_5^{2}$,} & \multicolumn{1}{l}{$X_{35}=x_1^{3}x_2^{5}x_3^{4}x_4^{2}x_5^{7}$,} & \multicolumn{1}{l}{$X_{36}=x_1^{3}x_2^{5}x_3^{4}x_4^{7}x_5^{2}$,} \\
$X_{37}=x_1^{3}x_2^{7}x_3^{4}x_4^{5}x_5^{2}$, & \multicolumn{1}{l}{$X_{38}=x_1^{3}x_2^{5}x_3^{7}x_4^{4}x_5^{2}$,} & \multicolumn{1}{l}{$X_{39}=x_1^{3}x_2^{7}x_3^{5}x_4^{4}x_5^{2}$,} & \multicolumn{1}{l}{$X_{40}=x_1^{7}x_2^{3}x_3^{4}x_4^{5}x_5^{2}$,} \\
$X_{41}=x_1^{7}x_2^{3}x_3^{5}x_4^{4}x_5^{2}$, & \multicolumn{1}{l}{$X_{42}=x_1x_2^{6}x_3^{6}x_4x_5^{7}$,} & \multicolumn{1}{l}{$X_{43}=x_1^{3}x_2^{5}x_3^{6}x_4^{2}x_5^{5}$,} & \multicolumn{1}{l}{$X_{44}=x_1^{3}x_2^{5}x_3^{6}x_4^{5}x_5^{2}$,} \\
$X_{45}=x_1^{3}x_2^{5}x_3^{5}x_4^{2}x_5^{6}$, & \multicolumn{1}{l}{$X_{46}=x_1^{3}x_2^{5}x_3^{5}x_4^{6}x_5^{2}$,} & \multicolumn{1}{l}{$X_{47}=x_1^{3}x_2^{4}x_3^{4}x_4^{3}x_5^{7}$,} & \multicolumn{1}{l}{$X_{48}=x_1^{3}x_2^{4}x_3^{4}x_4^{7}x_5^{3}$,} \\
$X_{49}=x_1^{3}x_2^{4}x_3^{3}x_4^{4}x_5^{7}$, & \multicolumn{1}{l}{$X_{50}=x_1^{3}x_2^{4}x_3^{7}x_4^{4}x_5^{3}$,} & \multicolumn{1}{l}{$X_{51}=x_1^{3}x_2^{4}x_3^{3}x_4^{7}x_5^{4}$,} & \multicolumn{1}{l}{$X_{52}=x_1^{3}x_2^{4}x_3^{7}x_4^{3}x_5^{4}$,} \\
$X_{53}=x_1^{3}x_2^{7}x_3^{4}x_4^{4}x_5^{3}$, & \multicolumn{1}{l}{$X_{54}=x_1^{3}x_2^{7}x_3^{4}x_4^{3}x_5^{4}$,} & \multicolumn{1}{l}{$X_{55}=x_1^{7}x_2^{3}x_3^{4}x_4^{4}x_5^{3}$,} & \multicolumn{1}{l}{$X_{56}=x_1^{7}x_2^{3}x_3^{4}x_4^{3}x_5^{4}$,} \\
$X_{57}=x_1^{3}x_2^{4}x_3^{5}x_4^{6}x_5^{3}$, & \multicolumn{1}{l}{$X_{58}=x_1^{3}x_2^{4}x_3^{5}x_4^{3}x_5^{6}$,} & \multicolumn{1}{l}{$X_{59}=x_1^{3}x_2^{5}x_3^{4}x_4^{6}x_5^{3}$,} & \multicolumn{1}{l}{$X_{60}=x_1^{3}x_2^{5}x_3^{6}x_4^{4}x_5^{3}$,} \\
$X_{61}=x_1^{3}x_2^{5}x_3^{4}x_4^{3}x_5^{6}$. &       &       &  
\end{tabular}
\end{center}
\end{lema}

\begin{proof}
It is easily seen that $\omega(X_1) = (3,3,1,1)$ and $\omega(X_j) = (3,3,3)$ for $j = 2, 3, \ldots, 61.$ We prove the lemma for the monomials $X_1 = x_1x_2^{6}x_3^{8}x_4^{3}x_5^{3}$, $X_{2}= x_1x_2^{3}x_3^{6}x_4^{6}x_5^{5},$ and  $X_{3} = x_1^3x_2^{5}x_3^{2}x_4^6x_5^{5}$. The others can be proven by a similar computation. By a direct computation, we have
$$ \begin{array}{ll}
X_1 &=  x_1x_2^{3}x_3^{3}x_4^{6}x_5^{8} + x_1x_2^{3}x_3^{3}x_4^{8}x_5^{6} + x_1x_2^{3}x_3^{6}x_4^{3}x_5^{8} + x_1x_2^{3}x_3^{6}x_4^{8}x_5^{3}+ x_1x_2^{3}x_3^{8}x_4^{3}x_5^{6}  \\
&\quad\quad + x_1x_2^{3}x_3^{8}x_4^{6}x_5^{3} + x_1x_2^{4}x_3^{3}x_4^{3}x_5^{10} + x_1x_2^{4}x_3^{3}x_4^{10}x_5^{3}+ x_1x_2^{4}x_3^{10}x_4^{3}x_5^{3}\\
 &\quad\quad  + x_1x_2^{6}x_3^{3}x_4^{3}x_5^{8} + x_1x_2^{6}x_3^{3}x_4^{8}x_5^{3} + Sq^1(A_1) + Sq^2(A_2) + Sq^4(A_4)\ {\rm modulo}\,(\mathscr P_5^-(3,3,1,1)),
\end{array}$$
where
$$ \begin{array}{ll}
A_1 &= x_1^2x_2^{3}x_3^{5}x_4^{5}x_5^{5} + x_1^2x_2^{5}x_3^{3}x_4^{5}x_5^{5} + x_1^2x_2^{5}x_3^{5}x_4^{3}x_5^{5} + x_1^2x_2^{5}x_3^{5}x_4^{5}x_5^{3},\\
A_2 &=x_1x_2^{3}x_3^{3}x_4^{6}x_5^{6} + x_1x_2^{3}x_3^{5}x_4^{5}x_5^{5} + x_1x_2^{3}x_3^{6}x_4^{3}x_5^{6} + x_1x_2^{3}x_3^{6}x_4^{6}x_5^{3}+  x_1x_2^{5}x_3^{3}x_4^{5}x_5^{5}\\
&\quad  + x_1x_2^{5}x_3^{5}x_4^{3}x_5^{5} + x_1x_2^{5}x_3^{5}x_4^{5}x_5^{3}+ x_1x_2^{6}x_3^{3}x_4^{3}x_5^{6} + x_1x_2^{6}x_3^{3}x_4^{6}x_5^{3} + x_1x_2^{6}x_3^{6}x_4^{3}x_5^{3},\\
A_4 &=x_1x_2^{4}x_3^{3}x_4^{3}x_5^{6} + x_1x_2^{4}x_3^{3}x_4^{6}x_5^{3} + x_1x_2^{4}x_3^{6}x_4^{3}x_5^{3}.
\end{array}$$
This relation implies that $X_1$  is strictly inadmissible. By a similar technique, we obtain
$$ 
X_2 = Sq^1(B_1) + Sq^2(B_2) + Sq^8(B_8) + x_1x_2^{3}x_3^{5}x_4^{6}x_5^{6} + x_1x_2^{3}x_3^{6}x_4^{5}x_5^{6}\ {\rm modulo}\,(\mathscr P_5^-(3,3,3)),$$
where 
$$ \begin{array}{ll}
B_1 &= x_1^2x_2^{5}x_3^{3}x_4^{5}x_5^{5} +  x_1^2x_2^{5}x_3^{5}x_4^{3}x_5^{5} + x_1^2x_2^{5}x_3^{5}x_4^{5}x_5^{3},\\
B_2&= x_1x_2^{3}x_3^{3}x_4^{6}x_5^{6} + x_1x_2^{3}x_3^{5}x_4^{5}x_5^{5} + x_1x_2^{3}x_3^{6}x_4^{3}x_5^{6}+ x_1x_2^{6}x_3^{3}x_4^{3}x_5^{6} + x_1x_2^{6}x_3^{3}x_4^{6}x_5^{3} + x_1x_2^{6}x_3^{6}x_4^{3}x_5^{3},\\
B_8&=  x_1x_2^{3}x_3^{3}x_4^{3}x_5^{3},
\end{array}$$
 $$ X_3 = Sq^1(C_1) + Sq^2(C_2) + Sq^8(C_8) + x_1^3x_2^{5}x_3x_4^{6}x_5^{6} + x_1^3x_2^{5}x_3^{2}x_4^{5}x_5^{6}\ {\rm modulo}\,(\mathscr P_5^-(3,3,3)),$$
where 
$$ \begin{array}{ll}
C_1 &= x_1^3x_2^{5}x_3x_4^{5}x_5^{6} + x_1^3x_2^{5}x_3x_4^{6}x_5^{5} + x_1^5x_2^{3}x_3^{2}x_4^{5}x_5^{5} + x_1^5x_2^{5}x_3^{2}x_4^{3}x_5^{5} + x_1^5x_2^{5}x_3^{2}x_4^{5}x_5^{3},\\
C_2&= x_1^3x_2^{3}x_3x_4^{6}x_5^{6} + x_1^3x_2^{5}x_3x_4^{5}x_5^{5} + x_1^3x_2^{6}x_3x_4^{3}x_5^{6} + x_1^3x_2^{6}x_3x_4^{6}x_5^{3} + x_1^6x_2^{3}x_3x_4^{3}x_5^{6}\\
&\quad + x_1^6x_2^{3}x_3x_4^{6}x_5^{3} + x_1^6x_2^{6}x_3x_4^{3}x_5^{3},\\
C_8&= x_1^3x_2^{3}x_3x_4^{3}x_5^{3}.
\end{array}$$
The lemma is proved.
\end{proof}

Now we denote by $\mathcal C$ the set of the following monomials: 

\begin{center}
\begin{tabular}{lllr}
$ x_1^{3}x_2^{12}x_3x_4^{2}x_5^{3}$ & $ x_1^{3}x_2^{12}x_3x_4^{3}x_5^{2}$, & $ x_1^{3}x_2^{12}x_3^{3}x_4x_5^{2}$, & \multicolumn{1}{l}{$ x_1^{3}x_2^{4}x_3x_4^{2}x_5^{11}$,} \\
$ x_1^{3}x_2^{4}x_3x_4^{11}x_5^{2}$, & $ x_1^{3}x_2^{4}x_3^{11}x_4x_5^{2}$, & $ x_1^{7}x_2^{8}x_3x_4^{2}x_5^{3}$, & \multicolumn{1}{l}{$ x_1^{7}x_2^{8}x_3x_4^{3}x_5^{2}$,} \\
$ x_1^{7}x_2^{8}x_3^{3}x_4x_5^{2}$, & $ x_1^{3}x_2^{4}x_3x_4^{3}x_5^{10}$, & $ x_1^{3}x_2^{4}x_3x_4^{10}x_5^{3}$, & \multicolumn{1}{l}{$ x_1^{3}x_2^{4}x_3^{3}x_4x_5^{10}$,} \\
$ x_1^{3}x_2^{4}x_3^{3}x_4^{9}x_5^{2}$, & $ x_1^{3}x_2^{4}x_3^{9}x_4^{2}x_5^{3}$, & $ x_1^{3}x_2^{4}x_3^{9}x_4^{3}x_5^{2}$. &  
\end{tabular}
\end{center}
A direct computation shows that $\overline{\Phi}^+(\mathscr B_4(3,3,1,1)) \cup \mathcal C$  is the set of $196$ monomials: $\mathcal Y_j:=\mathcal Y_{21, j},\, 401\leq j\leq 596$ (see Sect.\ref{s5.3}.)

\begin{propo}\label{md21.2}
Under the above notations, the $\mathbb Z/2$-vector space $Q\mathscr P_5^+(3,3,1,1)$ is spanned by the set 
$$[\overline{\Phi}^+(\mathscr B_4(3,3,1,1)) \cup \mathcal C].$$
\end{propo}

\begin{proof}
Let $X$ be an admissible monomial in $\mathscr P_5$ such that $\omega(X) = (3,3,1,1).$ Then $X = X_{(\{k, \ell\},\,5)}Y^2$ with $1\leq k < \ell\leq 5$ and $Y$ a monomial of degree $9$ in $\mathscr P_5.$ Since $X$  is admissible, according to Theorem \ref{dlKS}, $Y\in \mathscr B_5(3,1,1).$

A direct computation shows that if $z\in \mathscr B_5^+(3,1,1),\ 1\leq k < \ell\leq 5,$ and $X_{(\{k, \ell\},\,5)}z^2\neq \mathcal Y_j,\,\forall j,\, 401\leq j\leq 596 $ then there exists a monomial $w$ which is given in one of Lemmas \ref{bd21.2}- \ref{bd21.4} such that $X_{(\{k, \ell\},\,5)}z^2 = wz_1^{2^a}$ with a monomial $z_1\in \mathscr P_5$ and $a = {\rm max}\{m\in\mathbb Z:\, \omega_m(w) >0\}.$ By Theorem \ref{dlKS}, $X_{(\{k, \ell\},\,5)}z^2$ is inadmissible. Since $X = X_{(\{k, \ell\},\,5)}Y^2$ with $Y\in \mathscr B_5(3,1,1)$  and $X$ is admissible, one can see that $X = \mathcal Y_j,\, 401\leq j\leq 596.$ The lemma follows. 
\end{proof}

By a direct computation, we see that 
$$ \overline{\Phi}^+(\mathscr B_4(3,3,3)) \cup \{x_1^{3}x_2^{4}x_3x_4^{5}x_5^{6},\, x_1^{3}x_2^{5}x_3x_4^{6}x_5^{4},\, x_1^{3}x_2^{5}x_3^6x_4^{3}x_5^{4}\}$$ 
is the set consisting of $70$ monomials: $\mathcal Y_t:= \mathcal Y_{21,\,t},\, 597\leq t\leq 666$ (see Sect.\ref{s5.3}.)

\begin{propo}\label{md21.3}
The $\mathbb Z/2$-vector space $Q\mathscr P_5^+(3,3,3)$ is spanned by the set $\{[\mathcal Y_t]:\, 597\leq t\leq 666\}.$
\end{propo}

\begin{proof}
Let $u$ be an admissible monomial in $\mathscr P_5$ such that $\omega(u) = (3,3,3).$ Then $u = x_ix_jx_{\ell}y^2$ with $1\leq  i<j <\ell\leq 5$ and $y\in \mathscr{B}_5(3,3).$

By a direct computation, we can verify that for any $X\in \mathscr{B}_5(3,3),\, 1\leq  i<j <\ell\leq 5,$ such that $x_ix_jx_{\ell}X^2\neq \mathcal Y_t,\, \forall t,\, 597\leq t\leq 666,$ there is a monomial $z$ which is given in one of Lemmas \ref{bd21.2}- \ref{bd21.4} such that $x_ix_jx_{\ell}X^2 = zw^{2^b}$ with suitable monomial $w\in \mathscr P_5$ and $b = {\rm max}\{r\in\mathbb Z:\, \omega_r(z) >0\}.$ Then, according to Theorem \ref{dlKS}, $x_ix_jx_{\ell}X^2$ nadmissible. Since $u = x_ix_jx_{\ell}y^2$ is admissible and $y\in \mathscr B_5(3,3),$  one gets $u = \mathcal Y_t,$ for some $t.$ This proves the proposition.
\end{proof}

\begin{proof}[{\it Proof of Proposition \ref{md21.1}}]
From Propositions \ref{md21.2} and \ref{md21.3}, the space ${\rm Ker}((\widetilde {Sq^0_*})_{(5,21)})\cap (Q\mathscr P_5^+)_{21}$  is spanned by the set $\{[\mathcal Y_i:=\mathcal Y_{21,\,i}]:\ 401\leq i\leqslant 666\}.$  Futhermore,  this set is linearly independent in $(Q\mathscr P_5)_{21}$ . Indeed, suppose there is a linear relation
$$ \mathcal S = \sum\limits_{401\leq i \leqslant 666}\gamma_i\mathcal Y_i= 0\ {\rm modulo}(\mathcal A_2^+\mathscr P_5),$$ where $\gamma_i\in\mathbb Z/2.$ Let $\mathscr H$ be a sequence of non-negative integers and $\gamma_h\in \mathbb Z/2$ for $h\in \mathscr H.$  Denote $\gamma_{\mathscr H} = \sum_{h\in\mathscr H}\gamma_h\in\mathbb Z/2.$ 
Based on Theorem \ref{dlsig} and Proposition \ref{mdPS}, for $(k;\mathscr K) \in \mathcal N_5$, we explicitly compute $\pi_{(k;\mathscr K)}(\mathcal S)$ in terms of a given minimal set of $\mathcal A_2$-generators in $\mathscr P_4\ ({\rm modulo}(\mathcal A_2^+\mathscr P_{4})).$
Computing directly from the relations $\pi_{(k; p)}(\mathcal S)\equiv 0,\, 1\leq k <p \leq 5,$ one gets $\gamma_h = 0,\ h\in \mathscr H,$
where the set $\mathscr H = \{401, 402, \ldots, 425, 434, 435, 439, 441, 442, 443, 446, 451, 453, 454, 457, 458,460, 461, 462,\\ 463, 465, 467, 470, 471, \ldots, 477, 479, 483, 487, 489, 497, 498, 500, 501, 502, 504, 505, 506, 507, 510, 513,\\
 515, 517, 519, 520, 523, 527, 532, 533, 540, 542, 543,  560, 561, 562, 565, 566, 567, 568, 570, 573, 575, 577, \\
578, 579, 580, 581, 585, 597, 598, 599, 600, 601, 602, 604, 606, 607, 609, 610,  613, 615, 616, 617, 620, 622,\\
 629, 630, 632, 633, \ldots, 647, 650, 651, 652, 653, 655, 656, 657, 658, 659,  660, 662, 663\},$ and 
$$ \left\{\begin{array}{ll}
\gamma_{430} = \gamma_{593},\\
\gamma_{431} = \gamma_{594},
\gamma_{536} = \gamma_{541},\\
\gamma_{603} = \gamma_{664},\ \gamma_{h} =  \gamma_{426},\ h \in \mathbb I_1,\\ 
\gamma_{608} = \gamma_{627},\ \gamma_{h} =  \gamma_{427},\ h\in \mathbb I_2,\\
\gamma_{619} = \gamma_{626},\ \gamma_{h} = \gamma_{432},\ h = 468, 491, 518, 595,\\
\gamma_{621} = \gamma_{665},\ \gamma_{h} = \gamma_{433},\ h = 596, 623, 666,\\
\gamma_{h} = \gamma_{436},\ h = 437, 438, 469, 492, 495, 496, 499, 534, 539, 545, 614, 654,\\
\gamma_{\{450, 595, 661\}} = \gamma_{\{522, 648, 661\}} = \gamma_{\{531, 648, 661\}} = \gamma_{\{549, 654, 661\}} = 0,\\
\gamma_{\{555, 595, 654\}} = \gamma_{\{618, 626, 628\}} = 0,\\
\gamma_{\{582, 654, 661\}} = \gamma_{\{583, 654, 661\}} = \gamma_{\{584, 654, 661\}} = 0,\\
\gamma_{\{429, 648, 595, 661\}} = \gamma_{\{455, 595, 654, 661\}} = \gamma_{\{456, 654, 661, 666\}} = \gamma_{\{486, 648, 594, 654\}} = 0,\\
\gamma_{\{493, 664, 665, 594\}} = \gamma_{\{503, 654, 552, 666\}} = \gamma_{\{512, 648, 595, 661\}} = \gamma_{\{529, 664, 665, 594\}} = 0,\\
\gamma_{\{544, 654, 661, 595\}} = \gamma_{\{547, 654, 661, 595\}} = \gamma_{\{548, 654, 661, 666\}} = \gamma_{\{550, 654, 661, 595\}} = 0,\\
\gamma_{\{551, 648, 594, 654\}} = \gamma_{\{557, 654, 661, 666\}} = \gamma_{\{558, 654, 661, 648\}}=\gamma_{\{592, 648, 595, 661\}} = 0,\\
\gamma_{\{481, 648, 664, 665, 594\}} = \gamma_{\{526, 648, 664, 665, 594\}} = \gamma_{\{554, 654, 628, 595, 626\}} =0,\\
  \gamma_{\{624, 626, 628, 627, 664\}} = \gamma_{\{459, 628, 648, 595, 626, 666\}}=  0,\\
\gamma_{\{488, 628, 595, 626, 648, 594\}} = \gamma_{\{535, 627, 654, 595, 661, 626\}} = 0,\\
\gamma_{\{546, 661, 666, 648, 594, 552\}} = \gamma_{\{612, 628, 627, 665, 666, 664\}} = 0,\\
\gamma_{\{485, 627, 654, 648, 595, 661, 626\}} = \gamma_{\{490, 626, 628, 627, 665, 666, 664\}} = 0,\\
\gamma_{\{509, 666, 552, 654, 628, 595, 626\}} = \gamma_{\{524, 666, 552, 654, 628, 595, 626\}} = 0,\\
\gamma_{\{538, 627, 654, 648, 595, 661, 626\}} = \gamma_{\{556, 654, 628, 648, 595, 661, 626\}} = 0,\\
\gamma_{\{559, 648, 626, 628, 627, 666, 661\}} = \gamma_{\{625, 626, 628, 627, 665, 666, 664\}} = 0,\\
\gamma_{\{521, 654, 666, 648, 594, 552, 628, 664, 665\}} = \gamma_{\{553, 654, 648, 594, 626, 628, 627,666, 664\}} = 0.
\end{array}\right.$$
Here $ \mathbb I_1 = \{440, 447, 448, 464, 466, 478, 484, 508, 514, 516, 586, 605,,648\},$ and \\
$ \mathbb I_2 = \{428, 444, 445, 449, 452, 480, 482, 494, 511, 525, 528, 530, 537, 563, 564, 569, 571, 572, 574, 576, 587,\\ 588, 589, 590, 591, 611, 631, 649, 661\},$

Combining the above computations and the relations $$\pi_{(1, (2; j))}(\mathcal S)\equiv 0,\ j = 3, 4, 5,\ \mbox{and}\ \pi_{(1, (3; 4))}(\mathcal S)\equiv 0,$$ we obtain $\gamma_i = 0,\ \forall i,\, 401\leq i\leq 666.$ This finishes the proof. 

\end{proof}

\subsubsection{ The case $t = 2$}\label{s3.2.2}

For $t =2,$ we have $13.2^t-5= 47$ and $\mu(47) = 3 < 5.$ Since Kameko's operation
$(\widetilde {Sq^0_*})_{(5,47)}: (Q\mathscr P_5)_{47} \to (Q\mathscr P_5)_{21}$ is an epimorphism of the $\mathbb Z/2GL_5$-modules, hence $$(Q\mathscr P_5)_{47} = {\rm Ker}((\widetilde {Sq^0_*})_{(5,47)})\bigoplus  (Q\mathscr P_5)_{21}.$$ Thus, we need to compute ${\rm Ker}((\widetilde {Sq^0_*})_{(5,47)}).$

\begin{nx}\label{nx1}
If $Y\in\mathscr B_5(47)$  and $[Y]\in {\rm Ker}((\widetilde {Sq^0_*})_{(5,47)}),$ then $\omega_1(Y) = 3.$
\end{nx}

Indeed, we see that $z = x_1^{31}x_2^{15}x_3$ is the minimal spike of degree $47$ in $\mathscr P_5.$ By Proposition \ref{PS},  $z\in\mathscr B_5(47)$. Since $[Y]\neq 0,$ by Theorem \ref{dlsig}, either $\omega_1(Y) = 3$ or $\omega_1(Y) = 5.$ If $\omega_1(Y) = 5,$ then $Y = X_{(\emptyset,\, 5)}Z^2$ with $Z$ a monomial of degree $21$ in $\mathscr P_5.$ Since $Y$ is admissible, by Theorem \ref{dlKS}, $Z\in \mathscr{B}_5(21).$ So, we have $(\widetilde {Sq^0_*})_{(5,47)}([Y]) = [Z]\neq [0].$ This contradicts the face that $[Y]\in {\rm Ker}((\widetilde {Sq^0_*})_{(5,47)});$ hence we get $\omega_1(Y) = 3.$ 

From Remark \ref{nx1}, we have $Y = x_kx_{\ell}x_mg^2$ with $1 \leq k < \ell < m\leq 5$ and $g\in\mathscr B_5(22).$ Thus, to determine ${\rm Ker}((\widetilde {Sq^0_*})_{(5,47)}),$ we need to compute all the admissible monomials of degree $22$ in the $\mathcal A_2$-module $\mathscr P_5.$

\subsection*{Computation of $(Q\mathscr P_5)_{22}$}

We consider the following weight vectors:
$$ \omega_{(1)} = (2,2,2,1),\ \omega_{(2)} = (2,4,1, 1),\ \omega_{(3)} = (2,4,3),\  \omega_{(4)} = (4,3,1,1),\  \omega_{(5)} = (4,3,3).$$
It is easy to see that $\deg \omega_{(i)} = 22,\ 1\leq i\leq 5.$ By Proposition \ref{PS}, $x_1^{15}x_2^7$ is the minimal spike in $\mathscr B_5(22)$ and $\omega(x_1^{15}x_2^7) = \omega_{(1)}.$ Let $u$ be an admissible monomial of degree $22$ in $\mathscr P_5.$ Then $[u]\neq [0]$ and by Theorem \ref{dlsig}, either $\omega_1(u) = 2$ or $\omega_1(u) = 4.$ Since $u\in\mathscr B_5(22),$ by Theorem \ref{dlKS}, if $\omega_1(u) = 2,$ then $u = X_{(\{i,j,k\},\,5)}y^2$ with $y\in \mathscr B_5(10)$ and $1\leqslant i < j <k\leq 5.$ According to T\'in \cite{N.T1}, $\omega(y)$ is one of the sequences $ (2,2,1),\  (4,1,1),$ and $(4,3).$ If $\omega_1(u) = 4,$ then $u = X_{(\{i\},\,5)}y_1^2$ with $y_1$ a monomial of degree $9$ in $\mathscr P_5$ and $1\leq i \leq 5.$  By T\'in \cite{N.T1}, either $\omega(y_1) = (3,1,1)$ or $\omega(y_1) = (3,3).$ Hence, we have the following.

\begin{nx}\label{nx2} 
If $u\in\mathscr B_5(22),$ then $\omega(u)$ is one of the sequences $\omega_{(t)},\ 1\leq t\leq 5.$
\end{nx}

As it is known, $(Q\mathscr P_5)_{22} = (Q\mathscr P_5^0)_{22}\bigoplus (Q\mathscr P_5^+)_{22}.$ By Sum \cite{N.S3}, $Q\mathscr P_5^+$ has dimension $72$ in degree $22.$ Then, combining the rerults in \cite{F.P1} and \cite{M.K} with fact that $(Q\mathscr P_5^0)_{22} = \bigoplus_{1\leq s\leq 4}\bigoplus_{1\leq u\leq \binom{5}{u}}(Q\mathscr P_s^+)_{22},$ we deduce that $ \dim(Q\mathscr P_5^0)_{22} = \binom 52.2 + \binom 53.8 + \binom 54.72 = 460,$ and that $\mathscr B_5^0(22) = \overline{\Phi}^0(\mathscr B_4(22)) =  \{\mathcal Y_{22,\, t}\ :\ 1\leq t\leq 460\},$ where the monomials $\mathcal Y_{22,\,t},\ 1\leq t\leq 460,$ are determined in Sect.\ref{s5.4}.

Next, we compute $(Q\mathscr P_5^+)_{22}.$ For $r,\ k\in\mathbb N$ and $1\leq k\leq 5,$ we denote
$$ \overline{\mathscr B}(k, 22) := \big\{x_k^{2^{r} - 1}\rho_{(k,\,5)}(x)\in (\mathscr P_5)_{22}\ :\ x\in \mathscr B_4(23-2^{r}),\ \alpha(27-2^{r})\leq 4\big\}.$$
 By Mothebe and Uys \cite{M.M2}, $\overline{\mathscr B}(5, 22)\subseteq \mathscr B_5(22),\ 1\leq k\leq 5.$ We set 
$$ \overline{\mathscr B}(k, \omega_{(t)})  := \overline{\mathscr B}(k, 22) \cap \mathscr P_5(\omega_{(t)}),\ \overline{\mathscr B}^+(k, \omega_{(t)})  := \overline{\mathscr B}(k, \omega_{(t)}) \cap (\mathscr P^+_5)_{22},$$
for all $1\leq t, \ k\leq 5.$

By a simple computation, we find that $$\overline{\Phi}^+(\mathscr B_4(\omega_{(1)})\bigcup \big(\bigcup_{1\leq k\leq 5} \overline{\mathscr B}^+(k, \omega_{(1)})\big)$$ is the set of $31$ admissible monomials: $\mathcal Y_{22,\, i},\ 461\leq i\leq 491$ (see Sect.\ref{s5.5}.) 

\newpage
Denote by $\mathcal D$ is  the set of the following monomials:

\begin{center}
   \begin{tabular}{lll}
$\mathcal Y_{22,\,492}=x_1x_2x_3^{6}x_4^{6}x_5^{8}$, & $\mathcal Y_{22,\,493}=x_1x_2x_3^{6}x_4^{10}x_5^{4}$, & $\mathcal Y_{22,\,494}=x_1x_2^{2}x_3^{3}x_4^{4}x_5^{12}$, \\
 \multicolumn{1}{l}{$\mathcal Y_{22,\,495}=x_1x_2^{2}x_3^{3}x_4^{12}x_5^{4}$,} &
$\mathcal Y_{22,\,496}=x_1x_2^{2}x_3^{4}x_4^{9}x_5^{6}$, & $\mathcal Y_{22,\,497}=x_1x_2^{2}x_3^{5}x_4^{8}x_5^{6}$,\\
  $\mathcal Y_{22,\,498}=x_1x_2^{3}x_3^{2}x_4^{4}x_5^{12}$, & \multicolumn{1}{l}{$\mathcal Y_{22,\,499}=x_1x_2^{3}x_3^{2}x_4^{12}x_5^{4}$,} &
$\mathcal Y_{22,\,500}=x_1x_2^{3}x_3^{4}x_4^{8}x_5^{6}$,\\
  $\mathcal Y_{22,\,501}=x_1x_2^{3}x_3^{6}x_4^{4}x_5^{8}$, & $\mathcal Y_{22,\,502}=x_1x_2^{3}x_3^{6}x_4^{8}x_5^{4}$, & \multicolumn{1}{l}{$\mathcal Y_{22,\,503}=x_1^{3}x_2x_3^{2}x_4^{4}x_5^{12}$,} \\
$\mathcal Y_{22,\,504}=x_1^{3}x_2x_3^{2}x_4^{12}x_5^{4}$, & $\mathcal Y_{22,\,505}=x_1^{3}x_2x_3^{4}x_4^{8}x_5^{6}$, & $\mathcal Y_{22,\,506}=x_1^{3}x_2x_3^{6}x_4^{4}x_5^{8}$,\\
  \multicolumn{1}{l}{$\mathcal Y_{22,\,507}=x_1^{3}x_2x_3^{6}x_4^{8}x_5^{4}$,} &
$\mathcal Y_{22,\,508}=x_1^{3}x_2^{5}x_3^{2}x_4^{4}x_5^{8}$, & $\mathcal Y_{22,\,509}=x_1^{3}x_2^{5}x_3^{2}x_4^{8}x_5^{4}$,\\
  $\mathcal Y_{22,\,510}=x_1^{3}x_2^{5}x_3^{8}x_4^{2}x_5^{4}$. &  &
\end{tabular}
\end{center}

\begin{propo}\label{md22-1}
$\mathscr B^+_5(\omega_{(1)}) = \overline{\Phi}^+(\mathscr B_4(\omega_{(1)})\cup \mathscr B^+(5, \omega_{(1)}) \cup \mathcal D.$ 
\end{propo}

In order to prove the proposition, we need some lemmas. 

\begin{lema}\label{bd22.2}
The following monomials are strictly inadmissible:
\begin{enumerate}
\item[(i)] $x_1^2x_jx_k^2x_l^3x_m^6, \  x_1^6x_jx_kx_l^2x_m^4,\ l < m,\ x_1^2x_jx_kx_l^4x_m^6.$\\
 Here $(j, k, l, m)$ is a permutation of $(2, 3, 4, 5);$

\item[(ii)] 

    \begin{tabular}[t]{llllr}
    $ x_1x_2^{2}x_3^{4}x_4^{6}x_5,$ & $ x_1x_2^{2}x_3^{6}x_4x_5^{4},$ & $ x_1x_2^{2}x_3^{6}x_4^{4}x_5,$ & $ x_1x_2^{6}x_3^{2}x_4x_5^{4},$ \\
 $ x_1x_2^{6}x_3^{2}x_4^{4}x_5,$ & $ x_1^{3}x_2^{3}x_3^{4}x_4^{2}x_5^{2},$ &  $ x_1^{3}x_2^{4}x_3^{2}x_4^{2}x_5^{3},$ & $ x_1^{3}x_2^{4}x_3^{2}x_4^{3}x_5^{2},$\\
  $ x_1^{3}x_2^{4}x_3^{3}x_4^{2}x_5^{2},$ & $ x_1^{3}x_2^{2}x_3x_4^{2}x_5^{6},$ & $ x_1^{3}x_2^{2}x_3x_4^{6}x_5^{2},$ & $ x_1^{3}x_2^{2}x_3^{2}x_4x_5^{6},$\\
  $ x_1^{3}x_2^{2}x_3^{2}x_4^{6}x_5,$& $ x_1^{3}x_2^{2}x_3^{6}x_4x_5^{2},$ & $ x_1^{3}x_2^{2}x_3^{6}x_4^{2}x_5,$ & $ x_1^{3}x_2^{6}x_3x_4^{2}x_5^{2},$\\
  $ x_1^{3}x_2^{6}x_3^{2}x_4x_5^{2},$ & $ x_1^{3}x_2^{6}x_3^{2}x_4^{2}x_5,$ &  $ x_1^{2}x_2^{3}x_3^{3}x_4^{3}x_5^{3},$ & $ x_1^{3}x_2^{2}x_3^{3}x_4^{3}x_5^{3},$ \\
 $ x_1^{3}x_2^{3}x_3^{2}x_4^{3}x_5^{3},$ &   $ x_1^{3}x_2^{3}x_3^{3}x_4^{2}x_5^{3},$ & $ x_1^{3}x_2^{3}x_3^{3}x_4^{3}x_5^{2}.$ &  
    \end{tabular}%
 \end{enumerate}
\end{lema}

\begin{proof}
We prove the lemma for the monomials $u = x_1^2x_jx_k^2x_l^3x_m^6,$ and $v = x_1x_2^{2}x_3^{4}x_4^{6}x_5.$ The others can be proved by a similar computation. We have $\omega(u) = (2,4,1)$ and $\omega(v) = (2,2,2).$ By a simple computation, one gets
$$ \begin{array}{ll}
u &= x_1x_j^2x_k^2x_l^3x_m^6 + Sq^1(x_1x_jx_k^2x_l^3x_m^6) \ {\rm modulo}(\mathscr P_5^-(2,4,1)),\\
v &= x_1x_2x_3^{4}x_4^{6}x_5^2 + x_1x_2^2x_3^{4}x_4^{5}x_5^2 + Sq^1(f_1) + Sq^2(f_2) \ {\rm modulo}(\mathscr P_5^-(2,2,2)),
\end{array}$$
where $f_1 = x_1^2x_2x_3^{4}x_4^{5}x_5$ and $f_2 = x_1x_2x_3^{4}x_4^{5}x_5.$  Hence, $u$ and $v$ are strictly inadmissible. The lemma follows.
\end{proof}

The following lemma can easily be proved by a direct computation.

\begin{lema}\label{bd22.1}
If $(i, j, k, l, m)$ is a permutation of $(1, 2, 3, 4, 5),$ then the following monomials are strictly inadmissible:
\begin{enumerate}
\item[(i)] $x_i^6x_jx_k^7,\ x_i^2x_j^5x_k^7,\
x_i^3x_j^4x_k^7,\  x_i^2x_jx_k^2x_l^2x_m^7,\\  x_i^2x_j^5x_k^2x_l^2x_m^3, x_i^2x_j^4x_k^2x_l^3x_m^3, x_i^2x_jx_k^2x_l^2x_m^3, \ i < j;$  
\item[(ii)] $x_ix_j^6x_k^3x_l^2x_m^2,\ j < k,\ x_i^3x_j^6x_k^5,\ x_i^6x_j^3x_k^5,\ x_i^2x_jx_kx_l^3x_m^3,\ x_i^2x_jx_kx_l^2 ,\ i < j < k;$
\item[(iii)] $x_i^2x_jx_k^4x_l^3x_m^4,\  x_i^2x_j^4x_kx_l^3x_m^4,\ i < j < k,\ l <m.$
\end{enumerate}
\end{lema}

\begin{lema}\label{bd22.4}
The following monomials are strictly inadmissible:

\begin{center}
    \begin{tabular}[t]{lllll}
    $ x_1x_2^{2}x_3^{2}x_4^{7}x_5^{10},$ & $ x_1x_2^{2}x_3^{4}x_4^{3}x_5^{12},$ & $ x_1x_2^{2}x_3^{4}x_4^{11}x_5^{4},$ & $ x_1x_2^{2}x_3^{7}x_4^{2}x_5^{10},$\\
  $ x_1x_2^{2}x_3^{7}x_4^{8}x_5^{4},$ & $ x_1x_2^{2}x_3^{7}x_4^{10}x_5^{2},$ &
    $ x_1x_2^{2}x_3^{12}x_4^{3}x_5^{4},$ & $ x_1x_2^{6}x_3x_4^{6}x_5^{8},$\\
  $ x_1x_2^{6}x_3x_4^{10}x_5^{4},$ & $ x_1x_2^{6}x_3^{3}x_4^{6}x_5^{6},$ & $ x_1x_2^{6}x_3^{6}x_4^{3}x_5^{6},$ & $ x_1x_2^{6}x_3^{6}x_4^{6}x_5^{3},$ \\
    $ x_1x_2^{6}x_3^{9}x_4^{2}x_5^{4},$ & $ x_1x_2^{7}x_3^{2}x_4^{2}x_5^{10},$ & $ x_1x_2^{7}x_3^{2}x_4^{8}x_5^{4},$ & $ x_1x_2^{7}x_3^{2}x_4^{10}x_5^{2},$\\
  $ x_1x_2^{7}x_3^{8}x_4^{2}x_5^{4},$ & $ x_1x_2^{7}x_3^{10}x_4^{2}x_5^{2},$ &
    $ x_1^{3}x_2^{3}x_3^{4}x_4^{4}x_5^{8},$ & $ x_1^{3}x_2^{3}x_3^{4}x_4^{8}x_5^{4},$\\
  $ x_1^{3}x_2^{4}x_3x_4^{2}x_5^{12},$ & $ x_1^{3}x_2^{4}x_3x_4^{10}x_5^{4},$ & $ x_1^{3}x_2^{4}x_3^{4}x_4^{4}x_5^{7},$ & $ x_1^{3}x_2^{4}x_3^{4}x_4^{5}x_5^{6},$ \\
    $ x_1^{3}x_2^{4}x_3^{4}x_4^{7}x_5^{4},$ & $ x_1^{3}x_2^{4}x_3^{5}x_4^{4}x_5^{6},$ & $ x_1^{3}x_2^{4}x_3^{5}x_4^{6}x_5^{4},$ & $ x_1^{3}x_2^{4}x_3^{7}x_4^{4}x_5^{4},$ \\
$ x_1^{3}x_2^{4}x_3^{9}x_4^{2}x_5^{4},$ & $ x_1^{3}x_2^{5}x_3^{4}x_4^{4}x_5^{6},$ &
    $ x_1^{3}x_2^{5}x_3^{4}x_4^{6}x_5^{4},$ & $ x_1^{3}x_2^{5}x_3^{6}x_4^{4}x_5^{4},$\\
 $ x_1^{3}x_2^{7}x_3^{4}x_4^{4}x_5^{4},$ & $ x_1^{3}x_2^{12}x_3x_4^{2}x_5^{4},$ & $ x_1^{7}x_2x_3^{2}x_4^{2}x_5^{10},$ & $ x_1^{7}x_2x_3^{2}x_4^{8}x_5^{4},$ \\
    $ x_1^{7}x_2x_3^{2}x_4^{10}x_5^{2},$ & $ x_1^{7}x_2x_3^{8}x_4^{2}x_5^{4},$ & $ x_1^{7}x_2x_3^{10}x_4^{2}x_5^{2},$ & $ x_1^{7}x_2^{3}x_3^{4}x_4^{4}x_5^{4},$\\
  $ x_1^{7}x_2^{8}x_3x_4^{2}x_5^{4},$ & $ x_1^{7}x_2^{9}x_3^{2}x_4^{2}x_5^{2}.$ 
    \end{tabular}%
\end{center}
\end{lema}

 \begin{proof}
We prove the lemma for the monomials $x = x_1x_2^{2}x_3^{4}x_4^{3}x_5^{12},$ and $y = x_1^{3}x_2^{3}x_3^{4}x_4^{4}x_5^{8}.$ The others can be proved by a similar computation. By a direct computation using the Cartan formula, we have
$$\begin{array}{ll}
x &=  x_1x_2^{2}x_3^{3}x_4^{4}x_5^{12}+
 x_1x_2^{2}x_3^{4}x_4^{2}x_5^{13}+
 x_1x_2^{2}x_3^{2}x_4^{4}x_5^{13} + x_1x_2x_3^{6}x_4^{6}x_5^{8}+ x_1x_2x_3^{6}x_4^{4}x_5^{10}\\
\medskip
 &\quad + x_1x_2x_3^{4}x_4^{6}x_5^{10}+ Sq^1(f_1) + Sq^2(f_2) + Sq^4(f_4) \ {\rm modulo}(\mathscr P_5^-(\omega_{(1)})),\, \mbox{ where}\\
f_1 &=  x_1x_2^{4}x_3^{3}x_4^{4}x_5^{9}+
 x_1x_2^{4}x_3^{3}x_4^{5}x_5^{8}+
 x_1x_2^{4}x_3^{4}x_4^{3}x_5^{9}+
 x_1x_2^{4}x_3^{4}x_4^{5}x_5^{7}\\
&\quad + x_1x_2^{4}x_3^{5}x_4^{3}x_5^{8}+ x_1x_2^{4}x_3^{5}x_4^{4}x_5^{7}+
 x_1^{2}x_2x_3^{4}x_4^{5}x_5^{9}+
 x_1^{2}x_2x_3^{5}x_4^{4}x_5^{9}\\
\medskip
&\quad +   x_1^{2}x_2x_3^{5}x_4^{5}x_5^{8}+ x_1^{4}x_2^{4}x_3^{3}x_4^{3}x_5^{7},\\
f_2&=  x_1x_2x_3^{4}x_4^{5}x_5^{9}+
 x_1x_2x_3^{5}x_4^{4}x_5^{9}+
 x_1x_2x_3^{5}x_4^{5}x_5^{8}+
 x_1x_2^{2}x_3^{3}x_4^{4}x_5^{10}\\
 &\quad+  x_1x_2^{2}x_3^{3}x_4^{6}x_5^{8} +  x_1x_2^{2}x_3^{4}x_4^{3}x_5^{10}+
 x_1x_2^{2}x_3^{4}x_4^{6}x_5^{7}+
 x_1x_2^{2}x_3^{6}x_4^{3}x_5^{8}\\
 &\quad + x_1x_2^{2}x_3^{6}x_4^{4}x_5^{7}+
 x_1x_2^{4}x_3^{2}x_4^{2}x_5^{11} + x_1^{2}x_2^{4}x_3^{3}x_4^{3}x_5^{8}+
 x_1^{2}x_2^{4}x_3^{3}x_4^{4}x_5^{7}\\
\medskip
  &\quad + x_1^{2}x_2^{4}x_3^{4}x_4^{3}x_5^{7},\\
f_4 &= x_1x_2^{2}x_3^{2}x_4^{2}x_5^{11}+ x_1x_2^{2}x_3^{4}x_4^{4}x_5^{7}.
\end{array}$$
The above equalities show that $x$ is strictly inadmissible. By a similar computation, we obtain
$$ \begin{array}{ll}
y &= x_1^2x_2x_3^2x_4^{13}x_5^4 + x_1^2x_2x_3^3x_4^{12}x_5^4 + x_1^2x_2x_3^4x_4^{7}x_5^8 + x_1^2x_2x_3^4x_4^{13}x_5^2 +  x_1^{2}x_2x_3^{5}x_4^{6}x_5^{8}\\ 
&\quad  +  x_1^{2}x_2x_3^{6}x_4^{9}x_5^{4}+  x_1^{2}x_2x_3^{8}x_4^{7}x_5^{4} +  x_1^{2}x_2x_3^{12}x_4^{5}x_5^{2} +   x_1^{2}x_2^{3}x_3^{5}x_4^{8}x_5^{4}+ x_1^{2}x_2^{3}x_3^{8}x_4^{5}x_5^{4}\\
&\quad  + x_1^{2}x_2^{5}x_3^{2}x_4^{5}x_5^{8}+
 x_1^{2}x_2^{5}x_3^{2}x_4^{9}x_5^{4} + x_1^{2}x_2^{5}x_3^{3}x_4^{8}x_5^{4}+
 x_1^{2}x_2^{5}x_3^{4}x_4^{3}x_5^{8}+
 x_1^{2}x_2^{5}x_3^{4}x_4^{9}x_5^{2}\\
&\quad +  x_1^{2}x_2^{5}x_3^{8}x_4^{5}x_5^{2} +  x_1^{3}x_2x_3^{2}x_4^{12}x_5^{4}+
 x_1^{3}x_2x_3^{4}x_4^{6}x_5^{8} + x_1^{3}x_2x_3^{4}x_4^{10}x_5^{4}+
 x_1^{3}x_2x_3^{4}x_4^{12}x_5^{2} \\
&\quad +x_1^{3}x_2x_3^{4}x_4^{12}x_5^{2}+
 x_1^{3}x_2^{2}x_3^{4}x_4^{9}x_5^{4}+
 x_1^{3}x_2^{2}x_3^{8}x_4^{5}x_5^{4} \\ 
\medskip
&\quad+  Sq^1(g_1) + Sq^2(g_2)+ Sq^4(g_4) + Sq^8(x_1^{3}x_2^{3}x_3^{2}x_4^{4}x_5^{2})\ {\rm modulo}(\mathscr P_5^-(\omega_{(1)})),\,\mbox{ where}\\
g_1&=  x_1^{3}x_2^{3}x_3^{2}x_4^{9}x_5^{4}+
 x_1^{3}x_2^{3}x_3^{4}x_4^{9}x_5^{2}+
 x_1^{3}x_2^{3}x_3^{2}x_4^{5}x_5^{8}+
 x_1^{3}x_2^{3}x_3^{8}x_4^{5}x_5^{2}+ x_1^{5}x_2x_3^{2}x_4^{9}x_5^{4} \\
&\quad +x_1^{3}x_2x_3^{4}x_4^{9}x_5^{4}+
 x_1^{5}x_2x_3^{4}x_4^{9}x_5^{2}+
 x_1^{5}x_2x_3^{8}x_4^{5}x_5^{2}+ x_1^{3}x_2x_3^{8}x_4^{5}x_5^{4}+
 x_1^{5}x_2x_3^{6}x_4^{5}x_5^{4}\\
&\quad  + x_1^{3}x_2^{4}x_3^{5}x_4^{5}x_5^{4}+
 x_1^{5}x_2x_3^{5}x_4^{6}x_5^{4} + x_1^{5}x_2x_3^{3}x_4^{8}x_5^{4}+
 x_1^{3}x_2^{3}x_3^{3}x_4^{8}x_5^{4}+
 x_1^{5}x_2x_3^{2}x_4^{5}x_5^{8} \\
\medskip
&\quad + x_1^{5}x_2x_3^{4}x_4^{3}x_5^{8}+ x_1^{3}x_2^{3}x_3^{4}x_4^{3}x_5^{8},\\
g_2&=  x_1^{5}x_2^{3}x_3^{2}x_4^{8}x_5^{2}+
 x_1^{5}x_2^{3}x_3^{2}x_4^{6}x_5^{4}+
 x_1^{5}x_2^{3}x_3^{4}x_4^{6}x_5^{2}+
 x_1^{2}x_2^{3}x_3^{2}x_4^{9}x_5^{4} + x_1^{3}x_2^{2}x_3^{2}x_4^{9}x_5^{4}\\
 &\quad + x_1^{3}x_2x_3^{2}x_4^{10}x_5^{4}+
 x_1^{6}x_2x_3^{2}x_4^{7}x_5^{4}+
 x_1^{2}x_2^{3}x_3^{4}x_4^{9}x_5^{2}+ x_1^{3}x_2^{2}x_3^{4}x_4^{9}x_5^{2}+
 x_1^{3}x_2x_3^{4}x_4^{10}x_5^{2} \\
 &\quad  +  x_1^{6}x_2x_3^{4}x_4^{7}x_5^{2}+
 x_1^{2}x_2^{3}x_3^{8}x_4^{5}x_5^{2}+ x_1^{2}x_2^{3}x_3^{2}x_4^{5}x_5^{8}+
 x_1^{3}x_2^{2}x_3^{8}x_4^{5}x_5^{2}+
 x_1^{3}x_2^{2}x_3^{2}x_4^{5}x_5^{8}\\
 &\quad+ x_1^{3}x_2x_3^{8}x_4^{6}x_5^{2} + x_1^{6}x_2x_3^{6}x_4^{5}x_5^{2}+
 x_1^{3}x_2^{2}x_3^{6}x_4^{5}x_5^{4}+
 x_1^{3}x_2^{2}x_3^{5}x_4^{6}x_5^{4}+
 x_1^{3}x_2^{2}x_3^{3}x_4^{8}x_5^{4}\\
 &\quad +x_1^{2}x_2^{3}x_3^{3}x_4^{8}x_5^{4}+
 x_1^{6}x_2x_3^{2}x_4^{3}x_5^{8}+
 x_1^{3}x_2x_3^{2}x_4^{6}x_5^{8}+
 x_1^{3}x_2^{2}x_3^{4}x_4^{3}x_5^{8}+ x_1^{6}x_2x_3^{3}x_4^{6}x_5^{4}\\
 &\quad  + x_1^{2}x_2^{3}x_3^{4}x_4^{3}x_5^{8}+
 x_1^{2}x_2x_3^{10}x_4^{5}x_5^{2}+
 x_1^{2}x_2x_3^{2}x_4^{7}x_5^{8} + x_1^{2}x_2x_3^{2}x_4^{11}x_5^{4}+
 x_1^{2}x_2x_3^{3}x_4^{6}x_5^{8}\\
&\quad +   x_1^{2}x_2x_3^{3}x_4^{10}x_5^{4}+
\medskip
 x_1^{2}x_2x_3^{4}x_4^{11}x_5^{2}+ x_1^{2}x_2x_3^{6}x_4^{9}x_5^{2}+  x_1^{2}x_2x_3^{8}x_4^{7}x_5^{2},\\
g_4&=  x_1^{3}x_2^{3}x_3^{2}x_4^{6}x_5^{4}+
 x_1^{3}x_2^{3}x_3^{4}x_4^{6}x_5^{2}+
 x_1^{4}x_2x_3^{2}x_4^{7}x_5^{4} \\
&\quad + x_1^{4}x_2x_3^{4}x_4^{7}x_5^{2}+ x_1^{4}x_2x_3^{6}x_4^{5}x_5^{2}+ x_1^{4}x_2x_3^{3}x_4^{6}x_5^{4}.
\end{array}$$
The above relations imply that $y$ is also strictly inadmissible. The lemma is proved.
\end{proof}

\begin{proof}[{\it Proof of Proposition \ref{md22-1}}]
We denote by $\mathcal Y_t:= \mathcal Y_{22,\,t},\, 461\leq t\leq 510$ the admissible monomials in $\mathscr B_5^+(\omega_{(1)})$ (see Sect.\ref{s5.5}.)
For $x\in\mathscr B_5^+(\omega_{(1)}),$ we have $x = X_{\{i,j,k\}}y^2$ with $y$ a monomial of degree $10$ in $\mathscr P_5,$  and $1\leq i < j < k\leq 5.$ Since $x$ is admissible, by Theorem \ref{dlKS}, $y\in\mathscr B_5(2,2,1).$

Let $y_1\in \mathscr B_5(2,2,1)$ such that $X_{(\{i,j,k\},\,5)}y^2_1\in \mathscr P_5^+.$ By a direct computation, we see that if $X_{(\{i,j,k\},\,5)}y^2_1\neq \mathcal Y_t,$ for all $t,\ 461\leq t\leq 510,$ then there is a monomial $w$ which is given  in one of Lemmas \ref{bd22.1} - \ref{bd22.4} such that  $X_{(\{i,j,k\},\,5)}y^2_1 = wz^{2^u}$ with suitable monomial $z\in \mathscr P_5$ and $u = {\rm max}\{j\in \mathbb Z\ :\ \omega_j(w) > 0\}.$ By Theorem \ref{dlKS}, $X_{(\{i,j,k\},\,5)}y^2_1$ is inadmissible. Since $x  = X_{(\{i,j,k\},\,5)}y^2$ and $x$ is admissible, one gets $x = \mathcal Y_t$. This implies $Q\mathscr P_5^+(\omega_{(1)})$ is spanned by the set $\{[\mathcal Y_t:=\mathcal Y_{22,\,t}]_{\omega_{(1)}}:\, 461\leq t\leq 510\}.$

We now prove the set $\{[\mathcal Y_t]_{\omega_{(1)}}:\,461\leq t\leq 510\}$  is linearly independent in $Q\mathscr P_5(\omega_{(1)}).$ Suppose there is a linear relation 
\begin{equation}\label{thtt1}
\mathcal S = \sum\limits_{461\leq t\leq 510}\gamma_t\mathcal Y_t \equiv_{\omega_{(1)}} 0,
\end{equation}
where $\gamma_t\in\mathbb Z/2.$ From a result in \cite{N.S3}, $\dim Q\mathscr P_4^+(\omega_{(1)}) =  26,$  with the basis $\{[u_j]_{\omega_{(1)}}\ : \ 1\leq j\leq 26\},$ where

\begin{center}
    \begin{tabular}{lllllll}
    $u_{1}.\ x_1x_2x_3^{6}x_4^{14},$ & $u_{2}.\ x_1x_2x_3^{14}x_4^{6},$ & $u_{3}.\ x_1x_2^{2}x_3^{4}x_4^{15},$ &
  $u_{4}.\ x_1x_2^{2}x_3^{5}x_4^{14},$ & $u_{5}.\ x_1x_2^{2}x_3^{7}x_4^{12},$\\
 $u_{6}.\ x_1x_2^{2}x_3^{12}x_4^{7},$ &
  $u_{7}.\ x_1x_2^{2}x_3^{13}x_4^{6},$ & $u_{8}.\ x_1x_2^{2}x_3^{15}x_4^{4},$ & $u_{9}.\ x_1x_2^{3}x_3^{4}x_4^{14},$ &
 $u_{10}.\ x_1x_2^{3}x_3^{6}x_4^{12},$\\
  $u_{11}.\ x_1x_2^{3}x_3^{12}x_4^{6},$ &
  $u_{12}.\ x_1x_2^{3}x_3^{14}x_4^{4},$ &
$u_{13}.\ x_1x_2^{6}x_3x_4^{14},$ & $u_{14}.\ x_1x_2^{7}x_3^{2}x_4^{12},$ & $u_{15}.\ x_1x_2^{14}x_3x_4^{6},$\\
    $u_{16}.\ x_1x_2^{15}x_3^{2}x_4^{4},$ & $u_{17}.\ x_1^{3}x_2x_3^{4}x_4^{14},$ &
  $u_{18}.\ x_1^{3}x_2x_3^{6}x_4^{12},$ &
  $u_{19}.\ x_1^{3}x_2x_3^{12}x_4^{6},$ & $u_{20}.\ x_1^{3}x_2x_3^{14}x_4^{4},$ \\
 $u_{21}.\ x_1^{3}x_2^{5}x_3^{2}x_4^{12},$ &
  $u_{22}.\ x_1^{3}x_2^{5}x_3^{6}x_4^{8},$ &
  $u_{23}.\ x_1^{3}x_2^{5}x_3^{10}x_4^{4},$ & $u_{24}.\ x_1^{3}x_2^{13}x_3^{2}x_4^{4},$ &
 $u_{25}.\ x_1^{7}x_2x_3^{2}x_4^{12},$ \\
   $u_{26}.\ x_1^{15}x_2x_3^{2}x_4^{4}.$ &  &&&       
    \end{tabular}%
\end{center}

Consider the homomorphism $\pi_{(1; 2)}: \mathscr P_5\to \mathscr P_4.$ By a direct computation using Theorem \ref{dlsig} and Proposition \ref{mdPS}, we have
$$ \begin{array}{ll}
\pi_{(1; 2)}(\mathcal S)&\equiv_{\omega_{(1)}} (\gamma_{473} + \gamma_{475} + \gamma_{481} + \gamma_{483})u_1 + \gamma_{477}u_6 + \gamma_{478}u_7 \\
&+ (\gamma_{476} + \gamma_{477} + \gamma_{478})u_4   + \gamma_{483}(u_9 +u_{14})  + \gamma_{481}(u_{10} + u_{11}) 
\\
&+  (\gamma_{475} + \gamma_{476} + \gamma_{477} + \gamma_{478})u_{13} + (\gamma_{470} + \gamma_{481} + \gamma_{484})u_{18} \\
& + (\gamma_{471} + \gamma_{478} + \gamma_{481} + \gamma_{484})u_{19}+ \gamma_{472}u_{20} + (\gamma_{479} + \gamma_{484})u_{21} \\
& + (\gamma_{480} + \gamma_{481} + \gamma_{484})u_{22}+ \gamma_{482}u_{23} + \gamma_{485}u_{24} + \gamma_{495}u_{25}+ \gamma_{497}u_{26}\\
 &+ (\gamma_{474} + \gamma_{478} + \gamma_{481} + \gamma_{484})u_2 +\gamma_{476}u_5 + (\gamma_{469} + \gamma_{475})u_{17}+\gamma_{484}u_{15}\equiv_{\omega_{(1)}} 0.
\end{array}$$
This relation implies 
\begin{equation}\label{hs1}
 \begin{array}{ll}
\gamma_{469} &= \gamma_{470} = \gamma_{471} = \gamma_{472} = \gamma_{473} = \gamma_{474} = \gamma_{475} = \gamma_{476}\\
&  = \gamma_{477} = \gamma_{478} = \gamma_{479} = \gamma_{480} = \gamma_{481} = \gamma_{482} = \gamma_{483}\\
 &= \gamma_{484}=  \gamma_{485} = \gamma_{495} = \gamma_{497} = 0.
\end{array}
\end{equation}

Substituting \eqref{hs1} into the relation \eqref{thtt1}, we have
 \begin{equation}\label{thtt2}
 \sum\limits_{461\leq t\leq 468}\gamma_t\mathcal Y_t +  \sum\limits_{486\leqslant t\leqslant 494}\gamma_t\mathcal Y_t + \gamma_{496}\mathcal Y_{36} + \sum\limits_{498\leq t\leq 510}\gamma_t\mathcal Y_t \equiv_{\omega_{(1)}} 0.
\end{equation}

Applying the homomorphisms $\pi_{(1; 3)}, \pi_{(1; 4)}: \mathscr P_5\to \mathscr P_4$ to \eqref{thtt2}, we get
\begin{equation}\label{hs2}
\left\{\begin{array}{ll}
\gamma_t = 0,\ t\in \mathbb J,\\
\gamma_{464} = \gamma_{487} = \gamma_{499} = \gamma_{508}, \gamma_{468} = \gamma_{494} = \gamma_{506} = \gamma_{509}, \gamma_{498} =  \gamma_{502},\\
\gamma_{467} +\gamma_{468} + \gamma_{503} = \gamma_{467} +\gamma_{468} + \gamma_{505} = 0,\\
\gamma_{462} +\gamma_{468} + \gamma_{502} + \gamma_{508} = \gamma_{463} +\gamma_{468} + \gamma_{502} + \gamma_{508}= 0,\\
\gamma_{466} +\gamma_{501} + \gamma_{502} + \gamma_{508} = \gamma_{466} +\gamma_{468} + \gamma_{504} + \gamma_{507} + \gamma_{508} = 0.
\end{array}\right.
\end{equation}
Here $\mathbb J = \{461, 465, 486,488,489,490,491,492,493,496,500\}.$ Then, 
combining \eqref{hs1}, \eqref{hs2}, and the relation $\pi_{(1; 5)}(\mathcal S)\equiv_{\omega_{(1)}} 0,$ we obtain $\gamma_t = 0$ for $461\leq t\leq 510.$ The proposition is proved.
\end{proof}

Using a similar technique as mentioned in the proof of Proposition \ref{md22-1}, we obtain

\begin{propo}\label{md22-2}
\emph{}

\begin{enumerate}
\item[(I)] $\mathscr B^+_5(\omega_{(2)})  = \mathscr B^+(5, \omega_{(2)})\cup \mathcal E,$
 where $\mathcal E$ is  the set of the following monomials:

\begin{center}
    \begin{tabular}{lllll}
    $ x_1x_2^{2}x_3^{2}x_4^{3}x_5^{14},$ & $ x_1x_2^{2}x_3^{3}x_4^{2}x_5^{14},$ & $ x_1x_2^{2}x_3^{3}x_4^{6}x_5^{10},$ &
  $ x_1x_2^{2}x_3^{3}x_4^{14}x_5^{2},$ & $ x_1x_2^{3}x_3^{2}x_4^{2}x_5^{14},$\\
  $ x_1x_2^{3}x_3^{2}x_4^{6}x_5^{10},$ & $ x_1x_2^{3}x_3^{2}x_4^{14}x_5^{2},$ & $ x_1x_2^{3}x_3^{6}x_4^{2}x_5^{10},$ & $ x_1x_2^{3}x_3^{6}x_4^{10}x_5^{2},$ & $ x_1x_2^{3}x_3^{14}x_4^{2}x_5^{2},$\\
  $ x_1^{3}x_2x_3^{2}x_4^{2}x_5^{14},$ & $ x_1^{3}x_2x_3^{2}x_4^{6}x_5^{10},$ & $ x_1^{3}x_2x_3^{2}x_4^{14}x_5^{2},$ & $ x_1^{3}x_2x_3^{6}x_4^{2}x_5^{10},$ & $ x_1^{3}x_2x_3^{6}x_4^{10}x_5^{2},$ \\
  $ x_1^{3}x_2x_3^{14}x_4^{2}x_5^{2},$ & $ x_1^{3}x_2^{5}x_3^{2}x_4^{2}x_5^{10},$ & $ x_1^{3}x_2^{5}x_3^{2}x_4^{10}x_5^{2},$ &
  $ x_1^{3}x_2^{5}x_3^{10}x_4^{2}x_5^{2},$ & $ x_1^{3}x_2^{13}x_3^{2}x_4^{2}x_5^{2},$ 
    \end{tabular}%
\end{center}

\item[(II)] $\mathscr B^+_5(\omega_{(3)}) = \{x_1x_2^{3}x_3^{6}x_4^{6}x_5^{6},\ x_1^{3}x_2x_3^{6}x_4^{6}x_5^{6}, \ x_1^{3}x_2^{5}x_3^{2}x_4^{6}x_5^{6},\ x_1^{3}x_2^{5}x_3^{6}x_4^{2}x_5^{6},\ x_1^{3}x_2^{5}x_3^{6}x_4^{6}x_5^{2}\},$ 
 
\item[(III)] $\mathscr B^+_5(\omega_{(4)}) =  \overline{\Phi}^+(\mathscr B_4(\omega_{(4)})\bigcup \big(\bigcup_{1\leq k\leq 5} \overline{\mathscr B}^+(k, \omega_{(4)})\big),$

\item[(IV)] $\mathscr B^+_5(\omega_{(5)}) =  \overline{\Phi}^+(\mathscr B_4(\omega_{(5)})\bigcup \big(\bigcup_{1\leq k\leq 5} \overline{\mathscr B}^+(k, \omega_{(5)})\big).$
\end{enumerate}
\end{propo}

A direct computation shows: $|\mathscr{B}_5^+(\omega_{(2)})| = 25,\ |\mathscr{B}_5^+(\omega_{(4)})|  =300$ and $|\mathscr{B}_5^+(\omega_{(5)})|  = 125$ (see Sect.\ref{s5.5}.) On the other hand, by Remark \ref{nx2}, we have $(Q\mathscr P_5)_{22}^+ \cong \bigoplus_{1\leq j\leq 5}Q\mathscr P_5^+(\omega_{(j)}).$ Combining this with the above results, we obtain

\begin{corls}\label{hq22}
$(Q\mathscr P_5^+)_{22}$ is the $\mathbb Z/2$-vector space of dimension $505$ with a basis consisting of all the classes represented by the monomials $\mathcal Y_{22,\,t},\ 1\leq t\leq 505,$ which are determined in Sect.\ref{s5.5}.
\end{corls}

\subsection*{Structure of  the kernel of Kameko's map $(\widetilde {Sq^0_*})_{(5,47)}$}

The following weight vectors that have the same degrees are $47$:
$$ \overline{\omega}_{(1)} = (3,2,2,2,1),\ \overline{\omega}_{(2)} = (3,2,4,1,1),\ \overline{\omega}_{(3)} = (3,2,4,3),$$ 
 $$ \overline{\omega}_{(4)} = (3,4,3,1,1),\ \overline{\omega}_{(5)} = (3,4,3,3).$$ 
From Remarks \ref{nx1} and \ref{nx2}, we conclude that if $X\in \mathscr B_5(47)$ and $[X]$ belongs to the kernel of $(\widetilde {Sq^0_*})_{(5,47)}$ then the weight vector of $X$ is one of the above sequences $\overline{\omega}_{(k)}, \ 1\leq k\leq 5.$  This implies that the dimension of ${\rm Ker}(\widetilde {Sq^0_*})_{(5,47)}$ is equal to the sum of the dimensions of $Q\mathscr P_5^0$ and $Q\mathscr P_5^+(\overline{\omega}_{(k)})$ in degree $47$ for all $1\leq k\leq 5.$ Since $(Q\mathscr P_5^0)_{47}$ is isomorphic to $\bigoplus_{1\leq t\leq 4}\bigoplus_{1\leq \ell\leq \binom{5}{t}}(Q\mathscr P_t^+)_{47},$ by a direct computation using a result in \cite{M.K}, \cite{F.P1} and \cite{N.S3}, we claim $ \dim(Q\mathscr P_5^0)_{47} = \binom 53.14 + \binom 54.84  = 560.$ Furthermore, $$\mathscr B_5^0(47) = \mathscr B_5^0(\overline{\omega}_{(1)})= \overline{\Phi}^0(\mathscr B_4(47)) =  \{\mathcal Y_{47,\, i}\ :\ 1\leq i\leq 560\},$$ where the monomials $\mathcal Y_{47,\, i}\in \mathscr B_5^0(47) $ are explicitly described in Sect.\ref{s5.6} of the Appendix. 

We now determine the $\mathbb Z/2$-subspaces $Q\mathscr P_5^+(\overline{\omega}_{(k)})$ for $k = 1, 2, \ldots, 5.$


\begin{lema}\label{bd47.1}
The following monomials are strictly inadmissible:
\begin{enumerate}
\item[I)] 
$X_1 = x_1^{3}x_2^{12}x_kx_{\ell}^3x_m^{12},\ X_2 = x_1^{3}x_2^{4}x_k^3x_{\ell}^8x_m^{13},\ X_3 = x_1x_2^{14}x_k^{2}x_{\ell}x_m^{13},\
 \\
 X_5 = x_1^3x_2^{14}x_k^{12}x_{\ell}x_m,\
 X_6 = x_1^7x_2^{10}x_k^{12}x_{\ell}x_m,\ 
X_7 = x_1^3x_2^{2}x_k^{12}x_{\ell}x_m^{13},\\
 X_8 = x_1^3x_2^{12}x_k^{2}x_{\ell}x_m^{13},\
X_9 = x_1^{15}x_2^{2}x_kx_{\ell}^4x_m^{9},\ X_{10} = x_1^{15}x_2^{2}x_kx_{\ell}^5x_m^{8},\\
X_{11} = x_1^{7}x_2^{2}x_kx_{\ell}^8x_m^{13},\ X_{12} = x_1^{7}x_2^{2}x_kx_{\ell}^9x_m^{12},\
 X_{13} = x_1^{15}x_2^{2}x_k^{12}x_{\ell}x_m,\\
 X_{14} = x_1^{3}x_2^{4}x_kx_{\ell}^8x_m^{15},\ X_{15} = x_1^{3}x_2^{4}x_kx_{\ell}^9x_m^{14},\ X_{16} = x_1^{3}x_2^{14}x_kx_{\ell}^4x_m^{9},\\
X_{17} =  x_1^{3}x_2^{5}x_kx_{\ell}^8x_m^{14},\ X_{18} = x_1^{3}x_2^{6}x_kx_{\ell}^8x_m^{13},\ X_{19} = x_1^{3}x_2^{6}x_kx_{\ell}^9x_m^{12},\\
 X_{20} = x_1^{7}x_2^{10}x_kx_{\ell}^4x_m^{9}, \
X_{21} =  x_1^{7}x_2^{10}x_kx_{\ell}^5x_m^{8},\ X_{22} = x_1^{3}x_2^{2}x_k^5x_{\ell}^8x_m^{13},\\
 X_{23} = x_1^{3}x_2^{2}x_k^5x_{\ell}^9x_m^{12},\ X_{24} = x_1^{7}x_2^{2}x_k^5x_{\ell}^8x_m^{9},\
 X_{25} = x_1x_2^{2}x_k^{14}x_{\ell}x_m^{13},\\
 X_{26} =  x_1^{3}x_2^{6}x_k^5x_{\ell}^8x_m^{9},\ X_{27} = x_1^3x_2^{12}x_k^{14}x_{\ell}x_m.$ Here $(k, \ell, m)$ is a permutation of $(3, 4, 5);$

\item[II)] $ X_{27} = x_i^{3}x_j^{2}x_k^{13}x_{\ell}^4x_m^{9},\ j < k,\  X_{28} =x_ix_j^{2}x_k^6x_{\ell}^9x_m^{13},\  X_{29} =x_i^3x_j^{4}x_k^6x_{\ell}^9x_m^{9},$ where $(i, j, k, \ell, m)$ is a permutation of $(1, 2, 3, 4, 5).$
\end{enumerate}
\end{lema}

\begin{proof}
It is easy to see that $\omega(X_t) = \omega^* := (3,2,2,2)$ for $1\leq t\leq 29.$ We prove the lemma for the monomials $X_1 = x_1^{3}x_2^{12}x_kx_{\ell}^3x_m^{12}$ and $X_2 = x_1^{3}x_2^{4}x_k^3x_{\ell}^8x_m^{13},$ where $(k, \ell, m)$ is a permutation of $(3, 4, 5).$ The others can be proved by a similar technique. We have
$$ \begin{array}{ll}
X_1&= x_1^{2}x_2^{11}x_kx_{\ell}^{5}x_m^{12}+
x_1^{2}x_2^{11}x_k^{4}x_{\ell}^{5}x_m^{9}+
x_1^{2}x_2^{11}x_k^{8}x_{\ell}^{5}x_m^{5}+
x_1^{2}x_2^{13}x_kx_{\ell}^{3}x_m^{12}\\
&\quad + x_1^{2}x_2^{13}x_k^{4}x_{\ell}^{3}x_m^{9}+
x_1^{2}x_2^{13}x_k^{8}x_{\ell}^{3}x_m^{5}+
x_1^{3}x_2^{3}x_k^{8}x_{\ell}^{5}x_m^{12}+
x_1^{3}x_2^{5}x_k^{8}x_{\ell}^{5}x_m^{10}\\
&\quad + x_1^{3}x_2^{5}x_k^{8}x_{\ell}^{6}x_m^{9}+
x_1^{3}x_2^{7}x_k^{4}x_{\ell}^{8}x_m^{9}+
x_1^{3}x_2^{7}x_k^{8}x_{\ell}^{5}x_m^{8}+
x_1^{3}x_2^{7}x_k^{8}x_{\ell}^{8}x_m^{5}\\
&\quad + x_1^{3}x_2^{9}x_k^{2}x_{\ell}^{5}x_m^{12}+
x_1^{3}x_2^{9}x_k^{4}x_{\ell}^{3}x_m^{12}+
x_1^{3}x_2^{9}x_k^{4}x_{\ell}^{5}x_m^{10}+
x_1^{3}x_2^{9}x_k^{4}x_{\ell}^{6}x_m^{9}\\
&\quad + x_1^{3}x_2^{9}x_k^{8}x_{\ell}^{5}x_m^{6}+
x_1^{3}x_2^{9}x_k^{8}x_{\ell}^{6}x_m^{5}+
x_1^{3}x_2^{11}x_kx_{\ell}^{4}x_m^{12}\\
\medskip
&\quad + Sq^1(g_1) + Sq^2(g_2) + Sq^4(g_4) + Sq^8(x_1^{6}x_2^{5}x_k^4x_{\ell}^{3}x_m^{5})\ {\rm modulo}(\mathscr P_5^-(\omega^*)),\ \mbox{where}\\
g_1&=  x_1^{3}x_2^{7}x_kx_{\ell}^{3}x_m^{16}+
 x_1^{3}x_2^{11}x_kx_{\ell}^{3}x_m^{12}+
 x_1^{5}x_2^{3}x_k^{8}x_{\ell}^{5}x_m^{9}+
 x_1^{5}x_2^{7}x_k^{4}x_{\ell}^{5}x_m^{9}+
\medskip
 x_1^{5}x_2^{7}x_k^{8}x_{\ell}^{5}x_m^{5},\\
g_2&=  x_1^{2}x_2^{11}x_kx_{\ell}^{3}x_m^{12}+
 x_1^{2}x_2^{11}x_k^{4}x_{\ell}^{3}x_m^{9}+
 x_1^{2}x_2^{11}x_k^{8}x_{\ell}^{3}x_m^{5}+
 x_1^{3}x_2^{3}x_k^{8}x_{\ell}^{5}x_m^{10}+
 x_1^{3}x_2^{3}x_k^{8}x_{\ell}^{6}x_m^{9}\\
&\quad + x_1^{3}x_2^{7}x_k^{4}x_{\ell}^{5}x_m^{12}+
 x_1^{3}x_2^{7}x_k^{4}x_{\ell}^{6}x_m^{9}+
 x_1^{3}x_2^{7}x_k^{8}x_{\ell}^{5}x_m^{6}+
 x_1^{3}x_2^{7}x_k^{8}x_{\ell}^{6}x_m^{5}+
 x_1^{5}x_2^{7}x_k^{2}x_{\ell}^{3}x_m^{12}\\
&\quad + x_1^{6}x_2^{3}x_k^{8}x_{\ell}^{3}x_m^{9}+
 x_1^{6}x_2^{7}x_k^{4}x_{\ell}^{3}x_m^{9}+
\medskip
 x_1^{6}x_2^{7}x_k^{8}x_{\ell}^{3}x_m^{5},\\
g_4 &= x_1^{3}x_2^{7}x_k^{2}x_{\ell}^{3}x_m^{12}+
 x_1^{4}x_2^{7}x_k^{4}x_{\ell}^{3}x_m^{9}+
 x_1^{4}x_2^{7}x_k^{8}x_{\ell}^{3}x_m^{5}+
 x_1^{10}x_2^{5}x_k^{4}x_{\ell}^{3}x_m^{5}.
\end{array} $$
This equality implies that $X_1$ is strictly inadmissible. 

Next, we show that $X_2$ is also strictly inadmissible. Indeed, using Cartan's formula, we obtain
$$ \begin{array}{ll}
 X_2 &= x_1^{2}x_2x_k^{3}x_{\ell}^{12}x_m^{13}+
 x_1^{2}x_2x_k^{5}x_{\ell}^{9}x_m^{14}+
 x_1^{2}x_2x_k^{5}x_{\ell}^{10}x_m^{13}+
 x_1^{2}x_2x_k^{6}x_{\ell}^{9}x_m^{13}+
 x_1^{2}x_2x_k^{10}x_{\ell}^{5}x_m^{13}\\
&\quad + x_1^{2}x_2x_k^{12}x_{\ell}^{5}x_m^{11}+
 x_1^{2}x_2^{5}x_kx_{\ell}^{9}x_m^{14}+
 x_1^{2}x_2^{5}x_kx_{\ell}^{10}x_m^{13}+
 x_1^{2}x_2^{5}x_k^{8}x_{\ell}^{3}x_m^{13}+
 x_1^{2}x_2^{5}x_k^{8}x_{\ell}^{5}x_m^{11}\\
&\quad + x_1^{3}x_2x_k^{4}x_{\ell}^{9}x_m^{14}+
 x_1^{3}x_2x_k^{4}x_{\ell}^{10}x_m^{13}+
 x_1^{3}x_2x_k^{6}x_{\ell}^{8}x_m^{13}+
 x_1^{3}x_2x_k^{8}x_{\ell}^{6}x_m^{13}+
 x_1^{3}x_2^{2}x_k^{4}x_{\ell}^{9}x_m^{13}\\
&\quad + x_1^{3}x_2^{2}x_k^{5}x_{\ell}^{8}x_m^{13}+
 x_1^{3}x_2^{3}x_k^{8}x_{\ell}^{4}x_m^{13}+
 x_1^{3}x_2^{3}x_k^{8}x_{\ell}^{5}x_m^{12}+
 x_1^{3}x_2^{4}x_kx_{\ell}^{9}x_m^{14}+
 x_1^{3}x_2^{4}x_kx_{\ell}^{10}x_m^{13}\\
&\quad + x_1^{3}x_2^{4}x_k^{2}x_{\ell}^{9}x_m^{13} \\
&\quad + Sq^1(Z_1) + Sq^2(Z_2) + Sq^4(Z_4) + Sq^8(x_1^{3}x_2^{4}x_k^4x_{\ell}^{5}x_m^{7})\ {\rm modulo}(\mathscr P_5^-(\omega^*)),
\end{array} $$
where
$$ \begin{array}{ll}
Z_1&= x_1^{3}x_2x_k^{3}x_{\ell}^{5}x_m^{18}+
 x_1^{3}x_2x_k^{3}x_{\ell}^{6}x_m^{17}+
 x_1^{3}x_2x_k^{3}x_{\ell}^{9}x_m^{14}+
 x_1^{3}x_2x_k^{3}x_{\ell}^{10}x_m^{13}\\
 &\quad + x_1^{3}x_2x_k^{6}x_{\ell}^{9}x_m^{11}+
x_1^{3}x_2x_k^{10}x_{\ell}^{5}x_m^{11}+
 x_1^{3}x_2^{3}x_kx_{\ell}^{6}x_m^{17}+
 x_1^{3}x_2^{3}x_kx_{\ell}^{10}x_m^{13}\\
&\quad + x_1^{3}x_2^{3}x_kx_{\ell}^{12}x_m^{11}+
 x_1^{3}x_2^{3}x_kx_{\ell}^{16}x_m^{7}+
  x_1^{3}x_2^{3}x_k^{4}x_{\ell}^{3}x_m^{17}+
 x_1^{3}x_2^{3}x_k^{8}x_{\ell}^{3}x_m^{13}\\
&\quad + x_1^{3}x_2^{3}x_k^{8}x_{\ell}^{5}x_m^{11}+
 x_1^{3}x_2^{3}x_k^{8}x_{\ell}^{9}x_m^{7}+
 x_1^{3}x_2^{4}x_k^{5}x_{\ell}^{5}x_m^{13}+
\medskip
 x_1^{5}x_2^{5}x_kx_{\ell}^{5}x_m^{14},\\
Z_2&=  x_1^{2}x_2x_k^{3}x_{\ell}^{9}x_m^{14}+
 x_1^{2}x_2x_k^{3}x_{\ell}^{10}x_m^{13}+
 x_1^{2}x_2x_k^{6}x_{\ell}^{9}x_m^{11}+
 x_1^{2}x_2x_k^{10}x_{\ell}^{5}x_m^{11}\\
&\quad + x_1^{2}x_2^{3}x_kx_{\ell}^{9}x_m^{14}+
 x_1^{2}x_2^{3}x_kx_{\ell}^{10}x_m^{13}+
 x_1^{2}x_2^{3}x_kx_{\ell}^{12}x_m^{11}+
 x_1^{2}x_2^{3}x_k^{8}x_{\ell}^{3}x_m^{13}\\
&\quad + x_1^{2}x_2^{3}x_k^{8}x_{\ell}^{5}x_m^{11}+
 x_1^{2}x_2^{3}x_k^{8}x_{\ell}^{9}x_m^{7}+
 x_1^{3}x_2^{5}x_kx_{\ell}^{6}x_m^{14}+
 x_1^{3}x_2^{5}x_k^{2}x_{\ell}^{5}x_m^{14}\\
&\quad + x_1^{3}x_2^{6}x_kx_{\ell}^{5}x_m^{14}+
 x_1^{5}x_2x_k^{6}x_{\ell}^{6}x_m^{11}+
 x_1^{5}x_2^{2}x_k^{2}x_{\ell}^{9}x_m^{11}+
 x_1^{5}x_2^{2}x_k^{3}x_{\ell}^{5}x_m^{14}\\
&\quad + x_1^{5}x_2^{2}x_k^{3}x_{\ell}^{6}x_m^{13}+
 x_1^{5}x_2^{2}x_k^{6}x_{\ell}^{5}x_m^{11}+
 x_1^{5}x_2^{2}x_k^{8}x_{\ell}^{3}x_m^{11}+
 x_1^{5}x_2^{3}x_kx_{\ell}^{6}x_m^{14}\\
&\quad + x_1^{5}x_2^{3}x_k^{2}x_{\ell}^{6}x_m^{13}+
 x_1^{5}x_2^{3}x_k^{2}x_{\ell}^{8}x_m^{11}+
 x_1^{5}x_2^{3}x_k^{2}x_{\ell}^{10}x_m^{9}+
 x_1^{5}x_2^{3}x_k^{2}x_{\ell}^{12}x_m^{7}\\
&\quad + x_1^{5}x_2^{3}x_k^{4}x_{\ell}^{3}x_m^{14}+
 x_1^{5}x_2^{3}x_k^{4}x_{\ell}^{6}x_m^{11}+
 x_1^{5}x_2^{3}x_k^{4}x_{\ell}^{10}x_m^{7}+
\medskip
 x_1^{6}x_2^{3}x_kx_{\ell}^{5}x_m^{14},\\
Z_4&=  x_1^{3}x_2x_k^{6}x_{\ell}^{5}x_m^{12}+
 x_1^{3}x_2x_k^{6}x_{\ell}^{6}x_m^{11}+
 x_1^{3}x_2^{2}x_k^{2}x_{\ell}^{9}x_m^{11}+
 x_1^{3}x_2^{2}x_k^{3}x_{\ell}^{6}x_m^{13}+
x_1^{3}x_2^{2}x_k^{6}x_{\ell}^{5}x_m^{11}\\
&\quad + x_1^{3}x_2^{2}x_k^{8}x_{\ell}^{3}x_m^{11}+
 x_1^{3}x_2^{3}x_kx_{\ell}^{6}x_m^{14}+
 x_1^{3}x_2^{3}x_k^{2}x_{\ell}^{6}x_m^{13}+
x_1^{3}x_2^{3}x_k^{2}x_{\ell}^{8}x_m^{11}+
 x_1^{3}x_2^{3}x_k^{2}x_{\ell}^{10}x_m^{9}\\
&\quad + x_1^{3}x_2^{3}x_k^{2}x_{\ell}^{12}x_m^{7}+
 x_1^{3}x_2^{3}x_k^{4}x_{\ell}^{3}x_m^{14}+
x_1^{3}x_2^{3}x_k^{4}x_{\ell}^{6}x_m^{11}+
 x_1^{3}x_2^{3}x_k^{4}x_{\ell}^{10}x_m^{7}+
 x_1^{3}x_2^{4}x_kx_{\ell}^{5}x_m^{14}\\
&\quad + x_1^{3}x_2^{8}x_k^{4}x_{\ell}^{5}x_m^{7}+
 x_1^{4}x_2^{3}x_kx_{\ell}^{5}x_m^{14}.
\end{array} $$
The above relations imply that $X_2$ is strictly inadmissible. The lemma follows.
\end{proof}

\begin{lema}\label{bd47.2}
The following monomials are strictly inadmissible:
\begin{center}
\begin{tabular}{lllr}
$Y_{1}=x_1^{3}x_2^{4}x_3^{7}x_4^{8}x_5^{9}$, & $Y_{2}=x_1^{3}x_2^{7}x_3^{8}x_4^{5}x_5^{8}$, & $Y_{3}=x_1x_2^{6}x_3^{10}x_4x_5^{13}$, & \multicolumn{1}{l}{$Y_{4}=x_1x_2^{6}x_3^{10}x_4^{13}x_5$,} \\
$Y_{5}=x_1x_2^{6}x_3^{11}x_4^{12}x_5$, & $Y_{6}=x_1x_2^{7}x_3^{10}x_4^{12}x_5$, & $Y_{7}=x_1^{7}x_2x_3^{10}x_4^{12}x_5$, & \multicolumn{1}{l}{$Y_{8}=x_1x_2^{2}x_3^{2}x_4^{13}x_5^{13}$,} \\
$Y_{9}=x_1x_2^{2}x_3^{14}x_4^{5}x_5^{9}$, & $Y_{10}=x_1x_2^{14}x_3^{2}x_4^{5}x_5^{9}$, & $Y_{11}=x_1^{3}x_2^{3}x_3^{12}x_4^{12}x_5$, & \multicolumn{1}{l}{$Y_{12}=x_1^{3}x_2^{15}x_3^{4}x_4^{8}x_5$,} \\
$Y_{13}=x_1^{3}x_2^{15}x_3^{4}x_4x_5^{8}$, & $Y_{14}=x_1^{15}x_2^{3}x_3^{4}x_4^{8}x_5$, & $Y_{15}=x_1^{15}x_2^{3}x_3^{4}x_4x_5^{8}$, & \multicolumn{1}{l}{$Y_{16}=x_1x_2^{14}x_3^{3}x_4^{4}x_5^{9}$,} \\
$Y_{17}=x_1^{3}x_2^{4}x_3^{10}x_4x_5^{13}$, & $Y_{18}=x_1^{3}x_2^{4}x_3^{10}x_4^{13}x_5$, & $Y_{19}=x_1^{3}x_2^{4}x_3^{11}x_4^{12}x_5$, & \multicolumn{1}{l}{$Y_{20}=x_1x_2^{14}x_3^{3}x_4^{5}x_5^{8}$,} \\
$Y_{21}=x_1^{3}x_2^{14}x_3^{5}x_4^{8}x_5$, & $Y_{22}=x_1^{3}x_2^{14}x_3x_4^{5}x_5^{8}$, & $Y_{23}=x_1^{3}x_2^{14}x_3^{5}x_4x_5^{8}$, & \multicolumn{1}{l}{$Y_{24}=x_1x_2^{6}x_3^{3}x_4^{13}x_5^{8}$,} \\
$Y_{25}=x_1x_2^{14}x_3^{14}x_4x_5$, & $Y_{26}=x_1^{3}x_2^{7}x_3^{8}x_4^{12}x_5$, & $Y_{27}=x_1^{3}x_2^{7}x_3^{12}x_4^{8}x_5$, & \multicolumn{1}{l}{$Y_{28}=x_1^{3}x_2^{7}x_3^{12}x_4x_5^{8}$,} \\
$Y_{29}=x_1^{7}x_2^{3}x_3^{8}x_4^{12}x_5$, & $Y_{30}=x_1^{7}x_2^{3}x_3^{12}x_4^{8}x_5$, & $Y_{31}=x_1^{7}x_2^{3}x_3^{12}x_4x_5^{8}$, & \multicolumn{1}{l}{$Y_{32}=x_1x_2^{6}x_3^{11}x_4^{4}x_5^{9}$,} \\
$Y_{33}=x_1^{7}x_2^{11}x_3^{4}x_4^{8}x_5$, & $Y_{34}=x_1^{7}x_2^{11}x_3^{4}x_4x_5^{8}$, & $Y_{35}=x_1x_2^{6}x_3^{11}x_4^{5}x_5^{8}$, & \multicolumn{1}{l}{$Y_{36}=x_1x_2^{6}x_3^{10}x_4^{5}x_5^{9}$,} \\
$Y_{37}=x_1^{3}x_2^{5}x_3^{8}x_4^{2}x_5^{13}$, & $Y_{38}=x_1^{3}x_2^{12}x_3^{2}x_4^{5}x_5^{9}$, & $Y_{39}=x_1^{3}x_2^{5}x_3^{9}x_4^{2}x_5^{12}$, & \multicolumn{1}{l}{$Y_{40}=x_1^{7}x_2^{2}x_3^{4}x_4^{9}x_5^{9}$,} \\
$Y_{41}=x_1^{7}x_2^{2}x_3^{9}x_4^{4}x_5^{9}$, & $Y_{42}=x_1^{3}x_2^{4}x_3^{3}x_4^{12}x_5^{9}$, & $Y_{43}=x_1^{3}x_2^{12}x_3^{3}x_4^{4}x_5^{9}$, & \multicolumn{1}{l}{$Y_{44}=x_1^{3}x_2^{4}x_3^{9}x_4^{3}x_5^{12}$,} \\
$Y_{45}=x_1^{3}x_2^{12}x_3^{3}x_4^{5}x_5^{8}$, & $Y_{46}=x_1^{3}x_2^{5}x_3^{8}x_4^{3}x_5^{12}$, & $Y_{47}=x_1^{3}x_2^{4}x_3^{11}x_4^{4}x_5^{9}$, & \multicolumn{1}{l}{$Y_{48}=x_1^{3}x_2^{4}x_3^{11}x_4^{5}x_5^{8}$,} \\
$Y_{49}=x_1^{3}x_2^{4}x_3^{10}x_4^{5}x_5^{9}$, & $Y_{50}=x_1x_2^{6}x_3^{3}x_4^{12}x_5^{9}$, & $Y_{51}=x_1^{3}x_2^{4}x_3^{7}x_4^{9}x_5^{8}$, & \multicolumn{1}{l}{$Y_{52}=x_1^{3}x_2^{7}x_3^{8}x_4^{4}x_5^{9}$,} \\
$Y_{53}=x_1^{7}x_2^{3}x_3^{8}x_4^{4}x_5^{9}$, & $Y_{54}=x_1x_2^{14}x_3^{3}x_4^{12}x_5$, & $Y_{55}=x_1^{7}x_2^{3}x_3^{8}x_4^{5}x_5^{8}$. &  
\end{tabular}
\end{center}
\end{lema}

\begin{proof}
Note that $\omega(Y_j) = \omega^*,\ j = 1, 2, \ldots, 55.$ We prove this lemma for the monomials $Y_1 = x_1^3x_2^{4}x_3^{7}x_4^{8}x_5^{9},$ and $Y_2 = x_1^{3}x_2^{7}x_3^{8}x_4^{5}x_5^{8}.$ The others can be proved by a similar computation. A direct computation shows:
$$ \begin{array}{ll}
Y_1&=  x_1^{2}x_2x_3^{7}x_4^{8}x_5^{13}+
 x_1^{2}x_2x_3^{13}x_4x_5^{14}+
 x_1^{2}x_2x_3^{13}x_4^{8}x_5^{7}+
 x_1^{2}x_2^{3}x_3^{9}x_4^{4}x_5^{13}+
 x_1^{2}x_2^{3}x_3^{12}x_4x_5^{13}\\
&\quad +  x_1^{2}x_2^{5}x_3^{9}x_4^{4}x_5^{11}+
 x_1^{2}x_2^{5}x_3^{9}x_4^{8}x_5^{7}+
 x_1^{2}x_2^{5}x_3^{12}x_4x_5^{11}+
 x_1^{3}x_2x_3^{5}x_4^{8}x_5^{14}+
 x_1^{3}x_2x_3^{7}x_4^{8}x_5^{12}\\
&\quad + x_1^{3}x_2x_3^{12}x_4x_5^{14}+
 x_1^{3}x_2x_3^{12}x_4^{4}x_5^{11}+
 x_1^{3}x_2^{2}x_3^{9}x_4^{4}x_5^{13}+
 x_1^{3}x_2^{3}x_3^{9}x_4^{4}x_5^{12}+
 x_1^{3}x_2^{3}x_3^{12}x_4x_5^{12}\\
\medskip
&\quad + x_1^{3}x_2^{3}x_3^{12}x_4^{4}x_5^{9} + Sq^1(u_1) + Sq^2(u_2) + Sq^4(u_4) + Sq^8(u_8)\ {\rm modulo}(\mathscr P_5^-(\omega^*)),\, \mbox{where}\\
u_1&=  x_1^{3}x_2x_3^{5}x_4^{8}x_5^{13}+
 x_1^{3}x_2x_3^{7}x_4x_5^{18}+
 x_1^{3}x_2x_3^{9}x_4^{4}x_5^{13}+
 x_1^{3}x_2x_3^{11}x_4x_5^{14}\\
 &\quad +  x_1^{3}x_2^{3}x_3^{9}x_4^{4}x_5^{11}+
x_1^{3}x_2^{3}x_3^{9}x_4^{8}x_5^{7}+
 x_1^{3}x_2^{3}x_3^{12}x_4x_5^{11}+
 x_1^{3}x_2^{3}x_3^{16}x_4x_5^{7}\\
&\quad +  x_1^{3}x_2^{4}x_3^{9}x_4x_5^{13}+
 x_1^{5}x_2x_3^{5}x_4^{8}x_5^{11}+
\medskip
x_1^{5}x_2x_3^{7}x_4^{8}x_5^{9},\\
u_2&=  x_1^{2}x_2x_3^{7}x_4^{8}x_5^{11}+
 x_1^{2}x_2x_3^{11}x_4x_5^{14}+
 x_1^{2}x_2x_3^{11}x_4^{8}x_5^{7}+
 x_1^{2}x_2^{3}x_3^{9}x_4^{4}x_5^{11}\\
&\quad + x_1^{2}x_2^{3}x_3^{9}x_4^{8}x_5^{7}+
 x_1^{2}x_2^{3}x_3^{12}x_4x_5^{11}+
 x_1^{3}x_2x_3^{6}x_4^{8}x_5^{11}+
 x_1^{3}x_2x_3^{7}x_4^{8}x_5^{10}\\
&\quad + x_1^{3}x_2x_3^{10}x_4^{4}x_5^{11}+
 x_1^{3}x_2^{2}x_3^{5}x_4^{8}x_5^{11}+
 x_1^{3}x_2^{2}x_3^{7}x_4^{8}x_5^{9}+
 x_1^{3}x_2^{2}x_3^{9}x_4^{4}x_5^{11}\\
&\quad +  x_1^{5}x_2x_3^{7}x_4^{2}x_5^{14}+
 x_1^{5}x_2^{2}x_3^{7}x_4x_5^{14}+
 x_1^{5}x_2^{2}x_3^{9}x_4^{2}x_5^{11}+
 x_1^{5}x_2^{3}x_3^{10}x_4^{2}x_5^{9}\\
&\quad + x_1^{5}x_2^{3}x_3^{10}x_4^{4}x_5^{7}+
 x_1^{5}x_2^{3}x_3^{12}x_4^{2}x_5^{7}+
\medskip
 x_1^{6}x_2x_3^{7}x_4^{8}x_5^{7},\\
u_4&=  x_1^{3}x_2x_3^{7}x_4^{2}x_5^{14}+
 x_1^{3}x_2^{2}x_3^{7}x_4x_5^{14}+
 x_1^{3}x_2^{2}x_3^{9}x_4^{2}x_5^{11}+
 x_1^{3}x_2^{2}x_3^{10}x_4x_5^{11}\\
&\quad + x_1^{3}x_2^{3}x_3^{10}x_4^{2}x_5^{9}+
 x_1^{3}x_2^{3}x_3^{10}x_4^{4}x_5^{7}+
 x_1^{3}x_2^{3}x_3^{12}x_4^{2}x_5^{7}+
 x_1^{3}x_2^{8}x_3^{5}x_4^{4}x_5^{7}\\
&\quad + x_1^{4}x_2x_3^{7}x_4^{8}x_5^{7}+
\medskip
 x_1^{10}x_2x_3^{5}x_4^{4}x_5^{7},\\
\medskip
u_8&=x_1^{3}x_2^{4}x_3^{5}x_4^{4}x_5^{7}+ x_1^{6}x_2x_3^{5}x_4^{4}x_5^{7}.
\end{array}$$
By a similar computation, we have
$$ \begin{array}{ll}
Y_2 &=  x_1^{2}x_2^{7}x_3^{8}x_4^{5}x_5^{9}+
 x_1^{2}x_2^{11}x_3x_4^{5}x_5^{12}+
 x_1^{2}x_2^{11}x_3^{8}x_4^{5}x_5^{5}+
 x_1^{2}x_2^{13}x_3x_4^{5}x_5^{10}+
 x_1^{2}x_2^{13}x_3x_4^{9}x_5^{6}\\
&\quad + x_1^{2}x_2^{13}x_3^{8}x_4^{3}x_5^{5}+
 x_1^{3}x_2^{3}x_3x_4^{12}x_5^{12}+
 x_1^{3}x_2^{3}x_3^{4}x_4^{9}x_5^{12}+
 x_1^{3}x_2^{3}x_3^{8}x_4^{5}x_5^{12}+
 x_1^{3}x_2^{5}x_3^{4}x_4^{9}x_5^{10}\\
&\quad + x_1^{3}x_2^{5}x_3^{4}x_4^{10}x_5^{9}+
 x_1^{3}x_2^{5}x_3^{8}x_4^{5}x_5^{10}+
 x_1^{3}x_2^{5}x_3^{8}x_4^{6}x_5^{9}+
 x_1^{3}x_2^{5}x_3^{8}x_4^{9}x_5^{6}+
 x_1^{3}x_2^{5}x_3^{8}x_4^{10}x_5^{5}\\
&\quad + x_1^{3}x_2^{7}x_3x_4^{8}x_5^{12}+
 x_1^{3}x_2^{7}x_3x_4^{12}x_5^{8}+
 x_1^{3}x_2^{7}x_3^{4}x_4^{9}x_5^{8}+
 x_1^{3}x_2^{7}x_3^{8}x_4^{4}x_5^{9}\\
\medskip
&\quad + Sq^1(v_1) + Sq^2(v_2) + Sq^4(v_4) + Sq^8(v_8)\ {\rm modulo}(\mathscr P_5^-(\omega^*)), \,\mbox{where}\\
v_1&=  x_1^{5}x_2^{3}x_3x_4^{9}x_5^{12}+
 x_1^{5}x_2^{3}x_3^{8}x_4^{5}x_5^{9}+
 x_1^{5}x_2^{7}x_3x_4^{5}x_5^{12}\\
&\quad + x_1^{5}x_2^{7}x_3x_4^{9}x_5^{8}+
 x_1^{5}x_2^{7}x_3^{4}x_4^{5}x_5^{9}+
\medskip
 x_1^{5}x_2^{7}x_3^{8}x_4^{5}x_5^{5},\\
v_2&=  x_1^{2}x_2^{7}x_3^{8}x_4^{3}x_5^{9}+
 x_1^{2}x_2^{11}x_3x_4^{5}x_5^{10}+
 x_1^{2}x_2^{11}x_3x_4^{9}x_5^{6}+
 x_1^{2}x_2^{11}x_3^{8}x_4^{3}x_5^{5}+
 x_1^{3}x_2^{3}x_3x_4^{10}x_5^{12}\\
&\quad + x_1^{3}x_2^{3}x_3^{2}x_4^{9}x_5^{12}+
 x_1^{3}x_2^{3}x_3^{8}x_4^{5}x_5^{10}+
 x_1^{3}x_2^{3}x_3^{8}x_4^{6}x_5^{9}+
 x_1^{3}x_2^{7}x_3x_4^{6}x_5^{12}+
 x_1^{3}x_2^{7}x_3x_4^{10}x_5^{8}\\
&\quad + x_1^{3}x_2^{7}x_3^{2}x_4^{5}x_5^{12}+
 x_1^{3}x_2^{7}x_3^{2}x_4^{9}x_5^{8}+
 x_1^{3}x_2^{7}x_3^{4}x_4^{5}x_5^{10}+
 x_1^{3}x_2^{7}x_3^{4}x_4^{6}x_5^{9}+
 x_1^{3}x_2^{7}x_3^{8}x_4^{5}x_5^{6}\\
&\quad + x_1^{3}x_2^{7}x_3^{8}x_4^{6}x_5^{5}+
 x_1^{6}x_2^{3}x_3x_4^{9}x_5^{10}+
 x_1^{6}x_2^{3}x_3^{8}x_4^{3}x_5^{9}+
 x_1^{6}x_2^{7}x_3x_4^{5}x_5^{10}+
 x_1^{6}x_2^{7}x_3x_4^{9}x_5^{6}\\
&\quad + x_1^{6}x_2^{7}x_3^{4}x_4^{3}x_5^{9}+
\medskip
 x_1^{6}x_2^{7}x_3^{8}x_4^{3}x_5^{5},\\
v_4&=  x_1^{3}x_2^{5}x_3x_4^{6}x_5^{12}+
 x_1^{3}x_2^{5}x_3^{2}x_4^{5}x_5^{12}+
 x_1^{3}x_2^{5}x_3^{4}x_4^{5}x_5^{10}+
 x_1^{3}x_2^{5}x_3^{4}x_4^{6}x_5^{9}+
 x_1^{3}x_2^{5}x_3^{8}x_4^{5}x_5^{6}\\
&\quad + x_1^{3}x_2^{5}x_3^{8}x_4^{6}x_5^{5}+
 x_1^{3}x_2^{11}x_3^{4}x_4^{4}x_5^{5}+
 x_1^{4}x_2^{7}x_3x_4^{5}x_5^{10}+
 x_1^{4}x_2^{7}x_3x_4^{9}x_5^{6}+
 x_1^{4}x_2^{7}x_3^{8}x_4^{3}x_5^{5}\\
&\quad + x_1^{10}x_2^{5}x_3x_4^{5}x_5^{6}+
\medskip
 x_1^{10}x_2^{5}x_3^{4}x_4^{3}x_5^{5},\\
\medskip
v_8&=x_1^{3}x_2^{7}x_3^{4}x_4^{4}x_5^{5} + x_1^{6}x_2^5x_3^{1}x_4^{5}x_5^{6}+x_1^{6}x_2^5x_3^{4}x_4^{3}x_5^{5}.
\end{array}$$
From the above equalities, we see that $Y_1$ and $Y_2$ are strictly inadmissible. The lemma is proved.
\end{proof}

For $s,\ k\in\mathbb N$ and $1\leq k\leq 5,$ we denote
$$ \overline{\mathscr B}(k, 47) := \big\{x_k^{2^{s} - 1}\rho_{(k,\,5)}(x)\in (\mathscr P_5)_{47}\ :\ x\in \mathscr B_4(45-2^{s}),\ \alpha(52-2^{s})\leq 4\big\}.$$
It has been shown (see \cite{M.M2}) that $\overline{\mathscr B}(k, 47)\subseteq \mathscr B_5(47)$ for $k = 1, 2, \ldots, 5.$ We set $\overline{\mathscr B}(k, \overline{\omega}_{(1)})  := \overline{\mathscr B}(k, 47) \cap \mathscr P_5(\overline{\omega}_{(1)})$ and $\overline{\mathscr B}^+(k, \overline{\omega}_{(1)})  := \overline{\mathscr B}(k, \overline{\omega}_{(1)}) \cap (\mathscr P^+_5)_{47}.$ Then, we obtain the following result.

\begin{propo}\label{md47-1}
We have 
$$ \begin{array}{ll}
Q\mathscr P_5^+(\overline{\omega}_{(1)}) &= \big\langle[\overline{\Phi}^+(\mathscr B_4(\overline{\omega}_{(1)}))\bigcup \big(\bigcup_{1\leq k\leq 5} \overline{\mathscr B}^+(k, \overline{\omega}_{(1)})\big) \bigcup \mathcal F]_{\overline{\omega}_{(1)}}\big\rangle\\
& = \big\langle [\mathcal Y_{47,\, j}]_{\overline{\omega}_{(1)}}:\, 1\leq j\leq 370\big\rangle,
\end{array}$$
where the monomials $\mathcal Y_{47,\, j},\ 1\leq j\leq 370,$ are listed in Sect.\ref{s5.7}.
\end{propo}

The proof of the proposition is based on Lemmas \ref{bd47.1}, \ref{bd47.2} and some results below. 

The follwing lemma is an immediate consequence of a result in Sum \cite{N.S3}.

\begin{lema}\label{bd47.3}
The following monomials are strictly inadmissible:
\begin{enumerate}
\item[I)] 
$x_i^{15}x_j^2x_kx_{\ell}^{13},\ x_i^3x_j^{12}x_kx_{\ell}^{15}, \ x_i^3x_j^{14}x_kx_{\ell}^{13}, \\
 x_i^7x_j^{10}x_kx_{\ell}^{13}, \ x_i^3x_j^13x_k^8x_{\ell}^7, \ x_i^3x_j^{12}x_k^9x_{\ell}^7, 1\leq j < k \leq 5,\ 1\leq i,  \ell\leq 5,\ i\neq \ell,\ i, \ell \neq j, k;$
\item[II)] $x_i^{15}x_j^{14}x_kx_{\ell}, \ x_i^3x_j^2x_k^{13}x_{\ell}^{13}, \ x_i^{15}x_j^2x_k^5x_{\ell}^9, \ x_i^{13}x_j^2x_k^7x_{\ell}^9,\  x_i^{13}x_j^7x_k^2x_{\ell}^9, \ x_i^{15}x_j^3x_k^4x_{\ell}^9, x_i^{15}x_j^3x_k^5x_{\ell}^8, \\
 x_i^{14}x_j^3x_k^5x_{\ell}^9, \ x_i^{13}x_j^3x_k^6x_{\ell}^9, \ x_i^{13}x_j^6x_k^3x_{\ell}^9, \ 1\leq j < k < \ell\leq 5,\ 1\leq i\leq 5,\ i\neq j, k, \ell;$
\item[III)] $\rho_{(1,\, 5)}(x_1^2x_2^5x_3^{11}x_4^{13}),\ \rho_{(1,\, 5)}(x_1^6x_2^{11}x_3^{5}x_4^{9}),\ \rho_{(1,\, 5)}(x_1^6x_2x_3^{11}x_4^{13}),\\
 \rho_{(1,\, 5)}(x_1^6x_2^{11}x_3x_4^{13}),\ \rho_{(1,\, 5)}(x_1^6x_2^{11}x_3^{13}x_4).$
\end{enumerate}
\end{lema}

\begin{lema}\label{bd47.4}
The following monomials are strictly inadmissible:
\begin{center}
\begin{tabular}{llll}
$Z_{1}=x_1^{3}x_2^{3}x_3^{12}x_4^{5}x_5^{24}$, & $Z_{2}=x_1^{7}x_2x_3^{11}x_4^{20}x_5^{8}$, & $Z_{3}=x_1^{7}x_2x_3^{2}x_4^{28}x_5^{9}$, & $Z_{4}=x_1x_2^{2}x_3^{7}x_4^{12}x_5^{25}$, \\
$Z_{5}=x_1x_2^{7}x_3^{2}x_4^{12}x_5^{25}$, & $Z_{6}=x_1^{7}x_2x_3^{2}x_4^{12}x_5^{25}$, & $Z_{7}=x_1x_2^{6}x_3^{3}x_4^{8}x_5^{29}$, & $Z_{8}=x_1x_2^{6}x_3^{3}x_4^{24}x_5^{13}$, \\
$Z_{9}=x_1^{3}x_2^{7}x_3^{8}x_4x_5^{28}$, & $Z_{10}=x_1^{7}x_2^{3}x_3^{8}x_4x_5^{28}$, & $Z_{11}=x_1^{3}x_2^{7}x_3^{24}x_4x_5^{12}$, & $Z_{12}=x_1^{7}x_2^{3}x_3^{24}x_4x_5^{12}$, \\
$Z_{13}=x_1x_2^{7}x_3^{26}x_4^{4}x_5^{9}$, & $Z_{14}=x_1^{7}x_2x_3^{26}x_4^{4}x_5^{9}$, & $Z_{15}=x_1x_2^{7}x_3^{10}x_4^{4}x_5^{25}$, & $Z_{16}=x_1^{7}x_2x_3^{10}x_4^{4}x_5^{25}$, \\
$Z_{17}=x_1x_2^{7}x_3^{26}x_4^{5}x_5^{8}$, & $Z_{18}=x_1^{7}x_2x_3^{26}x_4^{5}x_5^{8}$, & $Z_{19}=x_1x_2^{7}x_3^{10}x_4^{5}x_5^{24}$, & $Z_{20}=x_1^{7}x_2x_3^{10}x_4^{5}x_5^{24}$, \\
$Z_{21}=x_1x_2^{6}x_3^{11}x_4^{13}x_5^{16}$, & $Z_{22}=x_1x_2^{7}x_3^{10}x_4^{21}x_5^{8}$, & $Z_{23}=x_1^{7}x_2x_3^{10}x_4^{21}x_5^{8}$, & $Z_{24}=x_1x_2^{7}x_3^{11}x_4^{20}x_5^{8}$, \\
$Z_{25}=x_1x_2^{2}x_3^{7}x_4^{28}x_5^{9}$, & $Z_{26}=x_1^{7}x_2^{11}x_3x_4^{20}x_5^{8}$, & $Z_{27}=x_1x_2^{7}x_3^{10}x_4^{20}x_5^{9}$, & $Z_{28}=x_1^{7}x_2x_3^{10}x_4^{20}x_5^{9}$, \\
$Z_{29}=x_1x_2^{7}x_3^{10}x_4^{13}x_5^{16}$, & $Z_{30}=x_1^{7}x_2x_3^{10}x_4^{13}x_5^{16}$, & $Z_{31}=x_1x_2^{7}x_3^{11}x_4^{12}x_5^{16}$, & $Z_{32}=x_1^{7}x_2x_3^{11}x_4^{12}x_5^{16}$, \\
$Z_{33}=x_1^{7}x_2^{11}x_3x_4^{12}x_5^{16}$, & $Z_{34}=x_1^{3}x_2^{3}x_3^{4}x_4^{28}x_5^{9}$, & $Z_{35}=x_1^{3}x_2^{3}x_3^{28}x_4^{4}x_5^{9}$, & $Z_{36}=x_1^{3}x_2^{3}x_3^{4}x_4^{12}x_5^{25}$, \\
$Z_{37}=x_1^{3}x_2^{3}x_3^{12}x_4^{4}x_5^{25}$, & $Z_{38}=x_1^{3}x_2^{3}x_3^{28}x_4^{5}x_5^{8}$, & $Z_{39}=x_1x_2^{7}x_3^{2}x_4^{28}x_5^{9}$, & $Z_{40}=x_1^{3}x_2^{3}x_3^{12}x_4^{21}x_5^{8}.$
\end{tabular}
\end{center}
\end{lema}

\begin{proof}
We have $\omega(Z_u) = \overline{\omega}_{(1)},\,\forall u,\, 1\leq u\leq 40.$ Consider the monomials $Z_1 = x_1^{3}x_2^{3}x_3^{12}x_4^{5}x_5^{24}$ and $Z_2 = x_1^{7}x_2x_3^{11}x_4^{20}x_5^{8}$. We prove that these monomials are strictly inadmissible. The others can be proved by a similar technique.
Computing the monomials $Z_1,\, Z_2$ is long and technical. Indeed, by using Cartan's formula, we get
$$ Z_1 = \sum X + Sq^1(\sum \sigma_1) +  Sq^2(\sum \sigma_2) + Sq^4(\sum \sigma_4) + Sq^8(\sum \sigma_8)\ {\rm modulo}(\mathscr P_5^-(\overline{\omega}_{(1)})),\ \mbox{where}$$
$$ 
\begin{array}{ll}
 \sum X &=  x_1^{2}x_2x_3^{3}x_4^{12}x_5^{29}+
 x_1^{2}x_2x_3^{3}x_4^{13}x_5^{28}+
 x_1^{2}x_2x_3^{5}x_4^{12}x_5^{27}+
 x_1^{2}x_2x_3^{5}x_4^{13}x_5^{26}+
 x_1^{2}x_2x_3^{5}x_4^{24}x_5^{15}\\
&\quad + x_1^{2}x_2x_3^{5}x_4^{25}x_5^{14}+
 x_1^{2}x_2x_3^{7}x_4^{9}x_5^{28}+
 x_1^{2}x_2x_3^{7}x_4^{13}x_5^{24}+
 x_1^{2}x_2x_3^{12}x_4^{3}x_5^{29}+
 x_1^{2}x_2x_3^{12}x_4^{5}x_5^{27}\\
&\quad + x_1^{2}x_2x_3^{13}x_4^{7}x_5^{24}+
 x_1^{2}x_2x_3^{13}x_4^{9}x_5^{22}+
 x_1^{2}x_2x_3^{13}x_4^{16}x_5^{15}+
 x_1^{2}x_2x_3^{13}x_4^{17}x_5^{14}+
 x_1^{2}x_2x_3^{16}x_4^{13}x_5^{15}\\
&\quad + x_1^{2}x_2x_3^{24}x_4^{5}x_5^{15}+
 x_1^{2}x_2^{3}x_3^{5}x_4^{9}x_5^{28}+
 x_1^{2}x_2^{3}x_3^{5}x_4^{12}x_5^{25}+
 x_1^{2}x_2^{3}x_3^{9}x_4^{5}x_5^{28}+
 x_1^{2}x_2^{3}x_3^{12}x_4^{5}x_5^{25}\\
&\quad + x_1^{2}x_2^{5}x_3^{3}x_4^{9}x_5^{28}+
 x_1^{2}x_2^{5}x_3^{3}x_4^{12}x_5^{25}+
 x_1^{2}x_2^{5}x_3^{9}x_4^{5}x_5^{26}+
 x_1^{2}x_2^{5}x_3^{12}x_4^{3}x_5^{25}+
 x_1^{3}x_2x_3^{3}x_4^{12}x_5^{28}\\
&\quad + x_1^{3}x_2x_3^{4}x_4^{12}x_5^{27}+
 x_1^{3}x_2x_3^{4}x_4^{13}x_5^{26}+
 x_1^{3}x_2x_3^{4}x_4^{24}x_5^{15}+
 x_1^{3}x_2x_3^{4}x_4^{25}x_5^{14}+
 x_1^{3}x_2x_3^{7}x_4^{12}x_5^{24}\\
&\quad + x_1^{3}x_2x_3^{7}x_4^{12}x_5^{24}+
 x_1^{3}x_2x_3^{9}x_4^{12}x_5^{22}+
 x_1^{3}x_2x_3^{12}x_4^{3}x_5^{28}+
 x_1^{3}x_2x_3^{12}x_4^{4}x_5^{27}+
 x_1^{3}x_2x_3^{12}x_4^{7}x_5^{24}\\
&\quad + x_1^{3}x_2x_3^{12}x_4^{9}x_5^{22}+
 x_1^{3}x_2x_3^{12}x_4^{16}x_5^{15}+
 x_1^{3}x_2x_3^{12}x_4^{17}x_5^{14}+
 x_1^{3}x_2x_3^{16}x_4^{12}x_5^{15}+
 x_1^{3}x_2x_3^{24}x_4^{4}x_5^{15}\\
\end{array}$$

\newpage
$$ \begin{array}{ll}
\sum \sigma_2&=  x_1^{2}x_2x_3^{3}x_4^{12}x_5^{27}+
 x_1^{2}x_2x_3^{3}x_4^{13}x_5^{26}+
 x_1^{2}x_2x_3^{3}x_4^{24}x_5^{15}+
 x_1^{2}x_2x_3^{3}x_4^{25}x_5^{14}+
 x_1^{2}x_2x_3^{7}x_4^{9}x_5^{26}\\
&\quad+ x_1^{2}x_2x_3^{7}x_4^{11}x_5^{24}+
 x_1^{2}x_2x_3^{11}x_4^{7}x_5^{24}+
 x_1^{2}x_2x_3^{11}x_4^{9}x_5^{22}+
 x_1^{2}x_2x_3^{11}x_4^{16}x_5^{15}+
 x_1^{2}x_2x_3^{11}x_4^{17}x_5^{14}\\
&\quad+ x_1^{2}x_2x_3^{12}x_4^{3}x_5^{27}+
 x_1^{2}x_2x_3^{16}x_4^{11}x_5^{15}+
 x_1^{2}x_2x_3^{24}x_4^{3}x_5^{15}+
 x_1^{2}x_2^{3}x_3^{3}x_4^{9}x_5^{28}+
 x_1^{2}x_2^{3}x_3^{3}x_4^{12}x_5^{25}\\
&\quad+ x_1^{2}x_2^{3}x_3^{9}x_4^{5}x_5^{26}+
 x_1^{2}x_2^{3}x_3^{9}x_4^{9}x_5^{22}+
 x_1^{2}x_2^{3}x_3^{12}x_4^{3}x_5^{25}+
 x_1^{5}x_2x_3^{3}x_4^{14}x_5^{22}+
 x_1^{5}x_2x_3^{3}x_4^{22}x_5^{14}\\
&\quad+ x_1^{5}x_2x_3^{7}x_4^{10}x_5^{22}+
 x_1^{5}x_2x_3^{7}x_4^{18}x_5^{14}+
 x_1^{5}x_2^{2}x_3^{3}x_4^{9}x_5^{26}+
 x_1^{5}x_2^{2}x_3^{3}x_4^{11}x_5^{24}+
 x_1^{5}x_2^{2}x_3^{3}x_4^{12}x_5^{23}\\
&\quad+  x_1^{5}x_2^{2}x_3^{3}x_4^{13}x_5^{22}+
 x_1^{5}x_2^{2}x_3^{3}x_4^{20}x_5^{15}+
 x_1^{5}x_2^{2}x_3^{3}x_4^{21}x_5^{14}+
 x_1^{5}x_2^{2}x_3^{7}x_4^{7}x_5^{24}+
 x_1^{5}x_2^{2}x_3^{7}x_4^{9}x_5^{22}\\
&\quad+ x_1^{5}x_2^{2}x_3^{7}x_4^{16}x_5^{15}+
 x_1^{5}x_2^{2}x_3^{7}x_4^{17}x_5^{14}+
 x_1^{5}x_2^{2}x_3^{9}x_4^{10}x_5^{19}+
 x_1^{5}x_2^{2}x_3^{9}x_4^{11}x_5^{18}+
 x_1^{5}x_2^{2}x_3^{10}x_4^{9}x_5^{19}\\
&\quad+ x_1^{5}x_2^{2}x_3^{12}x_4^{3}x_5^{23}+
 x_1^{5}x_2^{2}x_3^{16}x_4^{7}x_5^{15}+
 x_1^{5}x_2^{2}x_3^{20}x_4^{3}x_5^{15}+
 x_1^{5}x_2^{3}x_3^{3}x_4^{6}x_5^{28}+
 x_1^{5}x_2^{3}x_3^{3}x_4^{8}x_5^{26}\\
\end{array}$$
The above relations imply that $Z_1$ is strictly inadmissible. 

Next we prove that $Z_2$ is also strictly inadmissible strictly inadmissible. A direct computation shows that
$$ Z_2 = \sum Y + Sq^1(\sum\beta_1) +  Sq^2(\sum\beta_2) + Sq^4(\sum\beta_4) + Sq^8(\sum\beta_8)\ {\rm modulo}(\mathscr P_5^-(\overline{\omega}_{(1)})),$$
where the polynomials $\sum Y,\, \sum\beta_1,\, \sum\beta_2,\, \sum\beta_4,$ and $\sum\beta_8$ are as follows:
$$ \begin{array}{ll}
 \sum Y&=  x_1^{3}x_2x_3^{7}x_4^{12}x_5^{24}+
 x_1^{3}x_2x_3^{7}x_4^{28}x_5^{8}+
 x_1^{3}x_2x_3^{9}x_4^{20}x_5^{14}+
 x_1^{3}x_2x_3^{9}x_4^{28}x_5^{6}+
 x_1^{3}x_2x_3^{11}x_4^{20}x_5^{12}\\
&\quad + 
 x_1^{3}x_2x_3^{12}x_4^{17}x_5^{14}+
 x_1^{3}x_2x_3^{12}x_4^{19}x_5^{12}+
 x_1^{3}x_2x_3^{12}x_4^{21}x_5^{10}+
 x_1^{3}x_2x_3^{14}x_4^{17}x_5^{12}+
 x_1^{3}x_2x_3^{17}x_4^{12}x_5^{14}\\
&\quad + 
 x_1^{3}x_2x_3^{20}x_4^{9}x_5^{14}+
 x_1^{3}x_2^{2}x_3^{13}x_4^{13}x_5^{16}+
 x_1^{3}x_2^{2}x_3^{13}x_4^{21}x_5^{8}+
 x_1^{3}x_2^{4}x_3^{7}x_4^{25}x_5^{8}+
 x_1^{3}x_2^{4}x_3^{9}x_4^{11}x_5^{20}\\
&\quad + 
 x_1^{3}x_2^{4}x_3^{9}x_4^{21}x_5^{10}+
 x_1^{3}x_2^{4}x_3^{9}x_4^{25}x_5^{6}+
 x_1^{3}x_2^{4}x_3^{11}x_4^{12}x_5^{17}+
 x_1^{3}x_2^{4}x_3^{11}x_4^{20}x_5^{9}+
 x_1^{3}x_2^{4}x_3^{12}x_4^{11}x_5^{17}\\
&\quad + 
 x_1^{3}x_2^{4}x_3^{12}x_4^{19}x_5^{9}+
 x_1^{4}x_2x_3^{7}x_4^{11}x_5^{24}+
 x_1^{4}x_2x_3^{11}x_4^{17}x_5^{14}+
 x_1^{4}x_2x_3^{11}x_4^{19}x_5^{12}+
 x_1^{4}x_2x_3^{11}x_4^{21}x_5^{10}\\
&\quad + 
 x_1^{4}x_2x_3^{17}x_4^{11}x_5^{14}+
 x_1^{4}x_2^{2}x_3^{7}x_4^{25}x_5^{9}+
 x_1^{4}x_2^{2}x_3^{11}x_4^{13}x_5^{17}+
 x_1^{4}x_2^{2}x_3^{11}x_4^{21}x_5^{9}+
 x_1^{4}x_2^{4}x_3^{11}x_4^{11}x_5^{17}\\
&\quad + 
 x_1^{4}x_2^{4}x_3^{11}x_4^{19}x_5^{9}+
 x_1^{5}x_2x_3^{7}x_4^{24}x_5^{10}+
 x_1^{5}x_2x_3^{11}x_4^{12}x_5^{18}+
 x_1^{5}x_2x_3^{11}x_4^{14}x_5^{16}+
 x_1^{5}x_2x_3^{11}x_4^{20}x_5^{10}\\
&\quad + 
 x_1^{5}x_2x_3^{24}x_4^{7}x_5^{10}+
 x_1^{5}x_2x_3^{24}x_4^{11}x_5^{6}+
 x_1^{5}x_2^{2}x_3^{7}x_4^{24}x_5^{9}+
 x_1^{5}x_2^{2}x_3^{11}x_4^{12}x_5^{17}+
 x_1^{5}x_2^{2}x_3^{11}x_4^{20}x_5^{9}\\
\end{array}$$

\newpage
$$ \begin{array}{ll}
\sum \beta_4 &=  x_1^{3}x_2x_3^{7}x_4^{18}x_5^{14}+
 x_1^{3}x_2x_3^{7}x_4^{22}x_5^{10}+
 x_1^{3}x_2x_3^{10}x_4^{11}x_5^{18}+
 x_1^{3}x_2x_3^{10}x_4^{19}x_5^{10}+
 x_1^{3}x_2x_3^{11}x_4^{10}x_5^{18}\\
&\quad +
 x_1^{3}x_2x_3^{11}x_4^{22}x_5^{6}+
 x_1^{3}x_2x_3^{18}x_4^{7}x_5^{14}+
 x_1^{3}x_2^{2}x_3^{7}x_4^{11}x_5^{20}+
 x_1^{3}x_2^{2}x_3^{7}x_4^{21}x_5^{10}+
 x_1^{3}x_2^{2}x_3^{7}x_4^{25}x_5^{6}\\
&\quad +
 x_1^{3}x_2^{2}x_3^{11}x_4^{11}x_5^{16}+
 x_1^{3}x_2^{2}x_3^{11}x_4^{13}x_5^{14}+
 x_1^{3}x_2^{2}x_3^{11}x_4^{19}x_5^{8}+
 x_1^{3}x_2^{2}x_3^{11}x_4^{21}x_5^{6}+
 x_1^{3}x_2^{2}x_3^{13}x_4^{11}x_5^{14}\\
&\quad +
 x_1^{3}x_2^{2}x_3^{13}x_4^{19}x_5^{6}+
 x_1^{4}x_2^{2}x_3^{7}x_4^{13}x_5^{17}+
 x_1^{4}x_2^{2}x_3^{7}x_4^{21}x_5^{9}+
 x_1^{5}x_2x_3^{7}x_4^{12}x_5^{18}+
 x_1^{5}x_2x_3^{7}x_4^{18}x_5^{12}\\
&\quad +
 x_1^{5}x_2x_3^{7}x_4^{20}x_5^{10}+
 x_1^{5}x_2x_3^{11}x_4^{14}x_5^{12}+
 x_1^{5}x_2x_3^{14}x_4^{17}x_5^{6}+
 x_1^{5}x_2x_3^{20}x_4^{7}x_5^{10}+
 x_1^{5}x_2x_3^{20}x_4^{11}x_5^{6}\\
&\quad +
 x_1^{5}x_2^{2}x_3^{7}x_4^{11}x_5^{18}+
 x_1^{5}x_2^{2}x_3^{7}x_4^{12}x_5^{17}+
 x_1^{5}x_2^{2}x_3^{7}x_4^{19}x_5^{10}+
 x_1^{5}x_2^{2}x_3^{7}x_4^{20}x_5^{9}+
 x_1^{5}x_2^{4}x_3^{7}x_4^{17}x_5^{10}\\
&\quad +
 x_1^{5}x_2^{4}x_3^{11}x_4^{17}x_5^{6}+
 x_1^{11}x_2x_3^{7}x_4^{14}x_5^{10}+
 x_1^{11}x_2x_3^{10}x_4^{11}x_5^{10}+
 x_1^{11}x_2x_3^{18}x_4^{7}x_5^{6}+
\medskip
 x_1^{11}x_2^{2}x_3^{7}x_4^{17}x_5^{6},\\
\sum \beta_8 &=  x_1^{3}x_2x_3^{7}x_4^{10}x_5^{18}+
 x_1^{3}x_2x_3^{7}x_4^{22}x_5^{6}+
 x_1^{3}x_2x_3^{10}x_4^{7}x_5^{18}+
 x_1^{3}x_2x_3^{10}x_4^{19}x_5^{6}+
 x_1^{7}x_2x_3^{7}x_4^{14}x_5^{10}\\
&\quad +
 x_1^{7}x_2x_3^{10}x_4^{11}x_5^{10}+
 x_1^{7}x_2x_3^{12}x_4^{9}x_5^{10}+
 x_1^{7}x_2x_3^{12}x_4^{11}x_5^{8}+
 x_1^{7}x_2x_3^{18}x_4^{7}x_5^{6}+
\medskip
 x_1^{7}x_2^{2}x_3^{7}x_4^{17}x_5^{6}.
\end{array}$$
From the above equalities, we conclude that $Z_2$ is also strictly inadmissible.
\end{proof}

\begin{lema}\label{bd47.5}
The following monomials are strictly inadmissible:

\begin{center}
\begin{tabular}{lrrr}
$d_{1}= x_1^{3}x_2^{3}x_3^{12}x_4^{13}x_5^{16}$, & \multicolumn{1}{l}{$d_{2}=x_1^{3}x_2^{3}x_3^{13}x_4^{12}x_5^{16}$,} & \multicolumn{1}{l}{$d_{3}=x_1^{3}x_2^{3}x_3^{13}x_4^{20}x_5^{8}$,} & \multicolumn{1}{l}{$d_{4}= x_1^{3}x_2^{3}x_3^{12}x_4^{20}x_5^{9}$,} \\
$d_{5}=x_1^{3}x_2^{7}x_3^{4}x_4^{8}x_5^{25}$, & \multicolumn{1}{l}{$d_{6}=x_1^{3}x_2^{7}x_3^{4}x_4^{25}x_5^{8}$,} & \multicolumn{1}{l}{$d_{7}=x_1^{7}x_2^{3}x_3^{4}x_4^{8}x_5^{25}$,} & \multicolumn{1}{l}{$d_{8}=x_1^{7}x_2^{3}x_3^{4}x_4^{25}x_5^{8}$,} \\
$d_{9}=x_1^{3}x_2^{7}x_3^{4}x_4^{24}x_5^{9}$, & \multicolumn{1}{l}{$d_{10}=x_1^{3}x_2^{7}x_3^{4}x_4^{9}x_5^{24}$,} & \multicolumn{1}{l}{$d_{11}=x_1^{7}x_2^{3}x_3^{4}x_4^{24}x_5^{9}$,} & \multicolumn{1}{l}{$d_{12}=x_1^{7}x_2^{3}x_3^{4}x_4^{9}x_5^{24}$,} \\
$d_{13}=x_1^{3}x_2^{4}x_3^{11}x_4^{17}x_5^{12}$, & \multicolumn{1}{l}{$d_{14}=x_1^{3}x_2^{4}x_3^{11}x_4^{16}x_5^{13}$,} & \multicolumn{1}{l}{$d_{15}=x_1^{3}x_2^{4}x_3^{11}x_4^{13}x_5^{16}$,} & \multicolumn{1}{l}{$d_{16}=x_1^{3}x_2^{5}x_3^{6}x_4^{8}x_5^{25}$,} \\
$d_{17}=x_1^{3}x_2^{5}x_3^{6}x_4^{25}x_5^{8}$, & \multicolumn{1}{l}{$d_{18}=x_1^{3}x_2^{5}x_3^{6}x_4^{24}x_5^{9}$,} & \multicolumn{1}{l}{$d_{19}=x_1^{3}x_2^{5}x_3^{6}x_4^{9}x_5^{24}$,} & \multicolumn{1}{l}{$d_{20}=x_1^{3}x_2^{5}x_3^{7}x_4^{8}x_5^{24}$,} \\
$d_{21}=x_1^{3}x_2^{5}x_3^{7}x_4^{24}x_5^{8}$, & \multicolumn{1}{l}{$d_{22}=x_1^{3}x_2^{7}x_3^{5}x_4^{8}x_5^{24}$,} & \multicolumn{1}{l}{$d_{23}=x_1^{3}x_2^{7}x_3^{5}x_4^{24}x_5^{8}$,} & \multicolumn{1}{l}{$d_{24}=x_1^{7}x_2^{3}x_3^{5}x_4^{8}x_5^{24}$,} \\
$d_{25}=x_1^{7}x_2^{3}x_3^{5}x_4^{24}x_5^{8}$, & \multicolumn{1}{l}{$d_{26}=x_1^{3}x_2^{5}x_3^{10}x_4^{17}x_5^{12}$,} & \multicolumn{1}{l}{$d_{27}=x_1^{3}x_2^{5}x_3^{10}x_4^{16}x_5^{13}$,} & \multicolumn{1}{l}{$d_{28}=x_1^{3}x_2^{5}x_3^{11}x_4^{16}x_5^{12}$,} \\
\end{tabular}%
\end{center}
\end{lema}

\begin{proof}
Note that the weight vector of $d_i$ is $\overline{\omega}_{(1)}$ for $i = 1, 2, \ldots, 37.$ We prove the lemma for the monomial $d_1 = x_1^{3}x_2^{3}x_3^{12}x_4^{13}x_5^{16}$. The others can be proved by an argument similar to the proof of Lemma \ref{bd47.4}.
By a direct computation, we have
$$ \begin{array}{ll}
 d_1 &=  x_1^{2}x_2x_3^{3}x_4^{13}x_5^{28}+
 x_1^{2}x_2x_3^{3}x_4^{28}x_5^{13}+
 x_1^{2}x_2x_3^{5}x_4^{13}x_5^{26}+
 x_1^{2}x_2x_3^{5}x_4^{14}x_5^{25}+
 x_1^{2}x_2x_3^{5}x_4^{25}x_5^{14}\\
&\quad +
 x_1^{2}x_2x_3^{5}x_4^{26}x_5^{13}+
 x_1^{2}x_2x_3^{10}x_4^{13}x_5^{21}+
 x_1^{2}x_2x_3^{12}x_4^{7}x_5^{25}+
 x_1^{2}x_2x_3^{12}x_4^{11}x_5^{21}+
 x_1^{2}x_2x_3^{13}x_4^{14}x_5^{17}\\
&\quad +
 x_1^{2}x_2x_3^{13}x_4^{17}x_5^{14}+
 x_1^{2}x_2x_3^{18}x_4^{13}x_5^{13}+
 x_1^{2}x_2x_3^{20}x_4^{11}x_5^{13}+
 x_1^{2}x_2^{3}x_3^{5}x_4^{13}x_5^{24}+
 x_1^{2}x_2^{3}x_3^{5}x_4^{24}x_5^{13}\\
&\quad +
 x_1^{2}x_2^{3}x_3^{9}x_4^{13}x_5^{20}+
 x_1^{2}x_2^{3}x_3^{12}x_4^{13}x_5^{17}+
 x_1^{2}x_2^{3}x_3^{16}x_4^{13}x_5^{13}+
 x_1^{2}x_2^{5}x_3^{3}x_4^{13}x_5^{24}+
 x_1^{2}x_2^{5}x_3^{3}x_4^{24}x_5^{13}\\
&\quad +
 x_1^{2}x_2^{5}x_3^{9}x_4^{13}x_5^{18}+
 x_1^{2}x_2^{5}x_3^{9}x_4^{17}x_5^{14}+
 x_1^{2}x_2^{5}x_3^{12}x_4^{11}x_5^{17}+
 x_1^{2}x_2^{5}x_3^{16}x_4^{11}x_5^{13}+
 x_1^{3}x_2x_3^{4}x_4^{13}x_5^{26}\\
&\quad +
 x_1^{3}x_2x_3^{4}x_4^{14}x_5^{25}+
 x_1^{3}x_2x_3^{4}x_4^{25}x_5^{14}+
 x_1^{3}x_2x_3^{4}x_4^{26}x_5^{13}+
 x_1^{3}x_2x_3^{9}x_4^{14}x_5^{20}+
 x_1^{3}x_2x_3^{9}x_4^{20}x_5^{14}\\
&\quad +
 x_1^{3}x_2x_3^{10}x_4^{12}x_5^{21}+
 x_1^{3}x_2x_3^{12}x_4^{7}x_5^{24}+
 x_1^{3}x_2x_3^{12}x_4^{9}x_5^{22}+
 x_1^{3}x_2x_3^{12}x_4^{14}x_5^{17}+
 x_1^{3}x_2x_3^{12}x_4^{17}x_5^{14}\\
&\quad +
 x_1^{3}x_2x_3^{18}x_4^{12}x_5^{13}+
 x_1^{3}x_2x_3^{20}x_4^{9}x_5^{14}+
 x_1^{3}x_2^{2}x_3^{5}x_4^{13}x_5^{24}+
 x_1^{3}x_2^{2}x_3^{5}x_4^{24}x_5^{13}+
 x_1^{3}x_2^{2}x_3^{12}x_4^{17}x_5^{13}\\
\end{array}$$

\newpage
$$ \begin{array}{ll}
\sum Z &=  x_1^{3}x_2x_3^{7}x_4^{14}x_5^{18}+
 x_1^{3}x_2x_3^{7}x_4^{18}x_5^{14}+
 x_1^{3}x_2x_3^{18}x_4^{7}x_5^{14}+
 x_1^{3}x_2^{2}x_3^{3}x_4^{13}x_5^{22}+
 x_1^{3}x_2^{2}x_3^{3}x_4^{14}x_5^{21}\\
&\quad +
 x_1^{3}x_2^{2}x_3^{3}x_4^{21}x_5^{14}+
 x_1^{3}x_2^{2}x_3^{3}x_4^{22}x_5^{13}+
 x_1^{3}x_2^{2}x_3^{7}x_4^{14}x_5^{17}+
 x_1^{3}x_2^{2}x_3^{7}x_4^{17}x_5^{14}+
 x_1^{3}x_2^{2}x_3^{10}x_4^{7}x_5^{21}\\
&\quad +
 x_1^{3}x_2^{2}x_3^{18}x_4^{7}x_5^{13}+
 x_1^{3}x_2^{3}x_3^{3}x_4^{14}x_5^{20}+
 x_1^{3}x_2^{3}x_3^{3}x_4^{20}x_5^{14}+
 x_1^{3}x_2^{3}x_3^{4}x_4^{20}x_5^{13}+
 x_1^{3}x_2^{3}x_3^{5}x_4^{14}x_5^{18}\\
&\quad +
 x_1^{3}x_2^{3}x_3^{5}x_4^{18}x_5^{14}+
 x_1^{3}x_2^{3}x_3^{10}x_4^{11}x_5^{16}+
 x_1^{3}x_2^{3}x_3^{10}x_4^{13}x_5^{14}+
 x_1^{3}x_2^{3}x_3^{12}x_4^{11}x_5^{14}+
 x_1^{3}x_2^{8}x_3^{5}x_4^{13}x_5^{14}\\
&\quad +
 x_1^{3}x_2^{8}x_3^{5}x_4^{14}x_5^{13}+
\medskip
 x_1^{3}x_2^{8}x_3^{12}x_4^{7}x_5^{13},\\
\sum T &=  x_1^{3}x_2^{2}x_3^{12}x_4^{9}x_5^{13}+
 x_1^{3}x_2^{3}x_3^{8}x_4^{12}x_5^{13}+
 x_1^{3}x_2^{4}x_3^{5}x_4^{13}x_5^{14}\\
&\quad + x_1^{3}x_2^{4}x_3^{5}x_4^{14}x_5^{13}+
x_1^{3}x_2^{4}x_3^{10}x_4^{9}x_5^{13}+
 x_1^{3}x_2^{4}x_3^{12}x_4^{7}x_5^{13}.
\end{array}$$
The above relations show that $d_1$ is strictly inadmissible. The lemma follows.
\end{proof}

\begin{proof}[{\it Proof of  Proposition \ref{md47-1}}]
Let $a$ be an admissible monomial in $\mathscr P_5^+$ such that $\omega(a) = \overline{\omega}_{(1)}.$ Then $a = x_tx_kx_m u^2$ with $1\leq t < k <m \leq 5$ and $u\in (\mathscr P_5^+)_{22}.$ Since $a$ is admissible, by Theorem \ref{dlKS}, $u$ is also admissible. Further, $u\in \mathscr B_5^+(\omega_{(1)}).$

Let $X\in\mathscr{B}^+_5(\omega_{(1)})$ such that $x_tx_kx_mX^2\in (\mathscr P_5^+)_{47}.$ By a direct computation using Proposition \ref{md22-1}, we see that if $x_tx_kx_mX^2\neq \mathcal Y_{47,\, j},\ 1\leq j\leq 370,$ then there is a monomial $b$ which is given in one of Lemmas \ref{bd47.1} - \ref{bd47.5} such that $x_tx_kx_mX^2 = bY^{2^r}$ with suitable monomia $Y\in \mathscr P_5,$ and $r = {\rm max}\{s\in\mathbb{Z}:\, \omega_s(b) > 0\}.$ By Theorem \ref{dlKS}, $x_tx_kx_mX^2$ is  inadmissible. On the other hand, we have $a = x_tx_kx_mu^2$ and $a$  is admissible, hence $a = \mathcal Y_j$ for some $j = 1, 2, \ldots, 370.$ This completes the proof of the proposition.
\end{proof}

\begin{nx}\label{nx3}
For $1\leq j\leq 370,$ we have $[\mathcal Y_{j}:= \mathcal Y_{47,\, j}]_{\overline{\omega}_{(1)}}\neq [0].$ Indeed, suppose that there is a linear relation 
$$ \mathcal S = \sum\limits_{1\leq j\leq 170}\gamma_j\mathcal Y_{j}\equiv_{\overline{\omega}_{(1)}} 0,$$
where $\gamma_j\in\mathbb Z/2,$ for all $j.$ Based on Theorem \ref{dlsig} and Proposition \ref{mdPS}, for $(k;\mathscr K) \in \mathcal N_5$, we explicitly compute $\pi_{(k;\mathscr K)}(\mathcal S)$ in terms of a given minimal set of $\mathcal A_2$-generators in $\mathscr P_4\ ({\rm modulo}(\mathcal A_2^+\mathscr P_{4})).$ By computing from the relations $\pi_{(k;\mathscr K)}(\mathcal S) \equiv_{\overline{\omega}_{(1)}} 0,\ \ell(\mathscr K) = 1,$ one gets $\gamma_j = 0,\, \forall j,\, 1\leq j\leq 370.$ Note that these computations are similar to the proof of Propositions \ref{md21.1} and \ref{md22-1}. Combining this and Proposition \ref{md47-1}, we have a direct corollary.
\end{nx}

\begin{corls}\label{hq47.1}
The set $[\overline{\Phi}^+(\mathscr B_4(\overline{\omega}_{(1)}))\bigcup \big(\bigcup_{1\leq k\leq 5} \overline{\mathscr B}^+(k, \overline{\omega}_{(1)})\big) \bigcup \mathcal F]_{\overline{\omega}_{(1)}}$ is a basis of the $\mathbb Z/2$-vector space $Q\mathscr P_5^+(\overline{\omega}_{(1)}).$ This implies $\dim(Q\mathscr P_5^+(\overline{\omega}_{(1)}))  = 370.$
\end{corls}

\begin{nx}\label{nx4}
Consider the weight vector $\omega^{**} = (3,2,4)$ with $\deg \omega^{**} = 23.$ By using a result in \cite{N.T1}, we see that the following monomials are strictly inadmissible:

\begin{center}
\begin{tabular}{llrr}
$e_{1}=x_1^{7}x_2^{2}x_3^{4}x_4^{5}x_5^{5}$, & $e_{2}=x_1^{7}x_2^{2}x_3^{5}x_4^{4}x_5^{5}$, & \multicolumn{1}{l}{$e_{3}=x_1^{7}x_2^{2}x_3^{5}x_4^{5}x_5^{4}$,} & \multicolumn{1}{l}{$e_{4}=x_1^{3}x_2^{7}x_3^{4}x_4^{4}x_5^{5}$,} \\
$e_{5}=x_1^{3}x_2^{7}x_3^{4}x_4^{5}x_5^{4}$, & $e_{6}=x_1^{3}x_2^{7}x_3^{5}x_4^{4}x_5^{4}$, & \multicolumn{1}{l}{$e_{7}=x_1^{7}x_2^{3}x_3^{4}x_4^{4}x_5^{5}$,} & \multicolumn{1}{l}{$e_{8}=x_1^{7}x_2^{3}x_3^{4}x_4^{5}x_5^{4}$,} \\
$e_{9}=x_1^{7}x_2^{3}x_3^{5}x_4^{4}x_5^{4}$, & $e_{10}=x_1^{3}x_2^{4}x_3^{4}x_4^{5}x_5^{7}$, & \multicolumn{1}{l}{$e_{11}=x_1^{3}x_2^{4}x_3^{4}x_4^{7}x_5^{5}$,} & \multicolumn{1}{l}{$e_{12}=x_1^{3}x_2^{4}x_3^{5}x_4^{4}x_5^{7}$,} \\
$e_{13}=x_1^{3}x_2^{4}x_3^{7}x_4^{4}x_5^{5}$, & $e_{14}=x_1^{3}x_2^{4}x_3^{5}x_4^{7}x_5^{4}$, & \multicolumn{1}{l}{$e_{15}=x_1^{3}x_2^{4}x_3^{7}x_4^{5}x_5^{4}$,} & \multicolumn{1}{l}{$e_{16}=x_1^{3}x_2^{5}x_3^{4}x_4^{4}x_5^{7}$,} \\
$e_{17}=x_1^{3}x_2^{5}x_3^{4}x_4^{7}x_5^{4}$, & $e_{18}=x_1^{3}x_2^{5}x_3^{7}x_4^{4}x_5^{4}$, & \multicolumn{1}{l}{$e_{19}=x_1^{3}x_2^{6}x_3^{4}x_4^{5}x_5^{5}$,} & \multicolumn{1}{l}{$e_{20}=x_1^{3}x_2^{6}x_3^{5}x_4^{4}x_5^{5}$,} \\
$e_{21}=x_1^{3}x_2^{6}x_3^{5}x_4^{5}x_5^{4}$, & $e_{22}=x_1^{3}x_2^{4}x_3^{6}x_4^{5}x_5^{5}$, & \multicolumn{1}{l}{$e_{23}=x_1^{3}x_2^{4}x_3^{5}x_4^{6}x_5^{5}$,} & \multicolumn{1}{l}{$e_{24}=x_1^{3}x_2^{4}x_3^{5}x_4^{5}x_5^{6}$,} \\
$e_{25}=x_1^{3}x_2^{5}x_3^{4}x_4^{6}x_5^{5}$, & $e_{26}=x_1^{3}x_2^{5}x_3^{6}x_4^{4}x_5^{5}$, & \multicolumn{1}{l}{$e_{27}=x_1^{3}x_2^{5}x_3^{4}x_4^{5}x_5^{6}$,} & \multicolumn{1}{l}{$e_{28}=x_1^{3}x_2^{5}x_3^{6}x_4^{5}x_5^{4}$,} \\
$e_{29}=x_1^{3}x_2^{5}x_3^{5}x_4^{4}x_5^{6}$, & $e_{30}=x_1^{3}x_2^{5}x_3^{5}x_4^{6}x_5^{4}.$ &       &  \\
\end{tabular}%
\end{center}

Note that $\omega(e_k) = \omega^{**},\ k = 1, 2, \ldots, 30.$ Let $X\in \mathscr B_5^+(47)$ such that either $\omega(X) = \overline{\omega}_{(2)}$ or $\omega(X) = \overline{\omega}_{(3)}.$  Then $X = X_{(\{\ell, m\},\,5)}y^2$ with $1\leq \ell < m\leq 5,$ and $y\in (\mathscr P_5^+)_{22}.$ By a direct computation using Theorem \ref{dlKS}, and Proposition \ref{md22-2}(I), (II), we see that either $y\in\mathscr B_5^+(\omega_{(2)})$ or $y\in\mathscr B_5^+(\omega_{(3)})$ and there is a monomial $Z = e_i$ for some $i,\ 1\leq i\leq 30$ such that $X = X_{(\{\ell, m\},\,5)}y^2 = Zu^{2^t},\ 1\leq \ell < m\leq 5,$ with suitable monomia $u\in \mathscr P_5,$ and $t = {\rm max}\{a\in\mathbb{Z}:\, \omega_a(Z) > 0\}.$ By Theorem \ref{dlKS}, we get either $[X] = [0]_{\overline{\omega}_{(2)}}$ or $[X] = [0]_{\overline{\omega}_{(3)}}.$ As a consequence, we get the following.
\end{nx}

\begin{propo}\label{md47.2}
The $\mathbb Z/2$-vector spaces $Q\mathscr P_5^+(\overline{\omega}_{(2)})$ and $Q\mathscr{P}^+_5(\overline{\omega}_{(3)})$ are trivial.
\end{propo}

\begin{lema}\label{bd47-2}
The set $\{[\mathcal Y_{47,\, j}]_{\overline{\omega}_{(4)}}:\, 371\leq j\leq 479\}$ is a minimal system of generators for $Q\mathscr P_5^+(\overline{\omega}_{(4)})),$ where the monomials $\mathcal Y_{47,\, j},\ j = 371, \ldots, 479,$ are determined in Sect.\ref{s5.7}.
\end{lema}

In order to prove the lemma, we need to use some results. As an immediate consequence of a result in \cite{N.S3}, we obtain the following.

\begin{lema}\label{bd47.6}
The following monomials are strictly inadmissible:
\begin{enumerate}
\item[I)]  $x_a^7x_b^2x_c^{15}x_d^7, \ x_a^7x_b^{15}x_c^2x_d^7,\ x_a^7x_b^6x_c^3x_d^{15},\  x_a^7x_b^3x_c^6x_d^{15},\ x_a^7x_b^3x_c^{14}x_d^7, \\
\medskip
  x_a^7x_b^{14}x_c^3x_d^7, \ x_a^7x_b^6x_c^{11}x_d^7,\ x_a^7x_b^{11}x_c^6x_d^7, \ 1\leq b < c\leq 5,\ 1\leq a,\, d\leq 5,\ a\neq d,\ a, d \neq b, c;$

\item[II)] $\rho_{(k,\, 5)}(v),\, 1\leq k\leq 5,$ where $v$ is one of the following monomials:
$$ x_1^7x_2^7x_3^7x_4^{10},\  x_1^7x_2^7x_3^{10}x_4^{7},\  x_1^7x_2^{10}x_3^7x_4^{7}.$$
\end{enumerate}
\end{lema}

We now consider the weight vector $\omega^{***} = (3,4,3,1)$ with $\deg \omega^{***} = 31.$ The following lemma can be easily proved by a direct computation.

\begin{lema}\label{bd47.7}
All permutations of the following monomials are strictly inadmissible: 

\begin{center}
\begin{tabular}{llll}
$x_1x_2^2x_3^6x_4^7x_5^{15},$ & $x_1x_2^2x_3^14x_4^7x_5^7,$ & $x_1x_2^6x_3^7x_4^7x_5^{10},$ & $x_1^7x_2^2x_3^2x_4^5x_5^{15},$\\
$x_1^7x_2^2x_3^2x_4^7x_5^{13},$ & $x_1^3x_2^2x_3^4x_4^7x_5^{15},$ & $x_1^3x_2^2x_3^7x_4^7x_5^{12},$ & $x_1^7x_2^2x_3^4x_4^7x_5^{11},$\\
 $x_1^7x_2^2x_3^5x_4^7x_5^{10},$ & $x_1^7x_2^2x_3^6x_4^7x_5^{9},$ & $x_1^7x_2^2x_3^7x_4^7x_5^{8},$ & $x_1^3x_2^3x_3^4x_4^6x_5^{15},$\\
$x_1^3x_2^3x_3^4x_4^7x_5^{14},$ & $x_1^3x_2^3x_3^6x_4^7x_5^{12},$ & $x_1^3x_2^4x_3^6x_4^7x_5^{11},$ & $x_1^3x_2^4x_3^7x_4^7x_5^{10},$\\
$x_1^3x_2^6x_3^7x_4^7x_5^{8}.$ &&&
\end{tabular}
\end{center}
\end{lema}

\begin{lema}\label{bd47.7-1}
If $(m, n, p, q, r)$ is a permutation of $(1,2,3,4,5),$ then the following monomials are strictly inadmissible:
$$ \begin{array}{ll} 
x_m^3x_n^{13}x_p^6x_q^3x_r^{6}\neq X &\in \{x_1^3x_2^{3}x_3^{13}x_4^{6}x_5^{6},\ x_1^3x_2^{13}x_3^{3}x_4^{6}x_5^{6}\},\\
x_m^7x_n^2x_p^5x_q^6x_r^{11}\neq Y &\in \{ x_1^7x_2^{11}x_3^{5}x_4^2x_5^{6},\ x_1^7x_2^{11}x_3^{5}x_4^6x_5^{2}\},\\
x_m^3x_n^5x_p^6x_q^6x_r^{11}\neq Z &\in \{x_1^3x_2^{5}x_3^{6}x_4^{6}x_5^{11},\ x_1^3x_2^{5}x_3^{11}x_4^{6}x_5^{6}\},\\
x_m^3x_n^7x_p^9x_q^6x_r^{6}\neq G &\in \{x_1^3x_2^{7}x_3^{9}x_4^{6}x_5^{6},\ x_1^7x_2^{3}x_3^{9}x_4^{6}x_5^{6},\ 
x_1^7x_2^{9}x_3^{3}x_4^{6}x_5^{6}\},\\
x_m^3x_n^2x_p^6x_q^7x_r^{13}\neq H &\in \{ x_1^3x_2^7x_3^{13}x_4^2x_5^{6},\ x_1^3x_2^7x_3^{13}x_4^6x_5^{2},\ x_1^7x_2^3x_3^{13}x_4^2x_5^{6},\\ 
&\quad x_1^7x_2^3x_3^{13}x_4^6x_5^{2}\},\\
 x_m^3x_n^5x_p^6x_q^7x_r^{10}\neq F &\in \{ x_1^3x_2^{5}x_3^{6}x_4^{7}x_5^{10},\ x_1^3x_2^{5}x_3^{7}x_4^{6}x_5^{10},\ x_1^3x_2^{5}x_3^{7}x_4^{10}x_5^{6},\\ 
&\quad x_1^3x_2^{7}x_3^{5}x_4^{6}x_5^{10},\  x_1^3x_2^{7}x_3^{5}x_4^{10}x_5^{6},\ x_1^7x_2^{3}x_3^{5}x_4^{6}x_5^{10},\\ 
&\quad x_1^7x_2^{3}x_3^{5}x_4^{10}x_5^{6}\},\\
x_m^3x_n^5x_p^{14}x_q^3x_r^{6}\neq T &\in \{ x_1^3x_2^{3}x_3^{5}x_4^6x_5^{14},\ x_1^3x_2^{3}x_3^{5}x_4^{14}x_5^{6},\ x_1^3x_2^{5}x_3^{3}x_4^6x_5^{14},\\
 &\quad x_1^3x_2^{5}x_3^{3}x_4^{14}x_5^{6},\ x_1^3x_2^{5}x_3^{6}x_4^3x_5^{14},\ x_1^3x_2^{5}x_3^{6}x_4^{14}x_5^{3}\}.
\end{array}$$
\end{lema}

\begin{proof}
It is easy to see that the weight vector of these monomials is $\omega^{***}.$ Note that the monomials $X,\, Y,\, Z,\, G,\, H,\, F,$ and $T$ are admissible. We now prove the lemma for the monomials $f = x_m^3x_n^{13}x_p^{6}x_q^3x_r^{6}$ and $g = x_m^3x_n^5x_p^{14}x_q^3x_r^{6}.$  The others can be proved by an argument similar to the proofs of Lemmas \ref{bd47.1} and \ref{bd47.2}.
Applying the Cartan formula, we get
$$ f = x_m^3x_n^{13}x_p^{3}x_q^6x_r^{6} + Sq^1(\sum \overline{A}_1) + Sq^2(\sum \overline{B}_1) + Sq^4(\sum \overline{C}_1)\ {\rm modulo}(\mathscr P_5^-(\omega^{***})),\ {\rm where}$$ 
$$ \begin{array}{ll}
\sum \overline{A}_1 &=  x_m^{3}x_n^{7}x_p^{5}x_q^{5}x_r^{10}+
 x_m^{3}x_n^{7}x_p^{5}x_q^{6}x_r^{9}+
 x_m^{3}x_n^{7}x_p^{6}x_q^{5}x_r^{9}+
 x_m^{3}x_n^{7}x_p^{6}x_q^{9}x_r^{5}+
 x_m^{3}x_n^{7}x_p^{9}x_q^{6}x_r^{5}\\
&\quad +
 x_m^{3}x_n^{11}x_p^{5}x_q^{6}x_r^{5}+
 x_m^{3}x_n^{14}x_p^{3}x_q^{5}x_r^{5}+
\medskip
 x_m^{6}x_n^{11}x_p^{5}x_q^{3}x_r^{5},\\
\sum \overline{B}_1&=  x_m^{3}x_n^{11}x_p^{5}x_q^{5}x_r^{5}+
 x_m^{3}x_n^{11}x_p^{6}x_q^{3}x_r^{6}+
 x_m^{3}x_n^{13}x_p^{5}x_q^{3}x_r^{5}+
 x_m^{5}x_n^{7}x_p^{3}x_q^{5}x_r^{9}+
 x_m^{5}x_n^{7}x_p^{3}x_q^{9}x_r^{5}\\
&\quad +
 x_m^{5}x_n^{7}x_p^{5}x_q^{3}x_r^{9}+
 x_m^{5}x_n^{7}x_p^{9}x_q^{3}x_r^{5}+
 x_m^{5}x_n^{9}x_p^{3}x_q^{3}x_r^{9}+
\medskip
 x_m^{5}x_n^{11}x_p^{3}x_q^{5}x_r^{5},\\
\sum \overline{C}_1&=  x_m^{3}x_n^{7}x_p^{3}x_q^{5}x_r^{9}+
 x_m^{3}x_n^{7}x_p^{3}x_q^{9}x_r^{5}+
 x_m^{3}x_n^{7}x_p^{5}x_q^{3}x_r^{9}+
 x_m^{3}x_n^{7}x_p^{9}x_q^{3}x_r^{5}+
 x_m^{3}x_n^{9}x_p^{3}x_q^{3}x_r^{9}\\
&\quad +
 x_m^{3}x_n^{9}x_p^{5}x_q^{5}x_r^{5}+
 x_m^{3}x_n^{11}x_p^{3}x_q^{5}x_r^{5}+
\medskip
 x_m^{3}x_n^{11}x_p^{5}x_q^{3}x_r^{5},
\end{array}$$
Since $x_m^3x_n^{13}x_p^{3}x_q^6x_r^{6}  < f,$ the monomial $f$ is strictly inadmissible. Next, by a direct computation, we have
$$ g = \sum \mathcal L + Sq^1(\sum \overline{A}_2) + Sq^2(\sum \overline{B}_2) + Sq^4(\sum \overline{C}_2)\ {\rm modulo}(\mathscr P_5^-(\omega^{***})),\ {\rm where}$$ 
$$ \begin{array}{ll}
\medskip
\sum \mathcal L&= x_m^{3}x_n^{3}x_p^{14}x_q^{5}x_r^{6} + x_m^{3}x_n^{5}x_p^{7}x_q^{6}x_r^{10} + x_m^{3}x_n^{5}x_p^{11}x_q^{6}x_r^{6},\\
\sum \overline{A}_2&=  x_m^{3}x_n^{3}x_p^{13}x_q^{5}x_r^{6}+
 x_m^{3}x_n^{3}x_p^{14}x_q^{5}x_r^{5}+
 x_m^{3}x_n^{5}x_p^{11}x_q^{5}x_r^{6}+
 x_m^{3}x_n^{5}x_p^{13}x_q^{3}x_r^{6}+
 x_m^{3}x_n^{6}x_p^{7}x_q^{9}x_r^{5}\\
&\quad +
 x_m^{3}x_n^{6}x_p^{11}x_q^{5}x_r^{5}+
 x_m^{3}x_n^{10}x_p^{7}x_q^{5}x_r^{5}+
\medskip
 x_m^{6}x_n^{5}x_p^{11}x_q^{3}x_r^{5},\\
\sum \overline{B}_2&=  x_m^{3}x_n^{5}x_p^{11}x_q^{5}x_r^{5}+
 x_m^{3}x_n^{5}x_p^{13}x_q^{3}x_r^{5}+
 x_m^{3}x_n^{6}x_p^{11}x_q^{3}x_r^{6}+
 x_m^{5}x_n^{3}x_p^{7}x_q^{5}x_r^{9}+
 x_m^{5}x_n^{3}x_p^{7}x_q^{9}x_r^{5}\\
&\quad +
 x_m^{5}x_n^{3}x_p^{9}x_q^{3}x_r^{9}+
 x_m^{5}x_n^{3}x_p^{11}x_q^{5}x_r^{5}+
 x_m^{5}x_n^{5}x_p^{7}x_q^{3}x_r^{9}+
\medskip
 x_m^{5}x_n^{9}x_p^{7}x_q^{3}x_r^{5},\\
\sum \overline{C}_2&=  x_m^{3}x_n^{3}x_p^{7}x_q^{5}x_r^{9}+
 x_m^{3}x_n^{3}x_p^{7}x_q^{9}x_r^{5}+
 x_m^{3}x_n^{3}x_p^{9}x_q^{3}x_r^{9}+
 x_m^{3}x_n^{3}x_p^{11}x_q^{5}x_r^{5}+
 x_m^{3}x_n^{5}x_p^{7}x_q^{3}x_r^{9}\\
&\quad +
 x_m^{3}x_n^{5}x_p^{7}x_q^{6}x_r^{6}+
 x_m^{3}x_n^{5}x_p^{9}x_q^{5}x_r^{5}+
 x_m^{3}x_n^{5}x_p^{11}x_q^{3}x_r^{5}+
\medskip
 x_m^{3}x_n^{9}x_p^{7}x_q^{3}x_r^{5}.
\end{array}$$
These relations show that $g$ is strictly inadmissible. The lemmas follows.
\end{proof}

The proof of the following lemmas is analogous to the proofs of Lemmas \ref{bd47.1}, \ref{bd47.2}, \ref{bd47.7} and \ref{bd47.7-1}.

\begin{lema}\label{bd47.8}
If $(p, q, r)$ is a permutation of $(3,4,5),$ then the following monomials are strictly inadmissible:

\begin{center}
$u_1 = x_1x_2^6x_p^6x_q^3x_r^{15},\ u_2 = x_1^3x_2^6x_p^6x_qx_r^{15},\ u_3  =x_1^{16}x_2^6x_p^6x_qx_r^{3},\ u_4 =  x_1x_2^6x_p^{3}x_q^7x_r^{14},$

$u_5 =  x_1x_2^{14}x_p^{3}x_q^6x_r^{7},\ u_{6} =  x_1^3x_2^{6}x_px_q^7x_r^{14},\ u_{7} = x_1^3x_2^{14}x_px_q^6x_r^{7},\ u_{8} = x_1^7x_2^{6}x_px_q^3x_r^{14},$

$u_{9} = x_1^7x_2^{14}x_px_q^3x_r^{6},\  u_{10} = x_1x_2^6x_p^{6}x_q^7x_r^{11},\  u_{11} = x_1^7x_2^6x_px_q^6x_r^{11},\  u_{12} = x_1^3x_2^{2}x_p^{5}x_q^6x_r^{15},$

$ u_{13} = x_1^3x_2^{6}x_p^{2}x_q^5x_r^{15},\ u_{14} = x_1^{15}x_2^{2}x_p^{3}x_q^5x_r^{6},\ u_{15} = x_1^{15}x_2^{6}x_p^{2}x_q^3x_r^{5},\ u_{16} = x_1^{3}x_2^{2}x_p^{5}x_q^7x_r^{14},$

$ u_{17} = x_1^{3}x_2^{14}x_p^{2}x_q^5x_r^{7},\ u_{18} = x_1^{7}x_2^{2}x_p^{14}x_q^3x_r^{5},\ u_{19} = x_1^{7}x_2^{14}x_p^{2}x_q^3x_r^{5}.$
\end{center}
\end{lema}

\begin{lema}\label{bd47.9}
The following monomials are strictly inadmissible:
\begin{enumerate}
\item[I)] $u_{20} = x_1x_2^{15}x_3^6x_q^3x_r^{6},\ u_{21} = x_1^{3}x_2^{15}x_3^6x_qx_r^{6},\ u_{22} = x_1^{15}x_2x_3^6x_q^3x_r^{6},\ u_{23} = x_1^{15}x_2^3x_3^6x_qx_r^{6},\\  u_{24} = x_1^7x_2x_3^6x_q^3x_r^{14},\ u_{25} = x_1^7x_2x_3^{14}x_q^3x_r^{6},\ 
 u_{26} = x_1^7x_2^3x_3^{6}x_qx_r^{14},\ u_{27} = x_1^7x_2^3x_3^{14}x_qx_r^{6},\\
 u_{28} = x_1x_2^3x_3^{14}x_q^6x_r^{7},\ u_{29} = x_1x_2^7x_3^{6}x_q^3x_r^{14},\ u_{30} = x_1x_2^7x_3^{14}x_q^3x_r^{6},\ u_{31} = x_1^3x_2x_3^{14}x_q^6x_r^{7},\\ u_{32} =  x_1^3x_2^7x_3^{6}x_qx_r^{14},\ u_{33}  = x_1^3x_2^7x_3^{14}x_qx_r^{6},\ u_{34} = x_1x_2^7x_3^{6}x_q^6x_r^{11},\ u_{35} = x_1^7x_2x_3^{6}x_q^6x_r^{11},\\ u_{36} = x_1^7x_2^{11}x_3^{6}x_qx_r^{6},\ u_{37} = x_1^{3}x_2^{15}x_3^{2}x_q^5x_r^{6},\ u_{38} = x_1^{3}x_2^{15}x_3^{6}x_q^2x_r^{5},\ u_{39} = x_1^{15}x_2^{3}x_3^{2}x_q^5x_r^{6},\\ u_{40} = x_1^{15}x_2^{3}x_3^{6}x_q^2x_r^{5},\ u_{41} = x_1^{3}x_2^{5}x_3^{14}x_q^2x_r^{7},\ u_{42} = x_1^{3}x_2^{5}x_3^{14}x_q^2x_r^{7},\ u_{43} = x_1^{3}x_2^{7}x_3^{2}x_q^5x_r^{14},\\ u_{44} = x_1^{3}x_2^{7}x_3^{14}x_q^2x_r^{5},\ u_{45} = x_1^{7}x_2^{3}x_3^{2}x_q^5x_r^{14},\ u_{46} = x_1^{7}x_2^{3}x_3^{14}x_q^2x_r^{5},$\\ where $q,\, r = 4,\, 5,\ q\neq r.$

\medskip

\item[II)] $u_{47} = x_1x_2^{3}x_3^{6}x_4^{14}x_5^{7},\ u_{48} = x_1^3x_2x_3^{6}x_4^{14}x_5^{7},\ u_{49} = x_1^3x_2^{5}x_3^{2}x_4^{14}x_5^{7}.$
\end{enumerate}
\end{lema}
Note that $\omega(u_t) = \omega^{***}$ for $t = 1, 2, \ldots, 49.$ 

\begin{lema}\label{bd47.9}
The following monomials are strictly inadmissible:

\begin{center}
\begin{tabular}{llll}
$A_{1}=x_1x_2^{3}x_3^{6}x_4^{15}x_5^{22}$, & $A_{2}=x_1x_2^{3}x_3^{15}x_4^{6}x_5^{22}$, & $A_{3}=x_1x_2^{3}x_3^{15}x_4^{22}x_5^{6}$, & $A_{4}=x_1x_2^{15}x_3^{3}x_4^{6}x_5^{22}$, \\
$A_{5}=x_1x_2^{15}x_3^{3}x_4^{22}x_5^{6}$, & $A_{6}=x_1x_2^{15}x_3^{19}x_4^{6}x_5^{6}$, & $A_{7}=x_1^{3}x_2x_3^{6}x_4^{15}x_5^{22}$, & $A_{8}=x_1^{3}x_2^{5}x_3^{2}x_4^{15}x_5^{22}$, \\
$A_{9}=x_1^{3}x_2^{5}x_3^{6}x_4^{15}x_5^{18}$, & $A_{10}=x_1^{3}x_2x_3^{15}x_4^{6}x_5^{22}$, & $A_{11}=x_1^{3}x_2x_3^{15}x_4^{22}x_5^{6}$, & $A_{12}=x_1^{3}x_2^{5}x_3^{15}x_4^{2}x_5^{22}$, \\
$A_{13}=x_1^{3}x_2^{5}x_3^{15}x_4^{6}x_5^{18}$, & $A_{14}=x_1^{3}x_2^{5}x_3^{15}x_4^{18}x_5^{6}$, & $A_{15}=x_1^{3}x_2^{5}x_3^{15}x_4^{22}x_5^{2}$, & $A_{16}=x_1^{3}x_2^{13}x_3^{3}x_4^{6}x_5^{22}$, \\
$A_{17}=x_1^{3}x_2^{13}x_3^{3}x_4^{22}x_5^{6}$, & $A_{18}=x_1^{3}x_2^{13}x_3^{19}x_4^{6}x_5^{6}$, & $A_{19}=x_1^{3}x_2^{15}x_3x_4^{6}x_5^{22}$, & $A_{20}=x_1^{3}x_2^{15}x_3x_4^{22}x_5^{6}$, \\
$A_{21}=x_1^{3}x_2^{15}x_3^{5}x_4^{2}x_5^{22}$, & $A_{22}=x_1^{3}x_2^{15}x_3^{5}x_4^{6}x_5^{18}$, & $A_{23}=x_1^{3}x_2^{15}x_3^{5}x_4^{18}x_5^{6}$, & $A_{24}=x_1^{3}x_2^{15}x_3^{5}x_4^{22}x_5^{2}$, \\
$A_{25}=x_1^{3}x_2^{15}x_3^{17}x_4^{6}x_5^{6}$, & $A_{26}=x_1^{3}x_2^{15}x_3^{21}x_4^{2}x_5^{6}$, & $A_{27}=x_1^{3}x_2^{15}x_3^{21}x_4^{6}x_5^{2}$, & $A_{28}=x_1^{7}x_2^{9}x_3^{3}x_4^{6}x_5^{22}$, \\
$A_{29}=x_1^{7}x_2^{9}x_3^{3}x_4^{22}x_5^{6}$, & $A_{30}=x_1^{7}x_2^{9}x_3^{19}x_4^{6}x_5^{6}$, & $A_{31}=x_1^{7}x_2^{11}x_3^{17}x_4^{6}x_5^{6}$, & $A_{32}=x_1^{7}x_2^{25}x_3^{3}x_4^{6}x_5^{6}$, \\
$A_{33}=x_1^{15}x_2x_3^{3}x_4^{6}x_5^{22}$, & $A_{34}=x_1^{15}x_2x_3^{3}x_4^{22}x_5^{6}$, & $A_{35}=x_1^{15}x_2x_3^{19}x_4^{6}x_5^{6}$, & $A_{36}=x_1^{15}x_2^{3}x_3x_4^{6}x_5^{22}$, \\
$A_{37}=x_1^{15}x_2^{3}x_3x_4^{22}x_5^{6}$, & $A_{38}=x_1^{15}x_2^{3}x_3^{5}x_4^{2}x_5^{22}$, & $A_{39}=x_1^{15}x_2^{3}x_3^{5}x_4^{6}x_5^{18}$, & $A_{40}=x_1^{15}x_2^{3}x_3^{5}x_4^{18}x_5^{6}$, \\
$A_{41}=x_1^{15}x_2^{3}x_3^{5}x_4^{22}x_5^{2}$, & $A_{42}=x_1^{15}x_2^{3}x_3^{17}x_4^{6}x_5^{6}$, & $A_{43}=x_1^{15}x_2^{3}x_3^{21}x_4^{2}x_5^{6}$, & $A_{44}=x_1^{15}x_2^{3}x_3^{21}x_4^{6}x_5^{2}$, \\
$A_{45}=x_1^{15}x_2^{17}x_3^{3}x_4^{6}x_5^{6}$, & $A_{46}=x_1^{15}x_2^{19}x_3x_4^{6}x_5^{6}$, & $A_{47}=x_1^{15}x_2^{19}x_3^{5}x_4^{2}x_5^{6}$, & $A_{48}=x_1^{15}x_2^{19}x_3^{5}x_4^{6}x_5^{2}$,\\
$A_{49}=x_1^{3}x_2^{5}x_3^{6}x_4^{14}x_5^{19}$, & $A_{50}=x_1^{3}x_2^{5}x_3^{6}x_4^{6}x_5^{27},$ & $A_{51}=x_1^{3}x_2^{5}x_3^{27}x_4^{6}x_5^{6}.$
\end{tabular}
\end{center}
\end{lema}

\begin{proof}
We have $\omega(A_i) = \overline{\omega}_{(4)},\ 1\leq i\leq 51.$  We prove the lemma for the monomials $A_{23} = x_1^3x_2^{15}x_3^{5}x_4^{18}x_5^{6}$ and $A_{29} = x_1^7x_2^{9}x_3^{3}x_4^{22}x_5^{6}.$  The others can be proved by using a similar technique as in Lemmas \ref{bd47.4} and \ref{bd47.5}.
Direct computing from Cartan's formula, we get
$$ \begin{array}{ll}
\medskip
A_{23} &= \sum_{1\leq i\leq 3} b_i + Sq^1(\sum f)  + Sq^2(\sum g)  + Sq^4(\sum h) + Sq^8(\sum p)\ {\rm modulo}(\mathscr P_5^-(\overline{\omega}_{(4)}),\ {\rm where} \\ 
\medskip
b_1&= x_1^3x_2^{11}x_3^{5}x_4^{22}x_5^{6},\ b_2 =  x_1^3x_2^{13}x_3^{3}x_4^{22}x_5^{6},\ b_3 = x_1^3x_2^{13}x_3^{6}x_4^{19}x_5^{6},\\
\sum f&=   x_1^{3}x_2^{7}x_3^{6}x_4^{21}x_5^{9}+
 x_1^{3}x_2^{7}x_3^{6}x_4^{25}x_5^{5}+
 x_1^{3}x_2^{11}x_3^{6}x_4^{21}x_5^{5}+
 x_1^{3}x_2^{14}x_3^{3}x_4^{21}x_5^{5}+
 x_1^{3}x_2^{15}x_3^{5}x_4^{14}x_5^{9}\\
&\quad +
 x_1^{3}x_2^{15}x_3^{5}x_4^{17}x_5^{6}+
 x_1^{3}x_2^{15}x_3^{9}x_4^{14}x_5^{5}+
 x_1^{3}x_2^{19}x_3^{5}x_4^{14}x_5^{5}+
\medskip
 x_1^{3}x_2^{22}x_3^{5}x_4^{11}x_5^{5},\\
\sum g &=  x_1^{5}x_2^{7}x_3^{3}x_4^{21}x_5^{9}+
 x_1^{5}x_2^{7}x_3^{3}x_4^{25}x_5^{5}+
 x_1^{5}x_2^{9}x_3^{9}x_4^{19}x_5^{3}+
 x_1^{5}x_2^{11}x_3^{3}x_4^{17}x_5^{9}+
 x_1^{5}x_2^{15}x_3^{5}x_4^{11}x_5^{9}\\
&\quad +
 x_1^{5}x_2^{15}x_3^{9}x_4^{11}x_5^{5}+
 x_1^{5}x_2^{17}x_3^{3}x_4^{17}x_5^{3}+
 x_1^{5}x_2^{17}x_3^{9}x_4^{11}x_5^{3}+
\medskip
 x_1^{5}x_2^{19}x_3^{5}x_4^{11}x_5^{5},\\
\sum h&= x_1^{3}x_2^{7}x_3^{3}x_4^{21}x_5^{9}+
 x_1^{3}x_2^{7}x_3^{3}x_4^{25}x_5^{5}+
 x_1^{3}x_2^{9}x_3^{5}x_4^{21}x_5^{5}+
 x_1^{3}x_2^{9}x_3^{9}x_4^{19}x_5^{3}+
 x_1^{3}x_2^{11}x_3^{3}x_4^{17}x_5^{9}\\
&\quad +
 x_1^{3}x_2^{11}x_3^{3}x_4^{21}x_5^{5}+
 x_1^{3}x_2^{13}x_3^{9}x_4^{13}x_5^{5}+
 x_1^{3}x_2^{15}x_3^{5}x_4^{11}x_5^{9}+
 x_1^{3}x_2^{15}x_3^{6}x_4^{13}x_5^{6}+
 x_1^{3}x_2^{15}x_3^{9}x_4^{11}x_5^{5}\\
&\quad +
 x_1^{3}x_2^{17}x_3^{3}x_4^{17}x_5^{3}+
 x_1^{3}x_2^{17}x_3^{9}x_4^{11}x_5^{3}+
\medskip
 x_1^{3}x_2^{19}x_3^{5}x_4^{11}x_5^{5},\\
\sum p&= x_1^3x_2^{13}x_3^{5}x_4^{13}x_5^{5} + x_1^3x_2^{13}x_3^{6}x_4^{11}x_5^{6}.
\end{array}$$
Since $b_i <  A_1,\ 1\leq i\leq 3,$ $A_{23}$ is strictly inadmissible. By a similar computation, we obtain
$$ A_{29} = \sum Z + Sq^1(\sum \overline{f})  + Sq^2(\sum \overline{g})  + Sq^4(\sum \overline{h}) + Sq^8(x_1^{7}x_2^{5}x_3^{6}x_4^{15}x_5^{6})\ {\rm modulo}(\mathscr P_5^-(\overline{\omega}_{(4)}),$$
where the polynomials $\sum Z,\, \sum\overline{f},\, \sum\overline{g},$ and $\sum\overline{h}$ are as follows:
$$ \begin{array}{ll}
\sum Z &=  x_1^{5}x_2^{3}x_3^{6}x_4^{23}x_5^{10}+
 x_1^{5}x_2^{3}x_3^{6}x_4^{27}x_5^{6}+
 x_1^{5}x_2^{3}x_3^{10}x_4^{23}x_5^{6}+
 x_1^{5}x_2^{7}x_3^{3}x_4^{22}x_5^{10}+
 x_1^{5}x_2^{7}x_3^{3}x_4^{26}x_5^{6}\\
&\quad +
 x_1^{5}x_2^{7}x_3^{6}x_4^{19}x_5^{10}+
 x_1^{5}x_2^{7}x_3^{10}x_4^{19}x_5^{6}+
 x_1^{5}x_2^{11}x_3^{3}x_4^{22}x_5^{6}+
 x_1^{5}x_2^{11}x_3^{6}x_4^{19}x_5^{6}+
 x_1^{7}x_2^{3}x_3^{5}x_4^{22}x_5^{10}\\
&\quad +
 x_1^{7}x_2^{3}x_3^{6}x_4^{23}x_5^{8}+
 x_1^{7}x_2^{3}x_3^{6}x_4^{25}x_5^{6}+
 x_1^{7}x_2^{3}x_3^{8}x_4^{23}x_5^{6}+
 x_1^{7}x_2^{3}x_3^{9}x_4^{22}x_5^{6}+
 x_1^{7}x_2^{7}x_3^{3}x_4^{22}x_5^{8}\\
\medskip
&\quad +x_1^{7}x_2^{7}x_3^{3}x_4^{24}x_5^{6},\\
\sum \overline{f}&=  x_1^{7}x_2^{3}x_3^{5}x_4^{19}x_5^{12}+
 x_1^{7}x_2^{3}x_3^{5}x_4^{21}x_5^{10}+
 x_1^{7}x_2^{3}x_3^{9}x_4^{21}x_5^{6}+
 x_1^{7}x_2^{3}x_3^{12}x_4^{19}x_5^{5}+
 x_1^{7}x_2^{6}x_3^{9}x_4^{15}x_5^{9}\\
&\quad + x_1^{7}x_2^{7}x_3^{5}x_4^{19}x_5^{8}+
 x_1^{7}x_2^{7}x_3^{5}x_4^{21}x_5^{6}+
 x_1^{7}x_2^{7}x_3^{8}x_4^{19}x_5^{5}+
 x_1^{7}x_2^{10}x_3^{5}x_4^{15}x_5^{9}\\
\medskip
&\quad + x_1^{7}x_2^{10}x_3^{9}x_4^{15}x_5^{5},\\
\sum \overline{g}&=  x_1^{7}x_2^{3}x_3^{6}x_4^{19}x_5^{10}+
 x_1^{7}x_2^{3}x_3^{6}x_4^{23}x_5^{6}+
 x_1^{7}x_2^{3}x_3^{9}x_4^{17}x_5^{9}+
 x_1^{7}x_2^{3}x_3^{10}x_4^{19}x_5^{6}+
 x_1^{7}x_2^{5}x_3^{9}x_4^{15}x_5^{9}\\
&\quad +
 x_1^{7}x_2^{7}x_3^{3}x_4^{22}x_5^{6}+
 x_1^{7}x_2^{7}x_3^{6}x_4^{19}x_5^{6}+
 x_1^{7}x_2^{9}x_3^{3}x_4^{17}x_5^{9}+
 x_1^{7}x_2^{9}x_3^{5}x_4^{15}x_5^{9}+
 x_1^{7}x_2^{9}x_3^{9}x_4^{15}x_5^{5}\\
&\quad +
 x_1^{7}x_2^{9}x_3^{9}x_4^{17}x_5^{3}+
 x_1^{9}x_2^{3}x_3^{9}x_4^{15}x_5^{9}+
 x_1^{9}x_2^{9}x_3^{3}x_4^{15}x_5^{9}+
\medskip
 x_1^{9}x_2^{9}x_3^{9}x_4^{15}x_5^{3},\\
\sum \overline{h}&=  x_1^{5}x_2^{3}x_3^{6}x_4^{23}x_5^{6}+
 x_1^{5}x_2^{7}x_3^{3}x_4^{22}x_5^{6}+
\medskip
 x_1^{5}x_2^{7}x_3^{6}x_4^{19}x_5^{6}.
\end{array}$$
The above equalities imply that $A_{29}$ is strictly inadmissible. The lemma is proved.
\end{proof}

\begin{proof}[{\it Proof of  Lemma \ref{bd47-2}}]
Let $b$ be an admissible monomial in $(\mathscr P_5^+)_{47}$ such that $\omega(b) = \overline{\omega}_{(4)}.$ Then $\omega_1(b) = 3$ and $b = X_{(\{t,\, k\},\, 5)}Y^2$ with $1\leq t < k \leq 5$ and $Y$ a monomial of degree $22$ in $\mathscr P_5.$ Since $b$ is admissible, according to Theorem \ref{dlKS}, $Y\in \mathscr B_5(\omega_{(4)}).$

Using Proposition \ref{md22-2}(III) and a simple computation shows that if $Z\in \mathscr B_5(\omega_{(4)}),\ 1\leq t < k \leq 5,$ and $X_{(\{t,\, k\},\, 5)}Z^2\neq \mathcal Y_{47,\,j},\ \forall j,\ 371\leq j\leq 479,$ then there is a monomial $u$ which is given in one of Lemmas \ref{bd47.6} - \ref{bd47.9} 
such that $X_{(\{t,\, k\},\, 5)}Z^2 = ug^{2^s}$ with a monomial $g\in \mathscr P_5,$ and $s = {\rm max}\{\ell\in\mathbb{Z}:\, \omega_{\ell}(u) > 0\}.$ Then, by Theorem \ref{dlKS}, $X_{(\{t,\, k\},\, 5)}Z^2$ is  inadmissible. Finally, we see that $b = X_{(\{t,\, k\},\, 5)}Y^2$ is admissible with $Y\in \mathscr B_5(\omega_{(4)});$ hence $b = \mathcal Y_{47,\,j}$ for some $j,\ j\in \{371, \ldots, 479\}.$ This implies $\mathscr B_5^+(\overline{\omega}_{(4)}) \subseteq \{\mathcal Y_{47,\,j}:\ 371\leq j\leq 479\}$. 
\end{proof}

\begin{propo}\label{md47.3}
$Q\mathscr P_5^+(\overline{\omega}_{(4)})$ is the $\mathbb Z/2$-vector space of dimension $109$ with a basis consisting of all the classes represented by the monomials $Y_{47,\,j},\ 371\leq j\leq 479.$
\end{propo}

\begin{proof}
First, we show that the set $[V:=\{Y_{47,\,j}:\ 371\leq j\leq 479\}]_{\overline{\omega}_{(4)}}$ is  linearly independent in the space $Q\mathscr P_5^+(\overline{\omega}_{(4)}).$ Indeed,  suppose there is a linear relation $\mathcal S = \sum_{371\leq j\leq 479}\gamma_j\mathcal Y_{j}\equiv_{\overline{\omega}_{(4)}} 0$ with $\gamma_j\in\mathbb Z/2$ and $\mathcal Y_{47,\,j}\in V.$ By using Theorem \ref{dlsig} and Proposition \ref{mdPS},  we determine explicitly  $\pi_{(k;\mathscr K)}(\mathcal S)$ in terms of the admissible monomials in $(\mathscr P_4^+)_{47}.$ From the relations $\pi_{(k;\mathscr K)}(\mathcal S) \equiv_{\overline{\omega}_{(4)}} 0$ with $\ell(\mathscr K) > 0,$ one gets $\gamma_j = 0$ for $j = 371, \ldots, 479.$ 

Now, by Lemma \ref{bd47-2}, to prove $[V]_{\overline{\omega}_{(4)}}$ is a basis of $Q\mathscr P_5^+(\overline{\omega}_{(4)})$ we need to show that $[\mathcal Y_{47,\,j}]_{\overline{\omega}_{(4)}}\neq [0]$ for all $\mathcal Y_{47,\,j}\in V.$ By a similar argument as given in the proof of Propositions \ref{md21.1} and \ref{md22-1},  we can prove that the set $[\mathscr B_5^+(\overline{\omega}_{(1)})\cup V]$ is  linearly independent in $(Q\mathscr P_5^+)_{47}.$ This fact shows that $[\mathcal Y_{47,\,j}]_{\overline{\omega}_{(4)}}\neq [0]$ for all $\mathcal Y_{47,\,j}.$ The proposition is proved.
 \end{proof}

\begin{propo}\label{md47.4}
There exist exactly $15$ admissible monomials in $(\mathscr P_5^+)_{47}$ such that their weight vectors are $\overline{\omega}_{(5)}.$ Consequently $\dim(QP_5^+(\overline{\omega}_{(5)})) = 15.$
\end{propo}

We prove the proposition by showing that 
$$ \mathscr B_5^+(\overline{\omega}_{(5)}) = \{\mathcal Y_{47,\, m}:\ 480\leq m\leq 494\},$$ 
where the monomials $\mathcal Y_{47,\, m}:\ 480\leq m\leq 494,$ are described in Sect.\ref{s5.7}. We need some lemmas for the proof of this proposition. The following lemma can be proved by using a result in \cite{N.T1}.
\begin{lema}\label{bd47.10}
The following monomials are strictly inadmissible:
\begin{enumerate}
\item[a)] $x_ux_v^2x_m^6x_n^7x_p^7,\ x_u^7x_v^2x_m^2x_n^5x_p^7,\ x_u^3x_v^2x_m^4x_n^7x_p^7,\ x_u^3x_v^4x_m^6x_n^3x_p^7,\ x_u^3x_v^6x_m^6x_n^3x_p^5,$\\ where $(u, v, m, n, p)$ is a permutation of $(1,2,3,4,5)$;

\medskip

\item[b)] $x_1x_2^6x_q^6x_r^3x_t^7,\ x_1^3x_2^6x_q^6x_rx_t^7,\ x_1^7x_2^6x_q^6x_rx_t^3,\ x_1^3x_2^2x_q^6x_r^5x_t^7,\ x_1^3x_2^6x_q^2x_r^5x_t^7,\ x_1^7x_2^2x_q^6x_r^3x_t^5,\\ x_1^7x_2^6x_q^2x_r^3x_t^5,$ where $(q, r, t)$ is a permutation of $(3,4,5)$;

\medskip

\item[c)] 
\begin{tabular}[t]{lrrrrrr}
$ x_1x_2^{7}x_3^{6}x_4^{6}x_5^{3}$, & $ x_1x_2^{7}x_3^{6}x_4^{3}x_5^{6}$, & $ x_1^{3}x_2^{7}x_3^{6}x_4^{6}x_5$, & $ x_1^{3}x_2^{7}x_3^{6}x_4x_5^{6}$,\\
  \multicolumn{1}{l}{$ x_1^{7}x_2x_3^{6}x_4^{6}x_5^{3}$,} & \multicolumn{1}{l}{$ x_1^{7}x_2^{3}x_3^{6}x_4^{6}x_5$,} &
$ x_1^{7}x_2x_3^{6}x_4^{3}x_5^{6}$, & $ x_1^{7}x_2^{3}x_3^{6}x_4x_5^{6}$,\\
  $ x_1^{3}x_2^{7}x_3^{2}x_4^{6}x_5^{5}$, & $ x_1^{3}x_2^{7}x_3^{6}x_4^{2}x_5^{5}$, & \multicolumn{1}{l}{$ x_1^{3}x_2^{7}x_3^{2}x_4^{5}x_5^{6}$,} & \multicolumn{1}{l}{$ x_1^{3}x_2^{7}x_3^{6}x_4^{5}x_5^{2}$,} \\
$ x_1^{7}x_2^{3}x_3^{2}x_4^{6}x_5^{5}$, & $ x_1^{7}x_2^{3}x_3^{6}x_4^{2}x_5^{5}$, & $ x_1^{7}x_2^{3}x_3^{2}x_4^{5}x_5^{6}$, & $ x_1^{7}x_2^{3}x_3^{6}x_4^{5}x_5^{2}$; 
\end{tabular}

\medskip

\item[d)]  $\rho_{(k,\, 5)}(Z),\ 1\leq k\leq 5,$ where $Z$ is one of the following monomials:

\begin{tabular}[t]{lrrrrrr}
$ x_1^{2}x_2^{7}x_3^{7}x_4^{7}$, & $ x_1^{7}x_2^{2}x_3^{7}x_4^{7}$, & $ x_1^{7}x_2^{7}x_3^{2}x_4^{7}$, & $ x_1^{7}x_2^{7}x_3^{7}x_4^{2}$, & \multicolumn{1}{l}{$ x_1^{6}x_2^{3}x_3^{7}x_4^{7}$,} & $ x_1^{6}x_2^{7}x_3^{3}x_4^{7}$, \\
$ x_1^{6}x_2^{7}x_3^{7}x_4^{3}$, & $ x_1^{3}x_2^{6}x_3^{7}x_4^{7}$, & $ x_1^{7}x_2^{6}x_3^{3}x_4^{7}$, & $ x_1^{7}x_2^{6}x_3^{7}x_4^{3}$, & \multicolumn{1}{l}{$ x_1^{3}x_2^{7}x_3^{6}x_4^{7}$,} & $ x_1^{7}x_2^{3}x_3^{6}x_4^{7}$, \\
$ x_1^{7}x_2^{7}x_3^{6}x_4^{3}$, & $ x_1^{3}x_2^{7}x_3^{7}x_4^{6}$, & $ x_1^{7}x_2^{3}x_3^{7}x_4^{6}$, & $ x_1^{7}x_2^{7}x_3^{3}x_4^{6}$. &       &  
\end{tabular}%
\end{enumerate} 
\end{lema}

\begin{lema}\label{bd47.11}
The following monomials are strictly inadmissible:

\begin{center}
\begin{tabular}{llll}
$T_{1}=x_1x_2^{3}x_3^{14}x_4^{14}x_5^{15}$, & $T_{2}=x_1x_2^{3}x_3^{14}x_4^{15}x_5^{14}$, & $T_{3}=x_1x_2^{3}x_3^{15}x_4^{14}x_5^{14}$, & $T_{4}=x_1x_2^{15}x_3^{3}x_4^{14}x_5^{14}$, \\
$T_{5}=x_1^{3}x_2x_3^{14}x_4^{14}x_5^{15}$, & $T_{6}=x_1^{3}x_2x_3^{14}x_4^{15}x_5^{14}$, & $T_{7}=x_1^{3}x_2x_3^{15}x_4^{14}x_5^{14}$, & $T_{8}=x_1^{3}x_2^{15}x_3x_4^{14}x_5^{14}$, \\
$T_{9}=x_1^{15}x_2x_3^{3}x_4^{14}x_5^{14}$, & $T_{10}=x_1^{15}x_2^{3}x_3x_4^{14}x_5^{14}$, & $T_{11}=x_1^{3}x_2^{13}x_3^{2}x_4^{14}x_5^{15}$, & $T_{12}=x_1^{3}x_2^{13}x_3^{2}x_4^{15}x_5^{14}$, \\
$T_{13}=x_1^{3}x_2^{13}x_3^{14}x_4^{2}x_5^{15}$, & $T_{14}=x_1^{3}x_2^{13}x_3^{14}x_4^{15}x_5^{2}$, & $T_{15}=x_1^{3}x_2^{13}x_3^{15}x_4^{2}x_5^{14}$, & $T_{16}=x_1^{3}x_2^{13}x_3^{15}x_4^{14}x_5^{2}$, \\
$T_{17}=x_1^{3}x_2^{15}x_3^{13}x_4^{2}x_5^{14}$, & $T_{18}=x_1^{3}x_2^{15}x_3^{13}x_4^{14}x_5^{2}$, & $T_{19}=x_1^{15}x_2^{3}x_3^{13}x_4^{2}x_5^{14}$, & $T_{20}=x_1^{15}x_2^{3}x_3^{13}x_4^{14}x_5^{2}$, \\
$T_{21}=x_1^{3}x_2^{3}x_3^{13}x_4^{14}x_5^{14}$, & $T_{22}=x_1^{3}x_2^{13}x_3^{3}x_4^{14}x_5^{14}$, & $T_{23}=x_1^{3}x_2^{13}x_3^{14}x_4^{3}x_5^{14}$, & $T_{24}=x_1^{3}x_2^{13}x_3^{14}x_4^{14}x_5^{3}$, \\
$T_{25}=x_1^{3}x_2^{5}x_3^{10}x_4^{14}x_5^{15}$, & $T_{26}=x_1^{3}x_2^{5}x_3^{10}x_4^{15}x_5^{14}$, & $T_{27}=x_1^{3}x_2^{5}x_3^{14}x_4^{10}x_5^{15}$, & $T_{28}=x_1^{3}x_2^{5}x_3^{14}x_4^{15}x_5^{10}$, \\
$T_{29}=x_1^{3}x_2^{5}x_3^{15}x_4^{10}x_5^{14}$, & $T_{30}=x_1^{3}x_2^{5}x_3^{15}x_4^{14}x_5^{10}$, & $T_{31}=x_1^{3}x_2^{15}x_3^{5}x_4^{10}x_5^{14}$, & $T_{32}=x_1^{3}x_2^{15}x_3^{5}x_4^{14}x_5^{10}$, \\
$T_{33}=x_1^{15}x_2^{3}x_3^{5}x_4^{10}x_5^{14}$, & $T_{34}=x_1^{15}x_2^{3}x_3^{5}x_4^{14}x_5^{10}$, & 
 $T_{35}=x_1^{3}x_2^{5}x_3^{14}x_4^{11}x_5^{14}$, & $T_{36}=x_1^{3}x_2^{5}x_3^{14}x_4^{14}x_5^{11}$, \\ $T_{37}=x_1^{3}x_2^{13}x_3^{6}x_4^{10}x_5^{15}$, & $T_{38}=x_1^{3}x_2^{13}x_3^{6}x_4^{15}x_5^{10}$, &
$T_{39}=x_1^{3}x_2^{13}x_3^{15}x_4^{6}x_5^{10}$, & $T_{40}=x_1^{3}x_2^{15}x_3^{13}x_4^{6}x_5^{10}$, \\ $T_{41}=x_1^{15}x_2^{3}x_3^{13}x_4^{6}x_5^{10}$, & $T_{42}=x_1^{3}x_2^{13}x_3^{6}x_4^{11}x_5^{14}$, &
$T_{43}=x_1^{3}x_2^{13}x_3^{6}x_4^{14}x_5^{11}$, & $T_{44}=x_1^{3}x_2^{13}x_3^{14}x_4^{6}x_5^{11}$, \\
$T_{45}=x_1^{7}x_2^{9}x_3^{3}x_4^{14}x_5^{14}$, & $T_{46}=x_1^{3}x_2^{13}x_3^{7}x_4^{10}x_5^{14}$, &
$T_{47}=x_1^{3}x_2^{13}x_3^{7}x_4^{14}x_5^{10}$, & $T_{48}=x_1^{3}x_2^{13}x_3^{14}x_4^{7}x_5^{10}.$ 
\end{tabular}%
\end{center}
\end{lema}

\begin{proof}
It is easy to see that $\omega(T_s) = \overline{\omega}_{(5)}$ for $s = 1, 2, \ldots, 48.$ We prove the lemma for the monomials $T_{35}$ and $T_{42}.$  The others can be proved by using a similar technique as in Lemmas \ref{bd47.4}, \ref{bd47.5} and \ref{bd47.9}. Computing from the Cartan formula, we get
$$ \begin{array}{ll}
T_{35} &= Sq^1\big(x_1^3x_2^{6}x_3^{11}x_4^{13}x_5^{13} + x_1^3x_2^{6}x_3^{13}x_4^{11}x_5^{13}\big) + Sq^2\big( x_1^{5}x_2^{3}x_3^{7}x_4^{13}x_5^{17}+ x_1^{5}x_2^{3}x_3^{7}x_4^{17}x_5^{13}\\
&\quad +
 x_1^{5}x_2^{3}x_3^{9}x_4^{11}x_5^{17}+
 x_1^{5}x_2^{3}x_3^{9}x_4^{17}x_5^{11}+
 x_1^{5}x_2^{3}x_3^{11}x_4^{9}x_5^{17}+
 x_1^{5}x_2^{3}x_3^{11}x_4^{13}x_5^{13}\\
&\quad + x_1^{5}x_2^{3}x_3^{13}x_4^{7}x_5^{17}+
x_1^{5}x_2^{3}x_3^{13}x_4^{11}x_5^{13}+
 x_1^{5}x_2^{3}x_3^{17}x_4^{7}x_5^{13}+
 x_1^{5}x_2^{3}x_3^{17}x_4^{9}x_5^{11}\big)\\
&\quad  + Sq^4\big( x_1^{3}x_2^{3}x_3^{7}x_4^{13}x_5^{17}+
 x_1^{3}x_2^{3}x_3^{7}x_4^{17}x_5^{13}+
 x_1^{3}x_2^{3}x_3^{9}x_4^{11}x_5^{17}+
 x_1^{3}x_2^{3}x_3^{9}x_4^{17}x_5^{11}\\
&\quad+x_1^{3}x_2^{3}x_3^{11}x_4^{9}x_5^{17}
 x_1^{3}x_2^{3}x_3^{11}x_4^{13}x_5^{13}+
 x_1^{3}x_2^{3}x_3^{13}x_4^{7}x_5^{17}+
 x_1^{3}x_2^{3}x_3^{13}x_4^{11}x_5^{13}\\
&\quad + x_1^{3}x_2^{3}x_3^{17}x_4^{7}x_5^{13}+
 x_1^{3}x_2^{3}x_3^{17}x_4^{9}x_5^{11}\big)+ x_1^3x_2^{5}x_3^{11}x_4^{14}x_5^{14}\ {\rm modulo}(\mathscr P_5^-(\overline{\omega}_{(5)})).
\end{array}$$
Obviously, $T_{35} > x_1^3x_2^{5}x_3^{11}x_4^{14}x_5^{14},$ hence $T_{35}$ is strictly inadmissible. By a similar computation, we claim that
$$ T_{42}  = (p_1  + p_2)\ {\rm modulo}(\mathcal A_2^+\mathscr P_5 + \mathscr P_5^-(\overline{\omega}_{(5)})),$$
where $p_1 = x_1^3x_2^{7}x_3^{10}x_4^{13}x_5^{14},$ and $p_2 = x_1^3x_2^{13}x_3^{3}x_4^{14}x_5^{14}.$ This completes the proof.
\end{proof}

\begin{proof}[{\it Proof of  Proposition \ref{md47.4}}]
Suppose that $X$ is a admissible monomial in $\mathscr P_5^+$ and $\omega(X) = \overline{\omega}_{(5)}.$ Then $\omega_1(X) = 3$ and $X = x_ax_bx_cu^2$ with $1\leq a < b <c \leq 5$ and $u\in (\mathscr P_5^+)_{22}.$ Then, by Theorem \ref{dlKS}, $u$ is admissible and $u\in \mathscr B_5^+(\omega_{(5)}).$

Let $Z\in B_5^+(\omega_{(5)})$ such that $x_ax_bx_cZ^2\in (\mathscr P_5^+)_{47}$ with $1\leq a < b <c \leq 5.$ Denote by $\overline{V}$ the set of all the monomials as given in Proposition \ref{md47.4}. By a direct computation using Proposition \ref{md22-2}(IV), we see that if $x_ax_bx_cZ^2\not\in \overline{V}$  then there is a monomial $X_1$ which is given in Lemmas \ref{bd47.10} and \ref{bd47.11} such that $x_ax_bx_cZ^2 = X_1Z_1^{2^t}$ with suitable monomial $Z_1\in \mathscr P_5$ and $t = {\rm max}\{q\in\mathbb{Z}:\, \omega_q(X_1) > 0\}.$ Based on Theorem \ref{dlKS}, we deduce that $x_ax_bx_cZ^2$ is inadmissible. Combining this with the above data, one gets $X\in \overline{V}.$ This means $\mathscr B_5^+(\overline{\omega}_{(5)}) \subseteq  \overline{V}.$

Next, we show that the set $[\overline{V}]_{\overline{\omega}_{(5)}}$ is  linearly independent in the space $Q\mathscr P_5^+(\overline{\omega}_{(5)}).$ Indeed,  suppose there is a linear relation $\mathcal S = \sum_{480\leq m \leq 494}\gamma_m\mathcal Y_{47,\,m}\equiv_{\overline{\omega}_{(5)}} 0,$ where with $\gamma_m\in\mathbb Z/2,\ m = 480, \ldots, 494$ and $\mathcal Y_{47,\,m}\in \overline{V}.$ By combining Theorem \ref{dlsig} and Proposition \ref{mdPS},  we  explicitly calculate  $\pi_{(k;\mathscr K)}(\mathcal S)$ in terms of a given minimal set of $\mathcal A_2$-generators in $\mathscr P_4\ ({\rm modulo}(\mathcal A_2^+\mathscr P_{4})).$ From the relations $\pi_{(k;\mathscr K)}(\mathcal S) \equiv_{\overline{\omega}_{(5)}} 0$ with $\ell(\mathscr K) \leq 2,$ we get $\gamma_m = 0$ for all $m.$ 

To prove $[\overline{V}]_{\overline{\omega}_{(5)}}$ is a basis of $Q\mathscr P_5^+(\overline{\omega}_{(5)})$ we need to show that $[\mathcal Y_{47,\,m}]_{\overline{\omega}_{(5)}}\neq [0]$ for all $\mathcal Y_{47,\, m}\in \overline{V}.$ Denote by $\overline{V}$ the set of all the monomials as given in Proposition \ref{md47.4}.
By a similar argument as given in the proof of Propositions \ref{md21.1} and \ref{md22-1},  we can prove that the set $[\mathscr B_5^+(\overline{\omega}_{(1)})\cup V\cup \overline{V}]$ is  linearly independent in $(Q\mathscr P_5^+)_{47},$ where $V$ the set of all the admissible monomials as given in Proposition \ref{md47.3}. This implies $[\mathcal Y_{47,\,m}]_{\overline{\omega}_{(5)}}\neq [0]$ for all $\mathcal Y_{47,\,m}.$ The proposition is proved.
\end{proof}

Now, since $\dim(Q\mathscr P_5^{0})_{47} = 560$ and $\dim(Q\mathscr P_5^0)_{21} = 460,$ by Corollaries \ref{hq22}, \ref{hq47.1} and Propositions \ref{md47.2}, \ref{md47.3}, \ref{md47.4}, we conclude that $Q\mathscr P_5$ has dimension $1894$ in degree $47.$ The proof of Theorem \ref{dl1} is completed.

{\bf Final remarks}. Recall that Kameko's map $(\widetilde {Sq_*^0})_{(5, 13.2^t-5)}$ is an epimorphism of $\mathbb Z/2(GL_5)$-modules. This implies that $(Q\mathscr P_5)_{13.2^t-5}\cong \mbox{Ker}(\widetilde {Sq_*^0})_{(5,13.2^t-5)}\bigoplus (Q\mathscr P_5)_{13.2^{t-1}-5}.$ According to Lemma \ref{bd21.1}, we get $\mbox{Ker}(\widetilde {Sq_*^0})_{(5, 13.2^1-5)} \cong Q\mathscr P_5(\omega) \bigoplus Q\mathscr P_5(\omega'),$ where $\omega = (3,3,1,1),$ and $\omega' = (3,3,3).$ A simple computation shows $ \mathscr B_5^0(13.2^1-5) =  \overline{\Phi}(\mathscr B_4(13.2^1-5)) = \mathscr B_5^0(\omega)\cup \mathscr B_5^0(\omega').$ Combining this with Propositions \ref{md21.2} and \ref{md21.3} gives $$ \overline{\Phi}(\mathscr B_4(\omega))\subset \mathscr B_5(\omega),\quad \overline{\Phi}(\mathscr B_4(\omega'))\subset \mathscr B_5(\omega').$$
Next, we have $(Q\mathscr P_5)_{13.2^t-5}\cong (Q\mathscr P_5)_{13.2^2-5}$ for all $t \geq 2$ and $$ \mbox{Ker}(\widetilde {Sq_*^0})_{(5,13.2^2-5)} \cong \bigoplus_{1\leq k\leq 5}Q\mathscr P_5(\overline{\omega}_{(k)}).$$ From the above computations, $$\mathscr B_5^0(\overline{\omega}_{(1)}) = \mathscr B_5^0(13.2^2-5) =  \overline{\Phi}(\mathscr B_4(13.2^2-5)) =  \overline{\Phi}(\mathscr B_4(\overline{\omega}_{(1)})).$$ Then, by Corollary \ref{hq47.1}, and Propositions \ref{md47.2}, \ref{md47.3}, \ref{md47.4}, we conclude
$$\overline{\Phi}(\mathscr B_4(\overline{\omega}_{(1)})) \subset \mathscr B_5(\overline{\omega}_{(1)}).$$ If $\overline{\omega}$ is a weight vector of degree $13.2^2-5$ and $\overline{\omega}\neq \overline{\omega}_{(1)},$ then $\mathscr B_4(\overline{\omega}) = \emptyset.$ Furthermore, if $t > 2,$ then $\mathscr B_4(13.2^t-5) = \emptyset.$

From the above remarks, Conjecture \ref{gtS} also satisfies in case of five variables and generic degree $13.2^t-5$ for $t$ an arbitrary positive integer.

 \section{An application of Theorem \ref{dl1}}\label{s4}
The goal of this section is to prove Theorem \ref{dl2}. More precisely, by using our results in Sect.\ref{s3} and a result in \cite{N.T2}, we describe the $\mathbb Z/2(GL_d)$-modules structure of  $Q\mathscr P_5$ in degree $13.2^{t} - 5$ for $t\in \{0, 1\}.$ Then, we  explicitly determine all $GL_5$-invariants of these spaces.

\medskip

Before coming to the proof of the theorem, we introduce some notations and homomorphisms.  We note that $(\mathbb Z/2)^{\times d}$ regarded as a $\mathbb Z/2$-vector space of dimension $d$ and $$(\mathbb Z/2)^{\times d}\cong \langle x_1, \ldots, x_d\rangle \subset \mathscr P_d.$$ For $1\leq t\leq d,$  define the $\mathbb Z/2$-linear map $\tau_t: (\mathbb Z/2)^{\times d}\to (\mathbb Z/2)^{\times d},$ which is determined by  $$\tau_t(x_t) = x_{t+1},\;\tau_t(x_{t+1}) = x_t,\;\tau_t(x_m) = x_m\ \ (m\neq t, t +1,\; 1\leq t \leq d-1),$$ 
and $$\tau_d(x_1) = x_1 + x_2,\; \tau_d(x_m) = x_m\ \ (m  > 1).$$ 
Denote by $S_d$ the symmetric group of degree $d.$ Then, $S_d$ is generated by $\tau_t,\ 1\leq t\leq d-1.$ For each permutation in $S_d$, consider corresponding permutation matrix; these form a group of matrices isomorphic to $S_d.$ So, $GL_d = GL(d,\mathbb Z/2)\cong GL((\mathbb Z/2)^{\times d})$ is generated by $S_d$ and $\tau_d.$ Let $X = x_1^{a_1}x_2^{a_2}\ldots x_d^{a_d}$ be an monomial in $\mathscr P_d.$ Then, the weight vector $\omega(X)$ is invariant under the permutation of the generators $x_j,\ j = 1, 2, \ldots, d;$ hence $Q\mathscr P_d(\omega)$ also has a $S_d$-module structure. We have a homomorphism $\tau_t: \mathscr P_d\to \mathscr P_d$ of algebras, which is induced by $\tau_t.$ Hence, a class $[u]_{\omega}\in Q\mathscr P_d(\omega)$ is an $GL_d$-invariant if and only if $\tau_t(u) + u \equiv_{\omega} 0$ for $1\leq t\leq d.$ If  $\tau_t(u) + u \equiv_{\omega} 0$ for $1\leq t\leq d-1,$ then $[u]_{\omega}$ is an $S_d$-invariant. Note that $\dim((QP_d)^{GL_d}_n)\leq \sum_{\deg(\omega) = n}\dim(QP_d(\omega)^{GL_d})$  (see Sect.\ref{s2}).

Let $\omega$ be a weight vector of degree $n$ and let $\mathcal Y_1, \mathcal Y_2, \ldots, \mathcal Y_s$ be the monomials in $\mathscr P_d(\omega)$ for $s\geq 1.$ We consider a subgroup $L\subseteq GL_d$ and denote by
$$ \begin{array}{ll}
\medskip
L(\mathcal Y_1, \mathcal Y_2, \ldots, \mathcal Y_s)&= \{\sigma(\mathcal Y_j):\, \sigma \in L,\, 1\leq j\leq s\}\subset \mathscr P_d(\omega),\\
\medskip
[\mathscr B(\mathcal Y_1, \mathcal Y_2, \ldots, \mathcal Y_s)]_{\omega}&= [\mathscr B_d(\omega)]_{\omega}\cap \langle [L(\mathcal Y_1, \mathcal Y_2, \ldots, \mathcal Y_s)]_{\omega}\rangle,\\
\theta (\mathcal Y_j) &= \sum_{x\in \mathscr B_d(n)\cap L(\mathcal Y_j)} x, 
\end{array}$$ 
where $\langle [L(\mathcal Y_1, \mathcal Y_2, \ldots, \mathcal Y_s)]_{\omega}\rangle$ is the $L$-submodule of $Q\mathscr P_d(\omega)$ generated by the set $$\{[\mathcal Y_1]_{\omega}, [\mathcal Y_2]_{\omega},\ldots, [\mathcal Y_s]_{\omega}\}.$$

Now, we have $13.2^{1} - 5 = 8,$ and $13.2^{1} - 5 = 21.$ Recall that the squaring operation $(\widetilde {Sq^0_*})_{(5, 21)}:  (Q\mathscr P_5)_{21} \to (Q\mathscr P_5)_{8}$ is an epimorphism of $GL_5$-modules. So, to prove Theorem \ref{dl2}, we need to compute all $GL_5$-invariants of $(Q\mathscr P_5)_8.$

\subsection{Computation of  $(Q\mathscr P_5)^{GL_5}_8$}\label{b-8}\label{s4.1}

According to T\'in \cite{N.T1}, the $\mathbb Z/2$-vector space $(Q\mathscr P_5)_8$ has the basis $[\{\mathcal Y_{8,\,i}:\, 1\leq i\leq 174\}],$ where the monomials $\mathcal Y_i:= \mathcal Y_{8,\,i}, 1\leq i\leq 174,$ are given in Sect.\ref{s5.1}.

\begin{propo}\label{bb8}
The space $(Q\mathscr P_5)^{GL_5}$ is trivial in degree $8.$
\end{propo}
We prepare some lemmas for the proof of the proposition. We have
$$(Q\mathscr P_5)_8  = Q\mathscr P_5^0(\widetilde\omega_{(1)})\bigoplus Q\mathscr P_5(\widetilde\omega_{(2)})\bigoplus Q\mathscr P_5(\widetilde\omega_{(3)}),$$
where $\widetilde\omega_{(1)} = (2,1,1),\ \widetilde\omega_{(2)} = (2,3),\ \widetilde\omega_{(3)} = (4,2).$
We see that $\dim Q\mathscr P_5^0(\widetilde\omega_{(1)}) = 105$ with the basis $\bigcup_{1\leq i\leq 6}[\mathscr B(\mathcal Y_i)]_{\widetilde\omega_{(1)}},$ where $ \mathcal Y_1 = x_4x_5^7,\ \mathcal Y_2 =  x_4^3x_5^5,\  \mathcal Y_3 = x_3x_4x_5^6,\ \mathcal Y_4 =x_3x_4^2x_5^5,\ \mathcal Y_5 =x_3x_4^3x_5^4,$ and $\mathcal Y_6 = x_2x_3x_4^2x_5^4.$ Observe that $\widetilde\omega_{(1)}$ is the weight vector of the mimimal spike $x_1^7x_2,$ so $[X]_{\widetilde\omega_{(1)}} = [X]$ for all $X\in \mathscr P_5.$ A direct computation shows that 
$$ \begin{array}{ll}
\medskip
[S_5(\mathcal Y_1)] &= \langle [\mathcal Y_i]:\, 1\leq i\leq 20\rangle,\\
\medskip
[S_5(\mathcal Y_2)] &= \langle [\mathcal Y_i]:\, 21\leq i\leq 30\rangle,\\
\medskip
[S_5(\mathcal Y_3, \mathcal Y_4, \mathcal Y_5)] &= \langle [\mathcal Y_i]:\, 31\leq i\leq 90\rangle,\\
\medskip
[S_5(\mathcal Y_6)] &= \langle [\mathcal Y_i]:\, 91\leq i\leq 105\rangle
\end{array}$$\
are  $S_5$-submodules of $Q\mathscr P_5^0(\widetilde\omega_{(1)}).$ Hence, we have a direct summand decomposition of $S_5$-modules:
$ Q\mathscr P_5^0(\widetilde\omega_{(1)}) =  [S_5(\mathcal Y_1)]\bigoplus [S_5(\mathcal Y_2)] \bigoplus [S_5(\mathcal Y_3, \mathcal Y_4, \mathcal Y_5)]\bigoplus  [S_5(\mathcal Y_6)].$

\begin{lema}\label{bb8.1}
$Q\mathscr P_5^0(\widetilde\omega_{(1)}))^{S_5}$ has dimension $4.$
\end{lema}

\begin{proof}
We prove the following:
$$ \begin{array}{ll}
\medskip
[S_5(\mathcal Y_j)]^{S_5}&= \langle [\theta(\mathcal Y_j)]\rangle,\ j = 1, 2, 6\\
\medskip
[S_5(\mathcal Y_3, \mathcal Y_4, \mathcal Y_5)]^{S_5}&= \langle [q:= \mathcal Y_{31} + \mathcal Y_{32} + \cdots + \mathcal Y_{70}]\rangle.
\end{array}$$
Indeed, we compute $[S_5(\mathcal Y_j)]^{S_5}$  for $j  = 2$ and $[S_5(\mathcal Y_3, \mathcal Y_4, \mathcal Y_5)]^{S_5}.$ The others can be proved by a similar technique. Note that $\dim[S_5(\mathcal Y_2)]  = 10$ with a basis consisting of all the classes represented by the monomials $\mathcal Y_i:\, 21\leq i\leq 30.$ Suppose that $a = \sum_{21\leq i\leq 30}\gamma_i\mathcal Y_i$ with $\gamma_i\in\mathbb Z/2$ and $[a]\in [S_5(\mathcal Y_2)]^{S_5}.$  By a direct computation using Theorem \ref{dlsig}, we have
$$ \begin{array}{ll}
\medskip
\tau_1(a) + a  = \sum X \ \ {\rm modulo}(\mathcal A_2^+\mathscr P_5) =  0\ \ {\rm modulo}(\mathcal A_2^+\mathscr P_5) ,\\
\medskip
\tau_2(a) + a = \sum Y \ \ {\rm modulo}(\mathcal A_2^+\mathscr P_5) =  0\ \ { \rm modulo}(\mathcal A_2^+\mathscr P_5) ,\\
\medskip
\tau_3(a) + a =  \sum Z \ \ {\rm modulo}(\mathcal A_2^+\mathscr P_5) =  0\ \ {\rm modulo}(\mathcal A_2^+\mathscr P_5) ,\\
\medskip
\tau_4(a) + a =  \sum W \ \ {\rm modulo}(\mathcal A_2^+\mathscr P_5) =  0\ \ {\rm modulo}(\mathcal A_2^+\mathscr P_5) ,\\
\end{array}$$
where
$$ \begin{array}{ll}
\medskip
\sum X  &= (\gamma_{24} + \gamma_{27})(\mathcal Y_{24} + \mathcal Y_{27}) + (\gamma_{25} + \gamma_{28})(\mathcal Y_{25} + \mathcal Y_{28}) + (\gamma_{26} + \gamma_{29})(\mathcal Y_{26} + \mathcal Y_{29}),\\
\medskip
\sum Y  &= (\gamma_{22} + \gamma_{24})(\mathcal Y_{22} + \mathcal Y_{24}) + (\gamma_{23} + \gamma_{25})(\mathcal Y_{23} + \mathcal Y_{25}) + (\gamma_{29} + \gamma_{30})(\mathcal Y_{29} + \mathcal Y_{30}),\\
\medskip
\sum Z  &= (\gamma_{21} + \gamma_{22})(\mathcal Y_{21} + \mathcal Y_{22}) + (\gamma_{25} + \gamma_{26})(\mathcal Y_{25} + \mathcal Y_{26}) + (\gamma_{28} + \gamma_{29})(\mathcal Y_{28} + \mathcal Y_{29}),\\
\medskip
\sum W  &= (\gamma_{22} + \gamma_{23})(\mathcal Y_{22} + \mathcal Y_{23}) + (\gamma_{24} + \gamma_{25})(\mathcal Y_{24} + \mathcal Y_{25}) + (\gamma_{27} + \gamma_{28})(\mathcal Y_{27} + \mathcal Y_{28}).
\end{array}$$
These relations imply that $\gamma_i = \gamma_{21}$ for $i = 22, \ldots, 30.$ Hence, we get $a = \theta(\mathcal Y_2)\ {\rm modulo}(\mathcal A_2^+\mathscr P_5)$ with $\theta(Y_2) = \sum_{21\leq j\leq 30}\mathcal Y_j.$

\medskip
Now, we have the set $\{[\mathcal Y_i]:\, 31\leq i\leq 90]\}$ is a basis of $[S_5(\mathcal Y_3, \mathcal Y_4, \mathcal Y_5)].$ Suppose that $b = \sum_{31\leq i\leq 90}\gamma_i\mathcal Y_i$ with $\gamma_i\in\mathbb Z/2$ and $[b]\in [S_5(\mathcal Y_3, \mathcal Y_4, \mathcal Y_5)]^{S_5}$. By a similar computation from the relations $\tau_t(b) +  b = 0\ {\rm modulo}(\mathcal A_2^+\mathscr P_5)  ,\ t = 1,2, 3, 4,$ one gets 
$\gamma_i = 0$ for $71\leq i\leq 90$ and $\gamma_i = \gamma_{31},\ i = 32, 33, \ldots, 70.$ This means $b =  q\ {\rm modulo}(\mathcal A_2^+\mathscr P_5)$ with $q = \sum_{31\leq i\leq 70}.$ The lemma is proved.
\end{proof}

\begin{lema}\label{bb8.2}
The subspace $(Q\mathscr P_5(\widetilde\omega_{(3)}))^{GL_5}$ is trivial.
\end{lema}

\begin{proof}
We see that $Q\mathscr P_5(\widetilde\omega_{(3)})$ is the $\mathbb Z/2$-vector space of dimension $45$ with the basis $[\mathscr B(\mathcal Y_{130} = x_2x_3x_4^3x_5^3)]_{\widetilde\omega_{(3)}}\,\bigcup\, [\mathscr B(\mathcal Y_{160} = x_1x_2x_3x_4^2x_5^3)]_{\widetilde\omega_{(3)}}.$ Furthermore, $[S_5(\mathcal Y_{130})]$ and $[S_5(\mathcal Y_{160})]$ are $S_5$-submodules of $Q\mathscr P_5^0(\widetilde\omega_{(3)}),$ where $[S_5(\mathcal Y_{130})]_{\widetilde\omega_{(3)}} =\langle [\mathcal Y_i:\, 130\leq i\leq 159]_{(\widetilde\omega_{(3)})}\rangle$  and $[S_5(\mathcal Y_{160})]_{\widetilde\omega_{(3)}} =\langle [\mathcal Y_i:\, 160\leq i\leq 174]_{(\widetilde\omega_{(3)})}\rangle.$ So, we have a direct summand decomposition of $S_5$-modules: $Q\mathscr P_5(\widetilde\omega_{(3)}) = [S_5(\mathcal Y_{130})]_{\widetilde\omega_{(3)}}\bigoplus [S_5(\mathcal Y_{160})]_{\widetilde\omega_{(3)}}.$ The set $[\mathscr B(\mathcal Y_{130})]_{\widetilde\omega_{(3)}}$ is a basis of $[S_5(\mathcal Y_{130})]_{\widetilde\omega_{(3)}}.$ The action of $S_5$ on $Q\mathscr P_5$  induces the one of it on $[\mathscr B(\mathcal Y_{130})]_{\widetilde\omega_{(3)}}.$ On the other hand, this action is transitive, hence if $a = \sum_{130\leq i\leq 159}\gamma_i\mathcal Y_i$ with $\gamma_i\in\mathbb Z/2$ and $[a]\in [S_5(\mathcal Y_{130})]^{S_5},$ then from the relations $\tau_t(a)  + a \equiv_{\widetilde\omega_{(3)}}  0,\ 1\leq t\leq 4,$ we get $\gamma_i = \gamma_{130},\, \forall i,\, 131\leq i\leq 159.$  In other words, $[S_5(\mathcal Y_{130})]_{\widetilde\omega_{(3)}}^{S_5} = \langle [\theta(\mathcal Y_{130})]_{\widetilde\omega_{(3)}}\rangle$ with $\theta(\mathcal Y_{130}) = \sum_{130\leq i\leq 159}\mathcal Y_i.$

\medskip
Next, we have $\dim [S_5(\mathcal Y_{160})]_{\widetilde\omega_{(3)}} = 15$ with the basis $[\mathscr B(\mathcal Y_{160})]_{\widetilde\omega_{(3)}}.$ Suppose $b = \sum_{160\leq i\leq 174}\gamma_j\mathcal Y_j$ with $\gamma_j\in\mathbb Z/2$ and $[b]\in [S_5(\mathcal Y_{160})]^{S_5}.$ A direct computation shows:
$$  \begin{array}{ll}
\tau_1(b) + b &\equiv_{\widetilde\omega_{(3)}} \gamma_{169}(\mathcal Y_{160} +  \mathcal Y_{162})+ \gamma_{170}(\mathcal Y_{161} + \mathcal Y_{163}) + \gamma_{171}(\mathcal Y_{164} + \mathcal Y_{165})\\
&\quad  + (\gamma_{166} + \gamma_{172})(\mathcal Y_{166}+\mathcal Y_{172}) +  (\gamma_{167} + \gamma_{173})(\mathcal Y_{167} +\mathcal Y_{173})\\
&\quad+  (\gamma_{168} + \gamma_{174})(\mathcal Y_{168}+ \mathcal Y_{174})\\
\tau_2(b) + b &\equiv_{\widetilde\omega_{(3)}} (\gamma_{162} + \gamma_{169})(\mathcal Y_{162}+ \mathcal Y_{169}) + (\gamma_{163} + \gamma_{170})(\mathcal Y_{163}+ \mathcal Y_{170}) \\
&\quad+ (\gamma_{164} + \gamma_{166})(\mathcal Y_{164}+\mathcal Y_{166})  + (\gamma_{165} + \gamma_{167})(\mathcal Y_{165} + \mathcal Y_{167})\\
&\quad + (\gamma_{168} + \gamma_{171})(\mathcal Y_{168} + \mathcal Y_{171})  + (\gamma_{165} + \gamma_{167})(\mathcal Y_{165} + \mathcal Y_{167}) \\
&\quad +\gamma_{174}(\mathcal Y_{172} + \mathcal Y_{173}),\\
\tau_3(b) + b &\equiv_{\widetilde\omega_{(3)}} (\gamma_{160} + \gamma_{162})(\mathcal Y_{160} + \mathcal Y_{162}) + (\gamma_{161} + \gamma_{164})(\mathcal Y_{161}+ \mathcal Y_{164}) \\
&\quad   + (\gamma_{163} + \gamma_{165})(\mathcal Y_{163} + \mathcal Y_{165}) + (\gamma_{167} + \gamma_{168})(\mathcal Y_{167} + \mathcal Y_{168}) \\
&\quad  + (\gamma_{170} + \gamma_{171})(\mathcal Y_{170} + \mathcal Y_{171}) + (\gamma_{173} + \gamma_{174})(\mathcal Y_{173} + \mathcal Y_{174}),\\
\tau_4(b) + b &\equiv_{\widetilde\omega_{(3)}}  (\gamma_{160} + \gamma_{161})(\mathcal Y_{160} + \mathcal Y_{161}) +  (\gamma_{162} + \gamma_{163})(\mathcal Y_{162} + \mathcal Y_{163})\\
&\quad  + (\gamma_{164} + \gamma_{165})(\mathcal Y_{164} + \mathcal Y_{165}) + (\gamma_{166} + \gamma_{167})(\mathcal Y_{166} + \mathcal Y_{167})\\
&\quad  + (\gamma_{169} + \gamma_{170})(\mathcal Y_{169} + \mathcal Y_{170}) +  (\gamma_{172} + \gamma_{173})(\mathcal Y_{172} + \mathcal Y_{173}).
\end{array}$$
Then, from the relations $\tau_t(b) + b \equiv_{\widetilde\omega_{(3)}} 0,$ we conclude $\gamma_j = 0,\ \forall j.$ 

\medskip

Now, let $[X]_{\widetilde\omega_{(3)}}\in (Q\mathscr P_5(\widetilde\omega_{(3)}))^{GL_5}$ with $X\in \mathscr P_5(\widetilde\omega_{(3)}),$ then $[X]_{\widetilde\omega_{(3)}}\in (Q\mathscr P_5(\widetilde\omega_{(3)}))^{S_5}.$ So, we have $X \equiv_{\widetilde\omega_{(3)}} \gamma \theta(\mathcal Y_{130})$ with $\gamma\in\mathbb Z/2.$ By a direct computation, we obtain
$$ \tau_5(X) + X \equiv_{\widetilde\omega_{(3)}} \gamma \mathcal Y_{130} + \, \mbox{ other terms  }\equiv_{\widetilde\omega_{(3)}} 0.$$
This implies $\gamma = 0.$ The proposition follows.
  \end{proof}

By a simple computation using the techniques as given in the proof of Lemmas \ref{bb8.1} and \ref{bb8.2}, we claim that

\begin{lema}\label{bb8.3}
The following results are true:

\begin{enumerate}

\item[i)] We have a direct summand decomposition of the $S_5$-modules:
$$ Q\mathscr P_5(\widetilde\omega_{(2)}) = [S_5(\mathcal Y_{106})]_{\widetilde\omega_{(2)}}\bigoplus [S_5(\mathcal Y_{126})]_{\widetilde\omega_{(2)}}.$$

\item[ii)] The subspace $(Q\mathscr P_5(\widetilde\omega_{(2)}))^{GL_5}$ is trivial.
\end{enumerate}
\end{lema}

\begin{proof}[{\it Proof of Proposition \ref{bb8}}]
Let $[X]\in (Q\mathscr P_5)_{8}^{GL_5}.$ Then, from Lemmas \ref{bb8.1} - \ref{bb8.3}, we have 
$$X = \beta_1\theta(\mathcal Y_1) + \beta_2\theta(\mathcal Y_2) + \beta_3q +  \beta_4\theta(\mathcal Y_3)\ \ {\rm modulo}(\mathcal A_2^+\mathscr P_5),$$
with $\beta_t\in\mathbb Z/2,\ 1\leq t\leq 4.$ By using Theorem \ref{dlsig} and computing $\tau_5(X) + X$ in terms of the admissible monomials $({\rm modulo}(\mathcal A_2^+\mathscr P_5)),$ we conclude
$$ \tau_5(X) + X  \equiv \beta_1\mathcal Y_7 + (\beta_1 + \beta_2)\mathcal Y_{16} + \beta_3\mathcal Y_{33} + \beta_4(\mathcal Y_{91} + \mathcal Y_{92} + \mathcal Y_{93})\ + \mbox{other terms } \equiv 0.$$
This relation shows $\beta_1 = \beta_2 = \beta_3 = \beta_4 = 0.$ The proposition is proved.
\end{proof}

As an immediate consequence of Theorem \ref{bb8}, we get the following.

\begin{corls}
The fifth transfer 
$$Tr_5: \mathbb Z/2 \otimes_{GL_5}P_{\mathcal A_2}H_{8}(B(\mathbb Z/2)^{\times 5})\to {\rm Ext}_{\mathcal {A}_2}^{5,5+8}(\mathbb Z/2,\mathbb Z/2)$$ 
is a trivial isomorphism. 
\end{corls}

\subsection{Computation of  $({\rm Ker}(\widetilde {Sq^0_*})_{(5, 21)})^{S_5}$}\label{s4.2}

From the results in Sect.\ref{s3.2.1}, we see that $\dim{\rm Ker}(\widetilde {Sq^0_*})_{(5, 21)} = 666$ with the basis $\{[\mathcal Y_{21,\,t}]:\ 1\leq t\leq 666\}.$ Here,  the admissible monomials $\mathcal Y_t := \mathcal Y_{21,\,t},\ 1\leq t\leq 666,$ as given in Sect.\ref{s5.6} and Sect.\ref{s5.7}.
Recall that $(\widetilde {Sq^0_*})_{(5, 21)}$ is an epimorphism of $GL_5$-modules. Combining this and the results in Subsection \ref{b-8}, we get $(Q\mathscr P_5)^{GL_5}_{21}\subseteq ({\rm Ker}(\widetilde {Sq^0_*})_{(5, 21)})^{GL_5}.$ 
By Lemma \ref{bd21.1}, we have a direct summand decomposition of the $S_5$-modules: $${\rm Ker}(\widetilde {Sq^0_*})_{(5, 21)} = Q\mathscr P_5(3,3,1,1)\bigoplus Q\mathscr P_5(3,3,3).$$
For $\omega = (3,3,1,1),$ according to the results in Sect.\ref{s3.2.1},  we get $Q\mathscr P_5(\omega) = Q\mathscr P_5^0(\omega)\bigoplus Q\mathscr P_5^+(\omega)$ with $\dim Q\mathscr P_5^0(\omega) = 340$ and $\dim Q\mathscr P_5^+(\omega) = 196$. Note that $\mathscr Z = x_1^{15}x_2^3x_3^3$ is the minimal spike monomial in $(\mathscr P_5)_{21}$ and $\omega(\mathscr Z) = \omega.$ So, $[x]_{\omega} = [x]$ for any $x\in (\mathscr P_5)_{21}.$ By using the results in Sect.\ref{s3.2.1}, we see that there is a direct summand decomposition of the $S_5$-modules:
$$ Q\mathscr P_5^0(\omega) = \langle [S_5(\mathcal Y_1)]\rangle\, \bigoplus\,  \langle [S_5(\mathcal Y_{31})]\rangle\, \bigoplus\, \langle [S_5(\mathcal Y_{61})]\rangle\, \bigoplus\,  \langle [S_5(\mathcal Y_{121})]\rangle\, \bigoplus\, \langle[\mathbb V_1]\rangle,$$
where $$ \begin{array}{ll}
\medskip
\mathscr B(\mathcal Y_1) &= \{\mathcal Y_t:\, 1\leq t\leq 30\},\ \ \ \ \ \mathscr B(\mathcal Y_{31}) = \{\mathcal Y_t:\, 31\leq t\leq 60\},\\
\medskip
 \mathscr B(\mathcal Y_{61}) &= \{\mathcal Y_t:\, 61\leq t\leq 120\},\ \ \mathscr B(\mathcal Y_{121}) = \{\mathcal Y_t:\, 121\leq t\leq 150\},\\
\medskip
\mathbb V_1 &= \mathscr B(\mathcal Y_{151},\mathcal Y_{181}, \mathcal Y_{201}, \mathcal Y_{241}, \mathcal Y_{256}, \mathcal Y_{266}, \mathcal Y_{286},\mathcal Y_{296}, \mathcal Y_{301}, \mathcal Y_{316})\\
\medskip
&= \{\mathcal Y_t:\, 151\leq t\leq 340\}.\\
\end{array}$$

\begin{lema}\label{bb21.1}
The following results are true:

\begin{enumerate}
\item[a)] $ \langle [S_5(\mathcal Y_j)]\rangle^{S_5} = \langle [\theta(\mathcal Y_j)] \rangle$ for $j = 1, 31, 61, 121.$

\medskip

\item[b)] The subspace $\langle [S_5(\mathbb V_1)]\rangle^{S_5}$ is trivial.
\end{enumerate}
\end{lema} 

\begin{proof}[{\it Outline of the proof}]
For $j  = 1, $ let $[f_1]\in  \langle [S_5(\mathcal Y_1)]\rangle^{S_5}.$ Then, we have
$$ \tau_m(f_1) = \sum_{X\in \mathscr B(\mathcal Y_1)} \beta_X. X\ \ {\rm modulo}(\mathcal A_2^+\mathscr P_5),\ 0 < m <5$$
with $f_1 = \sum_{X\in \mathscr B(\mathcal Y_1)} \beta_X. X$ and $\beta_X\in \mathbb Z/2$. By a direct computation, we can see that the action of the symmetric group $S_5$ on $Q\mathscr P_5$ induces the one of it on the set $[\mathscr B(\mathcal Y_j)]$ and this action is transitive. So, we get $\beta_X = \beta_{X'} = \beta\in \mathbb Z/2$ for all $X, X'\in  \mathscr B(\mathcal Y_j).$ This means $f_1 =  \theta(\mathcal Y_1)\ {\rm modulo}(\mathcal A_2^+\mathscr P_5).$ For $j = 31, 61, 121,$ we determine $\tau_m(f_j) + f_j$ in terms of $\mathcal Y_j.$ Then, by a simple computation using the relations $\tau_m(f_j) = f_j\ \ {\rm modulo}(\mathcal A_2^+\mathscr P_5),$ we conclude $f_j =  \theta(\mathcal Y_j)\ {\rm modulo}(\mathcal A_2^+\mathscr P_5).$

Next, we have $\dim \langle\mathbb V_1\rangle = 190$ with the basis $\{[\mathcal Y_t]:\, 151\leq t\leq 340\}.$ Assume that $g = \sum_{u\in \mathbb V_1}\gamma_u.u$ with $\gamma_u\in\mathbb Z/2$ and $[g]\in \langle[\mathbb V_1]\rangle^{S_5}.$ By using Theorem \ref{dlsig} and a similar computation as given in the proof of Lemmas \ref{bb8.1} and \ref{bb8.2}, we obtain $\gamma_u = 0$ for all $u\in \mathbb V_1.$ This implies that $g$ is $\mathcal A_2$-decomposable. The lemma follows. 
\end{proof}

\begin{lema}\label{bb21.2}
$ Q\mathscr P_5^+(\omega)^{S_5}  = \langle [p:= \mathcal Y_{411} + \mathcal Y_{412} + \cdots + \mathcal Y_{419} ] \rangle.$ 
\end{lema}

\begin{proof}
From Proposition \ref{md21.2},  we see that the sets $[S_5(\mathcal Y_{401})] = \langle [\mathcal Y_t]:\, 401\leq t\leq 410\rangle$ and $\mathbb V_2 = \langle [\mathcal Y_t]:\, 411\leq t\leq 596\rangle,$ 
are $S_5$-submodules of $Q\mathscr P_5^+(\omega).$ Hence, we have a direct summand decomposition of the $S_5$-modules:
$$ Q\mathscr P_5^+(\omega)  = [S_5(\mathcal Y_{401})] \, \bigoplus\, \mathbb V_2 .$$
The set $[\mathscr B(\mathcal Y_{401})]$ is a basis of $[S_5(\mathcal Y_{401})].$ Assume that $Z$ is a polynomial such that $[Z]\in [S_5(\mathcal Y_{401})]^{S_5}$ and $Z = \sum_{401\leq t\leq 410}\ell_t\mathcal Y_t$ with $\ell_t\in \mathbb Z/2.$ For $1\leq j\leq 4,$ we explicitly compute $\tau_j(Z) + Z$ in the terms of the admissible monomials $\mathcal Y_t, \ 401\leq t\leq 410.$ By a direct computation using Theorem \ref{dlsig}, we get
$$ \begin{array}{ll}
\tau_1(Z) + Z &\equiv \ell_{401}\mathcal Y_{404} + \ell_{402}\mathcal Y_{405} + \ell_{403}\mathcal Y_{407} + (\ell_{406}+\ell_{409})(\mathcal Y_{406} + \mathcal Y_{409}) \\
\medskip
&\quad + (\ell_{408}+\ell_{410})(\mathcal Y_{408} + \mathcal Y_{410}),\\
\tau_2(Z) + Z &\equiv  (\ell_{401}+\ell_{404})(\mathcal Y_{401} + \mathcal Y_{404}) +  (\ell_{402}+\ell_{405})(\mathcal Y_{402} + \mathcal Y_{405})\\
\medskip
&\quad +  (\ell_{403}+\ell_{406})(\mathcal Y_{403} + \mathcal Y_{406}) +  (\ell_{407}+\ell_{408})(\mathcal Y_{407} + \mathcal Y_{408}) +  \ell_{409}\mathcal Y_{410},\\
\medskip
\tau_3(Z) + Z &\equiv  \ell_{401}\mathcal Y_{404}+ (\ell_{402}+\ell_{403})(\mathcal Y_{402} + \mathcal Y_{403}) + (\ell_{405}+\ell_{407})(\mathcal Y_{405} + \mathcal Y_{407}),\\
\medskip
&\quad + (\ell_{406}+\ell_{408})(\mathcal Y_{406} + \mathcal Y_{408}) + (\ell_{409}+\ell_{410})(\mathcal Y_{409} + \mathcal Y_{410}),\\
\tau_4(Z) + Z &\equiv (\ell_{401}+\ell_{402})(\mathcal Y_{401} + \mathcal Y_{402}) + \ell_{403}\mathcal Y_{407} +  (\ell_{404}+\ell_{405})(\mathcal Y_{404} + \mathcal Y_{405}) \\
&\quad+ \ell_{406}\mathcal Y_{408} + \ell_{409}\mathcal Y_{410}.
\end{array}$$
Then, by the relations $\tau_j(Z)  = Z\,\, {\rm modulo}(\mathcal A_2^+\mathscr P_5),\ j  =1, 2, 3, 4,$ one gets $\ell_t = 0,\ \forall t.$ This implies that $Z$ is $\mathcal A_2$-decomposable. 

Note that $\dim \mathbb V_2 = 186$ with the basis $\{[\mathcal Y_j]:\, 411\leq j\leq 596\}.$ Then if the polynomial $g = \sum_{u\in \mathbb V_2}\sigma_u.u,\ \sigma_u\in\mathbb Z/2$ such that $[g]\in \mathbb V_2^{S_5},$ then by a similar argument as given above, we obtain $g = p\,\, {\rm modulo}(\mathcal A_2^+\mathscr P_5).$ This completes the proof of the lemma.
\end{proof}

\begin{propo}\label{bb21.3}
For $\overline{\omega} = (3,3,3),$ we have $Q\mathscr P_5(\overline{\omega})^{GL_5} = \langle [q_2]_{\overline{\omega}} \rangle,$ where
$$ \begin{array}{ll}
\medskip
q_2 &= \mathcal Y_{652} + \mathcal Y_{653} + \mathcal Y_{654} + \mathcal Y_{656}+\mathcal Y_{657} + \mathcal Y_{658} +\sum_{661\leq t\leq 666}\mathcal Y_t\\
\medskip
&= x_1x_2^6x_3^3x_4^5x_5^6 + x_1x_2^3x_3^6x_4^5x_5^6 + x_1x_2^3x_3^5x_4^6x_5^6 + x_1^3x_2x_3^5x_4^6x_5^6 \\
\medskip
&\quad + x_1^3x_2^5x_3x_4^6x_5^6 + x_1^3x_2^5x_3^6x_4x_5^6 + x_1^3x_2^3x_3^5x_4^4x_5^6 + x_1^3x_2^3x_3^4x_4^5x_5^6\\
\medskip
&\quad + x_1^3x_2^3x_3^5x_4^6x_5^4 + x_1^3x_2^4x_3^3x_4^5x_5^6 + x_1^3x_2^5x_3^3x_4^6x_5^4 + x_1^3x_2^5x_3^6x_4^3x_5^4.
\end{array}$$
\end{propo}

Based on the results in Sect.\ref{s3.1}, we have $$\dim Q\mathscr P_5(\overline{\omega}) = \dim Q\mathscr P_5^0(\overline{\omega}) + \dim Q\mathscr P_5^+(\overline{\omega}) = 60 + 70  = 130.$$ 
Consider the following monomials: 
$$ \begin{array}{ll}
\medskip
 \mathcal Y_{341} &= x_3^7x_4^7x_5^7,\ \  \ \ \ \ \ \ \mathcal Y_{351} = x_2x_3^6x_4^7x_5^7, \ \ \ \ \ \mathcal Y_{381} = x_2^3x_3^5x_4^6x_5^7,\\
\medskip
 \mathcal Y_{597} &= x_1x_2^2x_3^4x_4^7x_5^7,\ \  \mathcal Y_{607} = x_1x_2^6x_3x_4^6x_5^7,\ \ \mathcal Y_{617} = x_1x_2^2x_3^5x_4^6x_5^7,\\
\medskip
 \mathcal Y_{622} &= x_1x_2^3x_3^4x_4^6x_5^7, \ \ \mathcal Y_{642} = x_1^3x_2^5x_3^2x_4^4x_5^7, \ \ \mathcal Y_{647} = x_1^3x_2^3x_3^4x_4^4x_5^7,\\
 \mathcal Y_{652} &= x_1x_2^6x_3^3x_4^5x_5^6, \ \ \mathcal Y_{659} = x_1^3x_2^5x_3^2x_4^5x_5^6, \ \ \mathcal Y_{660} = x_1^3x_2^5x_3^3x_4^4x_5^6.
\end{array}$$

The following lemma can be easily proved by a direct computation.

\begin{lema}\label{bb21.4}
\emph{}

\begin{enumerate}
\item[i)] The following subspaces are $S_5$-submodules of $Q\mathscr P_5(\overline{\omega})$:
$$ \langle [S_5(\mathcal Y_a)]_{\overline{\omega}} \rangle,\ a = 341, 351, 381, 597, \mathbb V_3:=\langle [S_5(\mathcal Y_{607}, \mathcal Y_{617}, \mathcal Y_{622}, \mathcal Y_{642}, \mathcal Y_{647})]_{\overline{\omega}} \rangle,$$
 $$ \mathbb V_4:= \langle [S_5(\mathcal Y_{652}, \mathcal Y_{659}, \mathcal Y_{660})]_{\overline{\omega}} \rangle.$$

\item[ii)]  We have a direct summand decomposition of the  $S_5$-modules:
$$ \begin{array}{ll}
Q\mathscr P_5(\overline{\omega})& =  \langle [S_5(\mathcal Y_{341})]_{\overline{\omega}}\rangle\, \bigoplus\,  \langle [S_5(\mathcal Y_{351})]_{\overline{\omega}}\rangle\, \bigoplus\,  \langle [S_5(\mathcal Y_{381})]_{\overline{\omega}}\rangle\\
&\quad \bigoplus\,  \langle [S_5(\mathcal Y_{597})]_{\overline{\omega}}\rangle\, \bigoplus\, \mathbb V_3\, \bigoplus\, \mathbb V_4. 
\end{array}$$
\end{enumerate}
\end{lema}

\newpage
\begin{lema}\label{bb21.5}
We have the following results:

\begin{enumerate}
\item[i)] $\langle [S_5(\mathcal Y_a)]_{\overline{\omega}} \rangle^{S_5} = \langle [\theta(\mathcal Y_a)]_{\overline{\omega}} \rangle$, for  $a = 341, 351, 381, 597.$

\item[ii)] $\langle [S_5(\mathbb V_3)]_{\overline{\omega}} \rangle^{S_5} = \langle [q_1:=\mathcal Y_{610} +  \mathcal Y_{611} +  \mathcal Y_{613}  + \mathcal Y_{614}  + \mathcal Y_{616} +\sum_{647\leq t\leq 651}\mathcal Y_t]_{\overline{\omega}}\rangle.$ 

\item[iii)]
$\langle [S_5(\mathbb V_4)]_{\overline{\omega}} \rangle^{S_5} = .\langle [q_2:= \mathcal Y_{652} + \mathcal Y_{653} + \mathcal Y_{654} +\mathcal Y_{656}+ \mathcal Y_{657} + \mathcal Y_{658} +\sum_{661\leq t\leq 666}\mathcal Y_t ]_{\overline{\omega}} \rangle,$ 
\end{enumerate}
\end{lema}

The proof of the lemma is straightforward.

\begin{proof}[{\it Proof of Proposition \ref{bb21.3}}]
By Lemmas \ref{bb21.4} and \ref{bb21.5}, we get
$$ Q\mathscr P_5(\overline{\omega})^{S_5} = \langle [\theta(\mathcal Y_{341})]_{\overline{\omega}},\ [\theta(\mathcal Y_{351})]_{\overline{\omega}}, \ [\theta(\mathcal Y_{381})]_{\overline{\omega}},\ [\theta(\mathcal Y_{597})]_{\overline{\omega}}, \ [q_1]_{\overline{\omega}}, \ [q_2]_{\overline{\omega}} \rangle.$$
Let $X$ be a polynomial in $\mathscr P_5(\overline{\omega})$ such that $[X]_{\overline{\omega}}\in Q\mathscr P_5(\overline{\omega})^{GL_5}.$ Then, we have 
$$ X \equiv_{\overline{\omega}} v_1\theta(\mathcal Y_{341}) + v_2\theta(\mathcal Y_{351}) + v_3\theta(\mathcal Y_{381}) + v_4\theta(\mathcal Y_{597}) +  v_5q_1 + v_6q_2,$$
with $v_j\in\mathbb Z/2$ for $1\leq j\leq 6.$ We explicitly compute $\tau_5(X)$ in terms of the admissible monomials $\mathcal Y_t$ with $t = 341,342, \ldots, 400, 597, 598, \ldots, 666.$ By a direct computation, one gets
$$ \begin{array}{ll}
\medskip
\tau_5(X)  +  X & \equiv_{\overline{\omega}}  (v_1 + v_2)\mathcal Y_{342} +  (v_2 + v_4)\mathcal Y_{351} +  v_2\mathcal Y_{354} + v_3\mathcal Y_{383}\\
&\quad  + v_5\mathcal Y_{650}\ + \mbox{other terms }.
\end{array}$$
Since $[X]_{\overline{\omega}}\in Q\mathscr P_5(\overline{\omega})^{GL_5},$  $v_j = 0,\ 1\leq j\leq 5.$ The proposition is proved.
\end{proof}

Combining Lemmas \ref{bb21.1}, \ref{bb21.2} and \ref{bb21.5} gives

\begin{corls}
There exist exactly $11$ non-zero classes in the kernel of $(\widetilde {Sq^0_*})_{(5, 21)}$ invariant under the action of $S_5.$
\end{corls}

\subsection{Proof of Theorem \ref{dl2}}\label{s4.3}
Suppose that $[T]\in (Q\mathscr P_5)_{21}^{GL_5}$ with $T$ is a polynomial in $(\mathscr P_5)_{21}.$ 
 From Proposition \ref{bb21.3}, we have $T = T^*+ \zeta_{6} q_2\ {\rm modulo}(\mathcal A_2^+\mathscr P_5)$ with $T^*\in \mathscr P_5^-(\overline{\omega})$ and $\zeta_6\in \mathbb Z/2.$ By a simple computation, we see that $[q_2]\in (Q\mathscr P_5)_{21}^{S_5}.$ This implies that $[T^*]$ is an $S_5$-invariant. On the other hand, $[\mathscr P_5^-(\overline{\omega})] = Q\mathscr P_5(\omega).$ Hence, by Lemmas \ref{bb21.1} and \ref{bb21.2}, we obtain 
$$T^* =  \zeta_1\theta(\mathcal Y_1) +  \zeta_2\theta(\mathcal Y_{31}) +  \zeta_3\theta(\mathcal Y_{61}) +  \zeta_4\theta(\mathcal Y_{121}) + \zeta_5p\ {\rm modulo}(\mathcal A_2^+\mathscr P_5),$$
where $\zeta_i\in\mathbb Z/2.$ Using Theorem \ref{dl1} and computing $\tau_5(T) + T$ in terms of the admissible monomials $\mathcal Y_t,\ 1\leq t\leq 666,$ we conclude 
$$ \begin{array}{ll}
\tau_5(T) + T &= (\zeta_1 + \zeta_3)\mathcal Y_4 +  (\zeta_2 + \zeta_4)\mathcal Y_{35} +  \zeta_3\mathcal Y_{61}\\
&\quad\quad\quad\quad\quad +  (\zeta_3 + \zeta_4)\mathcal Y_{107} + \zeta_5\mathcal Y_{153}\ + \mbox{other terms }\ {\rm modulo}(\mathcal A_2^+\mathscr P_5).
\end{array}$$
By the relation $\tau_5(T)  =  T\ {\rm modulo}(\mathcal A_2^+\mathscr P_5),$ one gets  $\zeta_i = 0,\ 1\leq i\leq 5.$ This shows that $$T = \zeta_6q_2\ {\rm modulo}(\mathcal A_2^+\mathscr P_5).$$ The proof of the theorem is completed.

\section{Proof of Theorem \ref{dl3}}\label{s5}

Obviously, $\lambda_3\in \Lambda^{1, 3}$ and 
$ \overline{f}_0 = \lambda_3\lambda_5\lambda_6\lambda_4 + \lambda^2_3\lambda_7\lambda_5 + \lambda_7\lambda_5\lambda_3^2 + \lambda_7\lambda_5\lambda_4\lambda_2  \in \Lambda^{4, 18}$
 are the cycles in the lambda algebra $\Lambda.$ By Lin \cite{W.L}, we have 
${\rm Ext}_{\mathcal{A}_2}^{5, 26}(\mathbb{Z}/2, \mathbb{Z}/2) = \langle h_2f_0\rangle,$ where $$h_2 = [\lambda_3]\in {\rm Ext}_{\mathcal{A}_2}^{1, 4}(\mathbb{Z}/2, \mathbb{Z}/2),\ \mbox{and} \ f_0 = [\overline{f}_0]\in {\rm Ext}_{\mathcal{A}_2}^{4, 22}(\mathbb{Z}/2, \mathbb{Z}/2).$$ Notice that $h_2f_0 = h_1g_1$ with $h_1\in {\rm Ext}_{\mathcal{A}_2}^{1, 2}(\mathbb{Z}/2, \mathbb{Z}/2)$ and $g_1\in {\rm Ext}_{\mathcal{A}_2}^{4, 24}(\mathbb{Z}/2, \mathbb{Z}/2).$ By direct computations, we find that the following element is $\mathcal A_2^{+}$-annihilated in $H_{21}(B(\mathbb Z/2)^{\times 5})$:
$$Z =  \left\{ \begin{array}{ll}
 &a_1^{(3)}a_2^{(3)}a_3^{(5)}a_4^{(1)}a_5^{(9)}+
 a_1^{(3)}a_2^{(3)}a_3^{(5)}a_4^{(2)}a_5^{(8)}+
 a_1^{(3)}a_2^{(3)}a_3^{(6)}a_4^{(1)}a_5^{(8)}+
\medskip
 a_1^{(3)}a_2^{(3)}a_3^{(6)}a_4^{(2)}a_5^{(7)}\\
&\quad +
 a_1^{(3)}a_2^{(3)}a_3^{(5)}a_4^{(4)}a_5^{(6)}+
 a_1^{(3)}a_2^{(3)}a_3^{(6)}a_4^{(3)}a_5^{(6)}+
 a_1^{(3)}a_2^{(5)}a_3^{(6)}a_4^{(1)}a_5^{(6)}+
\medskip
 a_1^{(3)}a_2^{(3)}a_3^{(5)}a_4^{(5)}a_5^{(5)}\\
&\quad +
 a_1^{(3)}a_2^{(3)}a_3^{(6)}a_4^{(4)}a_5^{(5)}+
 a_1^{(3)}a_2^{(5)}a_3^{(6)}a_4^{(2)}a_5^{(5)}+
 a_1^{(3)}a_2^{(3)}a_3^{(9)}a_4^{(1)}a_5^{(5)}+
\medskip
 a_1^{(3)}a_2^{(5)}a_3^{(7)}a_4^{(1)}a_5^{(5)}\\
&\quad +
 a_1^{(3)}a_2^{(3)}a_3^{(9)}a_4^{(2)}a_5^{(4)}+
 a_1^{(3)}a_2^{(5)}a_3^{(7)}a_4^{(2)}a_5^{(4)}+
 a_1^{(3)}a_2^{(3)}a_3^{(10)}a_4^{(1)}a_5^{(4)}+
\medskip
 a_1^{(3)}a_2^{(6)}a_3^{(7)}a_4^{(1)}a_5^{(4)}\\
&\quad +
 a_1^{(3)}a_2^{(5)}a_3^{(6)}a_4^{(4)}a_5^{(3)}+
 a_1^{(3)}a_2^{(6)}a_3^{(7)}a_4^{(2)}a_5^{(3)}+
 a_1^{(3)}a_2^{(3)}a_3^{(10)}a_4^{(2)}a_5^{(3)}+
\medskip
 a_1^{(3)}a_2^{(3)}a_3^{(11)}a_4^{(2)}a_5^{(2)}\\
&\quad +
 a_1^{(3)}a_2^{(5)}a_3^{(9)}a_4^{(2)}a_5^{(2)}+
 a_1^{(3)}a_2^{(6)}a_3^{(10)}a_4^{(1)}a_5^{(1)}+
 a_1^{(3)}a_2^{(3)}a_3^{(5)}a_4^{(9)}a_5^{(1)}+
\medskip
 a_1^{(3)}a_2^{(3)}a_3^{(5)}a_4^{(8)}a_5^{(2)}\\
&\quad +
 a_1^{(3)}a_2^{(3)}a_3^{(6)}a_4^{(8)}a_5^{(1)}+
 a_1^{(3)}a_2^{(3)}a_3^{(6)}a_4^{(7)}a_5^{(2)}+
 a_1^{(3)}a_2^{(3)}a_3^{(5)}a_4^{(6)}a_5^{(4)}+
\medskip
 a_1^{(3)}a_2^{(3)}a_3^{(6)}a_4^{(6)}a_5^{(3)}\\
&\quad +
 a_1^{(3)}a_2^{(5)}a_3^{(6)}a_4^{(6)}a_5^{(1)}+
 a_1^{(3)}a_2^{(3)}a_3^{(5)}a_4^{(5)}a_5^{(5)}+
 a_1^{(3)}a_2^{(3)}a_3^{(6)}a_4^{(5)}a_5^{(4)}+
\medskip
 a_1^{(3)}a_2^{(5)}a_3^{(6)}a_4^{(5)}a_5^{(2)}\\
&\quad +
 a_1^{(3)}a_2^{(3)}a_3^{(9)}a_4^{(5)}a_5^{(1)}+
 a_1^{(3)}a_2^{(5)}a_3^{(7)}a_4^{(5)}a_5^{(1)}+
 a_1^{(3)}a_2^{(3)}a_3^{(9)}a_4^{(4)}a_5^{(2)}+
\medskip
 a_1^{(3)}a_2^{(5)}a_3^{(7)}a_4^{(4)}a_5^{(2)}\\
&\quad +
 a_1^{(3)}a_2^{(3)}a_3^{(10)}a_4^{(4)}a_5^{(1)}+
 a_1^{(3)}a_2^{(6)}a_3^{(7)}a_4^{(4)}a_5^{(1)}+
 a_1^{(3)}a_2^{(5)}a_3^{(6)}a_4^{(3)}a_5^{(4)}+
\medskip
 a_1^{(3)}a_2^{(6)}a_3^{(7)}a_4^{(3)}a_5^{(2)}\\
&\quad +
 a_1^{(3)}a_2^{(3)}a_3^{(10)}a_4^{(3)}a_5^{(2)}+
 a_1^{(3)}a_2^{(3)}a_3^{(11)}a_4^{(2)}a_5^{(2)}+
 a_1^{(3)}a_2^{(5)}a_3^{(9)}a_4^{(2)}a_5^{(2)}+
\medskip
 a_1^{(3)}a_2^{(6)}a_3^{(10)}a_4^{(1)}a_5^{(1)}\\
&\quad+
 a_1^{(3)}a_2^{(3)}a_3^{(12)}a_4^{(1)}a_5^{(2)}+
 a_1^{(3)}a_2^{(7)}a_3^{(8)}a_4^{(1)}a_5^{(2)}+
 a_1^{(3)}a_2^{(11)}a_3^{(4)}a_4^{(1)}a_5^{(2)}+
\medskip
 a_1^{(3)}a_2^{(13)}a_3^{(2)}a_4^{(1)}a_5^{(2)}\\
&\quad +
 a_1^{(3)}a_2^{(14)}a_3^{(1)}a_4^{(1)}a_5^{(2)}+
 a_1^{(3)}a_2^{(12)}a_3^{(3)}a_4^{(1)}a_5^{(2)}+
 a_1^{(3)}a_2^{(8)}a_3^{(7)}a_4^{(1)}a_5^{(2)}+
\medskip
 a_1^{(3)}a_2^{(4)}a_3^{(11)}a_4^{(1)}a_5^{(2)}\\
&\quad +
 a_1^{(3)}a_2^{(2)}a_3^{(13)}a_4^{(1)}a_5^{(2)}+
 a_1^{(3)}a_2^{(1)}a_3^{(14)}a_4^{(1)}a_5^{(2)}+
 a_1^{(3)}a_2^{(6)}a_3^{(6)}a_4^{(3)}a_5^{(3)}+
\medskip
 a_1^{(3)}a_2^{(5)}a_3^{(5)}a_4^{(5)}a_5^{(3)}\\
&\quad +
 a_1^{(3)}a_2^{(3)}a_3^{(3)}a_4^{(9)}a_5^{(3)}+
 a_1^{(3)}a_2^{(5)}a_3^{(3)}a_4^{(7)}a_5^{(3)}+
 a_1^{(3)}a_2^{(7)}a_3^{(7)}a_4^{(2)}a_5^{(2)}+
\medskip
 a_1^{(3)}a_2^{(6)}a_3^{(9)}a_4^{(1)}a_5^{(2)}\\
&\quad +
 a_1^{(3)}a_2^{(9)}a_3^{(6)}a_4^{(1)}a_5^{(2)}+
 a_1^{(3)}a_2^{(10)}a_3^{(5)}a_4^{(1)}a_5^{(2)}+
 a_1^{(3)}a_2^{(5)}a_3^{(10)}a_4^{(2)}a_5^{(1)}+
\medskip
 a_1^{(3)}a_2^{(13)}a_3^{(3)}a_4^{(1)}a_5^{(1)}\\
&\quad +
 a_1^{(3)}a_2^{(5)}a_3^{(11)}a_4^{(1)}a_5^{(1)}+
 a_1^{(3)}a_2^{(9)}a_3^{(7)}a_4^{(1)}a_5^{(1)}.
\end{array}\right\}.$$
According to the proof of Theorem \ref{dl2}, $\{[q_2]\}$ is a basis of $(Q\mathscr P_5)^{GL_5}$ in degree $13.2^{1} - 5.$ Notice that $\langle [q_2], [Z]\rangle = 1.$ So, since $Z\in P_{\mathcal A_2}H_{13.2^{1} - 5}(B(\mathbb{Z}/2)^{\times 5}),$ $[Z]$ is dual to $[q_2].$ Using the representation of $Tr_5$ over the algebra $\Lambda$ and the differential \eqref{ct3} in Sect.\ref{s1}, we obtain 
$$  \begin{array}{ll}
\psi_5(Z) &= \lambda_3\lambda_4\lambda_6\lambda_5\lambda_3 + \lambda_3\lambda_5\lambda_7\lambda^2_3 + \lambda_3^3\lambda_2\lambda_5\lambda_7\\
\medskip
&\quad + \lambda_3\lambda_2\lambda_4\lambda_5\lambda_7 + \lambda_3\lambda_7\lambda_3\lambda_5\lambda_3 \\
\medskip
&= \lambda_3\overline{f}_0 + \partial(\lambda_3\lambda_{11}\lambda_5\lambda_3).
\end{array}$$  
Since $Z\in P_{\mathcal A_2}H_{13.2^{1} - 5}(B(\mathbb{Z}/2)^{\times 5}),$ $\psi_5(Z)$ is a cycle in $\Lambda^{5, 21}.$ This implies that $h_2f_0$ is in the image of $Tr_5.$ Further, by Theorem \ref{dl2}, $\mathbb Z/2 \otimes_{GL_5}P_{\mathcal A_2}H_{13.2^{1} - 5}(B(\mathbb Z/2)^{\times 5})$ and ${\rm Ext}_{\mathcal {A}_2}^{5,5 + (13.2^{1} - 5)}(\mathbb Z/2,\mathbb Z/2)$ that have the same dimensions are $1.$ Hence, $Tr_5$ is an isomorphism when acting on the space $\mathbb Z/2 \otimes_{GL_5}P_{\mathcal A_2}H_{13.2^{1} - 5}(B(\mathbb Z/2)^{\times 5}).$ Theorem \ref{dl3} is proved.

\newpage
\section{Appendix}\label{s6}
In this section, we describe all the admissible monomials in $\mathscr B_5(n)$ for $n \in \{8, 21, 22, 47\}.$ These monomials are mentioned in Sects.\ref{s3} and \ref{s4}.

\subsection{$\mathcal A_2$-generators for $\mathscr P_5$ in degree $8$}\label{s5.1}

From a result in \cite{N.T1}, we deduce that $\mathscr B_5(8) = \bigcup_{1\leq j\leq 3}\mathscr B_5(\widetilde\omega_{(j)}),$ where  $$\widetilde\omega_{(1)} = (2,1,1),\ \widetilde\omega_{(2)} = (2,3),\ \mbox{and}\  \widetilde\omega_{(3)} = (4,2).$$ For $m,\ k\in\mathbb N$ and $1\leq k\leq 5,$ we denote
$$ \overline{\mathscr B}(k, 8) := \big\{x_k^{2^{m} - 1}\rho_{(k,\,5)}(x)\in (\mathscr P_5)_{8}\ :\ x\in \mathscr B_4(9-2^{m}),\ \alpha(13-2^{m})\leq 4\big\}.$$
As well known (see \cite{M.M2}), $\overline{\mathscr B}(k, 8)\subseteq \mathscr B_5(8)$ for all $k,\ 1\leq k\leq 5.$ We set $$ \overline{\mathscr B}(k, \widetilde\omega_{(j)})  := \overline{\mathscr B}(k, 8) \cap \mathscr P_5(\widetilde\omega_{(j)}),\ 1\leq j\leq 3,\ 1\leq k\leq 5.$$ Then, by a simple computation, we get
$$ \big|\overline{\Phi}(\mathscr B_4(\widetilde\omega_{(1)}))\bigcup\big(\bigcup_{1\leq k\leq 5}\overline{\mathscr B}(k, \widetilde\omega_{(1)})\big)\big| = 105,\ \big|\overline{\Phi}(\mathscr B_4(\widetilde\omega_{(2)}))\bigcup\big(\bigcup_{1\leq k\leq 5}\overline{\mathscr B}(k, \widetilde\omega_{(2)})\big)\big| = 24,$$
$$ \big|\overline{\Phi}(\mathscr B_4(\widetilde\omega_{(3)}))\bigcup\big(\bigcup_{1\leq k\leq 5}\overline{\mathscr B}(k, \widetilde\omega_{(3)})\big)\big| = 45.$$ Furthermore, $$ \mathscr B_5(\widetilde\omega_{(j)}) = \overline{\Phi}(\mathscr B_4(\widetilde\omega_{(j)}))\bigcup\big(\bigcup_{1\leq k\leq 5}\overline{\mathscr B}(k, \widetilde\omega_{(j)})\big),\ j =1, 2, 3.$$ Note that $\sum_{1\leq j\leq 3}|\mathscr B_5(\widetilde\omega_{(j)})| = \dim(Q\mathscr P_5)_{13.2^{0}-5} = 105 + 24 + 45 = 174.$ 
This implies that Conjecture \ref{gtS} is true for $d = 5$ and in degree $8.$

Now, we have $\mathscr B_5(8) = \{\mathcal Y_{8,\,t}:\, 1\leq t\leq 174\},$ where the monomials $\mathcal Y_{8,\,t}:\, 1\leq t\leq 174,$ are determined as follows:  

\begin{center}
\begin{tabular}{llrr}
$\mathcal Y_{8,\,1}=x_4x_5^{7}$, & $\mathcal Y_{8,\,2}=x_4^{7}x_5$, & \multicolumn{1}{l}{$\mathcal Y_{8,\,3}=x_3x_5^{7}$,} & \multicolumn{1}{l}{$\mathcal Y_{8,\,4}=x_3x_4^{7}$,} \\
$\mathcal Y_{8,\,5}=x_3^{7}x_5$, & $\mathcal Y_{8,\,6}=x_3^{7}x_4$, & \multicolumn{1}{l}{$\mathcal Y_{8,\,7}=x_2x_5^{7}$,} & \multicolumn{1}{l}{$\mathcal Y_{8,\,8}=x_2x_4^{7}$,} \\
$\mathcal Y_{8,\,9}=x_2x_3^{7}$, & $\mathcal Y_{8,\,10}=x_2^{7}x_5$, & \multicolumn{1}{l}{$\mathcal Y_{8,\,11}=x_2^{7}x_4$,} & \multicolumn{1}{l}{$\mathcal Y_{8,\,12}=x_2^{7}x_3$,} \\
$\mathcal Y_{8,\,13}=x_1x_5^{7}$, & $\mathcal Y_{8,\,14}=x_1x_4^{7}$, & \multicolumn{1}{l}{$\mathcal Y_{8,\,15}=x_1x_3^{7}$,} & \multicolumn{1}{l}{$\mathcal Y_{8,\,16}=x_1x_2^{7}$,} \\
$\mathcal Y_{8,\,17}=x_1^{7}x_2$, & $\mathcal Y_{8,\,18}=x_1^{7}x_3$, & \multicolumn{1}{l}{$\mathcal Y_{8,\,19}=x_1^{7}x_4$,} & \multicolumn{1}{l}{$\mathcal Y_{8,\,20}=x_1^{7}x_5$,} \\
$\mathcal Y_{8,\,21}=x_4^{3}x_5^{5}$, & $\mathcal Y_{8,\,22}=x_3^{3}x_5^{5}$, & \multicolumn{1}{l}{$\mathcal Y_{8,\,23}=x_3^{3}x_4^{5}$,} & \multicolumn{1}{l}{$\mathcal Y_{8,\,24}=x_2^{3}x_5^{5}$,} \\
$\mathcal Y_{8,\,25}=x_2^{3}x_4^{5}$, & $\mathcal Y_{8,\,26}=x_2^{3}x_3^{5}$, & \multicolumn{1}{l}{$\mathcal Y_{8,\,27}=x_1^{3}x_5^{5}$,} & \multicolumn{1}{l}{$\mathcal Y_{8,\,28}=x_1^{3}x_4^{5}$,} \\
$\mathcal Y_{8,\,29}=x_1^{3}x_3^{5}$, & $\mathcal Y_{8,\,30}=x_1^{3}x_2^{5}$, & \multicolumn{1}{l}{$\mathcal Y_{8,\,31}=x_3x_4x_5^{6}$,} & \multicolumn{1}{l}{$\mathcal Y_{8,\,32}=x_3x_4^{6}x_5$,} \\
\end{tabular}
\end{center}

\newpage
\begin{center}
\begin{tabular}{llrr}
$\mathcal Y_{8,\,97}=x_1x_2x_3^{2}x_4^{4}$, & $\mathcal Y_{8,\,98}=x_1x_3^{2}x_4x_5^{4}$, & \multicolumn{1}{l}{$\mathcal Y_{8,\,99}=x_1x_3^{2}x_4^{4}x_5$,} & \multicolumn{1}{l}{$\mathcal Y_{8,\,100}=x_1x_2^{2}x_4x_5^{4}$,} \\
$\mathcal Y_{8,\,101}=x_1x_2^{2}x_4^{4}x_5$, & $\mathcal Y_{8,\,102}=x_1x_2^{2}x_3x_5^{4}$, & \multicolumn{1}{l}{$\mathcal Y_{8,\,103}=x_1x_2^{2}x_3x_4^{4}$,} & \multicolumn{1}{l}{$\mathcal Y_{8,\,104}=x_1x_2^{2}x_3^{4}x_5$,} \\
$\mathcal Y_{8,\,105}=x_1x_2^{2}x_3^{4}x_4$, & $\mathcal Y_{8,\,106}=x_2x_3^{2}x_4^{2}x_5^{3}$, & \multicolumn{1}{l}{$\mathcal Y_{8,\,107}=x_2x_3^{2}x_4^{3}x_5^{2}$,} & \multicolumn{1}{l}{$\mathcal Y_{8,\,108}=x_2x_3^{3}x_4^{2}x_5^{2}$,} \\
$\mathcal Y_{8,\,109}=x_2^{3}x_3x_4^{2}x_5^{2}$, & $\mathcal Y_{8,\,110}=x_1x_3^{2}x_4^{2}x_5^{3}$, & \multicolumn{1}{l}{$\mathcal Y_{8,\,111}=x_1x_3^{2}x_4^{3}x_5^{2}$,} & \multicolumn{1}{l}{$\mathcal Y_{8,\,112}=x_1x_3^{3}x_4^{2}x_5^{2}$,} \\
$\mathcal Y_{8,\,113}=x_1x_2^{2}x_4^{2}x_5^{3}$, & $\mathcal Y_{8,\,114}=x_1x_2^{2}x_4^{3}x_5^{2}$, & \multicolumn{1}{l}{$\mathcal Y_{8,\,115}=x_1x_2^{3}x_4^{2}x_5^{2}$,} & \multicolumn{1}{l}{$\mathcal Y_{8,\,116}=x_1x_2^{2}x_3^{2}x_5^{3}$,} \\
$\mathcal Y_{8,\,117}=x_1x_2^{2}x_3^{3}x_5^{2}$, & $\mathcal Y_{8,\,118}=x_1x_2^{3}x_3^{2}x_5^{2}$, & \multicolumn{1}{l}{$\mathcal Y_{8,\,119}=x_1x_2^{2}x_3^{2}x_4^{3}$,} & \multicolumn{1}{l}{$\mathcal Y_{8,\,120}=x_1x_2^{2}x_3^{3}x_4^{2}$,} \\
$\mathcal Y_{8,\,121}=x_1x_2^{3}x_3^{2}x_4^{2}$, & $\mathcal Y_{8,\,122}=x_1^{3}x_3x_4^{2}x_5^{2}$, & \multicolumn{1}{l}{$\mathcal Y_{8,\,123}=x_1^{3}x_2x_4^{2}x_5^{2}$,} & \multicolumn{1}{l}{$\mathcal Y_{8,\,124}=x_1^{3}x_2x_3^{2}x_5^{2}$,} \\
$\mathcal Y_{8,\,125}=x_1^{3}x_2x_3^{2}x_4^{2}$, & $\mathcal Y_{8,\,126}=x_1x_2x_3^{2}x_4^{2}x_5^{2}$, & \multicolumn{1}{l}{$\mathcal Y_{8,\,127}=x_1x_2^{2}x_3x_4^{2}x_5^{2}$,} & \multicolumn{1}{l}{$\mathcal Y_{8,\,128}=x_1x_2^{2}x_3^{2}x_4x_5^{2}$,} \\
\end{tabular}
\end{center}

\subsection{$\mathcal A_2$-generators for $\mathscr P_5^0$ in degree $21$}\label{s5.2}

Using the results in Sect.\ref{s3.2.1}, we have $\mathscr B_5^0(21) = \mathscr B_5^0(3,3,1,1)\bigcup \mathscr B_5^0(3,3,3),$ where $\mathscr B_5^0(3,3,1,1)$ is the set of $340$ admissible monomials: $\mathcal Y_{21,\, t},\ 1\leq t\leq 340$

\medskip

\begin{center}
\begin{tabular}{llll}
$\mathcal Y_{21,\,1}=x_3^{3}x_4^{3}x_5^{15}$, & $\mathcal Y_{21,\,2}=x_3^{3}x_4^{15}x_5^{3}$, & $\mathcal Y_{21,\,3}=x_3^{15}x_4^{3}x_5^{3}$, & $\mathcal Y_{21,\,4}=x_2^{3}x_4^{3}x_5^{15}$, \\
$\mathcal Y_{21,\,5}=x_2^{3}x_4^{15}x_5^{3}$, & $\mathcal Y_{21,\,6}=x_2^{3}x_3^{3}x_5^{15}$, & $\mathcal Y_{21,\,7}=x_2^{3}x_3^{3}x_4^{15}$, & $\mathcal Y_{21,\,8}=x_2^{3}x_3^{15}x_5^{3}$, \\
$\mathcal Y_{21,\,9}=x_2^{3}x_3^{15}x_4^{3}$, & $\mathcal Y_{21,\,10}=x_2^{15}x_4^{3}x_5^{3}$, & $\mathcal Y_{21,\,11}=x_2^{15}x_3^{3}x_5^{3}$, & $\mathcal Y_{21,\,12}=x_2^{15}x_3^{3}x_4^{3}$, \\
$\mathcal Y_{21,\,13}=x_1^{3}x_4^{3}x_5^{15}$, & $\mathcal Y_{21,\,14}=x_1^{3}x_4^{15}x_5^{3}$, & $\mathcal Y_{21,\,15}=x_1^{3}x_3^{3}x_5^{15}$, & $\mathcal Y_{21,\,16}=x_1^{3}x_3^{3}x_4^{15}$, \\
$\mathcal Y_{21,\,17}=x_1^{3}x_3^{15}x_5^{3}$, & $\mathcal Y_{21,\,18}=x_1^{3}x_3^{15}x_4^{3}$, & $\mathcal Y_{21,\,19}=x_1^{3}x_2^{3}x_5^{15}$, & $\mathcal Y_{21,\,20}=x_1^{3}x_2^{3}x_4^{15}$, \\
$\mathcal Y_{21,\,21}=x_1^{3}x_2^{3}x_3^{15}$, & $\mathcal Y_{21,\,22}=x_1^{3}x_2^{15}x_5^{3}$, & $\mathcal Y_{21,\,23}=x_1^{3}x_2^{15}x_4^{3}$, & $\mathcal Y_{21,\,24}=x_1^{3}x_2^{15}x_3^{3}$, \\
$\mathcal Y_{21,\,25}=x_1^{15}x_4^{3}x_5^{3}$, & $\mathcal Y_{21,\,26}=x_1^{15}x_3^{3}x_5^{3}$, & $\mathcal Y_{21,\,27}=x_1^{15}x_3^{3}x_4^{3}$, & $\mathcal Y_{21,\,28}=x_1^{15}x_2^{3}x_5^{3}$, \\
$\mathcal Y_{21,\,29}=x_1^{15}x_2^{3}x_4^{3}$, & $\mathcal Y_{21,\,30}=x_1^{15}x_2^{3}x_3^{3}$, & $\mathcal Y_{21,\,31}=x_3^{3}x_4^{7}x_5^{11}$, & $\mathcal Y_{21,\,32}=x_3^{7}x_4^{3}x_5^{11}$, \\
\end{tabular}%
\end{center}

\newpage
\begin{center}
\begin{tabular}{llll}
$\mathcal Y_{21,\,105}=x_1^{3}x_2x_3^{15}x_4^{2}$, & $\mathcal Y_{21,\,106}=x_1^{3}x_2^{15}x_4x_5^{2}$, & $\mathcal Y_{21,\,107}=x_1^{3}x_2^{15}x_3x_5^{2}$, & $\mathcal Y_{21,\,108}=x_1^{3}x_2^{15}x_3x_4^{2}$, \\
$\mathcal Y_{21,\,109}=x_1^{15}x_3x_4^{2}x_5^{3}$, & $\mathcal Y_{21,\,110}=x_1^{15}x_3x_4^{3}x_5^{2}$, & $\mathcal Y_{21,\,111}=x_1^{15}x_3^{3}x_4x_5^{2}$, & $\mathcal Y_{21,\,112}=x_1^{15}x_2x_4^{2}x_5^{3}$, \\
$\mathcal Y_{21,\,113}=x_1^{15}x_2x_4^{3}x_5^{2}$, & $\mathcal Y_{21,\,114}=x_1^{15}x_2x_3^{2}x_5^{3}$, & $\mathcal Y_{21,\,115}=x_1^{15}x_2x_3^{2}x_4^{3}$, & $\mathcal Y_{21,\,116}=x_1^{15}x_2x_3^{3}x_5^{2}$, \\
$\mathcal Y_{21,\,117}=x_1^{15}x_2x_3^{3}x_4^{2}$, & $\mathcal Y_{21,\,118}=x_1^{15}x_2^{3}x_4x_5^{2}$, & $\mathcal Y_{21,\,119}=x_1^{15}x_2^{3}x_3x_5^{2}$, & $\mathcal Y_{21,\,120}=x_1^{15}x_2^{3}x_3x_4^{2}$, \\
$\mathcal Y_{21,\,121}=x_2x_3^{2}x_4^{7}x_5^{11}$, & $\mathcal Y_{21,\,122}=x_2x_3^{7}x_4^{2}x_5^{11}$, & $\mathcal Y_{21,\,123}=x_2x_3^{7}x_4^{11}x_5^{2}$, & $\mathcal Y_{21,\,124}=x_2^{7}x_3x_4^{2}x_5^{11}$, \\
$\mathcal Y_{21,\,125}=x_2^{7}x_3x_4^{11}x_5^{2}$, & $\mathcal Y_{21,\,126}=x_2^{7}x_3^{11}x_4x_5^{2}$, & $\mathcal Y_{21,\,127}=x_1x_3^{2}x_4^{7}x_5^{11}$, & $\mathcal Y_{21,\,128}=x_1x_3^{7}x_4^{2}x_5^{11}$, \\
$\mathcal Y_{21,\,129}=x_1x_3^{7}x_4^{11}x_5^{2}$, & $\mathcal Y_{21,\,130}=x_1x_2^{2}x_4^{7}x_5^{11}$, & $\mathcal Y_{21,\,131}=x_1x_2^{2}x_3^{7}x_5^{11}$, & $\mathcal Y_{21,\,132}=x_1x_2^{2}x_3^{7}x_4^{11}$, \\
$\mathcal Y_{21,\,133}=x_1x_2^{7}x_4^{2}x_5^{11}$, & $\mathcal Y_{21,\,134}=x_1x_2^{7}x_4^{11}x_5^{2}$, & $\mathcal Y_{21,\,135}=x_1x_2^{7}x_3^{2}x_5^{11}$, & $\mathcal Y_{21,\,136}=x_1x_2^{7}x_3^{2}x_4^{11}$, \\
\end{tabular}%
\end{center}

\newpage
\begin{center}
\begin{tabular}{llll}
$\mathcal Y_{21,\,305}=x_1^{3}x_3^{5}x_4^{3}x_5^{10}$, & $\mathcal Y_{21,\,306}=x_1^{3}x_3^{5}x_4^{10}x_5^{3}$, & $\mathcal Y_{21,\,307}=x_1^{3}x_2^{3}x_4^{5}x_5^{10}$, & $\mathcal Y_{21,\,308}=x_1^{3}x_2^{3}x_3^{5}x_5^{10}$, \\
$\mathcal Y_{21,\,309}=x_1^{3}x_2^{3}x_3^{5}x_4^{10}$, & $\mathcal Y_{21,\,310}=x_1^{3}x_2^{5}x_4^{3}x_5^{10}$, & $\mathcal Y_{21,\,311}=x_1^{3}x_2^{5}x_4^{10}x_5^{3}$, & $\mathcal Y_{21,\,312}=x_1^{3}x_2^{5}x_3^{3}x_5^{10}$, \\
$\mathcal Y_{21,\,313}=x_1^{3}x_2^{5}x_3^{3}x_4^{10}$, & $\mathcal Y_{21,\,314}=x_1^{3}x_2^{5}x_3^{10}x_5^{3}$, & $\mathcal Y_{21,\,315}=x_1^{3}x_2^{5}x_3^{10}x_4^{3}$, & $\mathcal Y_{21,\,316}=x_2^{3}x_3^{3}x_4^{7}x_5^{8}$, \\
$\mathcal Y_{21,\,317}=x_2^{3}x_3^{7}x_4^{3}x_5^{8}$, & $\mathcal Y_{21,\,318}=x_2^{3}x_3^{7}x_4^{8}x_5^{3}$, & $\mathcal Y_{21,\,319}=x_2^{7}x_3^{3}x_4^{3}x_5^{8}$, & $\mathcal Y_{21,\,320}=x_2^{7}x_3^{3}x_4^{8}x_5^{3}$, \\
$\mathcal Y_{21,\,321}=x_1^{3}x_3^{3}x_4^{7}x_5^{8}$, & $\mathcal Y_{21,\,322}=x_1^{3}x_3^{7}x_4^{3}x_5^{8}$, & $\mathcal Y_{21,\,323}=x_1^{3}x_3^{7}x_4^{8}x_5^{3}$, & $\mathcal Y_{21,\,324}=x_1^{3}x_2^{3}x_4^{7}x_5^{8}$, \\
$\mathcal Y_{21,\,325}=x_1^{3}x_2^{3}x_3^{7}x_5^{8}$, & $\mathcal Y_{21,\,326}=x_1^{3}x_2^{3}x_3^{7}x_4^{8}$, & $\mathcal Y_{21,\,327}=x_1^{3}x_2^{7}x_4^{3}x_5^{8}$, & $\mathcal Y_{21,\,328}=x_1^{3}x_2^{7}x_4^{8}x_5^{3}$, \\
$\mathcal Y_{21,\,329}=x_1^{3}x_2^{7}x_3^{3}x_5^{8}$, & $\mathcal Y_{21,\,330}=x_1^{3}x_2^{7}x_3^{3}x_4^{8}$, & $\mathcal Y_{21,\,331}=x_1^{3}x_2^{7}x_3^{8}x_5^{3}$, & $\mathcal Y_{21,\,332}=x_1^{3}x_2^{7}x_3^{8}x_4^{3}$, \\
$\mathcal Y_{21,\,333}=x_1^{7}x_3^{3}x_4^{3}x_5^{8}$, & $\mathcal Y_{21,\,334}=x_1^{7}x_3^{3}x_4^{8}x_5^{3}$, & $\mathcal Y_{21,\,335}=x_1^{7}x_2^{3}x_4^{3}x_5^{8}$, & $\mathcal Y_{21,\,336}=x_1^{7}x_2^{3}x_4^{8}x_5^{3}$, \\
\end{tabular}%
\end{center}

\medskip

$\mathscr B_5^0(3,3,3)$ is the set of $60$ admissible monomials: $\mathcal Y_{21,\, t},\ 341\leq t\leq 400$

\begin{center}
\begin{tabular}{llll}

$\mathcal Y_{21,\,341}=x_3^{7}x_4^{7}x_5^{7}$, & $\mathcal Y_{21,\,342}=x_2^{7}x_4^{7}x_5^{7}$, & $\mathcal Y_{21,\,343}=x_2^{7}x_3^{7}x_5^{7}$, & $\mathcal Y_{21,\,344}=x_2^{7}x_3^{7}x_4^{7}$, \\
$\mathcal Y_{21,\,345}=x_1^{7}x_4^{7}x_5^{7}$, & $\mathcal Y_{21,\,346}=x_1^{7}x_3^{7}x_5^{7}$, & $\mathcal Y_{21,\,347}=x_1^{7}x_3^{7}x_4^{7}$, & $\mathcal Y_{21,\,348}=x_1^{7}x_2^{7}x_5^{7}$, \\
$\mathcal Y_{21,\,349}=x_1^{7}x_2^{7}x_4^{7}$, & $\mathcal Y_{21,\,350}=x_1^{7}x_2^{7}x_3^{7}$, & $\mathcal Y_{21,\,351}=x_2x_3^{6}x_4^{7}x_5^{7}$, & $\mathcal Y_{21,\,352}=x_2x_3^{7}x_4^{6}x_5^{7}$, \\
$\mathcal Y_{21,\,353}=x_2x_3^{7}x_4^{7}x_5^{6}$, & $\mathcal Y_{21,\,354}=x_2^{7}x_3x_4^{6}x_5^{7}$, & $\mathcal Y_{21,\,355}=x_2^{7}x_3x_4^{7}x_5^{6}$, & $\mathcal Y_{21,\,356}=x_2^{7}x_3^{7}x_4x_5^{6}$, \\
$\mathcal Y_{21,\,357}=x_1x_3^{6}x_4^{7}x_5^{7}$, & $\mathcal Y_{21,\,358}=x_1x_3^{7}x_4^{6}x_5^{7}$, & $\mathcal Y_{21,\,359}=x_1x_3^{7}x_4^{7}x_5^{6}$, & $\mathcal Y_{21,\,360}=x_1x_2^{6}x_4^{7}x_5^{7}$, \\
$\mathcal Y_{21,\,361}=x_1x_2^{6}x_3^{7}x_5^{7}$, & $\mathcal Y_{21,\,362}=x_1x_2^{6}x_3^{7}x_4^{7}$, & $\mathcal Y_{21,\,363}=x_1x_2^{7}x_4^{6}x_5^{7}$, & $\mathcal Y_{21,\,364}=x_1x_2^{7}x_4^{7}x_5^{6}$, \\
$\mathcal Y_{21,\,365}=x_1x_2^{7}x_3^{6}x_5^{7}$, & $\mathcal Y_{21,\,366}=x_1x_2^{7}x_3^{6}x_4^{7}$, & $\mathcal Y_{21,\,367}=x_1x_2^{7}x_3^{7}x_5^{6}$, & $\mathcal Y_{21,\,368}=x_1x_2^{7}x_3^{7}x_4^{6}$, \\
$\mathcal Y_{21,\,369}=x_1^{7}x_3x_4^{6}x_5^{7}$, & $\mathcal Y_{21,\,370}=x_1^{7}x_3x_4^{7}x_5^{6}$, & $\mathcal Y_{21,\,371}=x_1^{7}x_3^{7}x_4x_5^{6}$, & $\mathcal Y_{21,\,372}=x_1^{7}x_2x_4^{6}x_5^{7}$, \\
\end{tabular}
\end{center}

\subsection{$\mathcal A_2$-generators for $\mathscr P_5^+$ in degree $21$}\label{s5.3}

Note that Kameko's squaring operation $(\widetilde {Sq_*^0})_{(5,21)}: (Q\mathscr P_5)_{21}\to (Q\mathscr P_5)_{8}$ is an epimorphism of $\mathbb Z/2GL_5$-modules. Hence, we have
$$ \mathscr B_5(21) = \mathscr B_5^0(21)\bigcup \varphi(\mathscr B_5(8))\bigcup \big(\mathscr B_5^+(21)\bigcap {\rm Ker}((\widetilde {Sq_*^0})_{(5,21)})\big),$$ where $\mathscr B_5^0(21) = \overline{\Phi}^0(\mathscr B_4(21)),\ |\mathscr B_5^0(21)|  = 400,\ |\varphi(\mathscr B_5(8))| = 174$ with $$ \varphi: \mathscr P_5\to \mathscr P_5,\ \varphi(Z) = X_{(\emptyset,\, 5)}Z^2,\ \forall Z\in \mathscr P_5.$$ 
Based on the results in Sect.\ref{s3.2.1}, we see that
$$ \mathscr B_5^+(21)\,\bigcap\, {\rm Ker}((\widetilde {Sq_*^0})_{(5,21)}) = \mathscr B^+_5(3,3,1,1)\,\cup\, \mathscr B^+_5(3,3,3),$$
where $\mathscr B_5^+(3,3,1,1)$ is the set of $196$ admissible monomials: $\mathcal Y_{21,\, t},\ 401\leq t\leq 596$

\begin{center}
\begin{tabular}{llll}
$\mathcal Y_{21,\,401}=x_1x_2^{2}x_3x_4^{2}x_5^{15}$, & $\mathcal Y_{21,\,402}=x_1x_2^{2}x_3x_4^{15}x_5^{2}$, & $\mathcal Y_{21,\,403}=x_1x_2^{2}x_3^{15}x_4x_5^{2}$, & $\mathcal Y_{21,\,404}=x_1x_2x_3^{2}x_4^{2}x_5^{15}$, \\
$\mathcal Y_{21,\,405}=x_1x_2x_3^{2}x_4^{15}x_5^{2}$, & $\mathcal Y_{21,\,406}=x_1x_2^{15}x_3^{2}x_4x_5^{2}$, & $\mathcal Y_{21,\,407}=x_1x_2x_3^{15}x_4^{2}x_5^{2}$, & $\mathcal Y_{21,\,408}=x_1x_2^{15}x_3x_4^{2}x_5^{2}$, \\
$\mathcal Y_{21,\,409}=x_1^{15}x_2x_3^{2}x_4x_5^{2}$, & $\mathcal Y_{21,\,410}=x_1^{15}x_2x_3x_4^{2}x_5^{2}$, & $\mathcal Y_{21,\,411}=x_1x_2x_3^{6}x_4^{3}x_5^{10}$, & $\mathcal Y_{21,\,412}=x_1x_2x_3^{6}x_4^{10}x_5^{3}$, \\
$\mathcal Y_{21,\,413}=x_1x_2x_3^{3}x_4^{6}x_5^{10}$, & $\mathcal Y_{21,\,414}=x_1x_2^{2}x_3^{3}x_4^{3}x_5^{12}$, & $\mathcal Y_{21,\,415}=x_1x_2^{2}x_3^{3}x_4^{12}x_5^{3}$, & $\mathcal Y_{21,\,416}=x_1x_2^{2}x_3^{12}x_4^{3}x_5^{3}$, \\
$\mathcal Y_{21,\,417}=x_1^{3}x_2^{3}x_3^{4}x_4^{3}x_5^{8}$, & $\mathcal Y_{21,\,418}=x_1^{3}x_2^{3}x_3^{4}x_4^{8}x_5^{3}$, & $\mathcal Y_{21,\,419}=x_1^{3}x_2^{3}x_3^{3}x_4^{4}x_5^{8}$, & $\mathcal Y_{21,\,420}=x_1x_2^{2}x_3x_4^{3}x_5^{14}$, \\
$\mathcal Y_{21,\,421}=x_1x_2^{2}x_3x_4^{14}x_5^{3}$, & $\mathcal Y_{21,\,422}=x_1x_2^{2}x_3^{3}x_4x_5^{14}$, & $\mathcal Y_{21,\,423}=x_1x_2^{14}x_3x_4^{2}x_5^{3}$, & $\mathcal Y_{21,\,424}=x_1x_2^{14}x_3x_4^{3}x_5^{2}$, \\
$\mathcal Y_{21,\,425}=x_1x_2^{14}x_3^{3}x_4x_5^{2}$, & $\mathcal Y_{21,\,426}=x_1x_2x_3^{2}x_4^{3}x_5^{14}$, & $\mathcal Y_{21,\,427}=x_1x_2x_3^{2}x_4^{14}x_5^{3}$, & $\mathcal Y_{21,\,428}=x_1x_2x_3^{14}x_4^{2}x_5^{3}$, \\
\end{tabular}
\end{center}

\newpage
\begin{center}
\begin{tabular}{llll}
$\mathcal Y_{21,\,429}=x_1x_2x_3^{14}x_4^{3}x_5^{2}$, & $\mathcal Y_{21,\,430}=x_1x_2^{3}x_3^{2}x_4x_5^{14}$, & $\mathcal Y_{21,\,431}=x_1x_2^{3}x_3^{14}x_4x_5^{2}$, & $\mathcal Y_{21,\,432}=x_1x_2x_3^{3}x_4^{2}x_5^{14}$, \\
$\mathcal Y_{21,\,433}=x_1x_2x_3^{3}x_4^{14}x_5^{2}$, & $\mathcal Y_{21,\,434}=x_1x_2^{3}x_3x_4^{2}x_5^{14}$, & $\mathcal Y_{21,\,435}=x_1x_2^{3}x_3x_4^{14}x_5^{2}$, & $\mathcal Y_{21,\,436}=x_1x_2^{2}x_3x_4^{6}x_5^{11}$, \\
$\mathcal Y_{21,\,437}=x_1x_2^{6}x_3x_4^{2}x_5^{11}$, & $\mathcal Y_{21,\,438}=x_1x_2^{6}x_3x_4^{11}x_5^{2}$, & $\mathcal Y_{21,\,439}=x_1x_2^{6}x_3^{11}x_4x_5^{2}$, & $\mathcal Y_{21,\,440}=x_1x_2x_3^{2}x_4^{6}x_5^{11}$, \\
$\mathcal Y_{21,\,441}=x_1x_2x_3^{6}x_4^{2}x_5^{11}$, & $\mathcal Y_{21,\,442}=x_1x_2x_3^{6}x_4^{11}x_5^{2}$, & $\mathcal Y_{21,\,443}=x_1x_2^{2}x_3x_4^{7}x_5^{10}$, & $\mathcal Y_{21,\,444}=x_1x_2^{2}x_3^{7}x_4x_5^{10}$, \\
$\mathcal Y_{21,\,445}=x_1x_2x_3^{2}x_4^{7}x_5^{10}$, & $\mathcal Y_{21,\,446}=x_1x_2^{7}x_3^{10}x_4x_5^{2}$, & $\mathcal Y_{21,\,447}=x_1x_2^{7}x_3^{2}x_4x_5^{10}$, & $\mathcal Y_{21,\,448}=x_1x_2x_3^{7}x_4^{10}x_5^{2}$, \\
$\mathcal Y_{21,\,449}=x_1x_2x_3^{7}x_4^{2}x_5^{10}$, & $\mathcal Y_{21,\,450}=x_1x_2^{7}x_3x_4^{10}x_5^{2}$, & $\mathcal Y_{21,\,451}=x_1x_2^{7}x_3x_4^{2}x_5^{10}$, & $\mathcal Y_{21,\,452}=x_1x_2^{2}x_3^{3}x_4^{13}x_5^{2}$, \\
$\mathcal Y_{21,\,453}=x_1x_2^{2}x_3^{13}x_4^{2}x_5^{3}$, & $\mathcal Y_{21,\,454}=x_1x_2^{2}x_3^{13}x_4^{3}x_5^{2}$, & $\mathcal Y_{21,\,455}=x_1x_2^{3}x_3^{2}x_4^{13}x_5^{2}$, & $\mathcal Y_{21,\,456}=x_1x_2^{3}x_3^{13}x_4^{2}x_5^{2}$, \\
$\mathcal Y_{21,\,457}=x_1x_2^{2}x_3^{5}x_4^{2}x_5^{11}$, & $\mathcal Y_{21,\,458}=x_1x_2^{2}x_3^{5}x_4^{11}x_5^{2}$, & $\mathcal Y_{21,\,459}=x_1x_2^{2}x_3^{7}x_4^{9}x_5^{2}$, & $\mathcal Y_{21,\,460}=x_1x_2^{7}x_3^{2}x_4^{9}x_5^{2}$, \\
$\mathcal Y_{21,\,461}=x_1x_2^{7}x_3^{9}x_4^{2}x_5^{2}$, & $\mathcal Y_{21,\,462}=x_1x_2^{6}x_3x_4^{3}x_5^{10}$, & $\mathcal Y_{21,\,463}=x_1x_2^{6}x_3x_4^{10}x_5^{3}$, & $\mathcal Y_{21,\,464}=x_1x_2^{6}x_3^{3}x_4x_5^{10}$, \\
$\mathcal Y_{21,\,465}=x_1x_2^{3}x_3^{6}x_4x_5^{10}$, & $\mathcal Y_{21,\,466}=x_1x_2^{3}x_3x_4^{6}x_5^{10}$, & $\mathcal Y_{21,\,467}=x_1x_2^{3}x_3^{12}x_4^{2}x_5^{3}$, & $\mathcal Y_{21,\,468}=x_1x_2^{3}x_3^{12}x_4^{3}x_5^{2}$, \\
$\mathcal Y_{21,\,469}=x_1x_2^{3}x_3^{2}x_4^{3}x_5^{12}$, & $\mathcal Y_{21,\,470}=x_1x_2^{3}x_3^{2}x_4^{12}x_5^{3}$, & $\mathcal Y_{21,\,471}=x_1x_2^{3}x_3^{3}x_4^{12}x_5^{2}$, & $\mathcal Y_{21,\,472}=x_1x_2^{3}x_3^{3}x_4^{2}x_5^{12}$, \\
$\mathcal Y_{21,\,473}=x_1x_2^{2}x_3^{3}x_4^{4}x_5^{11}$, & $\mathcal Y_{21,\,474}=x_1x_2^{3}x_3^{4}x_4^{2}x_5^{11}$, & $\mathcal Y_{21,\,475}=x_1x_2^{3}x_3^{4}x_4^{11}x_5^{2}$, & $\mathcal Y_{21,\,476}=x_1x_2^{3}x_3^{2}x_4^{4}x_5^{11}$, \\
$\mathcal Y_{21,\,477}=x_1x_2^{2}x_3^{4}x_4^{3}x_5^{11}$, & $\mathcal Y_{21,\,478}=x_1x_2^{2}x_3^{4}x_4^{11}x_5^{3}$, & $\mathcal Y_{21,\,479}=x_1x_2^{2}x_3^{3}x_4^{5}x_5^{10}$, & $\mathcal Y_{21,\,480}=x_1x_2^{2}x_3^{5}x_4^{3}x_5^{10}$, \\
$\mathcal Y_{21,\,481}=x_1x_2^{2}x_3^{5}x_4^{10}x_5^{3}$, & $\mathcal Y_{21,\,482}=x_1x_2^{3}x_3^{2}x_4^{5}x_5^{10}$, & $\mathcal Y_{21,\,483}=x_1x_2^{3}x_3^{5}x_4^{10}x_5^{2}$, & $\mathcal Y_{21,\,484}=x_1x_2^{3}x_3^{5}x_4^{2}x_5^{10}$, \\
$\mathcal Y_{21,\,485}=x_1x_2^{6}x_3^{3}x_4^{9}x_5^{2}$, & $\mathcal Y_{21,\,486}=x_1x_2^{6}x_3^{9}x_4^{2}x_5^{3}$, & $\mathcal Y_{21,\,487}=x_1x_2^{6}x_3^{9}x_4^{3}x_5^{2}$, & $\mathcal Y_{21,\,488}=x_1x_2^{3}x_3^{6}x_4^{9}x_5^{2}$, \\
$\mathcal Y_{21,\,489}=x_1x_2^{2}x_3^{3}x_4^{7}x_5^{8}$, & $\mathcal Y_{21,\,490}=x_1x_2^{2}x_3^{7}x_4^{3}x_5^{8}$, & $\mathcal Y_{21,\,491}=x_1x_2^{2}x_3^{7}x_4^{8}x_5^{3}$, & $\mathcal Y_{21,\,492}=x_1x_2^{3}x_3^{2}x_4^{7}x_5^{8}$, \\
$\mathcal Y_{21,\,493}=x_1x_2^{7}x_3^{8}x_4^{2}x_5^{3}$, & $\mathcal Y_{21,\,494}=x_1x_2^{7}x_3^{8}x_4^{3}x_5^{2}$, & $\mathcal Y_{21,\,495}=x_1x_2^{7}x_3^{2}x_4^{3}x_5^{8}$, & $\mathcal Y_{21,\,496}=x_1x_2^{7}x_3^{2}x_4^{8}x_5^{3}$, \\
$\mathcal Y_{21,\,497}=x_1x_2^{3}x_3^{7}x_4^{8}x_5^{2}$, & $\mathcal Y_{21,\,498}=x_1x_2^{3}x_3^{7}x_4^{2}x_5^{8}$, & $\mathcal Y_{21,\,499}=x_1x_2^{7}x_3^{3}x_4^{8}x_5^{2}$, & $\mathcal Y_{21,\,500}=x_1x_2^{7}x_3^{3}x_4^{2}x_5^{8}$, \\
$\mathcal Y_{21,\,501}=x_1x_2^{3}x_3^{4}x_4^{3}x_5^{10}$, & $\mathcal Y_{21,\,502}=x_1x_2^{3}x_3^{4}x_4^{10}x_5^{3}$, & $\mathcal Y_{21,\,503}=x_1x_2^{3}x_3^{3}x_4^{4}x_5^{10}$, & $\mathcal Y_{21,\,504}=x_1x_2^{6}x_3^{3}x_4^{3}x_5^{8}$, \\
$\mathcal Y_{21,\,505}=x_1x_2^{6}x_3^{3}x_4^{8}x_5^{3}$, & $\mathcal Y_{21,\,506}=x_1x_2^{3}x_3^{6}x_4^{3}x_5^{8}$, & $\mathcal Y_{21,\,507}=x_1x_2^{3}x_3^{6}x_4^{8}x_5^{3}$, & $\mathcal Y_{21,\,508}=x_1x_2^{3}x_3^{3}x_4^{6}x_5^{8}$, \\
$\mathcal Y_{21,\,509}=x_1^{3}x_2x_3^{2}x_4x_5^{14}$, & $\mathcal Y_{21,\,510}=x_1^{3}x_2x_3^{14}x_4x_5^{2}$, & $\mathcal Y_{21,\,511}=x_1^{3}x_2x_3x_4^{2}x_5^{14}$, & $\mathcal Y_{21,\,512}=x_1^{3}x_2x_3x_4^{14}x_5^{2}$, \\
$\mathcal Y_{21,\,513}=x_1^{3}x_2x_3^{2}x_4^{13}x_5^{2}$, & $\mathcal Y_{21,\,514}=x_1^{3}x_2x_3^{13}x_4^{2}x_5^{2}$, & $\mathcal Y_{21,\,515}=x_1^{3}x_2^{13}x_3x_4^{2}x_5^{2}$, & $\mathcal Y_{21,\,516}=x_1^{3}x_2^{13}x_3^{2}x_4x_5^{2}$, \\
$\mathcal Y_{21,\,517}=x_1^{3}x_2x_3^{6}x_4x_5^{10}$, & $\mathcal Y_{21,\,518}=x_1^{3}x_2x_3x_4^{6}x_5^{10}$, & $\mathcal Y_{21,\,519}=x_1^{3}x_2x_3^{12}x_4^{2}x_5^{3}$, & $\mathcal Y_{21,\,520}=x_1^{3}x_2x_3^{12}x_4^{3}x_5^{2}$, \\
$\mathcal Y_{21,\,521}=x_1^{3}x_2x_3^{2}x_4^{3}x_5^{12}$, & $\mathcal Y_{21,\,522}=x_1^{3}x_2x_3^{2}x_4^{12}x_5^{3}$, & $\mathcal Y_{21,\,523}=x_1^{3}x_2x_3^{3}x_4^{12}x_5^{2}$, & $\mathcal Y_{21,\,524}=x_1^{3}x_2x_3^{3}x_4^{2}x_5^{12}$, \\
$\mathcal Y_{21,\,525}=x_1^{3}x_2^{3}x_3x_4^{12}x_5^{2}$, & $\mathcal Y_{21,\,526}=x_1^{3}x_2^{3}x_3x_4^{2}x_5^{12}$, & $\mathcal Y_{21,\,527}=x_1^{3}x_2^{3}x_3^{12}x_4x_5^{2}$, & $\mathcal Y_{21,\,528}=x_1^{3}x_2^{12}x_3x_4^{2}x_5^{3}$, \\
$\mathcal Y_{21,\,529}=x_1^{3}x_2^{12}x_3x_4^{3}x_5^{2}$, & $\mathcal Y_{21,\,530}=x_1^{3}x_2^{12}x_3^{3}x_4x_5^{2}$, & $\mathcal Y_{21,\,531}=x_1^{3}x_2x_3^{4}x_4^{2}x_5^{11}$, & $\mathcal Y_{21,\,532}=x_1^{3}x_2x_3^{4}x_4^{11}x_5^{2}$, \\
$\mathcal Y_{21,\,533}=x_1^{3}x_2x_3^{2}x_4^{4}x_5^{11}$, & $\mathcal Y_{21,\,534}=x_1^{3}x_2^{4}x_3x_4^{2}x_5^{11}$, & $\mathcal Y_{21,\,535}=x_1^{3}x_2^{4}x_3x_4^{11}x_5^{2}$, & $\mathcal Y_{21,\,536}=x_1^{3}x_2^{4}x_3^{11}x_4x_5^{2}$, \\
$\mathcal Y_{21,\,537}=x_1^{3}x_2x_3^{2}x_4^{5}x_5^{10}$, & $\mathcal Y_{21,\,538}=x_1^{3}x_2x_3^{5}x_4^{10}x_5^{2}$, & $\mathcal Y_{21,\,539}=x_1^{3}x_2x_3^{5}x_4^{2}x_5^{10}$, & $\mathcal Y_{21,\,540}=x_1^{3}x_2^{5}x_3x_4^{10}x_5^{2}$, \\
$\mathcal Y_{21,\,541}=x_1^{3}x_2^{5}x_3x_4^{2}x_5^{10}$, & $\mathcal Y_{21,\,542}=x_1^{3}x_2^{5}x_3^{2}x_4x_5^{10}$, & $\mathcal Y_{21,\,543}=x_1^{3}x_2^{5}x_3^{10}x_4x_5^{2}$, & $\mathcal Y_{21,\,544}=x_1^{3}x_2x_3^{6}x_4^{9}x_5^{2}$, \\
$\mathcal Y_{21,\,545}=x_1^{3}x_2x_3^{2}x_4^{7}x_5^{8}$, & $\mathcal Y_{21,\,546}=x_1^{3}x_2x_3^{7}x_4^{8}x_5^{2}$, & $\mathcal Y_{21,\,547}=x_1^{3}x_2x_3^{7}x_4^{2}x_5^{8}$, & $\mathcal Y_{21,\,548}=x_1^{3}x_2^{7}x_3x_4^{8}x_5^{2}$, \\
$\mathcal Y_{21,\,549}=x_1^{3}x_2^{7}x_3x_4^{2}x_5^{8}$, & $\mathcal Y_{21,\,550}=x_1^{3}x_2^{7}x_3^{8}x_4x_5^{2}$, & $\mathcal Y_{21,\,551}=x_1^{3}x_2x_3^{4}x_4^{3}x_5^{10}$, & $\mathcal Y_{21,\,552}=x_1^{3}x_2x_3^{4}x_4^{10}x_5^{3}$, \\
$\mathcal Y_{21,\,553}=x_1^{3}x_2x_3^{3}x_4^{4}x_5^{10}$, & $\mathcal Y_{21,\,554}=x_1^{3}x_2^{3}x_3x_4^{4}x_5^{10}$, & $\mathcal Y_{21,\,555}=x_1^{3}x_2^{3}x_3^{4}x_4x_5^{10}$, & $\mathcal Y_{21,\,556}=x_1^{3}x_2^{4}x_3x_4^{3}x_5^{10}$, \\
$\mathcal Y_{21,\,557}=x_1^{3}x_2^{4}x_3x_4^{10}x_5^{3}$, & $\mathcal Y_{21,\,558}=x_1^{3}x_2^{4}x_3^{3}x_4x_5^{10}$, & $\mathcal Y_{21,\,559}=x_1^{3}x_2x_3^{6}x_4^{3}x_5^{8}$, & $\mathcal Y_{21,\,560}=x_1^{3}x_2x_3^{6}x_4^{8}x_5^{3}$, \\
$\mathcal Y_{21,\,561}=x_1^{3}x_2x_3^{3}x_4^{6}x_5^{8}$, & $\mathcal Y_{21,\,562}=x_1^{3}x_2^{3}x_3x_4^{6}x_5^{8}$, & $\mathcal Y_{21,\,563}=x_1^{3}x_2^{5}x_3^{2}x_4^{9}x_5^{2}$, & $\mathcal Y_{21,\,564}=x_1^{3}x_2^{5}x_3^{9}x_4^{2}x_5^{2}$, \\
$\mathcal Y_{21,\,565}=x_1^{3}x_2^{3}x_3^{4}x_4^{9}x_5^{2}$, & $\mathcal Y_{21,\,566}=x_1^{3}x_2^{4}x_3^{3}x_4^{9}x_5^{2}$, & $\mathcal Y_{21,\,567}=x_1^{3}x_2^{4}x_3^{9}x_4^{2}x_5^{3}$, & $\mathcal Y_{21,\,568}=x_1^{3}x_2^{4}x_3^{9}x_4^{3}x_5^{2}$, \\
$\mathcal Y_{21,\,569}=x_1^{3}x_2^{5}x_3^{8}x_4^{2}x_5^{3}$, & $\mathcal Y_{21,\,570}=x_1^{3}x_2^{5}x_3^{8}x_4^{3}x_5^{2}$, & $\mathcal Y_{21,\,571}=x_1^{3}x_2^{5}x_3^{2}x_4^{3}x_5^{8}$, & $\mathcal Y_{21,\,572}=x_1^{3}x_2^{5}x_3^{2}x_4^{8}x_5^{3}$, \\
$\mathcal Y_{21,\,573}=x_1^{3}x_2^{5}x_3^{3}x_4^{8}x_5^{2}$, & $\mathcal Y_{21,\,574}=x_1^{3}x_2^{5}x_3^{3}x_4^{2}x_5^{8}$, & $\mathcal Y_{21,\,575}=x_1^{3}x_2^{3}x_3^{5}x_4^{8}x_5^{2}$, & $\mathcal Y_{21,\,576}=x_1^{3}x_2^{3}x_3^{5}x_4^{2}x_5^{8}$, \\
$\mathcal Y_{21,\,577}=x_1^{7}x_2x_3^{10}x_4x_5^{2}$, & $\mathcal Y_{21,\,578}=x_1^{7}x_2x_3^{2}x_4x_5^{10}$, & $\mathcal Y_{21,\,579}=x_1^{7}x_2x_3x_4^{10}x_5^{2}$, & $\mathcal Y_{21,\,580}=x_1^{7}x_2x_3x_4^{2}x_5^{10}$, \\
$\mathcal Y_{21,\,581}=x_1^{7}x_2x_3^{2}x_4^{9}x_5^{2}$, & $\mathcal Y_{21,\,582}=x_1^{7}x_2x_3^{9}x_4^{2}x_5^{2}$, & $\mathcal Y_{21,\,583}=x_1^{7}x_2^{9}x_3x_4^{2}x_5^{2}$, & $\mathcal Y_{21,\,584}=x_1^{7}x_2^{9}x_3^{2}x_4x_5^{2}$, \\
$\mathcal Y_{21,\,585}=x_1^{7}x_2x_3^{8}x_4^{2}x_5^{3}$, & $\mathcal Y_{21,\,586}=x_1^{7}x_2x_3^{8}x_4^{3}x_5^{2}$, & $\mathcal Y_{21,\,587}=x_1^{7}x_2x_3^{2}x_4^{3}x_5^{8}$, & $\mathcal Y_{21,\,588}=x_1^{7}x_2x_3^{2}x_4^{8}x_5^{3}$, \\
$\mathcal Y_{21,\,589}=x_1^{7}x_2x_3^{3}x_4^{8}x_5^{2}$, & $\mathcal Y_{21,\,590}=x_1^{7}x_2x_3^{3}x_4^{2}x_5^{8}$, & $\mathcal Y_{21,\,591}=x_1^{7}x_2^{3}x_3x_4^{8}x_5^{2}$, & $\mathcal Y_{21,\,592}=x_1^{7}x_2^{3}x_3x_4^{2}x_5^{8}$, \\
$\mathcal Y_{21,\,593}=x_1^{7}x_2^{3}x_3^{8}x_4x_5^{2}$, & $\mathcal Y_{21,\,594}=x_1^{7}x_2^{8}x_3x_4^{2}x_5^{3}$, & $\mathcal Y_{21,\,595}=x_1^{7}x_2^{8}x_3x_4^{3}x_5^{2}$, & $\mathcal Y_{21,\,596}=x_1^{7}x_2^{8}x_3^{3}x_4x_5^{2}.$
\end{tabular}%
\end{center}

\newpage
$\mathscr B_5^+(3,3,3)$ is the set of $70$ admissible monomials: $\mathcal Y_{21,\, t},\ 597\leq t\leq 666$

\begin{center}
\begin{tabular}{llrr}
$\mathcal Y_{21,\,597}=x_1x_2^{2}x_3^{4}x_4^{7}x_5^{7}$, & $\mathcal Y_{21,\,598}=x_1x_2^{2}x_3^{7}x_4^{4}x_5^{7}$, & \multicolumn{1}{l}{$\mathcal Y_{21,\,599}=x_1x_2^{2}x_3^{7}x_4^{7}x_5^{4}$,} & \multicolumn{1}{l}{$\mathcal Y_{21,\,600}=x_1x_2^{7}x_3^{2}x_4^{4}x_5^{7}$,} \\
$\mathcal Y_{21,\,601}=x_1x_2^{7}x_3^{2}x_4^{7}x_5^{4}$, & $\mathcal Y_{21,\,602}=x_1x_2^{7}x_3^{7}x_4^{2}x_5^{4}$, & \multicolumn{1}{l}{$\mathcal Y_{21,\,603}=x_1^{7}x_2x_3^{2}x_4^{4}x_5^{7}$,} & \multicolumn{1}{l}{$\mathcal Y_{21,\,604}=x_1^{7}x_2x_3^{2}x_4^{7}x_5^{4}$,} \\
$\mathcal Y_{21,\,605}=x_1^{7}x_2x_3^{7}x_4^{2}x_5^{4}$, & $\mathcal Y_{21,\,606}=x_1^{7}x_2^{7}x_3x_4^{2}x_5^{4}$, & \multicolumn{1}{l}{$\mathcal Y_{21,\,607}=x_1x_2^{6}x_3x_4^{6}x_5^{7}$,} & \multicolumn{1}{l}{$\mathcal Y_{21,\,608}=x_1x_2^{6}x_3x_4^{7}x_5^{6}$,} \\
$\mathcal Y_{21,\,609}=x_1x_2^{6}x_3^{7}x_4x_5^{6}$, & $\mathcal Y_{21,\,610}=x_1x_2x_3^{6}x_4^{6}x_5^{7}$, & \multicolumn{1}{l}{$\mathcal Y_{21,\,611}=x_1x_2x_3^{6}x_4^{7}x_5^{6}$,} & \multicolumn{1}{l}{$\mathcal Y_{21,\,612}=x_1x_2^{7}x_3^{6}x_4x_5^{6}$,} \\
$\mathcal Y_{21,\,613}=x_1x_2x_3^{7}x_4^{6}x_5^{6}$, & $\mathcal Y_{21,\,614}=x_1x_2^{7}x_3x_4^{6}x_5^{6}$, & \multicolumn{1}{l}{$\mathcal Y_{21,\,615}=x_1^{7}x_2x_3^{6}x_4x_5^{6}$,} & \multicolumn{1}{l}{$\mathcal Y_{21,\,616}=x_1^{7}x_2x_3x_4^{6}x_5^{6}$,} \\
$\mathcal Y_{21,\,617}=x_1x_2^{2}x_3^{5}x_4^{6}x_5^{7}$, & $\mathcal Y_{21,\,618}=x_1x_2^{2}x_3^{5}x_4^{7}x_5^{6}$, & \multicolumn{1}{l}{$\mathcal Y_{21,\,619}=x_1x_2^{2}x_3^{7}x_4^{5}x_5^{6}$,} & \multicolumn{1}{l}{$\mathcal Y_{21,\,620}=x_1x_2^{7}x_3^{2}x_4^{5}x_5^{6}$,} \\
$\mathcal Y_{21,\,621}=x_1^{7}x_2x_3^{2}x_4^{5}x_5^{6}$, & $\mathcal Y_{21,\,622}=x_1x_2^{3}x_3^{4}x_4^{6}x_5^{7}$, & \multicolumn{1}{l}{$\mathcal Y_{21,\,623}=x_1x_2^{3}x_3^{4}x_4^{7}x_5^{6}$,} & \multicolumn{1}{l}{$\mathcal Y_{21,\,624}=x_1x_2^{3}x_3^{7}x_4^{4}x_5^{6}$,} \\
$\mathcal Y_{21,\,625}=x_1x_2^{7}x_3^{3}x_4^{4}x_5^{6}$, & $\mathcal Y_{21,\,626}=x_1^{3}x_2x_3^{4}x_4^{6}x_5^{7}$, & \multicolumn{1}{l}{$\mathcal Y_{21,\,627}=x_1^{3}x_2x_3^{4}x_4^{7}x_5^{6}$,} & \multicolumn{1}{l}{$\mathcal Y_{21,\,628}=x_1^{3}x_2x_3^{7}x_4^{4}x_5^{6}$,} \\
$\mathcal Y_{21,\,629}=x_1^{3}x_2^{7}x_3x_4^{4}x_5^{6}$, & $\mathcal Y_{21,\,630}=x_1^{7}x_2x_3^{3}x_4^{4}x_5^{6}$, & \multicolumn{1}{l}{$\mathcal Y_{21,\,631}=x_1^{7}x_2^{3}x_3x_4^{4}x_5^{6}$,} & \multicolumn{1}{l}{$\mathcal Y_{21,\,632}=x_1x_2^{3}x_3^{6}x_4^{4}x_5^{7}$,} \\
$\mathcal Y_{21,\,633}=x_1x_2^{3}x_3^{6}x_4^{7}x_5^{4}$, & $\mathcal Y_{21,\,634}=x_1x_2^{3}x_3^{7}x_4^{6}x_5^{4}$, & \multicolumn{1}{l}{$\mathcal Y_{21,\,635}=x_1x_2^{7}x_3^{3}x_4^{6}x_5^{4}$,} & \multicolumn{1}{l}{$\mathcal Y_{21,\,636}=x_1^{3}x_2x_3^{6}x_4^{4}x_5^{7}$,} \\
$\mathcal Y_{21,\,637}=x_1^{3}x_2x_3^{6}x_4^{7}x_5^{4}$, & $\mathcal Y_{21,\,638}=x_1^{3}x_2x_3^{7}x_4^{6}x_5^{4}$, & \multicolumn{1}{l}{$\mathcal Y_{21,\,639}=x_1^{3}x_2^{7}x_3x_4^{6}x_5^{4}$,} & \multicolumn{1}{l}{$\mathcal Y_{21,\,640}=x_1^{7}x_2x_3^{3}x_4^{6}x_5^{4}$,} \\
$\mathcal Y_{21,\,641}=x_1^{7}x_2^{3}x_3x_4^{6}x_5^{4}$, & $\mathcal Y_{21,\,642}=x_1^{3}x_2^{5}x_3^{2}x_4^{4}x_5^{7}$, & \multicolumn{1}{l}{$\mathcal Y_{21,\,643}=x_1^{3}x_2^{5}x_3^{2}x_4^{7}x_5^{4}$,} & \multicolumn{1}{l}{$\mathcal Y_{21,\,644}=x_1^{3}x_2^{5}x_3^{7}x_4^{2}x_5^{4}$,} \\
$\mathcal Y_{21,\,645}=x_1^{3}x_2^{7}x_3^{5}x_4^{2}x_5^{4}$, & $\mathcal Y_{21,\,646}=x_1^{7}x_2^{3}x_3^{5}x_4^{2}x_5^{4}$, & \multicolumn{1}{l}{$\mathcal Y_{21,\,647}=x_1^{3}x_2^{3}x_3^{4}x_4^{4}x_5^{7}$,} & \multicolumn{1}{l}{$\mathcal Y_{21,\,648}=x_1^{3}x_2^{3}x_3^{4}x_4^{7}x_5^{4}$,} \\
$\mathcal Y_{21,\,649}=x_1^{3}x_2^{3}x_3^{7}x_4^{4}x_5^{4}$, & $\mathcal Y_{21,\,650}=x_1^{3}x_2^{7}x_3^{3}x_4^{4}x_5^{4}$, & \multicolumn{1}{l}{$\mathcal Y_{21,\,651}=x_1^{7}x_2^{3}x_3^{3}x_4^{4}x_5^{4}$,} & \multicolumn{1}{l}{$\mathcal Y_{21,\,652}=x_1x_2^{6}x_3^{3}x_4^{5}x_5^{6}$,} \\
$\mathcal Y_{21,\,653}=x_1x_2^{3}x_3^{6}x_4^{5}x_5^{6}$, & $\mathcal Y_{21,\,654}=x_1x_2^{3}x_3^{5}x_4^{6}x_5^{6}$, & \multicolumn{1}{l}{$\mathcal Y_{21,\,655}=x_1^{3}x_2x_3^{6}x_4^{5}x_5^{6}$,} & \multicolumn{1}{l}{$\mathcal Y_{21,\,656}=x_1^{3}x_2x_3^{5}x_4^{6}x_5^{6}$,} \\
$\mathcal Y_{21,\,657}=x_1^{3}x_2^{5}x_3x_4^{6}x_5^{6}$, & $\mathcal Y_{21,\,658}=x_1^{3}x_2^{5}x_3^{6}x_4x_5^{6}$, & \multicolumn{1}{l}{$\mathcal Y_{21,\,659}=x_1^{3}x_2^{5}x_3^{2}x_4^{5}x_5^{6}$,} & \multicolumn{1}{l}{$\mathcal Y_{21,\,660}=x_1^{3}x_2^{5}x_3^{3}x_4^{4}x_5^{6}$,} \\
$\mathcal Y_{21,\,661}=x_1^{3}x_2^{3}x_3^{5}x_4^{4}x_5^{6}$, & $\mathcal Y_{21,\,662}=x_1^{3}x_2^{3}x_3^{4}x_4^{5}x_5^{6}$, & \multicolumn{1}{l}{$\mathcal Y_{21,\,663}=x_1^{3}x_2^{3}x_3^{5}x_4^{6}x_5^{4}$,} & \multicolumn{1}{l}{$\mathcal Y_{21,\,664}=x_1^{3}x_2^{4}x_3^{3}x_4^{5}x_5^{6}$,} \\
$\mathcal Y_{21,\,665}=x_1^{3}x_2^{5}x_3^{3}x_4^{6}x_5^{4}$, & $\mathcal Y_{21,\,666}=x_1^{3}x_2^{5}x_3^{6}x_4^{3}x_5^{4}$. &       &  
\end{tabular}%
\end{center}

\subsection{$\mathcal A_2$-generators for $\mathscr P_5^0$ in degree $22$}\label{s5.4}

According to the results in Sect.\ref{s3.2.2}, we have 
$$ \mathscr B_5^0(22) = \overline{\Phi}^0(\mathscr B_4(\omega_{(1)}) \bigcup \overline{\Phi}^0(\mathscr B_4(\omega_{(4)}) \bigcup \overline{\Phi}^0(\mathscr B_4(\omega_{(5)}) = \{\mathcal Y_{22,\, t}:\ 1\leq t\leq 460\},$$
where $\omega_{(1)} = (2,2,2,1),\ \omega_{(4)} = (4,3,1,1),\ \omega_{(5)} = (4,3,3),$ and  the monomials $\mathcal Y_{22,\, t}:\ 1\leq t\leq 460,$ are determined as follows:

\begin{center}
\begin{tabular}{llll}
$\mathcal Y_{22,\,1}=x_4^{7}x_5^{15}$, & $\mathcal Y_{22,\,2}=x_4^{15}x_5^{7}$, & $\mathcal Y_{22,\,3}=x_3^{7}x_5^{15}$, & $\mathcal Y_{22,\,4}=x_3^{7}x_4^{15}$, \\
$\mathcal Y_{22,\,5}=x_3^{15}x_5^{7}$, & $\mathcal Y_{22,\,6}=x_3^{15}x_4^{7}$, & $\mathcal Y_{22,\,7}=x_2^{7}x_5^{15}$, & $\mathcal Y_{22,\,8}=x_2^{7}x_4^{15}$, \\
$\mathcal Y_{22,\,9}=x_2^{7}x_3^{15}$, & $\mathcal Y_{22,\,10}=x_2^{15}x_5^{7}$, & $\mathcal Y_{22,\,11}=x_2^{15}x_4^{7}$, & $\mathcal Y_{22,\,12}=x_2^{15}x_3^{7}$, \\
$\mathcal Y_{22,\,13}=x_1^{7}x_5^{15}$, & $\mathcal Y_{22,\,14}=x_1^{7}x_4^{15}$, & $\mathcal Y_{22,\,15}=x_1^{7}x_3^{15}$, & $\mathcal Y_{22,\,16}=x_1^{7}x_2^{15}$, \\
$\mathcal Y_{22,\,17}=x_1^{15}x_5^{7}$, & $\mathcal Y_{22,\,18}=x_1^{15}x_4^{7}$, & $\mathcal Y_{22,\,19}=x_1^{15}x_3^{7}$, & $\mathcal Y_{22,\,20}=x_1^{15}x_2^{7}$, \\
$\mathcal Y_{22,\,21}=x_3x_4^{6}x_5^{15}$, & $\mathcal Y_{22,\,22}=x_3x_4^{15}x_5^{6}$, & $\mathcal Y_{22,\,23}=x_3^{15}x_4x_5^{6}$, & $\mathcal Y_{22,\,24}=x_2x_4^{6}x_5^{15}$, \\
$\mathcal Y_{22,\,25}=x_2x_4^{15}x_5^{6}$, & $\mathcal Y_{22,\,26}=x_2x_3^{6}x_5^{15}$, & $\mathcal Y_{22,\,27}=x_2x_3^{6}x_4^{15}$, & $\mathcal Y_{22,\,28}=x_2x_3^{15}x_5^{6}$, \\
$\mathcal Y_{22,\,29}=x_2x_3^{15}x_4^{6}$, & $\mathcal Y_{22,\,30}=x_2^{15}x_4x_5^{6}$, & $\mathcal Y_{22,\,31}=x_2^{15}x_3x_5^{6}$, & $\mathcal Y_{22,\,32}=x_2^{15}x_3x_4^{6}$, \\
$\mathcal Y_{22,\,33}=x_1x_4^{6}x_5^{15}$, & $\mathcal Y_{22,\,34}=x_1x_4^{15}x_5^{6}$, & $\mathcal Y_{22,\,35}=x_1x_3^{6}x_5^{15}$, & $\mathcal Y_{22,\,36}=x_1x_3^{6}x_4^{15}$, \\
$\mathcal Y_{22,\,37}=x_1x_3^{15}x_5^{6}$, & $\mathcal Y_{22,\,38}=x_1x_3^{15}x_4^{6}$, & $\mathcal Y_{22,\,39}=x_1x_2^{6}x_5^{15}$, & $\mathcal Y_{22,\,40}=x_1x_2^{6}x_4^{15}$, \\
$\mathcal Y_{22,\,41}=x_1x_2^{6}x_3^{15}$, & $\mathcal Y_{22,\,42}=x_1x_2^{15}x_5^{6}$, & $\mathcal Y_{22,\,43}=x_1x_2^{15}x_4^{6}$, & $\mathcal Y_{22,\,44}=x_1x_2^{15}x_3^{6}$, \\
$\mathcal Y_{22,\,45}=x_1^{15}x_4x_5^{6}$, & $\mathcal Y_{22,\,46}=x_1^{15}x_3x_5^{6}$, & $\mathcal Y_{22,\,47}=x_1^{15}x_3x_4^{6}$, & $\mathcal Y_{22,\,48}=x_1^{15}x_2x_5^{6}$, \\
$\mathcal Y_{22,\,49}=x_1^{15}x_2x_4^{6}$, & $\mathcal Y_{22,\,50}=x_1^{15}x_2x_3^{6}$, & $\mathcal Y_{22,\,51}=x_3x_4^{7}x_5^{14}$, & $\mathcal Y_{22,\,52}=x_3x_4^{14}x_5^{7}$, \\
$\mathcal Y_{22,\,53}=x_3^{7}x_4x_5^{14}$, & $\mathcal Y_{22,\,54}=x_2x_4^{7}x_5^{14}$, & $\mathcal Y_{22,\,55}=x_2x_4^{14}x_5^{7}$, & $\mathcal Y_{22,\,56}=x_2x_3^{7}x_5^{14}$, \\
$\mathcal Y_{22,\,57}=x_2x_3^{7}x_4^{14}$, & $\mathcal Y_{22,\,58}=x_2x_3^{14}x_5^{7}$, & $\mathcal Y_{22,\,59}=x_2x_3^{14}x_4^{7}$, & $\mathcal Y_{22,\,60}=x_2^{7}x_4x_5^{14}$, \\
$\mathcal Y_{22,\,61}=x_2^{7}x_3x_5^{14}$, & $\mathcal Y_{22,\,62}=x_2^{7}x_3x_4^{14}$, & $\mathcal Y_{22,\,63}=x_1x_4^{7}x_5^{14}$, & $\mathcal Y_{22,\,64}=x_1x_4^{14}x_5^{7}$, \\
$\mathcal Y_{22,\,65}=x_1x_3^{7}x_5^{14}$, & $\mathcal Y_{22,\,66}=x_1x_3^{7}x_4^{14}$, & $\mathcal Y_{22,\,67}=x_1x_3^{14}x_5^{7}$, & $\mathcal Y_{22,\,68}=x_1x_3^{14}x_4^{7}$, \\
$\mathcal Y_{22,\,69}=x_1x_2^{7}x_5^{14}$, & $\mathcal Y_{22,\,70}=x_1x_2^{7}x_4^{14}$, & $\mathcal Y_{22,\,71}=x_1x_2^{7}x_3^{14}$, & $\mathcal Y_{22,\,72}=x_1x_2^{14}x_5^{7}$, \\
$\mathcal Y_{22,\,73}=x_1x_2^{14}x_4^{7}$, & $\mathcal Y_{22,\,74}=x_1x_2^{14}x_3^{7}$, & $\mathcal Y_{22,\,75}=x_1^{7}x_4x_5^{14}$, & $\mathcal Y_{22,\,76}=x_1^{7}x_3x_5^{14}$, \\
$\mathcal Y_{22,\,77}=x_1^{7}x_3x_4^{14}$, & $\mathcal Y_{22,\,78}=x_1^{7}x_2x_5^{14}$, & $\mathcal Y_{22,\,79}=x_1^{7}x_2x_4^{14}$, & $\mathcal Y_{22,\,80}=x_1^{7}x_2x_3^{14}$, \\
\end{tabular}%
\end{center}

\newpage
\begin{center}
\begin{tabular}{llll}
$\mathcal Y_{22,\,81}=x_3^{3}x_4^{5}x_5^{14}$, & $\mathcal Y_{22,\,82}=x_2^{3}x_4^{5}x_5^{14}$, & $\mathcal Y_{22,\,83}=x_2^{3}x_3^{5}x_5^{14}$, & $\mathcal Y_{22,\,84}=x_2^{3}x_3^{5}x_4^{14}$, \\
$\mathcal Y_{22,\,85}=x_1^{3}x_4^{5}x_5^{14}$, & $\mathcal Y_{22,\,86}=x_1^{3}x_3^{5}x_5^{14}$, & $\mathcal Y_{22,\,87}=x_1^{3}x_3^{5}x_4^{14}$, & $\mathcal Y_{22,\,88}=x_1^{3}x_2^{5}x_5^{14}$, \\
$\mathcal Y_{22,\,89}=x_1^{3}x_2^{5}x_4^{14}$, & $\mathcal Y_{22,\,90}=x_1^{3}x_2^{5}x_3^{14}$, & $\mathcal Y_{22,\,91}=x_3^{3}x_4^{13}x_5^{6}$, & $\mathcal Y_{22,\,92}=x_2^{3}x_4^{13}x_5^{6}$, \\
$\mathcal Y_{22,\,93}=x_2^{3}x_3^{13}x_5^{6}$, & $\mathcal Y_{22,\,94}=x_2^{3}x_3^{13}x_4^{6}$, & $\mathcal Y_{22,\,95}=x_1^{3}x_4^{13}x_5^{6}$, & $\mathcal Y_{22,\,96}=x_1^{3}x_3^{13}x_5^{6}$, \\
$\mathcal Y_{22,\,97}=x_1^{3}x_3^{13}x_4^{6}$, & $\mathcal Y_{22,\,98}=x_1^{3}x_2^{13}x_5^{6}$, & $\mathcal Y_{22,\,99}=x_1^{3}x_2^{13}x_4^{6}$, & $\mathcal Y_{22,\,100}=x_1^{3}x_2^{13}x_3^{6}$, \\
$\mathcal Y_{22,\,101}=x_2x_3x_4^{6}x_5^{14}$, & $\mathcal Y_{22,\,102}=x_2x_3x_4^{14}x_5^{6}$, & $\mathcal Y_{22,\,103}=x_2x_3^{6}x_4x_5^{14}$, & $\mathcal Y_{22,\,104}=x_2x_3^{14}x_4x_5^{6}$, \\
$\mathcal Y_{22,\,105}=x_1x_3x_4^{6}x_5^{14}$, & $\mathcal Y_{22,\,106}=x_1x_3x_4^{14}x_5^{6}$, & $\mathcal Y_{22,\,107}=x_1x_3^{6}x_4x_5^{14}$, & $\mathcal Y_{22,\,108}=x_1x_3^{14}x_4x_5^{6}$, \\
$\mathcal Y_{22,\,109}=x_1x_2x_4^{6}x_5^{14}$, & $\mathcal Y_{22,\,110}=x_1x_2x_4^{14}x_5^{6}$, & $\mathcal Y_{22,\,111}=x_1x_2x_3^{6}x_5^{14}$, & $\mathcal Y_{22,\,112}=x_1x_2x_3^{6}x_4^{14}$, \\
$\mathcal Y_{22,\,113}=x_1x_2x_3^{14}x_5^{6}$, & $\mathcal Y_{22,\,114}=x_1x_2x_3^{14}x_4^{6}$, & $\mathcal Y_{22,\,115}=x_1x_2^{6}x_4x_5^{14}$, & $\mathcal Y_{22,\,116}=x_1x_2^{6}x_3x_5^{14}$, \\
$\mathcal Y_{22,\,117}=x_1x_2^{6}x_3x_4^{14}$, & $\mathcal Y_{22,\,118}=x_1x_2^{14}x_4x_5^{6}$, & $\mathcal Y_{22,\,119}=x_1x_2^{14}x_3x_5^{6}$, & $\mathcal Y_{22,\,120}=x_1x_2^{14}x_3x_4^{6}$, \\
$\mathcal Y_{22,\,121}=x_2x_3^{2}x_4^{4}x_5^{15}$, & $\mathcal Y_{22,\,122}=x_2x_3^{2}x_4^{15}x_5^{4}$, & $\mathcal Y_{22,\,123}=x_2x_3^{15}x_4^{2}x_5^{4}$, & $\mathcal Y_{22,\,124}=x_2^{15}x_3x_4^{2}x_5^{4}$, \\
$\mathcal Y_{22,\,125}=x_1x_3^{2}x_4^{4}x_5^{15}$, & $\mathcal Y_{22,\,126}=x_1x_3^{2}x_4^{15}x_5^{4}$, & $\mathcal Y_{22,\,127}=x_1x_3^{15}x_4^{2}x_5^{4}$, & $\mathcal Y_{22,\,128}=x_1x_2^{2}x_4^{4}x_5^{15}$, \\
$\mathcal Y_{22,\,129}=x_1x_2^{2}x_4^{15}x_5^{4}$, & $\mathcal Y_{22,\,130}=x_1x_2^{2}x_3^{4}x_5^{15}$, & $\mathcal Y_{22,\,131}=x_1x_2^{2}x_3^{4}x_4^{15}$, & $\mathcal Y_{22,\,132}=x_1x_2^{2}x_3^{15}x_5^{4}$, \\
$\mathcal Y_{22,\,133}=x_1x_2^{2}x_3^{15}x_4^{4}$, & $\mathcal Y_{22,\,134}=x_1x_2^{15}x_4^{2}x_5^{4}$, & $\mathcal Y_{22,\,135}=x_1x_2^{15}x_3^{2}x_5^{4}$, & $\mathcal Y_{22,\,136}=x_1x_2^{15}x_3^{2}x_4^{4}$, \\
$\mathcal Y_{22,\,137}=x_1^{15}x_3x_4^{2}x_5^{4}$, & $\mathcal Y_{22,\,138}=x_1^{15}x_2x_4^{2}x_5^{4}$, & $\mathcal Y_{22,\,139}=x_1^{15}x_2x_3^{2}x_5^{4}$, & $\mathcal Y_{22,\,140}=x_1^{15}x_2x_3^{2}x_4^{4}$, \\
$\mathcal Y_{22,\,141}=x_2x_3^{2}x_4^{5}x_5^{14}$, & $\mathcal Y_{22,\,142}=x_1x_3^{2}x_4^{5}x_5^{14}$, & $\mathcal Y_{22,\,143}=x_1x_2^{2}x_4^{5}x_5^{14}$, & $\mathcal Y_{22,\,144}=x_1x_2^{2}x_3^{5}x_5^{14}$, \\
$\mathcal Y_{22,\,145}=x_1x_2^{2}x_3^{5}x_4^{14}$, & $\mathcal Y_{22,\,146}=x_2x_3^{2}x_4^{13}x_5^{6}$, & $\mathcal Y_{22,\,147}=x_1x_3^{2}x_4^{13}x_5^{6}$, & $\mathcal Y_{22,\,148}=x_1x_2^{2}x_4^{13}x_5^{6}$, \\
$\mathcal Y_{22,\,149}=x_1x_2^{2}x_3^{13}x_5^{6}$, & $\mathcal Y_{22,\,150}=x_1x_2^{2}x_3^{13}x_4^{6}$, & $\mathcal Y_{22,\,151}=x_2x_3^{2}x_4^{7}x_5^{12}$, & $\mathcal Y_{22,\,152}=x_2x_3^{2}x_4^{12}x_5^{7}$, \\
$\mathcal Y_{22,\,153}=x_2x_3^{7}x_4^{2}x_5^{12}$, & $\mathcal Y_{22,\,154}=x_2^{7}x_3x_4^{2}x_5^{12}$, & $\mathcal Y_{22,\,155}=x_1x_3^{2}x_4^{7}x_5^{12}$, & $\mathcal Y_{22,\,156}=x_1x_3^{2}x_4^{12}x_5^{7}$, \\
$\mathcal Y_{22,\,157}=x_1x_3^{7}x_4^{2}x_5^{12}$, & $\mathcal Y_{22,\,158}=x_1x_2^{2}x_4^{7}x_5^{12}$, & $\mathcal Y_{22,\,159}=x_1x_2^{2}x_4^{12}x_5^{7}$, & $\mathcal Y_{22,\,160}=x_1x_2^{2}x_3^{7}x_5^{12}$, \\
$\mathcal Y_{22,\,161}=x_1x_2^{2}x_3^{7}x_4^{12}$, & $\mathcal Y_{22,\,162}=x_1x_2^{2}x_3^{12}x_5^{7}$, & $\mathcal Y_{22,\,163}=x_1x_2^{2}x_3^{12}x_4^{7}$, & $\mathcal Y_{22,\,164}=x_1x_2^{7}x_4^{2}x_5^{12}$, \\
$\mathcal Y_{22,\,165}=x_1x_2^{7}x_3^{2}x_5^{12}$, & $\mathcal Y_{22,\,166}=x_1x_2^{7}x_3^{2}x_4^{12}$, & $\mathcal Y_{22,\,167}=x_1^{7}x_3x_4^{2}x_5^{12}$, & $\mathcal Y_{22,\,168}=x_1^{7}x_2x_4^{2}x_5^{12}$, \\
$\mathcal Y_{22,\,169}=x_1^{7}x_2x_3^{2}x_5^{12}$, & $\mathcal Y_{22,\,170}=x_1^{7}x_2x_3^{2}x_4^{12}$, & $\mathcal Y_{22,\,171}=x_2x_3^{3}x_4^{3}x_5^{15}$, & $\mathcal Y_{22,\,172}=x_2x_3^{3}x_4^{15}x_5^{3}$, \\
$\mathcal Y_{22,\,173}=x_2x_3^{15}x_4^{3}x_5^{3}$, & $\mathcal Y_{22,\,174}=x_2^{3}x_3x_4^{3}x_5^{15}$, & $\mathcal Y_{22,\,175}=x_2^{3}x_3x_4^{15}x_5^{3}$, & $\mathcal Y_{22,\,176}=x_2^{3}x_3^{3}x_4x_5^{15}$, \\
$\mathcal Y_{22,\,177}=x_2^{3}x_3^{3}x_4^{15}x_5$, & $\mathcal Y_{22,\,178}=x_2^{3}x_3^{15}x_4x_5^{3}$, & $\mathcal Y_{22,\,179}=x_2^{3}x_3^{15}x_4^{3}x_5$, & $\mathcal Y_{22,\,180}=x_2^{15}x_3x_4^{3}x_5^{3}$, \\
$\mathcal Y_{22,\,181}=x_2^{15}x_3^{3}x_4x_5^{3}$, & $\mathcal Y_{22,\,182}=x_2^{15}x_3^{3}x_4^{3}x_5$, & $\mathcal Y_{22,\,183}=x_1x_3^{3}x_4^{3}x_5^{15}$, & $\mathcal Y_{22,\,184}=x_1x_3^{3}x_4^{15}x_5^{3}$, \\
$\mathcal Y_{22,\,185}=x_1x_3^{15}x_4^{3}x_5^{3}$, & $\mathcal Y_{22,\,186}=x_1x_2^{3}x_4^{3}x_5^{15}$, & $\mathcal Y_{22,\,187}=x_1x_2^{3}x_4^{15}x_5^{3}$, & $\mathcal Y_{22,\,188}=x_1x_2^{3}x_3^{3}x_5^{15}$, \\
$\mathcal Y_{22,\,189}=x_1x_2^{3}x_3^{3}x_4^{15}$, & $\mathcal Y_{22,\,190}=x_1x_2^{3}x_3^{15}x_5^{3}$, & $\mathcal Y_{22,\,191}=x_1x_2^{3}x_3^{15}x_4^{3}$, & $\mathcal Y_{22,\,192}=x_1x_2^{15}x_4^{3}x_5^{3}$, \\
$\mathcal Y_{22,\,193}=x_1x_2^{15}x_3^{3}x_5^{3}$, & $\mathcal Y_{22,\,194}=x_1x_2^{15}x_3^{3}x_4^{3}$, & $\mathcal Y_{22,\,195}=x_1^{3}x_3x_4^{3}x_5^{15}$, & $\mathcal Y_{22,\,196}=x_1^{3}x_3x_4^{15}x_5^{3}$, \\
$\mathcal Y_{22,\,197}=x_1^{3}x_3^{3}x_4x_5^{15}$, & $\mathcal Y_{22,\,198}=x_1^{3}x_3^{3}x_4^{15}x_5$, & $\mathcal Y_{22,\,199}=x_1^{3}x_3^{15}x_4x_5^{3}$, & $\mathcal Y_{22,\,200}=x_1^{3}x_3^{15}x_4^{3}x_5$, \\
$\mathcal Y_{22,\,201}=x_1^{3}x_2x_4^{3}x_5^{15}$, & $\mathcal Y_{22,\,202}=x_1^{3}x_2x_4^{15}x_5^{3}$, & $\mathcal Y_{22,\,203}=x_1^{3}x_2x_3^{3}x_5^{15}$, & $\mathcal Y_{22,\,204}=x_1^{3}x_2x_3^{3}x_4^{15}$, \\
$\mathcal Y_{22,\,205}=x_1^{3}x_2x_3^{15}x_5^{3}$, & $\mathcal Y_{22,\,206}=x_1^{3}x_2x_3^{15}x_4^{3}$, & $\mathcal Y_{22,\,207}=x_1^{3}x_2^{3}x_4x_5^{15}$, & $\mathcal Y_{22,\,208}=x_1^{3}x_2^{3}x_4^{15}x_5$, \\
$\mathcal Y_{22,\,209}=x_1^{3}x_2^{3}x_3x_5^{15}$, & $\mathcal Y_{22,\,210}=x_1^{3}x_2^{3}x_3x_4^{15}$, & $\mathcal Y_{22,\,211}=x_1^{3}x_2^{3}x_3^{15}x_5$, & $\mathcal Y_{22,\,212}=x_1^{3}x_2^{3}x_3^{15}x_4$, \\
$\mathcal Y_{22,\,213}=x_1^{3}x_2^{15}x_4x_5^{3}$, & $\mathcal Y_{22,\,214}=x_1^{3}x_2^{15}x_4^{3}x_5$, & $\mathcal Y_{22,\,215}=x_1^{3}x_2^{15}x_3x_5^{3}$, & $\mathcal Y_{22,\,216}=x_1^{3}x_2^{15}x_3x_4^{3}$, \\
$\mathcal Y_{22,\,217}=x_1^{3}x_2^{15}x_3^{3}x_5$, & $\mathcal Y_{22,\,218}=x_1^{3}x_2^{15}x_3^{3}x_4$, & $\mathcal Y_{22,\,219}=x_1^{15}x_3x_4^{3}x_5^{3}$, & $\mathcal Y_{22,\,220}=x_1^{15}x_3^{3}x_4x_5^{3}$, \\
$\mathcal Y_{22,\,221}=x_1^{15}x_3^{3}x_4^{3}x_5$, & $\mathcal Y_{22,\,222}=x_1^{15}x_2x_4^{3}x_5^{3}$, & $\mathcal Y_{22,\,223}=x_1^{15}x_2x_3^{3}x_5^{3}$, & $\mathcal Y_{22,\,224}=x_1^{15}x_2x_3^{3}x_4^{3}$, \\
$\mathcal Y_{22,\,225}=x_1^{15}x_2^{3}x_4x_5^{3}$, & $\mathcal Y_{22,\,226}=x_1^{15}x_2^{3}x_4^{3}x_5$, & $\mathcal Y_{22,\,227}=x_1^{15}x_2^{3}x_3x_5^{3}$, & $\mathcal Y_{22,\,228}=x_1^{15}x_2^{3}x_3x_4^{3}$, \\
$\mathcal Y_{22,\,229}=x_1^{15}x_2^{3}x_3^{3}x_5$, & $\mathcal Y_{22,\,230}=x_1^{15}x_2^{3}x_3^{3}x_4$, & $\mathcal Y_{22,\,231}=x_2x_3^{3}x_4^{4}x_5^{14}$, & $\mathcal Y_{22,\,232}=x_2x_3^{3}x_4^{14}x_5^{4}$, \\
$\mathcal Y_{22,\,233}=x_2^{3}x_3x_4^{4}x_5^{14}$, & $\mathcal Y_{22,\,234}=x_2^{3}x_3x_4^{14}x_5^{4}$, & $\mathcal Y_{22,\,235}=x_1x_3^{3}x_4^{4}x_5^{14}$, & $\mathcal Y_{22,\,236}=x_1x_3^{3}x_4^{14}x_5^{4}$, \\
$\mathcal Y_{22,\,237}=x_1x_2^{3}x_4^{4}x_5^{14}$, & $\mathcal Y_{22,\,238}=x_1x_2^{3}x_4^{14}x_5^{4}$, & $\mathcal Y_{22,\,239}=x_1x_2^{3}x_3^{4}x_5^{14}$, & $\mathcal Y_{22,\,240}=x_1x_2^{3}x_3^{4}x_4^{14}$, \\
$\mathcal Y_{22,\,241}=x_1x_2^{3}x_3^{14}x_5^{4}$, & $\mathcal Y_{22,\,242}=x_1x_2^{3}x_3^{14}x_4^{4}$, & $\mathcal Y_{22,\,243}=x_1^{3}x_3x_4^{4}x_5^{14}$, & $\mathcal Y_{22,\,244}=x_1^{3}x_3x_4^{14}x_5^{4}$, \\
$\mathcal Y_{22,\,245}=x_1^{3}x_2x_4^{4}x_5^{14}$, & $\mathcal Y_{22,\,246}=x_1^{3}x_2x_4^{14}x_5^{4}$, & $\mathcal Y_{22,\,247}=x_1^{3}x_2x_3^{4}x_5^{14}$, & $\mathcal Y_{22,\,248}=x_1^{3}x_2x_3^{4}x_4^{14}$, \\
$\mathcal Y_{22,\,249}=x_1^{3}x_2x_3^{14}x_5^{4}$, & $\mathcal Y_{22,\,250}=x_1^{3}x_2x_3^{14}x_4^{4}$, & $\mathcal Y_{22,\,251}=x_2x_3^{3}x_4^{6}x_5^{12}$, & $\mathcal Y_{22,\,252}=x_2x_3^{3}x_4^{12}x_5^{6}$, \\
$\mathcal Y_{22,\,253}=x_2^{3}x_3x_4^{6}x_5^{12}$, & $\mathcal Y_{22,\,254}=x_2^{3}x_3x_4^{12}x_5^{6}$, & $\mathcal Y_{22,\,255}=x_1x_3^{3}x_4^{6}x_5^{12}$, & $\mathcal Y_{22,\,256}=x_1x_3^{3}x_4^{12}x_5^{6}$, \\
$\mathcal Y_{22,\,257}=x_1x_2^{3}x_4^{6}x_5^{12}$, & $\mathcal Y_{22,\,258}=x_1x_2^{3}x_4^{12}x_5^{6}$, & $\mathcal Y_{22,\,259}=x_1x_2^{3}x_3^{6}x_5^{12}$, & $\mathcal Y_{22,\,260}=x_1x_2^{3}x_3^{6}x_4^{12}$, \\
$\mathcal Y_{22,\,261}=x_1x_2^{3}x_3^{12}x_5^{6}$, & $\mathcal Y_{22,\,262}=x_1x_2^{3}x_3^{12}x_4^{6}$, & $\mathcal Y_{22,\,263}=x_1^{3}x_3x_4^{6}x_5^{12}$, & $\mathcal Y_{22,\,264}=x_1^{3}x_3x_4^{12}x_5^{6}$, \\
$\mathcal Y_{22,\,265}=x_1^{3}x_2x_4^{6}x_5^{12}$, & $\mathcal Y_{22,\,266}=x_1^{3}x_2x_4^{12}x_5^{6}$, & $\mathcal Y_{22,\,267}=x_1^{3}x_2x_3^{6}x_5^{12}$, & $\mathcal Y_{22,\,268}=x_1^{3}x_2x_3^{6}x_4^{12}$, \\
$\mathcal Y_{22,\,269}=x_1^{3}x_2x_3^{12}x_5^{6}$, & $\mathcal Y_{22,\,270}=x_1^{3}x_2x_3^{12}x_4^{6}$, & $\mathcal Y_{22,\,271}=x_2x_3^{3}x_4^{7}x_5^{11}$, & $\mathcal Y_{22,\,272}=x_2x_3^{7}x_4^{3}x_5^{11}$, \\
$\mathcal Y_{22,\,273}=x_2x_3^{7}x_4^{11}x_5^{3}$, & $\mathcal Y_{22,\,274}=x_2^{3}x_3x_4^{7}x_5^{11}$, & $\mathcal Y_{22,\,275}=x_2^{3}x_3^{7}x_4x_5^{11}$, & $\mathcal Y_{22,\,276}=x_2^{3}x_3^{7}x_4^{11}x_5$, \\
$\mathcal Y_{22,\,277}=x_2^{7}x_3x_4^{3}x_5^{11}$, & $\mathcal Y_{22,\,278}=x_2^{7}x_3x_4^{11}x_5^{3}$, & $\mathcal Y_{22,\,279}=x_2^{7}x_3^{3}x_4x_5^{11}$, & $\mathcal Y_{22,\,280}=x_2^{7}x_3^{3}x_4^{11}x_5$, \\
$\mathcal Y_{22,\,281}=x_2^{7}x_3^{11}x_4x_5^{3}$, & $\mathcal Y_{22,\,282}=x_2^{7}x_3^{11}x_4^{3}x_5$, & $\mathcal Y_{22,\,283}=x_1x_3^{3}x_4^{7}x_5^{11}$, & $\mathcal Y_{22,\,284}=x_1x_3^{7}x_4^{3}x_5^{11}$, \\
\end{tabular}%
\end{center}

\newpage
\begin{center}
\begin{tabular}{llll}
$\mathcal Y_{22,\,285}=x_1x_3^{7}x_4^{11}x_5^{3}$, & $\mathcal Y_{22,\,286}=x_1x_2^{3}x_4^{7}x_5^{11}$, & $\mathcal Y_{22,\,287}=x_1x_2^{3}x_3^{7}x_5^{11}$, & $\mathcal Y_{22,\,288}=x_1x_2^{3}x_3^{7}x_4^{11}$, \\
$\mathcal Y_{22,\,289}=x_1x_2^{7}x_4^{3}x_5^{11}$, & $\mathcal Y_{22,\,290}=x_1x_2^{7}x_4^{11}x_5^{3}$, & $\mathcal Y_{22,\,291}=x_1x_2^{7}x_3^{3}x_5^{11}$, & $\mathcal Y_{22,\,292}=x_1x_2^{7}x_3^{3}x_4^{11}$, \\
$\mathcal Y_{22,\,293}=x_1x_2^{7}x_3^{11}x_5^{3}$, & $\mathcal Y_{22,\,294}=x_1x_2^{7}x_3^{11}x_4^{3}$, & $\mathcal Y_{22,\,295}=x_1^{3}x_3x_4^{7}x_5^{11}$, & $\mathcal Y_{22,\,296}=x_1^{3}x_3^{7}x_4x_5^{11}$, \\
$\mathcal Y_{22,\,297}=x_1^{3}x_3^{7}x_4^{11}x_5$, & $\mathcal Y_{22,\,298}=x_1^{3}x_2x_4^{7}x_5^{11}$, & $\mathcal Y_{22,\,299}=x_1^{3}x_2x_3^{7}x_5^{11}$, & $\mathcal Y_{22,\,300}=x_1^{3}x_2x_3^{7}x_4^{11}$, \\
$\mathcal Y_{22,\,301}=x_1^{3}x_2^{7}x_4x_5^{11}$, & $\mathcal Y_{22,\,302}=x_1^{3}x_2^{7}x_4^{11}x_5$, & $\mathcal Y_{22,\,303}=x_1^{3}x_2^{7}x_3x_5^{11}$, & $\mathcal Y_{22,\,304}=x_1^{3}x_2^{7}x_3x_4^{11}$, \\
$\mathcal Y_{22,\,305}=x_1^{3}x_2^{7}x_3^{11}x_5$, & $\mathcal Y_{22,\,306}=x_1^{3}x_2^{7}x_3^{11}x_4$, & $\mathcal Y_{22,\,307}=x_1^{7}x_3x_4^{3}x_5^{11}$, & $\mathcal Y_{22,\,308}=x_1^{7}x_3x_4^{11}x_5^{3}$, \\
$\mathcal Y_{22,\,309}=x_1^{7}x_3^{3}x_4x_5^{11}$, & $\mathcal Y_{22,\,310}=x_1^{7}x_3^{3}x_4^{11}x_5$, & $\mathcal Y_{22,\,311}=x_1^{7}x_3^{11}x_4x_5^{3}$, & $\mathcal Y_{22,\,312}=x_1^{7}x_3^{11}x_4^{3}x_5$, \\
$\mathcal Y_{22,\,313}=x_1^{7}x_2x_4^{3}x_5^{11}$, & $\mathcal Y_{22,\,314}=x_1^{7}x_2x_4^{11}x_5^{3}$, & $\mathcal Y_{22,\,315}=x_1^{7}x_2x_3^{3}x_5^{11}$, & $\mathcal Y_{22,\,316}=x_1^{7}x_2x_3^{3}x_4^{11}$, \\
$\mathcal Y_{22,\,317}=x_1^{7}x_2x_3^{11}x_5^{3}$, & $\mathcal Y_{22,\,318}=x_1^{7}x_2x_3^{11}x_4^{3}$, & $\mathcal Y_{22,\,319}=x_1^{7}x_2^{3}x_4x_5^{11}$, & $\mathcal Y_{22,\,320}=x_1^{7}x_2^{3}x_4^{11}x_5$, \\
$\mathcal Y_{22,\,321}=x_1^{7}x_2^{3}x_3x_5^{11}$, & $\mathcal Y_{22,\,322}=x_1^{7}x_2^{3}x_3x_4^{11}$, & $\mathcal Y_{22,\,323}=x_1^{7}x_2^{3}x_3^{11}x_5$, & $\mathcal Y_{22,\,324}=x_1^{7}x_2^{3}x_3^{11}x_4$, \\
$\mathcal Y_{22,\,325}=x_1^{7}x_2^{11}x_4x_5^{3}$, & $\mathcal Y_{22,\,326}=x_1^{7}x_2^{11}x_4^{3}x_5$, & $\mathcal Y_{22,\,327}=x_1^{7}x_2^{11}x_3x_5^{3}$, & $\mathcal Y_{22,\,328}=x_1^{7}x_2^{11}x_3x_4^{3}$, \\
$\mathcal Y_{22,\,329}=x_1^{7}x_2^{11}x_3^{3}x_5$, & $\mathcal Y_{22,\,330}=x_1^{7}x_2^{11}x_3^{3}x_4$, & $\mathcal Y_{22,\,331}=x_2x_3^{7}x_4^{7}x_5^{7}$, & $\mathcal Y_{22,\,332}=x_2^{7}x_3x_4^{7}x_5^{7}$, \\
$\mathcal Y_{22,\,333}=x_2^{7}x_3^{7}x_4x_5^{7}$, & $\mathcal Y_{22,\,334}=x_2^{7}x_3^{7}x_4^{7}x_5$, & $\mathcal Y_{22,\,335}=x_1x_3^{7}x_4^{7}x_5^{7}$, & $\mathcal Y_{22,\,336}=x_1x_2^{7}x_4^{7}x_5^{7}$, \\
$\mathcal Y_{22,\,337}=x_1x_2^{7}x_3^{7}x_5^{7}$, & $\mathcal Y_{22,\,338}=x_1x_2^{7}x_3^{7}x_4^{7}$, & $\mathcal Y_{22,\,339}=x_1^{7}x_3x_4^{7}x_5^{7}$, & $\mathcal Y_{22,\,340}=x_1^{7}x_3^{7}x_4x_5^{7}$, \\
$\mathcal Y_{22,\,341}=x_1^{7}x_3^{7}x_4^{7}x_5$, & $\mathcal Y_{22,\,342}=x_1^{7}x_2x_4^{7}x_5^{7}$, & $\mathcal Y_{22,\,343}=x_1^{7}x_2x_3^{7}x_5^{7}$, & $\mathcal Y_{22,\,344}=x_1^{7}x_2x_3^{7}x_4^{7}$, \\
$\mathcal Y_{22,\,345}=x_1^{7}x_2^{7}x_4x_5^{7}$, & $\mathcal Y_{22,\,346}=x_1^{7}x_2^{7}x_4^{7}x_5$, & $\mathcal Y_{22,\,347}=x_1^{7}x_2^{7}x_3x_5^{7}$, & $\mathcal Y_{22,\,348}=x_1^{7}x_2^{7}x_3x_4^{7}$, \\
$\mathcal Y_{22,\,349}=x_1^{7}x_2^{7}x_3^{7}x_5$, & $\mathcal Y_{22,\,350}=x_1^{7}x_2^{7}x_3^{7}x_4$, & $\mathcal Y_{22,\,351}=x_2^{3}x_3^{13}x_4^{2}x_5^{4}$, & $\mathcal Y_{22,\,352}=x_1^{3}x_3^{13}x_4^{2}x_5^{4}$, \\
$\mathcal Y_{22,\,353}=x_1^{3}x_2^{13}x_4^{2}x_5^{4}$, & $\mathcal Y_{22,\,354}=x_1^{3}x_2^{13}x_3^{2}x_5^{4}$, & $\mathcal Y_{22,\,355}=x_1^{3}x_2^{13}x_3^{2}x_4^{4}$, & $\mathcal Y_{22,\,356}=x_2^{3}x_3^{5}x_4^{2}x_5^{12}$, \\
$\mathcal Y_{22,\,357}=x_1^{3}x_3^{5}x_4^{2}x_5^{12}$, & $\mathcal Y_{22,\,358}=x_1^{3}x_2^{5}x_4^{2}x_5^{12}$, & $\mathcal Y_{22,\,359}=x_1^{3}x_2^{5}x_3^{2}x_5^{12}$, & $\mathcal Y_{22,\,360}=x_1^{3}x_2^{5}x_3^{2}x_4^{12}$, \\
$\mathcal Y_{22,\,361}=x_2^{3}x_3^{3}x_4^{3}x_5^{13}$, & $\mathcal Y_{22,\,362}=x_2^{3}x_3^{3}x_4^{13}x_5^{3}$, & $\mathcal Y_{22,\,363}=x_2^{3}x_3^{13}x_4^{3}x_5^{3}$, & $\mathcal Y_{22,\,364}=x_1^{3}x_3^{3}x_4^{3}x_5^{13}$, \\
$\mathcal Y_{22,\,365}=x_1^{3}x_3^{3}x_4^{13}x_5^{3}$, & $\mathcal Y_{22,\,366}=x_1^{3}x_3^{13}x_4^{3}x_5^{3}$, & $\mathcal Y_{22,\,367}=x_1^{3}x_2^{3}x_4^{3}x_5^{13}$, & $\mathcal Y_{22,\,368}=x_1^{3}x_2^{3}x_4^{13}x_5^{3}$, \\
$\mathcal Y_{22,\,369}=x_1^{3}x_2^{3}x_3^{3}x_5^{13}$, & $\mathcal Y_{22,\,370}=x_1^{3}x_2^{3}x_3^{3}x_4^{13}$, & $\mathcal Y_{22,\,371}=x_1^{3}x_2^{3}x_3^{13}x_5^{3}$, & $\mathcal Y_{22,\,372}=x_1^{3}x_2^{3}x_3^{13}x_4^{3}$, \\
$\mathcal Y_{22,\,373}=x_1^{3}x_2^{13}x_4^{3}x_5^{3}$, & $\mathcal Y_{22,\,374}=x_1^{3}x_2^{13}x_3^{3}x_5^{3}$, & $\mathcal Y_{22,\,375}=x_1^{3}x_2^{13}x_3^{3}x_4^{3}$, & $\mathcal Y_{22,\,376}=x_2^{3}x_3^{3}x_4^{5}x_5^{11}$, \\
$\mathcal Y_{22,\,377}=x_2^{3}x_3^{5}x_4^{3}x_5^{11}$, & $\mathcal Y_{22,\,378}=x_2^{3}x_3^{5}x_4^{11}x_5^{3}$, & $\mathcal Y_{22,\,379}=x_1^{3}x_3^{3}x_4^{5}x_5^{11}$, & $\mathcal Y_{22,\,380}=x_1^{3}x_3^{5}x_4^{3}x_5^{11}$, \\
$\mathcal Y_{22,\,381}=x_1^{3}x_3^{5}x_4^{11}x_5^{3}$, & $\mathcal Y_{22,\,382}=x_1^{3}x_2^{3}x_4^{5}x_5^{11}$, & $\mathcal Y_{22,\,383}=x_1^{3}x_2^{3}x_3^{5}x_5^{11}$, & $\mathcal Y_{22,\,384}=x_1^{3}x_2^{3}x_3^{5}x_4^{11}$, \\
$\mathcal Y_{22,\,385}=x_1^{3}x_2^{5}x_4^{3}x_5^{11}$, & $\mathcal Y_{22,\,386}=x_1^{3}x_2^{5}x_4^{11}x_5^{3}$, & $\mathcal Y_{22,\,387}=x_1^{3}x_2^{5}x_3^{3}x_5^{11}$, & $\mathcal Y_{22,\,388}=x_1^{3}x_2^{5}x_3^{3}x_4^{11}$, \\
$\mathcal Y_{22,\,389}=x_1^{3}x_2^{5}x_3^{11}x_5^{3}$, & $\mathcal Y_{22,\,390}=x_1^{3}x_2^{5}x_3^{11}x_4^{3}$, & $\mathcal Y_{22,\,391}=x_2^{3}x_3^{3}x_4^{7}x_5^{9}$, & $\mathcal Y_{22,\,392}=x_2^{3}x_3^{7}x_4^{3}x_5^{9}$, \\
$\mathcal Y_{22,\,393}=x_2^{3}x_3^{7}x_4^{9}x_5^{3}$, & $\mathcal Y_{22,\,394}=x_2^{7}x_3^{3}x_4^{3}x_5^{9}$, & $\mathcal Y_{22,\,395}=x_2^{7}x_3^{3}x_4^{9}x_5^{3}$, & $\mathcal Y_{22,\,396}=x_2^{7}x_3^{9}x_4^{3}x_5^{3}$, \\
$\mathcal Y_{22,\,397}=x_1^{3}x_3^{3}x_4^{7}x_5^{9}$, & $\mathcal Y_{22,\,398}=x_1^{3}x_3^{7}x_4^{3}x_5^{9}$, & $\mathcal Y_{22,\,399}=x_1^{3}x_3^{7}x_4^{9}x_5^{3}$, & $\mathcal Y_{22,\,400}=x_1^{3}x_2^{3}x_4^{7}x_5^{9}$, \\
$\mathcal Y_{22,\,401}=x_1^{3}x_2^{3}x_3^{7}x_5^{9}$, & $\mathcal Y_{22,\,402}=x_1^{3}x_2^{3}x_3^{7}x_4^{9}$, & $\mathcal Y_{22,\,403}=x_1^{3}x_2^{7}x_4^{3}x_5^{9}$, & $\mathcal Y_{22,\,404}=x_1^{3}x_2^{7}x_4^{9}x_5^{3}$, \\
$\mathcal Y_{22,\,405}=x_1^{3}x_2^{7}x_3^{3}x_5^{9}$, & $\mathcal Y_{22,\,406}=x_1^{3}x_2^{7}x_3^{3}x_4^{9}$, & $\mathcal Y_{22,\,407}=x_1^{3}x_2^{7}x_3^{9}x_5^{3}$, & $\mathcal Y_{22,\,408}=x_1^{3}x_2^{7}x_3^{9}x_4^{3}$, \\
$\mathcal Y_{22,\,409}=x_1^{7}x_3^{3}x_4^{3}x_5^{9}$, & $\mathcal Y_{22,\,410}=x_1^{7}x_3^{3}x_4^{9}x_5^{3}$, & $\mathcal Y_{22,\,411}=x_1^{7}x_3^{9}x_4^{3}x_5^{3}$, & $\mathcal Y_{22,\,412}=x_1^{7}x_2^{3}x_4^{3}x_5^{9}$, \\
$\mathcal Y_{22,\,413}=x_1^{7}x_2^{3}x_4^{9}x_5^{3}$, & $\mathcal Y_{22,\,414}=x_1^{7}x_2^{3}x_3^{3}x_5^{9}$, & $\mathcal Y_{22,\,415}=x_1^{7}x_2^{3}x_3^{3}x_4^{9}$, & $\mathcal Y_{22,\,416}=x_1^{7}x_2^{3}x_3^{9}x_5^{3}$, \\
$\mathcal Y_{22,\,417}=x_1^{7}x_2^{3}x_3^{9}x_4^{3}$, & $\mathcal Y_{22,\,418}=x_1^{7}x_2^{9}x_4^{3}x_5^{3}$, & $\mathcal Y_{22,\,419}=x_1^{7}x_2^{9}x_3^{3}x_5^{3}$, & $\mathcal Y_{22,\,420}=x_1^{7}x_2^{9}x_3^{3}x_4^{3}$, \\
$\mathcal Y_{22,\,421}=x_2^{3}x_3^{5}x_4^{10}x_5^{4}$, & $\mathcal Y_{22,\,422}=x_1^{3}x_3^{5}x_4^{10}x_5^{4}$, & $\mathcal Y_{22,\,423}=x_1^{3}x_2^{5}x_4^{10}x_5^{4}$, & $\mathcal Y_{22,\,424}=x_1^{3}x_2^{5}x_3^{10}x_5^{4}$, \\
$\mathcal Y_{22,\,425}=x_1^{3}x_2^{5}x_3^{10}x_4^{4}$, & $\mathcal Y_{22,\,426}=x_2^{3}x_3^{5}x_4^{6}x_5^{8}$, & $\mathcal Y_{22,\,427}=x_1^{3}x_3^{5}x_4^{6}x_5^{8}$, & $\mathcal Y_{22,\,428}=x_1^{3}x_2^{5}x_4^{6}x_5^{8}$, \\
$\mathcal Y_{22,\,429}=x_1^{3}x_2^{5}x_3^{6}x_5^{8}$, & $\mathcal Y_{22,\,430}=x_1^{3}x_2^{5}x_3^{6}x_4^{8}$, & $\mathcal Y_{22,\,431}=x_2^{3}x_3^{5}x_4^{7}x_5^{7}$, & $\mathcal Y_{22,\,432}=x_2^{3}x_3^{7}x_4^{5}x_5^{7}$, \\
$\mathcal Y_{22,\,433}=x_2^{3}x_3^{7}x_4^{7}x_5^{5}$, & $\mathcal Y_{22,\,434}=x_2^{7}x_3^{3}x_4^{5}x_5^{7}$, & $\mathcal Y_{22,\,435}=x_2^{7}x_3^{3}x_4^{7}x_5^{5}$, & $\mathcal Y_{22,\,436}=x_2^{7}x_3^{7}x_4^{3}x_5^{5}$, \\
$\mathcal Y_{22,\,437}=x_1^{3}x_3^{5}x_4^{7}x_5^{7}$, & $\mathcal Y_{22,\,438}=x_1^{3}x_3^{7}x_4^{5}x_5^{7}$, & $\mathcal Y_{22,\,439}=x_1^{3}x_3^{7}x_4^{7}x_5^{5}$, & $\mathcal Y_{22,\,440}=x_1^{3}x_2^{5}x_4^{7}x_5^{7}$, \\
$\mathcal Y_{22,\,441}=x_1^{3}x_2^{5}x_3^{7}x_5^{7}$, & $\mathcal Y_{22,\,442}=x_1^{3}x_2^{5}x_3^{7}x_4^{7}$, & $\mathcal Y_{22,\,443}=x_1^{3}x_2^{7}x_4^{5}x_5^{7}$, & $\mathcal Y_{22,\,444}=x_1^{3}x_2^{7}x_4^{7}x_5^{5}$, \\
$\mathcal Y_{22,\,445}=x_1^{3}x_2^{7}x_3^{5}x_5^{7}$, & $\mathcal Y_{22,\,446}=x_1^{3}x_2^{7}x_3^{5}x_4^{7}$, & $\mathcal Y_{22,\,447}=x_1^{3}x_2^{7}x_3^{7}x_5^{5}$, & $\mathcal Y_{22,\,448}=x_1^{3}x_2^{7}x_3^{7}x_4^{5}$, \\
$\mathcal Y_{22,\,449}=x_1^{7}x_3^{3}x_4^{5}x_5^{7}$, & $\mathcal Y_{22,\,450}=x_1^{7}x_3^{3}x_4^{7}x_5^{5}$, & $\mathcal Y_{22,\,451}=x_1^{7}x_3^{7}x_4^{3}x_5^{5}$, & $\mathcal Y_{22,\,452}=x_1^{7}x_2^{3}x_4^{5}x_5^{7}$, \\
$\mathcal Y_{22,\,453}=x_1^{7}x_2^{3}x_4^{7}x_5^{5}$, & $\mathcal Y_{22,\,454}=x_1^{7}x_2^{3}x_3^{5}x_5^{7}$, & $\mathcal Y_{22,\,455}=x_1^{7}x_2^{3}x_3^{5}x_4^{7}$, & $\mathcal Y_{22,\,456}=x_1^{7}x_2^{3}x_3^{7}x_5^{5}$, \\
$\mathcal Y_{22,\,457}=x_1^{7}x_2^{3}x_3^{7}x_4^{5}$, & $\mathcal Y_{22,\,458}=x_1^{7}x_2^{7}x_4^{3}x_5^{5}$, & $\mathcal Y_{22,\,459}=x_1^{7}x_2^{7}x_3^{3}x_5^{5}$, & $\mathcal Y_{22,\,460}=x_1^{7}x_2^{7}x_3^{3}x_4^{5}$.
\end{tabular}%
\end{center}

\newpage
\subsection{$\mathcal A_2$-generators for $\mathscr P_5^+$ in degree $22$}\label{s5.5}

By Propositions \ref{md22-1} and \ref{md22-2}, we have 
$ \mathscr B_5^+(22) = \bigcup_{1\leq j\leq 5}\mathscr B_5^+(\omega_{(j)}).$
Here $\omega_{(1)} = (2,2,2,1),\ \omega_{(2)} = (2,4,1,1),\ \omega_{(3)} = (2,4,3),\  \omega_{(4)} = (4,3,1,1),\ \omega_{(5)} = (4,3,3).$

\medskip

$ \mathscr B_5^+(\omega_{(1)}) = \overline{\Phi}^+(\mathscr B_4((\omega_{(1)}))\bigcup \mathscr B^+(5, \omega_{(1)})\bigcup \mathscr D = \{\mathcal Y_{22,\,t}:\ 461 \leq t\leq 510\},$
where the monomials $\mathcal Y_{22,\,t}:\ 461 \leq t\leq 510,$ are listed as follows:

\begin{center}
\begin{tabular}{llrr}
$\mathcal Y_{22,\,461}=x_1x_2x_3^{2}x_4^{4}x_5^{14}$, & $\mathcal Y_{22,\,462}=x_1x_2x_3^{2}x_4^{14}x_5^{4}$, & \multicolumn{1}{l}{$\mathcal Y_{22,\,463}=x_1x_2x_3^{14}x_4^{2}x_5^{4}$,} & \multicolumn{1}{l}{$\mathcal Y_{22,\,464}=x_1x_2^{2}x_3x_4^{4}x_5^{14}$,} \\
$\mathcal Y_{22,\,465}=x_1x_2^{2}x_3x_4^{14}x_5^{4}$, & $\mathcal Y_{22,\,466}=x_1x_2^{2}x_3^{4}x_4x_5^{14}$, & \multicolumn{1}{l}{$\mathcal Y_{22,\,467}=x_1x_2^{14}x_3x_4^{2}x_5^{4}$,} & \multicolumn{1}{l}{$\mathcal Y_{22,\,468}=x_1x_2x_3^{2}x_4^{6}x_5^{12}$,} \\
$\mathcal Y_{22,\,469}=x_1x_2x_3^{2}x_4^{12}x_5^{6}$, & $\mathcal Y_{22,\,470}=x_1x_2x_3^{6}x_4^{2}x_5^{12}$, & \multicolumn{1}{l}{$\mathcal Y_{22,\,471}=x_1x_2^{2}x_3x_4^{6}x_5^{12}$,} & \multicolumn{1}{l}{$\mathcal Y_{22,\,472}=x_1x_2^{2}x_3x_4^{12}x_5^{6}$,} \\
$\mathcal Y_{22,\,473}=x_1x_2^{2}x_3^{12}x_4x_5^{6}$, & $\mathcal Y_{22,\,474}=x_1x_2^{6}x_3x_4^{2}x_5^{12}$, & \multicolumn{1}{l}{$\mathcal Y_{22,\,475}=x_1x_2^{2}x_3^{13}x_4^{2}x_5^{4}$,} & \multicolumn{1}{l}{$\mathcal Y_{22,\,476}=x_1x_2^{2}x_3^{5}x_4^{2}x_5^{12}$,} \\
$\mathcal Y_{22,\,477}=x_1x_2^{3}x_3^{4}x_4^{2}x_5^{12}$, & $\mathcal Y_{22,\,478}=x_1x_2^{3}x_3^{12}x_4^{2}x_5^{4}$, & \multicolumn{1}{l}{$\mathcal Y_{22,\,479}=x_1^{3}x_2x_3^{4}x_4^{2}x_5^{12}$,} & \multicolumn{1}{l}{$\mathcal Y_{22,\,480}=x_1^{3}x_2x_3^{12}x_4^{2}x_5^{4}$,} \\
$\mathcal Y_{22,\,481}=x_1x_2^{2}x_3^{5}x_4^{10}x_5^{4}$, & $\mathcal Y_{22,\,482}=x_1x_2^{2}x_3^{4}x_4^{7}x_5^{8}$, & \multicolumn{1}{l}{$\mathcal Y_{22,\,483}=x_1x_2^{2}x_3^{4}x_4^{8}x_5^{7}$,} & \multicolumn{1}{l}{$\mathcal Y_{22,\,484}=x_1x_2^{2}x_3^{7}x_4^{4}x_5^{8}$,} \\
$\mathcal Y_{22,\,485}=x_1x_2^{7}x_3^{2}x_4^{4}x_5^{8}$, & $\mathcal Y_{22,\,486}=x_1^{7}x_2x_3^{2}x_4^{4}x_5^{8}$, & \multicolumn{1}{l}{$\mathcal Y_{22,\,487}=x_1x_2^{2}x_3^{5}x_4^{6}x_5^{8}$,} & \multicolumn{1}{l}{$\mathcal Y_{22,\,488}=x_1x_2^{3}x_3^{4}x_4^{10}x_5^{4}$,} \\
$\mathcal Y_{22,\,489}=x_1^{3}x_2x_3^{4}x_4^{10}x_5^{4}$, & $\mathcal Y_{22,\,490}=x_1x_2^{3}x_3^{4}x_4^{6}x_5^{8}$, & \multicolumn{1}{l}{$\mathcal Y_{22,\,491}=x_1^{3}x_2x_3^{4}x_4^{6}x_5^{8}$,} & \multicolumn{1}{l}{$\mathcal Y_{22,\,492}=x_1x_2x_3^{6}x_4^{6}x_5^{8}$,} \\
$\mathcal Y_{22,\,493}=x_1x_2x_3^{6}x_4^{10}x_5^{4}$, & $\mathcal Y_{22,\,494}=x_1x_2^{2}x_3^{3}x_4^{4}x_5^{12}$, & \multicolumn{1}{l}{$\mathcal Y_{22,\,495}=x_1x_2^{2}x_3^{3}x_4^{12}x_5^{4}$,} & \multicolumn{1}{l}{$\mathcal Y_{22,\,496}=x_1x_2^{2}x_3^{4}x_4^{9}x_5^{6}$,} \\
$\mathcal Y_{22,\,497}=x_1x_2^{2}x_3^{5}x_4^{8}x_5^{6}$, & $\mathcal Y_{22,\,498}=x_1x_2^{3}x_3^{2}x_4^{4}x_5^{12}$, & \multicolumn{1}{l}{$\mathcal Y_{22,\,499}=x_1x_2^{3}x_3^{2}x_4^{12}x_5^{4}$,} & \multicolumn{1}{l}{$\mathcal Y_{22,\,500}=x_1x_2^{3}x_3^{4}x_4^{8}x_5^{6}$,} \\
$\mathcal Y_{22,\,501}=x_1x_2^{3}x_3^{6}x_4^{4}x_5^{8}$, & $\mathcal Y_{22,\,502}=x_1x_2^{3}x_3^{6}x_4^{8}x_5^{4}$, & \multicolumn{1}{l}{$\mathcal Y_{22,\,503}=x_1^{3}x_2x_3^{2}x_4^{4}x_5^{12}$,} & \multicolumn{1}{l}{$\mathcal Y_{22,\,504}=x_1^{3}x_2x_3^{2}x_4^{12}x_5^{4}$,} \\
$\mathcal Y_{22,\,505}=x_1^{3}x_2x_3^{4}x_4^{8}x_5^{6}$, & $\mathcal Y_{22,\,506}=x_1^{3}x_2x_3^{6}x_4^{4}x_5^{8}$, & \multicolumn{1}{l}{$\mathcal Y_{22,\,507}=x_1^{3}x_2x_3^{6}x_4^{8}x_5^{4}$,} & \multicolumn{1}{l}{$\mathcal Y_{22,\,508}=x_1^{3}x_2^{5}x_3^{2}x_4^{4}x_5^{8}$,} \\
$\mathcal Y_{22,\,509}=x_1^{3}x_2^{5}x_3^{2}x_4^{8}x_5^{4}$, & $\mathcal Y_{22,\,510}=x_1^{3}x_2^{5}x_3^{8}x_4^{2}x_5^{4}$. &       &  
\end{tabular}%
\end{center}

\medskip

$\mathscr B_5^+(\omega_{(2)}) = \mathscr B^+(5, \omega_{(2)})\bigcup \mathscr E = \{\mathcal Y_{22,\,t}:\ 511 \leq t\leq 535 \},$
where the monomials $\mathcal Y_{22,\,t}:\ 511 \leq t\leq 535,$ are listed as follows:

\begin{center}
\begin{tabular}{lrrr}
$\mathcal Y_{22,\,511}=x_1x_2^{2}x_3^{2}x_4^{2}x_5^{15}$, & \multicolumn{1}{l}{$\mathcal Y_{22,\,512}=x_1x_2^{2}x_3^{2}x_4^{15}x_5^{2}$,} & \multicolumn{1}{l}{$\mathcal Y_{22,\,513}=x_1x_2^{2}x_3^{15}x_4^{2}x_5^{2}$,} & \multicolumn{1}{l}{$\mathcal Y_{22,\,514}=x_1x_2^{15}x_3^{2}x_4^{2}x_5^{2}$,} \\
$\mathcal Y_{22,\,515}=x_1^{15}x_2x_3^{2}x_4^{2}x_5^{2}$, & \multicolumn{1}{l}{$\mathcal Y_{22,\,516}=x_1x_2^{2}x_3^{2}x_4^{3}x_5^{14}$,} & \multicolumn{1}{l}{$\mathcal Y_{22,\,517}=x_1x_2^{2}x_3^{3}x_4^{2}x_5^{14}$,} & \multicolumn{1}{l}{$\mathcal Y_{22,\,518}=x_1x_2^{2}x_3^{3}x_4^{14}x_5^{2}$,} \\
$\mathcal Y_{22,\,519}=x_1x_2^{3}x_3^{2}x_4^{2}x_5^{14}$, & \multicolumn{1}{l}{$\mathcal Y_{22,\,520}=x_1x_2^{3}x_3^{2}x_4^{14}x_5^{2}$,} & \multicolumn{1}{l}{$\mathcal Y_{22,\,521}=x_1x_2^{3}x_3^{14}x_4^{2}x_5^{2}$,} & \multicolumn{1}{l}{$\mathcal Y_{22,\,522}=x_1^{3}x_2x_3^{2}x_4^{2}x_5^{14}$,} \\
$\mathcal Y_{22,\,523}=x_1^{3}x_2x_3^{2}x_4^{14}x_5^{2}$, & \multicolumn{1}{l}{$\mathcal Y_{22,\,524}=x_1^{3}x_2x_3^{14}x_4^{2}x_5^{2}$,} & \multicolumn{1}{l}{$\mathcal Y_{22,\,525}=x_1x_2^{2}x_3^{3}x_4^{6}x_5^{10}$,} & \multicolumn{1}{l}{$\mathcal Y_{22,\,526}=x_1x_2^{3}x_3^{2}x_4^{6}x_5^{10}$,} \\
$\mathcal Y_{22,\,527}=x_1x_2^{3}x_3^{6}x_4^{2}x_5^{10}$, & \multicolumn{1}{l}{$\mathcal Y_{22,\,528}=x_1x_2^{3}x_3^{6}x_4^{10}x_5^{2}$,} & \multicolumn{1}{l}{$\mathcal Y_{22,\,529}=x_1^{3}x_2x_3^{2}x_4^{6}x_5^{10}$,} & \multicolumn{1}{l}{$\mathcal Y_{22,\,530}=x_1^{3}x_2x_3^{6}x_4^{2}x_5^{10}$,} \\
$\mathcal Y_{22,\,531}=x_1^{3}x_2x_3^{6}x_4^{10}x_5^{2}$, & \multicolumn{1}{l}{$\mathcal Y_{22,\,532}=x_1^{3}x_2^{13}x_3^{2}x_4^{2}x_5^{2}$,} & \multicolumn{1}{l}{$\mathcal Y_{22,\,533}=x_1^{3}x_2^{5}x_3^{2}x_4^{2}x_5^{10}$,} & \multicolumn{1}{l}{$\mathcal Y_{22,\,534}=x_1^{3}x_2^{5}x_3^{2}x_4^{10}x_5^{2}$,} \\
$\mathcal Y_{22,\,535}=x_1^{3}x_2^{5}x_3^{10}x_4^{2}x_5^{2}$. &       &       &  
\end{tabular}
\end{center}

\medskip

$\mathscr B^+_5(\omega_{(3)}) = \{\mathcal Y_{22,\,t}:\ 536 \leq t\leq 540\},$ where the monomials $\mathcal Y_{22,\,t}:\ 536 \leq t\leq 540,$ are determined as follows:
$$ \mathcal Y_{536} = x_1x_2^{3}x_3^{6}x_4^{6}x_5^{6},\ \mathcal Y_{537} =  x_1^{3}x_2x_3^{6}x_4^{6}x_5^{6}, \ \mathcal Y_{538} =  x_1^{3}x_2^{5}x_3^{2}x_4^{6}x_5^{6},$$
$$ \mathcal Y_{539} = x_1^{3}x_2^{5}x_3^{6}x_4^{2}x_5^{6},\ \mathcal Y_{540} =  x_1^{3}x_2^{5}x_3^{6}x_4^{6}x_5^{2},$$

\medskip
 
$\mathscr B^+_5(\omega_{(4)}) =  \overline{\Phi}^+(\mathscr B_4(\omega_{(4)})\cup \mathscr B^+(5, \omega_{(4)}) = \{\mathcal Y_{22,\,t}:\ 541 \leq t\leq 840 \},$ where the monomials $\mathcal Y_{22,\,t}:\ 541 \leq t\leq 840,$ are listed as follows:

\begin{center}
\begin{tabular}{llll}
$\mathcal Y_{22,\,541}=x_1x_2x_3^{2}x_4^{3}x_5^{15}$, & $\mathcal Y_{22,\,542}=x_1x_2x_3^{2}x_4^{15}x_5^{3}$, & $\mathcal Y_{22,\,543}=x_1x_2x_3^{3}x_4^{2}x_5^{15}$, & $\mathcal Y_{22,\,544}=x_1x_2x_3^{3}x_4^{15}x_5^{2}$, \\
$\mathcal Y_{22,\,545}=x_1x_2x_3^{15}x_4^{2}x_5^{3}$, & $\mathcal Y_{22,\,546}=x_1x_2x_3^{15}x_4^{3}x_5^{2}$, & $\mathcal Y_{22,\,547}=x_1x_2^{2}x_3x_4^{3}x_5^{15}$, & $\mathcal Y_{22,\,548}=x_1x_2^{2}x_3x_4^{15}x_5^{3}$, \\
$\mathcal Y_{22,\,549}=x_1x_2^{2}x_3^{3}x_4x_5^{15}$, & $\mathcal Y_{22,\,550}=x_1x_2^{2}x_3^{3}x_4^{15}x_5$, & $\mathcal Y_{22,\,551}=x_1x_2^{2}x_3^{15}x_4x_5^{3}$, & $\mathcal Y_{22,\,552}=x_1x_2^{2}x_3^{15}x_4^{3}x_5$, \\
$\mathcal Y_{22,\,553}=x_1x_2^{3}x_3x_4^{2}x_5^{15}$, & $\mathcal Y_{22,\,554}=x_1x_2^{3}x_3x_4^{15}x_5^{2}$, & $\mathcal Y_{22,\,555}=x_1x_2^{3}x_3^{2}x_4x_5^{15}$, & $\mathcal Y_{22,\,556}=x_1x_2^{3}x_3^{2}x_4^{15}x_5$, \\
$\mathcal Y_{22,\,557}=x_1x_2^{3}x_3^{15}x_4x_5^{2}$, & $\mathcal Y_{22,\,558}=x_1x_2^{3}x_3^{15}x_4^{2}x_5$, & $\mathcal Y_{22,\,559}=x_1x_2^{15}x_3x_4^{2}x_5^{3}$, & $\mathcal Y_{22,\,560}=x_1x_2^{15}x_3x_4^{3}x_5^{2}$, \\
$\mathcal Y_{22,\,561}=x_1x_2^{15}x_3^{2}x_4x_5^{3}$, & $\mathcal Y_{22,\,562}=x_1x_2^{15}x_3^{2}x_4^{3}x_5$, & $\mathcal Y_{22,\,563}=x_1x_2^{15}x_3^{3}x_4x_5^{2}$, & $\mathcal Y_{22,\,564}=x_1x_2^{15}x_3^{3}x_4^{2}x_5$, \\
$\mathcal Y_{22,\,565}=x_1^{3}x_2x_3x_4^{2}x_5^{15}$, & $\mathcal Y_{22,\,566}=x_1^{3}x_2x_3x_4^{15}x_5^{2}$, & $\mathcal Y_{22,\,567}=x_1^{3}x_2x_3^{2}x_4x_5^{15}$, & $\mathcal Y_{22,\,568}=x_1^{3}x_2x_3^{2}x_4^{15}x_5$, \\
$\mathcal Y_{22,\,569}=x_1^{3}x_2x_3^{15}x_4x_5^{2}$, & $\mathcal Y_{22,\,570}=x_1^{3}x_2x_3^{15}x_4^{2}x_5$, & $\mathcal Y_{22,\,571}=x_1^{3}x_2^{15}x_3x_4x_5^{2}$, & $\mathcal Y_{22,\,572}=x_1^{3}x_2^{15}x_3x_4^{2}x_5$, \\
$\mathcal Y_{22,\,573}=x_1^{15}x_2x_3x_4^{2}x_5^{3}$, & $\mathcal Y_{22,\,574}=x_1^{15}x_2x_3x_4^{3}x_5^{2}$, & $\mathcal Y_{22,\,575}=x_1^{15}x_2x_3^{2}x_4x_5^{3}$, & $\mathcal Y_{22,\,576}=x_1^{15}x_2x_3^{2}x_4^{3}x_5$, \\
$\mathcal Y_{22,\,577}=x_1^{15}x_2x_3^{3}x_4x_5^{2}$, & $\mathcal Y_{22,\,578}=x_1^{15}x_2x_3^{3}x_4^{2}x_5$, & $\mathcal Y_{22,\,579}=x_1^{15}x_2^{3}x_3x_4x_5^{2}$, & $\mathcal Y_{22,\,580}=x_1^{15}x_2^{3}x_3x_4^{2}x_5$, \\
$\mathcal Y_{22,\,581}=x_1x_2x_3^{2}x_4^{7}x_5^{11}$, & $\mathcal Y_{22,\,582}=x_1x_2x_3^{7}x_4^{2}x_5^{11}$, & $\mathcal Y_{22,\,583}=x_1x_2x_3^{7}x_4^{11}x_5^{2}$, & $\mathcal Y_{22,\,584}=x_1x_2^{2}x_3x_4^{7}x_5^{11}$, \\
\end{tabular}%
\end{center}

\newpage
\begin{center}
\begin{tabular}{llll}
$\mathcal Y_{22,\,585}=x_1x_2^{2}x_3^{7}x_4x_5^{11}$, & $\mathcal Y_{22,\,586}=x_1x_2^{2}x_3^{7}x_4^{11}x_5$, & $\mathcal Y_{22,\,587}=x_1x_2^{7}x_3x_4^{2}x_5^{11}$, & $\mathcal Y_{22,\,588}=x_1x_2^{7}x_3x_4^{11}x_5^{2}$, \\
$\mathcal Y_{22,\,589}=x_1x_2^{7}x_3^{2}x_4x_5^{11}$, & $\mathcal Y_{22,\,590}=x_1x_2^{7}x_3^{2}x_4^{11}x_5$, & $\mathcal Y_{22,\,591}=x_1x_2^{7}x_3^{11}x_4x_5^{2}$, & $\mathcal Y_{22,\,592}=x_1x_2^{7}x_3^{11}x_4^{2}x_5$, \\
$\mathcal Y_{22,\,593}=x_1^{7}x_2x_3x_4^{2}x_5^{11}$, & $\mathcal Y_{22,\,594}=x_1^{7}x_2x_3x_4^{11}x_5^{2}$, & $\mathcal Y_{22,\,595}=x_1^{7}x_2x_3^{2}x_4x_5^{11}$, & $\mathcal Y_{22,\,596}=x_1^{7}x_2x_3^{2}x_4^{11}x_5$, \\
$\mathcal Y_{22,\,597}=x_1^{7}x_2x_3^{11}x_4x_5^{2}$, & $\mathcal Y_{22,\,598}=x_1^{7}x_2x_3^{11}x_4^{2}x_5$, & $\mathcal Y_{22,\,599}=x_1^{7}x_2^{11}x_3x_4x_5^{2}$, & $\mathcal Y_{22,\,600}=x_1^{7}x_2^{11}x_3x_4^{2}x_5$, \\
$\mathcal Y_{22,\,601}=x_1x_2x_3^{3}x_4^{3}x_5^{14}$, & $\mathcal Y_{22,\,602}=x_1x_2x_3^{3}x_4^{14}x_5^{3}$, & $\mathcal Y_{22,\,603}=x_1x_2x_3^{14}x_4^{3}x_5^{3}$, & $\mathcal Y_{22,\,604}=x_1x_2^{3}x_3x_4^{3}x_5^{14}$, \\
$\mathcal Y_{22,\,605}=x_1x_2^{3}x_3x_4^{14}x_5^{3}$, & $\mathcal Y_{22,\,606}=x_1x_2^{3}x_3^{3}x_4x_5^{14}$, & $\mathcal Y_{22,\,607}=x_1x_2^{3}x_3^{3}x_4^{14}x_5$, & $\mathcal Y_{22,\,608}=x_1x_2^{3}x_3^{14}x_4x_5^{3}$, \\
$\mathcal Y_{22,\,609}=x_1x_2^{3}x_3^{14}x_4^{3}x_5$, & $\mathcal Y_{22,\,610}=x_1x_2^{14}x_3x_4^{3}x_5^{3}$, & $\mathcal Y_{22,\,611}=x_1x_2^{14}x_3^{3}x_4x_5^{3}$, & $\mathcal Y_{22,\,612}=x_1x_2^{14}x_3^{3}x_4^{3}x_5$, \\
$\mathcal Y_{22,\,613}=x_1^{3}x_2x_3x_4^{3}x_5^{14}$, & $\mathcal Y_{22,\,614}=x_1^{3}x_2x_3x_4^{14}x_5^{3}$, & $\mathcal Y_{22,\,615}=x_1^{3}x_2x_3^{3}x_4x_5^{14}$, & $\mathcal Y_{22,\,616}=x_1^{3}x_2x_3^{3}x_4^{14}x_5$, \\
$\mathcal Y_{22,\,617}=x_1^{3}x_2x_3^{14}x_4x_5^{3}$, & $\mathcal Y_{22,\,618}=x_1^{3}x_2x_3^{14}x_4^{3}x_5$, & $\mathcal Y_{22,\,619}=x_1^{3}x_2^{3}x_3x_4x_5^{14}$, & $\mathcal Y_{22,\,620}=x_1^{3}x_2^{3}x_3x_4^{14}x_5$, \\
$\mathcal Y_{22,\,621}=x_1x_2x_3^{3}x_4^{6}x_5^{11}$, & $\mathcal Y_{22,\,622}=x_1x_2x_3^{6}x_4^{3}x_5^{11}$, & $\mathcal Y_{22,\,623}=x_1x_2x_3^{6}x_4^{11}x_5^{3}$, & $\mathcal Y_{22,\,624}=x_1x_2^{3}x_3x_4^{6}x_5^{11}$, \\
$\mathcal Y_{22,\,625}=x_1x_2^{3}x_3^{6}x_4x_5^{11}$, & $\mathcal Y_{22,\,626}=x_1x_2^{3}x_3^{6}x_4^{11}x_5$, & $\mathcal Y_{22,\,627}=x_1x_2^{6}x_3x_4^{3}x_5^{11}$, & $\mathcal Y_{22,\,628}=x_1x_2^{6}x_3x_4^{11}x_5^{3}$, \\
$\mathcal Y_{22,\,629}=x_1x_2^{6}x_3^{3}x_4x_5^{11}$, & $\mathcal Y_{22,\,630}=x_1x_2^{6}x_3^{3}x_4^{11}x_5$, & $\mathcal Y_{22,\,631}=x_1x_2^{6}x_3^{11}x_4x_5^{3}$, & $\mathcal Y_{22,\,632}=x_1x_2^{6}x_3^{11}x_4^{3}x_5$, \\
$\mathcal Y_{22,\,633}=x_1^{3}x_2x_3x_4^{6}x_5^{11}$, & $\mathcal Y_{22,\,634}=x_1^{3}x_2x_3^{6}x_4x_5^{11}$, & $\mathcal Y_{22,\,635}=x_1^{3}x_2x_3^{6}x_4^{11}x_5$, & $\mathcal Y_{22,\,636}=x_1x_2x_3^{3}x_4^{7}x_5^{10}$, \\
$\mathcal Y_{22,\,637}=x_1x_2x_3^{7}x_4^{3}x_5^{10}$, & $\mathcal Y_{22,\,638}=x_1x_2x_3^{7}x_4^{10}x_5^{3}$, & $\mathcal Y_{22,\,639}=x_1x_2^{3}x_3x_4^{7}x_5^{10}$, & $\mathcal Y_{22,\,640}=x_1x_2^{3}x_3^{7}x_4x_5^{10}$, \\
$\mathcal Y_{22,\,641}=x_1x_2^{3}x_3^{7}x_4^{10}x_5$, & $\mathcal Y_{22,\,642}=x_1x_2^{7}x_3x_4^{3}x_5^{10}$, & $\mathcal Y_{22,\,643}=x_1x_2^{7}x_3x_4^{10}x_5^{3}$, & $\mathcal Y_{22,\,644}=x_1x_2^{7}x_3^{3}x_4x_5^{10}$, \\
$\mathcal Y_{22,\,645}=x_1x_2^{7}x_3^{3}x_4^{10}x_5$, & $\mathcal Y_{22,\,646}=x_1x_2^{7}x_3^{10}x_4x_5^{3}$, & $\mathcal Y_{22,\,647}=x_1x_2^{7}x_3^{10}x_4^{3}x_5$, & $\mathcal Y_{22,\,648}=x_1^{3}x_2x_3x_4^{7}x_5^{10}$, \\
$\mathcal Y_{22,\,649}=x_1^{3}x_2x_3^{7}x_4x_5^{10}$, & $\mathcal Y_{22,\,650}=x_1^{3}x_2x_3^{7}x_4^{10}x_5$, & $\mathcal Y_{22,\,651}=x_1^{3}x_2^{7}x_3x_4x_5^{10}$, & $\mathcal Y_{22,\,652}=x_1^{3}x_2^{7}x_3x_4^{10}x_5$, \\
$\mathcal Y_{22,\,653}=x_1^{7}x_2x_3x_4^{3}x_5^{10}$, & $\mathcal Y_{22,\,654}=x_1^{7}x_2x_3x_4^{10}x_5^{3}$, & $\mathcal Y_{22,\,655}=x_1^{7}x_2x_3^{3}x_4x_5^{10}$, & $\mathcal Y_{22,\,656}=x_1^{7}x_2x_3^{3}x_4^{10}x_5$, \\
$\mathcal Y_{22,\,657}=x_1^{7}x_2x_3^{10}x_4x_5^{3}$, & $\mathcal Y_{22,\,658}=x_1^{7}x_2x_3^{10}x_4^{3}x_5$, & $\mathcal Y_{22,\,659}=x_1^{7}x_2^{3}x_3x_4x_5^{10}$, & $\mathcal Y_{22,\,660}=x_1^{7}x_2^{3}x_3x_4^{10}x_5$, \\
$\mathcal Y_{22,\,661}=x_1x_2^{2}x_3^{3}x_4^{3}x_5^{13}$, & $\mathcal Y_{22,\,662}=x_1x_2^{2}x_3^{3}x_4^{13}x_5^{3}$, & $\mathcal Y_{22,\,663}=x_1x_2^{2}x_3^{13}x_4^{3}x_5^{3}$, & $\mathcal Y_{22,\,664}=x_1x_2^{3}x_3^{2}x_4^{3}x_5^{13}$, \\
$\mathcal Y_{22,\,665}=x_1x_2^{3}x_3^{2}x_4^{13}x_5^{3}$, & $\mathcal Y_{22,\,666}=x_1x_2^{3}x_3^{3}x_4^{2}x_5^{13}$, & $\mathcal Y_{22,\,667}=x_1x_2^{3}x_3^{3}x_4^{13}x_5^{2}$, & $\mathcal Y_{22,\,668}=x_1x_2^{3}x_3^{13}x_4^{2}x_5^{3}$, \\
$\mathcal Y_{22,\,669}=x_1x_2^{3}x_3^{13}x_4^{3}x_5^{2}$, & $\mathcal Y_{22,\,670}=x_1^{3}x_2x_3^{2}x_4^{3}x_5^{13}$, & $\mathcal Y_{22,\,671}=x_1^{3}x_2x_3^{2}x_4^{13}x_5^{3}$, & $\mathcal Y_{22,\,672}=x_1^{3}x_2x_3^{3}x_4^{2}x_5^{13}$, \\
$\mathcal Y_{22,\,673}=x_1^{3}x_2x_3^{3}x_4^{13}x_5^{2}$, & $\mathcal Y_{22,\,674}=x_1^{3}x_2x_3^{13}x_4^{2}x_5^{3}$, & $\mathcal Y_{22,\,675}=x_1^{3}x_2x_3^{13}x_4^{3}x_5^{2}$, & $\mathcal Y_{22,\,676}=x_1^{3}x_2^{3}x_3x_4^{2}x_5^{13}$, \\
$\mathcal Y_{22,\,677}=x_1^{3}x_2^{3}x_3x_4^{13}x_5^{2}$, & $\mathcal Y_{22,\,678}=x_1^{3}x_2^{3}x_3^{13}x_4x_5^{2}$, & $\mathcal Y_{22,\,679}=x_1^{3}x_2^{3}x_3^{13}x_4^{2}x_5$, & $\mathcal Y_{22,\,680}=x_1^{3}x_2^{13}x_3x_4^{2}x_5^{3}$, \\
$\mathcal Y_{22,\,681}=x_1^{3}x_2^{13}x_3x_4^{3}x_5^{2}$, & $\mathcal Y_{22,\,682}=x_1^{3}x_2^{13}x_3^{2}x_4x_5^{3}$, & $\mathcal Y_{22,\,683}=x_1^{3}x_2^{13}x_3^{2}x_4^{3}x_5$, & $\mathcal Y_{22,\,684}=x_1^{3}x_2^{13}x_3^{3}x_4x_5^{2}$, \\
$\mathcal Y_{22,\,685}=x_1^{3}x_2^{13}x_3^{3}x_4^{2}x_5$, & $\mathcal Y_{22,\,686}=x_1x_2^{2}x_3^{3}x_4^{5}x_5^{11}$, & $\mathcal Y_{22,\,687}=x_1x_2^{2}x_3^{5}x_4^{3}x_5^{11}$, & $\mathcal Y_{22,\,688}=x_1x_2^{2}x_3^{5}x_4^{11}x_5^{3}$, \\
$\mathcal Y_{22,\,689}=x_1x_2^{3}x_3^{2}x_4^{5}x_5^{11}$, & $\mathcal Y_{22,\,690}=x_1x_2^{3}x_3^{5}x_4^{2}x_5^{11}$, & $\mathcal Y_{22,\,691}=x_1x_2^{3}x_3^{5}x_4^{11}x_5^{2}$, & $\mathcal Y_{22,\,692}=x_1^{3}x_2x_3^{2}x_4^{5}x_5^{11}$, \\
$\mathcal Y_{22,\,693}=x_1^{3}x_2x_3^{5}x_4^{2}x_5^{11}$, & $\mathcal Y_{22,\,694}=x_1^{3}x_2x_3^{5}x_4^{11}x_5^{2}$, & $\mathcal Y_{22,\,695}=x_1^{3}x_2^{5}x_3x_4^{2}x_5^{11}$, & $\mathcal Y_{22,\,696}=x_1^{3}x_2^{5}x_3x_4^{11}x_5^{2}$, \\
$\mathcal Y_{22,\,697}=x_1^{3}x_2^{5}x_3^{2}x_4x_5^{11}$, & $\mathcal Y_{22,\,698}=x_1^{3}x_2^{5}x_3^{2}x_4^{11}x_5$, & $\mathcal Y_{22,\,699}=x_1^{3}x_2^{5}x_3^{11}x_4x_5^{2}$, & $\mathcal Y_{22,\,700}=x_1^{3}x_2^{5}x_3^{11}x_4^{2}x_5$, \\
$\mathcal Y_{22,\,701}=x_1x_2^{2}x_3^{3}x_4^{7}x_5^{9}$, & $\mathcal Y_{22,\,702}=x_1x_2^{2}x_3^{7}x_4^{3}x_5^{9}$, & $\mathcal Y_{22,\,703}=x_1x_2^{2}x_3^{7}x_4^{9}x_5^{3}$, & $\mathcal Y_{22,\,704}=x_1x_2^{3}x_3^{2}x_4^{7}x_5^{9}$, \\
$\mathcal Y_{22,\,705}=x_1x_2^{3}x_3^{7}x_4^{2}x_5^{9}$, & $\mathcal Y_{22,\,706}=x_1x_2^{3}x_3^{7}x_4^{9}x_5^{2}$, & $\mathcal Y_{22,\,707}=x_1x_2^{7}x_3^{2}x_4^{3}x_5^{9}$, & $\mathcal Y_{22,\,708}=x_1x_2^{7}x_3^{2}x_4^{9}x_5^{3}$, \\
$\mathcal Y_{22,\,709}=x_1x_2^{7}x_3^{3}x_4^{2}x_5^{9}$, & $\mathcal Y_{22,\,710}=x_1x_2^{7}x_3^{3}x_4^{9}x_5^{2}$, & $\mathcal Y_{22,\,711}=x_1x_2^{7}x_3^{9}x_4^{2}x_5^{3}$, & $\mathcal Y_{22,\,712}=x_1x_2^{7}x_3^{9}x_4^{3}x_5^{2}$, \\
$\mathcal Y_{22,\,713}=x_1^{3}x_2x_3^{2}x_4^{7}x_5^{9}$, & $\mathcal Y_{22,\,714}=x_1^{3}x_2x_3^{7}x_4^{2}x_5^{9}$, & $\mathcal Y_{22,\,715}=x_1^{3}x_2x_3^{7}x_4^{9}x_5^{2}$, & $\mathcal Y_{22,\,716}=x_1^{3}x_2^{7}x_3x_4^{2}x_5^{9}$, \\
$\mathcal Y_{22,\,717}=x_1^{3}x_2^{7}x_3x_4^{9}x_5^{2}$, & $\mathcal Y_{22,\,718}=x_1^{3}x_2^{7}x_3^{9}x_4x_5^{2}$, & $\mathcal Y_{22,\,719}=x_1^{3}x_2^{7}x_3^{9}x_4^{2}x_5$, & $\mathcal Y_{22,\,720}=x_1^{7}x_2x_3^{2}x_4^{3}x_5^{9}$, \\
$\mathcal Y_{22,\,721}=x_1^{7}x_2x_3^{2}x_4^{9}x_5^{3}$, & $\mathcal Y_{22,\,722}=x_1^{7}x_2x_3^{3}x_4^{2}x_5^{9}$, & $\mathcal Y_{22,\,723}=x_1^{7}x_2x_3^{3}x_4^{9}x_5^{2}$, & $\mathcal Y_{22,\,724}=x_1^{7}x_2x_3^{9}x_4^{2}x_5^{3}$, \\
$\mathcal Y_{22,\,725}=x_1^{7}x_2x_3^{9}x_4^{3}x_5^{2}$, & $\mathcal Y_{22,\,726}=x_1^{7}x_2^{3}x_3x_4^{2}x_5^{9}$, & $\mathcal Y_{22,\,727}=x_1^{7}x_2^{3}x_3x_4^{9}x_5^{2}$, & $\mathcal Y_{22,\,728}=x_1^{7}x_2^{3}x_3^{9}x_4x_5^{2}$, \\
$\mathcal Y_{22,\,729}=x_1^{7}x_2^{3}x_3^{9}x_4^{2}x_5$, & $\mathcal Y_{22,\,730}=x_1^{7}x_2^{9}x_3x_4^{2}x_5^{3}$, & $\mathcal Y_{22,\,731}=x_1^{7}x_2^{9}x_3x_4^{3}x_5^{2}$, & $\mathcal Y_{22,\,732}=x_1^{7}x_2^{9}x_3^{2}x_4x_5^{3}$, \\
$\mathcal Y_{22,\,733}=x_1^{7}x_2^{9}x_3^{2}x_4^{3}x_5$, & $\mathcal Y_{22,\,734}=x_1^{7}x_2^{9}x_3^{3}x_4x_5^{2}$, & $\mathcal Y_{22,\,735}=x_1^{7}x_2^{9}x_3^{3}x_4^{2}x_5$, & $\mathcal Y_{22,\,736}=x_1x_2^{3}x_3^{3}x_4^{3}x_5^{12}$, \\
$\mathcal Y_{22,\,737}=x_1x_2^{3}x_3^{3}x_4^{12}x_5^{3}$, & $\mathcal Y_{22,\,738}=x_1x_2^{3}x_3^{12}x_4^{3}x_5^{3}$, & $\mathcal Y_{22,\,739}=x_1^{3}x_2x_3^{3}x_4^{3}x_5^{12}$, & $\mathcal Y_{22,\,740}=x_1^{3}x_2x_3^{3}x_4^{12}x_5^{3}$, \\
$\mathcal Y_{22,\,741}=x_1^{3}x_2x_3^{12}x_4^{3}x_5^{3}$, & $\mathcal Y_{22,\,742}=x_1^{3}x_2^{3}x_3x_4^{3}x_5^{12}$, & $\mathcal Y_{22,\,743}=x_1^{3}x_2^{3}x_3x_4^{12}x_5^{3}$, & $\mathcal Y_{22,\,744}=x_1^{3}x_2^{3}x_3^{3}x_4x_5^{12}$, \\
$\mathcal Y_{22,\,745}=x_1^{3}x_2^{3}x_3^{3}x_4^{12}x_5$, & $\mathcal Y_{22,\,746}=x_1^{3}x_2^{3}x_3^{12}x_4x_5^{3}$, & $\mathcal Y_{22,\,747}=x_1^{3}x_2^{3}x_3^{12}x_4^{3}x_5$, & $\mathcal Y_{22,\,748}=x_1^{3}x_2^{12}x_3x_4^{3}x_5^{3}$, \\
$\mathcal Y_{22,\,749}=x_1^{3}x_2^{12}x_3^{3}x_4x_5^{3}$, & $\mathcal Y_{22,\,750}=x_1^{3}x_2^{12}x_3^{3}x_4^{3}x_5$, & $\mathcal Y_{22,\,751}=x_1x_2^{3}x_3^{3}x_4^{4}x_5^{11}$, & $\mathcal Y_{22,\,752}=x_1x_2^{3}x_3^{4}x_4^{3}x_5^{11}$, \\
$\mathcal Y_{22,\,753}=x_1x_2^{3}x_3^{4}x_4^{11}x_5^{3}$, & $\mathcal Y_{22,\,754}=x_1^{3}x_2x_3^{3}x_4^{4}x_5^{11}$, & $\mathcal Y_{22,\,755}=x_1^{3}x_2x_3^{4}x_4^{3}x_5^{11}$, & $\mathcal Y_{22,\,756}=x_1^{3}x_2x_3^{4}x_4^{11}x_5^{3}$, \\
$\mathcal Y_{22,\,757}=x_1^{3}x_2^{3}x_3x_4^{4}x_5^{11}$, & $\mathcal Y_{22,\,758}=x_1^{3}x_2^{3}x_3^{4}x_4x_5^{11}$, & $\mathcal Y_{22,\,759}=x_1^{3}x_2^{3}x_3^{4}x_4^{11}x_5$, & $\mathcal Y_{22,\,760}=x_1^{3}x_2^{4}x_3x_4^{3}x_5^{11}$, \\
$\mathcal Y_{22,\,761}=x_1^{3}x_2^{4}x_3x_4^{11}x_5^{3}$, & $\mathcal Y_{22,\,762}=x_1^{3}x_2^{4}x_3^{3}x_4x_5^{11}$, & $\mathcal Y_{22,\,763}=x_1^{3}x_2^{4}x_3^{3}x_4^{11}x_5$, & $\mathcal Y_{22,\,764}=x_1^{3}x_2^{4}x_3^{11}x_4x_5^{3}$, \\
$\mathcal Y_{22,\,765}=x_1^{3}x_2^{4}x_3^{11}x_4^{3}x_5$, & $\mathcal Y_{22,\,766}=x_1x_2^{3}x_3^{3}x_4^{5}x_5^{10}$, & $\mathcal Y_{22,\,767}=x_1x_2^{3}x_3^{5}x_4^{3}x_5^{10}$, & $\mathcal Y_{22,\,768}=x_1x_2^{3}x_3^{5}x_4^{10}x_5^{3}$, \\
$\mathcal Y_{22,\,769}=x_1^{3}x_2x_3^{3}x_4^{5}x_5^{10}$, & $\mathcal Y_{22,\,770}=x_1^{3}x_2x_3^{5}x_4^{3}x_5^{10}$, & $\mathcal Y_{22,\,771}=x_1^{3}x_2x_3^{5}x_4^{10}x_5^{3}$, & $\mathcal Y_{22,\,772}=x_1^{3}x_2^{3}x_3x_4^{5}x_5^{10}$, \\
$\mathcal Y_{22,\,773}=x_1^{3}x_2^{3}x_3^{5}x_4x_5^{10}$, & $\mathcal Y_{22,\,774}=x_1^{3}x_2^{3}x_3^{5}x_4^{10}x_5$, & $\mathcal Y_{22,\,775}=x_1^{3}x_2^{5}x_3x_4^{3}x_5^{10}$, & $\mathcal Y_{22,\,776}=x_1^{3}x_2^{5}x_3x_4^{10}x_5^{3}$, \\
$\mathcal Y_{22,\,777}=x_1^{3}x_2^{5}x_3^{3}x_4x_5^{10}$, & $\mathcal Y_{22,\,778}=x_1^{3}x_2^{5}x_3^{3}x_4^{10}x_5$, & $\mathcal Y_{22,\,779}=x_1^{3}x_2^{5}x_3^{10}x_4x_5^{3}$, & $\mathcal Y_{22,\,780}=x_1^{3}x_2^{5}x_3^{10}x_4^{3}x_5$, \\
\end{tabular}%
\end{center}

\newpage
\begin{center}
\begin{tabular}{llll}
$\mathcal Y_{22,\,781}=x_1x_2^{3}x_3^{3}x_4^{6}x_5^{9}$, & $\mathcal Y_{22,\,782}=x_1x_2^{3}x_3^{6}x_4^{3}x_5^{9}$, & $\mathcal Y_{22,\,783}=x_1x_2^{3}x_3^{6}x_4^{9}x_5^{3}$, & $\mathcal Y_{22,\,784}=x_1x_2^{6}x_3^{3}x_4^{3}x_5^{9}$, \\
$\mathcal Y_{22,\,785}=x_1x_2^{6}x_3^{3}x_4^{9}x_5^{3}$, & $\mathcal Y_{22,\,786}=x_1x_2^{6}x_3^{9}x_4^{3}x_5^{3}$, & $\mathcal Y_{22,\,787}=x_1^{3}x_2x_3^{3}x_4^{6}x_5^{9}$, & $\mathcal Y_{22,\,788}=x_1^{3}x_2x_3^{6}x_4^{3}x_5^{9}$, \\
$\mathcal Y_{22,\,789}=x_1^{3}x_2x_3^{6}x_4^{9}x_5^{3}$, & $\mathcal Y_{22,\,790}=x_1^{3}x_2^{3}x_3x_4^{6}x_5^{9}$, & $\mathcal Y_{22,\,791}=x_1x_2^{3}x_3^{3}x_4^{7}x_5^{8}$, & $\mathcal Y_{22,\,792}=x_1x_2^{3}x_3^{7}x_4^{3}x_5^{8}$, \\
$\mathcal Y_{22,\,793}=x_1x_2^{3}x_3^{7}x_4^{8}x_5^{3}$, & $\mathcal Y_{22,\,794}=x_1x_2^{7}x_3^{3}x_4^{3}x_5^{8}$, & $\mathcal Y_{22,\,795}=x_1x_2^{7}x_3^{3}x_4^{8}x_5^{3}$, & $\mathcal Y_{22,\,796}=x_1x_2^{7}x_3^{8}x_4^{3}x_5^{3}$, \\
$\mathcal Y_{22,\,797}=x_1^{3}x_2x_3^{3}x_4^{7}x_5^{8}$, & $\mathcal Y_{22,\,798}=x_1^{3}x_2x_3^{7}x_4^{3}x_5^{8}$, & $\mathcal Y_{22,\,799}=x_1^{3}x_2x_3^{7}x_4^{8}x_5^{3}$, & $\mathcal Y_{22,\,800}=x_1^{3}x_2^{3}x_3x_4^{7}x_5^{8}$, \\
$\mathcal Y_{22,\,801}=x_1^{3}x_2^{3}x_3^{7}x_4x_5^{8}$, & $\mathcal Y_{22,\,802}=x_1^{3}x_2^{3}x_3^{7}x_4^{8}x_5$, & $\mathcal Y_{22,\,803}=x_1^{3}x_2^{7}x_3x_4^{3}x_5^{8}$, & $\mathcal Y_{22,\,804}=x_1^{3}x_2^{7}x_3x_4^{8}x_5^{3}$, \\
$\mathcal Y_{22,\,805}=x_1^{3}x_2^{7}x_3^{3}x_4x_5^{8}$, & $\mathcal Y_{22,\,806}=x_1^{3}x_2^{7}x_3^{3}x_4^{8}x_5$, & $\mathcal Y_{22,\,807}=x_1^{3}x_2^{7}x_3^{8}x_4x_5^{3}$, & $\mathcal Y_{22,\,808}=x_1^{3}x_2^{7}x_3^{8}x_4^{3}x_5$, \\
$\mathcal Y_{22,\,809}=x_1^{7}x_2x_3^{3}x_4^{3}x_5^{8}$, & $\mathcal Y_{22,\,810}=x_1^{7}x_2x_3^{3}x_4^{8}x_5^{3}$, & $\mathcal Y_{22,\,811}=x_1^{7}x_2x_3^{8}x_4^{3}x_5^{3}$, & $\mathcal Y_{22,\,812}=x_1^{7}x_2^{3}x_3x_4^{3}x_5^{8}$, \\
$\mathcal Y_{22,\,813}=x_1^{7}x_2^{3}x_3x_4^{8}x_5^{3}$, & $\mathcal Y_{22,\,814}=x_1^{7}x_2^{3}x_3^{3}x_4x_5^{8}$, & $\mathcal Y_{22,\,815}=x_1^{7}x_2^{3}x_3^{3}x_4^{8}x_5$, & $\mathcal Y_{22,\,816}=x_1^{7}x_2^{3}x_3^{8}x_4x_5^{3}$, \\
$\mathcal Y_{22,\,817}=x_1^{7}x_2^{3}x_3^{8}x_4^{3}x_5$, & $\mathcal Y_{22,\,818}=x_1^{7}x_2^{8}x_3x_4^{3}x_5^{3}$, & $\mathcal Y_{22,\,819}=x_1^{7}x_2^{8}x_3^{3}x_4x_5^{3}$, & $\mathcal Y_{22,\,820}=x_1^{7}x_2^{8}x_3^{3}x_4^{3}x_5$, \\
$\mathcal Y_{22,\,821}=x_1^{3}x_2^{3}x_3^{5}x_4^{2}x_5^{9}$, & $\mathcal Y_{22,\,822}=x_1^{3}x_2^{3}x_3^{5}x_4^{9}x_5^{2}$, & $\mathcal Y_{22,\,823}=x_1^{3}x_2^{5}x_3^{2}x_4^{3}x_5^{9}$, & $\mathcal Y_{22,\,824}=x_1^{3}x_2^{5}x_3^{2}x_4^{9}x_5^{3}$, \\
$\mathcal Y_{22,\,825}=x_1^{3}x_2^{5}x_3^{3}x_4^{2}x_5^{9}$, & $\mathcal Y_{22,\,826}=x_1^{3}x_2^{5}x_3^{3}x_4^{9}x_5^{2}$, & $\mathcal Y_{22,\,827}=x_1^{3}x_2^{5}x_3^{9}x_4^{2}x_5^{3}$, & $\mathcal Y_{22,\,828}=x_1^{3}x_2^{5}x_3^{9}x_4^{3}x_5^{2}$, \\
$\mathcal Y_{22,\,829}=x_1^{3}x_2^{3}x_3^{3}x_4^{4}x_5^{9}$, & $\mathcal Y_{22,\,830}=x_1^{3}x_2^{3}x_3^{4}x_4^{3}x_5^{9}$, & $\mathcal Y_{22,\,831}=x_1^{3}x_2^{3}x_3^{4}x_4^{9}x_5^{3}$, & $\mathcal Y_{22,\,832}=x_1^{3}x_2^{4}x_3^{3}x_4^{3}x_5^{9}$, \\
$\mathcal Y_{22,\,833}=x_1^{3}x_2^{4}x_3^{3}x_4^{9}x_5^{3}$, & $\mathcal Y_{22,\,834}=x_1^{3}x_2^{4}x_3^{9}x_4^{3}x_5^{3}$, & $\mathcal Y_{22,\,835}=x_1^{3}x_2^{3}x_3^{3}x_4^{5}x_5^{8}$, & $\mathcal Y_{22,\,836}=x_1^{3}x_2^{3}x_3^{5}x_4^{3}x_5^{8}$, \\
$\mathcal Y_{22,\,837}=x_1^{3}x_2^{3}x_3^{5}x_4^{8}x_5^{3}$, & $\mathcal Y_{22,\,838}=x_1^{3}x_2^{5}x_3^{3}x_4^{3}x_5^{8}$, & $\mathcal Y_{22,\,839}=x_1^{3}x_2^{5}x_3^{3}x_4^{8}x_5^{3}$, & $\mathcal Y_{22,\,840}=x_1^{3}x_2^{5}x_3^{8}x_4^{3}x_5^{3}$.
\end{tabular}
\end{center}

 $\mathscr B^+_5(\omega_{(5)}) =  \overline{\Phi}^+(\mathscr B_4(\omega_{(5)})\cup \mathscr B^+(5, \omega_{(5)}) = \{\mathcal Y_{22,\,t}:\ 841 \leq t\leq  965\},$ where the monomials $\mathcal Y_{22,\,t}:\ 841 \leq t\leq 965,$ are listed as follows:

\begin{center}
\begin{tabular}{lrrr}
$\mathcal Y_{22,\,841}=x_1x_2x_3^{6}x_4^{7}x_5^{7}$, & \multicolumn{1}{l}{$\mathcal Y_{22,\,842}=x_1x_2x_3^{7}x_4^{6}x_5^{7}$,} & \multicolumn{1}{l}{$\mathcal Y_{22,\,843}=x_1x_2x_3^{7}x_4^{7}x_5^{6}$,} & \multicolumn{1}{l}{$\mathcal Y_{22,\,844}=x_1x_2^{6}x_3x_4^{7}x_5^{7}$,} \\
$\mathcal Y_{22,\,845}=x_1x_2^{6}x_3^{7}x_4x_5^{7}$, & \multicolumn{1}{l}{$\mathcal Y_{22,\,846}=x_1x_2^{6}x_3^{7}x_4^{7}x_5$,} & \multicolumn{1}{l}{$\mathcal Y_{22,\,847}=x_1x_2^{7}x_3x_4^{6}x_5^{7}$,} & \multicolumn{1}{l}{$\mathcal Y_{22,\,848}=x_1x_2^{7}x_3x_4^{7}x_5^{6}$,} \\
$\mathcal Y_{22,\,849}=x_1x_2^{7}x_3^{6}x_4x_5^{7}$, & \multicolumn{1}{l}{$\mathcal Y_{22,\,850}=x_1x_2^{7}x_3^{6}x_4^{7}x_5$,} & \multicolumn{1}{l}{$\mathcal Y_{22,\,851}=x_1x_2^{7}x_3^{7}x_4x_5^{6}$,} & \multicolumn{1}{l}{$\mathcal Y_{22,\,852}=x_1x_2^{7}x_3^{7}x_4^{6}x_5$,} \\
$\mathcal Y_{22,\,853}=x_1^{7}x_2x_3x_4^{6}x_5^{7}$, & \multicolumn{1}{l}{$\mathcal Y_{22,\,854}=x_1^{7}x_2x_3x_4^{7}x_5^{6}$,} & \multicolumn{1}{l}{$\mathcal Y_{22,\,855}=x_1^{7}x_2x_3^{6}x_4x_5^{7}$,} & \multicolumn{1}{l}{$\mathcal Y_{22,\,856}=x_1^{7}x_2x_3^{6}x_4^{7}x_5$,} \\
$\mathcal Y_{22,\,857}=x_1^{7}x_2x_3^{7}x_4x_5^{6}$, & \multicolumn{1}{l}{$\mathcal Y_{22,\,858}=x_1^{7}x_2x_3^{7}x_4^{6}x_5$,} & \multicolumn{1}{l}{$\mathcal Y_{22,\,859}=x_1^{7}x_2^{7}x_3x_4x_5^{6}$,} & \multicolumn{1}{l}{$\mathcal Y_{22,\,860}=x_1^{7}x_2^{7}x_3x_4^{6}x_5$,} \\
$\mathcal Y_{22,\,861}=x_1x_2^{2}x_3^{5}x_4^{7}x_5^{7}$, & \multicolumn{1}{l}{$\mathcal Y_{22,\,862}=x_1x_2^{2}x_3^{7}x_4^{5}x_5^{7}$,} & \multicolumn{1}{l}{$\mathcal Y_{22,\,863}=x_1x_2^{2}x_3^{7}x_4^{7}x_5^{5}$,} & \multicolumn{1}{l}{$\mathcal Y_{22,\,864}=x_1x_2^{7}x_3^{2}x_4^{5}x_5^{7}$,} \\
$\mathcal Y_{22,\,865}=x_1x_2^{7}x_3^{2}x_4^{7}x_5^{5}$, & \multicolumn{1}{l}{$\mathcal Y_{22,\,866}=x_1x_2^{7}x_3^{7}x_4^{2}x_5^{5}$,} & \multicolumn{1}{l}{$\mathcal Y_{22,\,867}=x_1^{7}x_2x_3^{2}x_4^{5}x_5^{7}$,} & \multicolumn{1}{l}{$\mathcal Y_{22,\,868}=x_1^{7}x_2x_3^{2}x_4^{7}x_5^{5}$,} \\
$\mathcal Y_{22,\,869}=x_1^{7}x_2x_3^{7}x_4^{2}x_5^{5}$, & \multicolumn{1}{l}{$\mathcal Y_{22,\,870}=x_1^{7}x_2^{7}x_3x_4^{2}x_5^{5}$,} & \multicolumn{1}{l}{$\mathcal Y_{22,\,871}=x_1x_2^{3}x_3^{4}x_4^{7}x_5^{7}$,} & \multicolumn{1}{l}{$\mathcal Y_{22,\,872}=x_1x_2^{3}x_3^{7}x_4^{4}x_5^{7}$,} \\
$\mathcal Y_{22,\,873}=x_1x_2^{3}x_3^{7}x_4^{7}x_5^{4}$, & \multicolumn{1}{l}{$\mathcal Y_{22,\,874}=x_1x_2^{7}x_3^{3}x_4^{4}x_5^{7}$,} & \multicolumn{1}{l}{$\mathcal Y_{22,\,875}=x_1x_2^{7}x_3^{3}x_4^{7}x_5^{4}$,} & \multicolumn{1}{l}{$\mathcal Y_{22,\,876}=x_1x_2^{7}x_3^{7}x_4^{3}x_5^{4}$,} \\
$\mathcal Y_{22,\,877}=x_1^{3}x_2x_3^{4}x_4^{7}x_5^{7}$, & \multicolumn{1}{l}{$\mathcal Y_{22,\,878}=x_1^{3}x_2x_3^{7}x_4^{4}x_5^{7}$,} & \multicolumn{1}{l}{$\mathcal Y_{22,\,879}=x_1^{3}x_2x_3^{7}x_4^{7}x_5^{4}$,} & \multicolumn{1}{l}{$\mathcal Y_{22,\,880}=x_1^{3}x_2^{4}x_3x_4^{7}x_5^{7}$,} \\
$\mathcal Y_{22,\,881}=x_1^{3}x_2^{4}x_3^{7}x_4x_5^{7}$, & \multicolumn{1}{l}{$\mathcal Y_{22,\,882}=x_1^{3}x_2^{4}x_3^{7}x_4^{7}x_5$,} & \multicolumn{1}{l}{$\mathcal Y_{22,\,883}=x_1^{3}x_2^{7}x_3x_4^{4}x_5^{7}$,} & \multicolumn{1}{l}{$\mathcal Y_{22,\,884}=x_1^{3}x_2^{7}x_3x_4^{7}x_5^{4}$,} \\
$\mathcal Y_{22,\,885}=x_1^{3}x_2^{7}x_3^{4}x_4x_5^{7}$, & \multicolumn{1}{l}{$\mathcal Y_{22,\,886}=x_1^{3}x_2^{7}x_3^{4}x_4^{7}x_5$,} & \multicolumn{1}{l}{$\mathcal Y_{22,\,887}=x_1^{3}x_2^{7}x_3^{7}x_4x_5^{4}$,} & \multicolumn{1}{l}{$\mathcal Y_{22,\,888}=x_1^{3}x_2^{7}x_3^{7}x_4^{4}x_5$,} \\
$\mathcal Y_{22,\,889}=x_1^{7}x_2x_3^{3}x_4^{4}x_5^{7}$, & \multicolumn{1}{l}{$\mathcal Y_{22,\,890}=x_1^{7}x_2x_3^{3}x_4^{7}x_5^{4}$,} & \multicolumn{1}{l}{$\mathcal Y_{22,\,891}=x_1^{7}x_2x_3^{7}x_4^{3}x_5^{4}$,} & \multicolumn{1}{l}{$\mathcal Y_{22,\,892}=x_1^{7}x_2^{3}x_3x_4^{4}x_5^{7}$,} \\
$\mathcal Y_{22,\,893}=x_1^{7}x_2^{3}x_3x_4^{7}x_5^{4}$, & \multicolumn{1}{l}{$\mathcal Y_{22,\,894}=x_1^{7}x_2^{3}x_3^{4}x_4x_5^{7}$,} & \multicolumn{1}{l}{$\mathcal Y_{22,\,895}=x_1^{7}x_2^{3}x_3^{4}x_4^{7}x_5$,} & \multicolumn{1}{l}{$\mathcal Y_{22,\,896}=x_1^{7}x_2^{3}x_3^{7}x_4x_5^{4}$,} \\
$\mathcal Y_{22,\,897}=x_1^{7}x_2^{3}x_3^{7}x_4^{4}x_5$, & \multicolumn{1}{l}{$\mathcal Y_{22,\,898}=x_1^{7}x_2^{7}x_3x_4^{3}x_5^{4}$,} & \multicolumn{1}{l}{$\mathcal Y_{22,\,899}=x_1^{7}x_2^{7}x_3^{3}x_4x_5^{4}$,} & \multicolumn{1}{l}{$\mathcal Y_{22,\,900}=x_1^{7}x_2^{7}x_3^{3}x_4^{4}x_5$,} \\
$\mathcal Y_{22,\,901}=x_1x_2^{3}x_3^{5}x_4^{6}x_5^{7}$, & \multicolumn{1}{l}{$\mathcal Y_{22,\,902}=x_1x_2^{3}x_3^{5}x_4^{7}x_5^{6}$,} & \multicolumn{1}{l}{$\mathcal Y_{22,\,903}=x_1x_2^{3}x_3^{6}x_4^{5}x_5^{7}$,} & \multicolumn{1}{l}{$\mathcal Y_{22,\,904}=x_1x_2^{3}x_3^{6}x_4^{7}x_5^{5}$,} \\
$\mathcal Y_{22,\,905}=x_1x_2^{3}x_3^{7}x_4^{5}x_5^{6}$, & \multicolumn{1}{l}{$\mathcal Y_{22,\,906}=x_1x_2^{3}x_3^{7}x_4^{6}x_5^{5}$,} & \multicolumn{1}{l}{$\mathcal Y_{22,\,907}=x_1x_2^{6}x_3^{3}x_4^{5}x_5^{7}$,} & \multicolumn{1}{l}{$\mathcal Y_{22,\,908}=x_1x_2^{6}x_3^{3}x_4^{7}x_5^{5}$,} \\
$\mathcal Y_{22,\,909}=x_1x_2^{6}x_3^{7}x_4^{3}x_5^{5}$, & \multicolumn{1}{l}{$\mathcal Y_{22,\,910}=x_1x_2^{7}x_3^{3}x_4^{5}x_5^{6}$,} & \multicolumn{1}{l}{$\mathcal Y_{22,\,911}=x_1x_2^{7}x_3^{3}x_4^{6}x_5^{5}$,} & \multicolumn{1}{l}{$\mathcal Y_{22,\,912}=x_1x_2^{7}x_3^{6}x_4^{3}x_5^{5}$,} \\
$\mathcal Y_{22,\,913}=x_1^{3}x_2x_3^{5}x_4^{6}x_5^{7}$, & \multicolumn{1}{l}{$\mathcal Y_{22,\,914}=x_1^{3}x_2x_3^{5}x_4^{7}x_5^{6}$,} & \multicolumn{1}{l}{$\mathcal Y_{22,\,915}=x_1^{3}x_2x_3^{6}x_4^{5}x_5^{7}$,} & \multicolumn{1}{l}{$\mathcal Y_{22,\,916}=x_1^{3}x_2x_3^{6}x_4^{7}x_5^{5}$,} \\
$\mathcal Y_{22,\,917}=x_1^{3}x_2x_3^{7}x_4^{5}x_5^{6}$, & \multicolumn{1}{l}{$\mathcal Y_{22,\,918}=x_1^{3}x_2x_3^{7}x_4^{6}x_5^{5}$,} & \multicolumn{1}{l}{$\mathcal Y_{22,\,919}=x_1^{3}x_2^{5}x_3x_4^{6}x_5^{7}$,} & \multicolumn{1}{l}{$\mathcal Y_{22,\,920}=x_1^{3}x_2^{5}x_3x_4^{7}x_5^{6}$,} \\
$\mathcal Y_{22,\,921}=x_1^{3}x_2^{5}x_3^{6}x_4x_5^{7}$, & \multicolumn{1}{l}{$\mathcal Y_{22,\,922}=x_1^{3}x_2^{5}x_3^{6}x_4^{7}x_5$,} & \multicolumn{1}{l}{$\mathcal Y_{22,\,923}=x_1^{3}x_2^{5}x_3^{7}x_4x_5^{6}$,} & \multicolumn{1}{l}{$\mathcal Y_{22,\,924}=x_1^{3}x_2^{5}x_3^{7}x_4^{6}x_5$,} \\
$\mathcal Y_{22,\,925}=x_1^{3}x_2^{7}x_3x_4^{5}x_5^{6}$, & \multicolumn{1}{l}{$\mathcal Y_{22,\,926}=x_1^{3}x_2^{7}x_3x_4^{6}x_5^{5}$,} & \multicolumn{1}{l}{$\mathcal Y_{22,\,927}=x_1^{3}x_2^{7}x_3^{5}x_4x_5^{6}$,} & \multicolumn{1}{l}{$\mathcal Y_{22,\,928}=x_1^{3}x_2^{7}x_3^{5}x_4^{6}x_5$,} \\
$\mathcal Y_{22,\,929}=x_1^{7}x_2x_3^{3}x_4^{5}x_5^{6}$, & \multicolumn{1}{l}{$\mathcal Y_{22,\,930}=x_1^{7}x_2x_3^{3}x_4^{6}x_5^{5}$,} & \multicolumn{1}{l}{$\mathcal Y_{22,\,931}=x_1^{7}x_2x_3^{6}x_4^{3}x_5^{5}$,} & \multicolumn{1}{l}{$\mathcal Y_{22,\,932}=x_1^{7}x_2^{3}x_3x_4^{5}x_5^{6}$,} \\
$\mathcal Y_{22,\,933}=x_1^{7}x_2^{3}x_3x_4^{6}x_5^{5}$, & \multicolumn{1}{l}{$\mathcal Y_{22,\,934}=x_1^{7}x_2^{3}x_3^{5}x_4x_5^{6}$,} & \multicolumn{1}{l}{$\mathcal Y_{22,\,935}=x_1^{7}x_2^{3}x_3^{5}x_4^{6}x_5$,} & \multicolumn{1}{l}{$\mathcal Y_{22,\,936}=x_1^{3}x_2^{5}x_3^{2}x_4^{5}x_5^{7}$,} \\
$\mathcal Y_{22,\,937}=x_1^{3}x_2^{5}x_3^{2}x_4^{7}x_5^{5}$, & \multicolumn{1}{l}{$\mathcal Y_{22,\,938}=x_1^{3}x_2^{5}x_3^{7}x_4^{2}x_5^{5}$,} & \multicolumn{1}{l}{$\mathcal Y_{22,\,939}=x_1^{3}x_2^{7}x_3^{5}x_4^{2}x_5^{5}$,} & \multicolumn{1}{l}{$\mathcal Y_{22,\,940}=x_1^{7}x_2^{3}x_3^{5}x_4^{2}x_5^{5}$,} \\
$\mathcal Y_{22,\,941}=x_1^{3}x_2^{3}x_3^{4}x_4^{5}x_5^{7}$, & \multicolumn{1}{l}{$\mathcal Y_{22,\,942}=x_1^{3}x_2^{3}x_3^{4}x_4^{7}x_5^{5}$,} & \multicolumn{1}{l}{$\mathcal Y_{22,\,943}=x_1^{3}x_2^{3}x_3^{5}x_4^{4}x_5^{7}$,} & \multicolumn{1}{l}{$\mathcal Y_{22,\,944}=x_1^{3}x_2^{3}x_3^{5}x_4^{7}x_5^{4}$,} \\
$\mathcal Y_{22,\,945}=x_1^{3}x_2^{3}x_3^{7}x_4^{4}x_5^{5}$, & \multicolumn{1}{l}{$\mathcal Y_{22,\,946}=x_1^{3}x_2^{3}x_3^{7}x_4^{5}x_5^{4}$,} & \multicolumn{1}{l}{$\mathcal Y_{22,\,947}=x_1^{3}x_2^{4}x_3^{3}x_4^{5}x_5^{7}$,} & \multicolumn{1}{l}{$\mathcal Y_{22,\,948}=x_1^{3}x_2^{4}x_3^{3}x_4^{7}x_5^{5}$,} \\
$\mathcal Y_{22,\,949}=x_1^{3}x_2^{4}x_3^{7}x_4^{3}x_5^{5}$, & \multicolumn{1}{l}{$\mathcal Y_{22,\,950}=x_1^{3}x_2^{5}x_3^{3}x_4^{4}x_5^{7}$,} & \multicolumn{1}{l}{$\mathcal Y_{22,\,951}=x_1^{3}x_2^{5}x_3^{3}x_4^{7}x_5^{4}$,} & \multicolumn{1}{l}{$\mathcal Y_{22,\,952}=x_1^{3}x_2^{5}x_3^{7}x_4^{3}x_5^{4}$,} \\
$\mathcal Y_{22,\,953}=x_1^{3}x_2^{7}x_3^{3}x_4^{4}x_5^{5}$, & \multicolumn{1}{l}{$\mathcal Y_{22,\,954}=x_1^{3}x_2^{7}x_3^{3}x_4^{5}x_5^{4}$,} & \multicolumn{1}{l}{$\mathcal Y_{22,\,955}=x_1^{3}x_2^{7}x_3^{4}x_4^{3}x_5^{5}$,} & \multicolumn{1}{l}{$\mathcal Y_{22,\,956}=x_1^{3}x_2^{7}x_3^{5}x_4^{3}x_5^{4}$,} \\
$\mathcal Y_{22,\,957}=x_1^{7}x_2^{3}x_3^{3}x_4^{4}x_5^{5}$, & \multicolumn{1}{l}{$\mathcal Y_{22,\,958}=x_1^{7}x_2^{3}x_3^{3}x_4^{5}x_5^{4}$,} & \multicolumn{1}{l}{$\mathcal Y_{22,\,959}=x_1^{7}x_2^{3}x_3^{4}x_4^{3}x_5^{5}$,} & \multicolumn{1}{l}{$\mathcal Y_{22,\,960}=x_1^{7}x_2^{3}x_3^{5}x_4^{3}x_5^{4}$,} \\
$\mathcal Y_{22,\,961}=x_1^{3}x_2^{3}x_3^{5}x_4^{5}x_5^{6}$, & \multicolumn{1}{l}{$\mathcal Y_{22,\,962}=x_1^{3}x_2^{3}x_3^{5}x_4^{6}x_5^{5}$,} & \multicolumn{1}{l}{$\mathcal Y_{22,\,963}=x_1^{3}x_2^{5}x_3^{3}x_4^{5}x_5^{6}$,} & \multicolumn{1}{l}{$\mathcal Y_{22,\,964}=x_1^{3}x_2^{5}x_3^{3}x_4^{6}x_5^{5}$,} \\
$\mathcal Y_{22,\,965}=x_1^{3}x_2^{5}x_3^{6}x_4^{3}x_5^{5}$. &       &       &  
\end{tabular}%
\end{center}

\newpage
From the above data, we see that $\overline{\Phi}(\mathscr B_4(\omega_{(j}) )\subset \mathscr B_5(\omega_{(j)})$ for $j = 1, 4, 5.$ Hence, Conjecture \ref{gtS} is true for $d = 5$ and the degree $22.$

\subsection{$\mathcal A_2$-generators for $\mathscr P_5^0$ in degree $47$}\label{s5.6}

Using the results in Sect.\ref{s3.2.2}, we have 
$$\mathscr B_5^0(47) = \overline{\Phi}(\mathscr B_4(47)) = \overline{\Phi}(\mathscr B_4(\overline{\omega}_{(1)}) = \{\mathcal Y_{47,\, t};\ 1\leq t\leq 560\},$$ 
where $\overline{\omega}_{(1)} = (3,2,2,2,1),$ and the monomials $\mathcal Y_{47,\,t}:\ 1 \leq t\leq 560,$ are determined as follows:

\begin{center}
\begin{tabular}{llll}
$\mathcal Y_{47,\,1}=x_3x_4^{15}x_5^{31}$, & $\mathcal Y_{47,\,2}=x_3x_4^{31}x_5^{15}$, & $\mathcal Y_{47,\,3}=x_3^{15}x_4x_5^{31}$, & $\mathcal Y_{47,\,4}=x_3^{15}x_4^{31}x_5$, \\
$\mathcal Y_{47,\,5}=x_3^{31}x_4x_5^{15}$, & $\mathcal Y_{47,\,6}=x_3^{31}x_4^{15}x_5$, & $\mathcal Y_{47,\,7}=x_2x_4^{15}x_5^{31}$, & $\mathcal Y_{47,\,8}=x_2x_4^{31}x_5^{15}$, \\
$\mathcal Y_{47,\,9}=x_2x_3^{15}x_5^{31}$, & $\mathcal Y_{47,\,10}=x_2x_3^{15}x_4^{31}$, & $\mathcal Y_{47,\,11}=x_2x_3^{31}x_5^{15}$, & $\mathcal Y_{47,\,12}=x_2x_3^{31}x_4^{15}$, \\
$\mathcal Y_{47,\,13}=x_2^{15}x_4x_5^{31}$, & $\mathcal Y_{47,\,14}=x_2^{15}x_4^{31}x_5$, & $\mathcal Y_{47,\,15}=x_2^{15}x_3x_5^{31}$, & $\mathcal Y_{47,\,16}=x_2^{15}x_3x_4^{31}$, \\
$\mathcal Y_{47,\,17}=x_2^{15}x_3^{31}x_5$, & $\mathcal Y_{47,\,18}=x_2^{15}x_3^{31}x_4$, & $\mathcal Y_{47,\,19}=x_2^{31}x_4x_5^{15}$, & $\mathcal Y_{47,\,20}=x_2^{31}x_4^{15}x_5$, \\
$\mathcal Y_{47,\,21}=x_2^{31}x_3x_5^{15}$, & $\mathcal Y_{47,\,22}=x_2^{31}x_3x_4^{15}$, & $\mathcal Y_{47,\,23}=x_2^{31}x_3^{15}x_5$, & $\mathcal Y_{47,\,24}=x_2^{31}x_3^{15}x_4$, \\
$\mathcal Y_{47,\,25}=x_1x_4^{15}x_5^{31}$, & $\mathcal Y_{47,\,26}=x_1x_4^{31}x_5^{15}$, & $\mathcal Y_{47,\,27}=x_1x_3^{15}x_5^{31}$, & $\mathcal Y_{47,\,28}=x_1x_3^{15}x_4^{31}$, \\
$\mathcal Y_{47,\,29}=x_1x_3^{31}x_5^{15}$, & $\mathcal Y_{47,\,30}=x_1x_3^{31}x_4^{15}$, & $\mathcal Y_{47,\,31}=x_1x_2^{15}x_5^{31}$, & $\mathcal Y_{47,\,32}=x_1x_2^{15}x_4^{31}$, \\
$\mathcal Y_{47,\,33}=x_1x_2^{15}x_3^{31}$, & $\mathcal Y_{47,\,34}=x_1x_2^{31}x_5^{15}$, & $\mathcal Y_{47,\,35}=x_1x_2^{31}x_4^{15}$, & $\mathcal Y_{47,\,36}=x_1x_2^{31}x_3^{15}$, \\
$\mathcal Y_{47,\,37}=x_1^{15}x_4x_5^{31}$, & $\mathcal Y_{47,\,38}=x_1^{15}x_4^{31}x_5$, & $\mathcal Y_{47,\,39}=x_1^{15}x_3x_5^{31}$, & $\mathcal Y_{47,\,40}=x_1^{15}x_3x_4^{31}$, \\
$\mathcal Y_{47,\,41}=x_1^{15}x_3^{31}x_5$, & $\mathcal Y_{47,\,42}=x_1^{15}x_3^{31}x_4$, & $\mathcal Y_{47,\,43}=x_1^{15}x_2x_5^{31}$, & $\mathcal Y_{47,\,44}=x_1^{15}x_2x_4^{31}$, \\
$\mathcal Y_{47,\,45}=x_1^{15}x_2x_3^{31}$, & $\mathcal Y_{47,\,46}=x_1^{15}x_2^{31}x_5$, & $\mathcal Y_{47,\,47}=x_1^{15}x_2^{31}x_4$, & $\mathcal Y_{47,\,48}=x_1^{15}x_2^{31}x_3$, \\
$\mathcal Y_{47,\,49}=x_1^{31}x_4x_5^{15}$, & $\mathcal Y_{47,\,50}=x_1^{31}x_4^{15}x_5$, & $\mathcal Y_{47,\,51}=x_1^{31}x_3x_5^{15}$, & $\mathcal Y_{47,\,52}=x_1^{31}x_3x_4^{15}$, \\
$\mathcal Y_{47,\,53}=x_1^{31}x_3^{15}x_5$, & $\mathcal Y_{47,\,54}=x_1^{31}x_3^{15}x_4$, & $\mathcal Y_{47,\,55}=x_1^{31}x_2x_5^{15}$, & $\mathcal Y_{47,\,56}=x_1^{31}x_2x_4^{15}$, \\
$\mathcal Y_{47,\,57}=x_1^{31}x_2x_3^{15}$, & $\mathcal Y_{47,\,58}=x_1^{31}x_2^{15}x_5$, & $\mathcal Y_{47,\,59}=x_1^{31}x_2^{15}x_4$, & $\mathcal Y_{47,\,60}=x_1^{31}x_2^{15}x_3$, \\
$\mathcal Y_{47,\,61}=x_3^{3}x_4^{13}x_5^{31}$, & $\mathcal Y_{47,\,62}=x_3^{3}x_4^{31}x_5^{13}$, & $\mathcal Y_{47,\,63}=x_3^{31}x_4^{3}x_5^{13}$, & $\mathcal Y_{47,\,64}=x_2^{3}x_4^{13}x_5^{31}$, \\
$\mathcal Y_{47,\,65}=x_2^{3}x_4^{31}x_5^{13}$, & $\mathcal Y_{47,\,66}=x_2^{3}x_3^{13}x_5^{31}$, & $\mathcal Y_{47,\,67}=x_2^{3}x_3^{13}x_4^{31}$, & $\mathcal Y_{47,\,68}=x_2^{3}x_3^{31}x_5^{13}$, \\
$\mathcal Y_{47,\,69}=x_2^{3}x_3^{31}x_4^{13}$, & $\mathcal Y_{47,\,70}=x_2^{31}x_4^{3}x_5^{13}$, & $\mathcal Y_{47,\,71}=x_2^{31}x_3^{3}x_5^{13}$, & $\mathcal Y_{47,\,72}=x_2^{31}x_3^{3}x_4^{13}$, \\
$\mathcal Y_{47,\,73}=x_1^{3}x_4^{13}x_5^{31}$, & $\mathcal Y_{47,\,74}=x_1^{3}x_4^{31}x_5^{13}$, & $\mathcal Y_{47,\,75}=x_1^{3}x_3^{13}x_5^{31}$, & $\mathcal Y_{47,\,76}=x_1^{3}x_3^{13}x_4^{31}$, \\
$\mathcal Y_{47,\,77}=x_1^{3}x_3^{31}x_5^{13}$, & $\mathcal Y_{47,\,78}=x_1^{3}x_3^{31}x_4^{13}$, & $\mathcal Y_{47,\,79}=x_1^{3}x_2^{13}x_5^{31}$, & $\mathcal Y_{47,\,80}=x_1^{3}x_2^{13}x_4^{31}$, \\
$\mathcal Y_{47,\,81}=x_1^{3}x_2^{13}x_3^{31}$, & $\mathcal Y_{47,\,82}=x_1^{3}x_2^{31}x_5^{13}$, & $\mathcal Y_{47,\,83}=x_1^{3}x_2^{31}x_4^{13}$, & $\mathcal Y_{47,\,84}=x_1^{3}x_2^{31}x_3^{13}$, \\
$\mathcal Y_{47,\,85}=x_1^{31}x_4^{3}x_5^{13}$, & $\mathcal Y_{47,\,86}=x_1^{31}x_3^{3}x_5^{13}$, & $\mathcal Y_{47,\,87}=x_1^{31}x_3^{3}x_4^{13}$, & $\mathcal Y_{47,\,88}=x_1^{31}x_2^{3}x_5^{13}$, \\
$\mathcal Y_{47,\,89}=x_1^{31}x_2^{3}x_4^{13}$, & $\mathcal Y_{47,\,90}=x_1^{31}x_2^{3}x_3^{13}$, & $\mathcal Y_{47,\,91}=x_3^{3}x_4^{15}x_5^{29}$, & $\mathcal Y_{47,\,92}=x_3^{3}x_4^{29}x_5^{15}$, \\
$\mathcal Y_{47,\,93}=x_3^{15}x_4^{3}x_5^{29}$, & $\mathcal Y_{47,\,94}=x_2^{3}x_4^{15}x_5^{29}$, & $\mathcal Y_{47,\,95}=x_2^{3}x_4^{29}x_5^{15}$, & $\mathcal Y_{47,\,96}=x_2^{3}x_3^{15}x_5^{29}$, \\
$\mathcal Y_{47,\,97}=x_2^{3}x_3^{15}x_4^{29}$, & $\mathcal Y_{47,\,98}=x_2^{3}x_3^{29}x_5^{15}$, & $\mathcal Y_{47,\,99}=x_2^{3}x_3^{29}x_4^{15}$, & $\mathcal Y_{47,\,100}=x_2^{15}x_4^{3}x_5^{29}$, \\
$\mathcal Y_{47,\,101}=x_2^{15}x_3^{3}x_5^{29}$, & $\mathcal Y_{47,\,102}=x_2^{15}x_3^{3}x_4^{29}$, & $\mathcal Y_{47,\,103}=x_1^{3}x_4^{15}x_5^{29}$, & $\mathcal Y_{47,\,104}=x_1^{3}x_4^{29}x_5^{15}$, \\
$\mathcal Y_{47,\,105}=x_1^{3}x_3^{15}x_5^{29}$, & $\mathcal Y_{47,\,106}=x_1^{3}x_3^{15}x_4^{29}$, & $\mathcal Y_{47,\,107}=x_1^{3}x_3^{29}x_5^{15}$, & $\mathcal Y_{47,\,108}=x_1^{3}x_3^{29}x_4^{15}$, \\
$\mathcal Y_{47,\,109}=x_1^{3}x_2^{15}x_5^{29}$, & $\mathcal Y_{47,\,110}=x_1^{3}x_2^{15}x_4^{29}$, & $\mathcal Y_{47,\,111}=x_1^{3}x_2^{15}x_3^{29}$, & $\mathcal Y_{47,\,112}=x_1^{3}x_2^{29}x_5^{15}$, \\
$\mathcal Y_{47,\,113}=x_1^{3}x_2^{29}x_4^{15}$, & $\mathcal Y_{47,\,114}=x_1^{3}x_2^{29}x_3^{15}$, & $\mathcal Y_{47,\,115}=x_1^{15}x_4^{3}x_5^{29}$, & $\mathcal Y_{47,\,116}=x_1^{15}x_3^{3}x_5^{29}$, \\
$\mathcal Y_{47,\,117}=x_1^{15}x_3^{3}x_4^{29}$, & $\mathcal Y_{47,\,118}=x_1^{15}x_2^{3}x_5^{29}$, & $\mathcal Y_{47,\,119}=x_1^{15}x_2^{3}x_4^{29}$, & $\mathcal Y_{47,\,120}=x_1^{15}x_2^{3}x_3^{29}$, \\
$\mathcal Y_{47,\,121}=x_3^{7}x_4^{11}x_5^{29}$, & $\mathcal Y_{47,\,122}=x_2^{7}x_4^{11}x_5^{29}$, & $\mathcal Y_{47,\,123}=x_2^{7}x_3^{11}x_5^{29}$, & $\mathcal Y_{47,\,124}=x_2^{7}x_3^{11}x_4^{29}$, \\
$\mathcal Y_{47,\,125}=x_1^{7}x_4^{11}x_5^{29}$, & $\mathcal Y_{47,\,126}=x_1^{7}x_3^{11}x_5^{29}$, & $\mathcal Y_{47,\,127}=x_1^{7}x_3^{11}x_4^{29}$, & $\mathcal Y_{47,\,128}=x_1^{7}x_2^{11}x_5^{29}$, \\
$\mathcal Y_{47,\,129}=x_1^{7}x_2^{11}x_4^{29}$, & $\mathcal Y_{47,\,130}=x_1^{7}x_2^{11}x_3^{29}$, & $\mathcal Y_{47,\,131}=x_3^{7}x_4^{27}x_5^{13}$, & $\mathcal Y_{47,\,132}=x_2^{7}x_4^{27}x_5^{13}$, \\
$\mathcal Y_{47,\,133}=x_2^{7}x_3^{27}x_5^{13}$, & $\mathcal Y_{47,\,134}=x_2^{7}x_3^{27}x_4^{13}$, & $\mathcal Y_{47,\,135}=x_1^{7}x_4^{27}x_5^{13}$, & $\mathcal Y_{47,\,136}=x_1^{7}x_3^{27}x_5^{13}$, \\
$\mathcal Y_{47,\,137}=x_1^{7}x_3^{27}x_4^{13}$, & $\mathcal Y_{47,\,138}=x_1^{7}x_2^{27}x_5^{13}$, & $\mathcal Y_{47,\,139}=x_1^{7}x_2^{27}x_4^{13}$, & $\mathcal Y_{47,\,140}=x_1^{7}x_2^{27}x_3^{13}$, \\
$\mathcal Y_{47,\,141}=x_2x_3x_4^{14}x_5^{31}$, & $\mathcal Y_{47,\,142}=x_2x_3x_4^{31}x_5^{14}$, & $\mathcal Y_{47,\,143}=x_2x_3^{14}x_4x_5^{31}$, & $\mathcal Y_{47,\,144}=x_2x_3^{14}x_4^{31}x_5$, \\
$\mathcal Y_{47,\,145}=x_2x_3^{31}x_4x_5^{14}$, & $\mathcal Y_{47,\,146}=x_2x_3^{31}x_4^{14}x_5$, & $\mathcal Y_{47,\,147}=x_2^{31}x_3x_4x_5^{14}$, & $\mathcal Y_{47,\,148}=x_2^{31}x_3x_4^{14}x_5$, \\
$\mathcal Y_{47,\,149}=x_1x_3x_4^{14}x_5^{31}$, & $\mathcal Y_{47,\,150}=x_1x_3x_4^{31}x_5^{14}$, & $\mathcal Y_{47,\,151}=x_1x_3^{14}x_4x_5^{31}$, & $\mathcal Y_{47,\,152}=x_1x_3^{14}x_4^{31}x_5$, \\
$\mathcal Y_{47,\,153}=x_1x_3^{31}x_4x_5^{14}$, & $\mathcal Y_{47,\,154}=x_1x_3^{31}x_4^{14}x_5$, & $\mathcal Y_{47,\,155}=x_1x_2x_4^{14}x_5^{31}$, & $\mathcal Y_{47,\,156}=x_1x_2x_4^{31}x_5^{14}$, \\
\end{tabular}%
\end{center}

\newpage
\begin{center}
\begin{tabular}{llll}
$\mathcal Y_{47,\,157}=x_1x_2x_3^{14}x_5^{31}$, & $\mathcal Y_{47,\,158}=x_1x_2x_3^{14}x_4^{31}$, & $\mathcal Y_{47,\,159}=x_1x_2x_3^{31}x_5^{14}$, & $\mathcal Y_{47,\,160}=x_1x_2x_3^{31}x_4^{14}$, \\
$\mathcal Y_{47,\,161}=x_1x_2^{14}x_4x_5^{31}$, & $\mathcal Y_{47,\,162}=x_1x_2^{14}x_4^{31}x_5$, & $\mathcal Y_{47,\,163}=x_1x_2^{14}x_3x_5^{31}$, & $\mathcal Y_{47,\,164}=x_1x_2^{14}x_3x_4^{31}$, \\
$\mathcal Y_{47,\,165}=x_1x_2^{14}x_3^{31}x_5$, & $\mathcal Y_{47,\,166}=x_1x_2^{14}x_3^{31}x_4$, & $\mathcal Y_{47,\,167}=x_1x_2^{31}x_4x_5^{14}$, & $\mathcal Y_{47,\,168}=x_1x_2^{31}x_4^{14}x_5$, \\
$\mathcal Y_{47,\,169}=x_1x_2^{31}x_3x_5^{14}$, & $\mathcal Y_{47,\,170}=x_1x_2^{31}x_3x_4^{14}$, & $\mathcal Y_{47,\,171}=x_1x_2^{31}x_3^{14}x_5$, & $\mathcal Y_{47,\,172}=x_1x_2^{31}x_3^{14}x_4$, \\
$\mathcal Y_{47,\,173}=x_1^{31}x_3x_4x_5^{14}$, & $\mathcal Y_{47,\,174}=x_1^{31}x_3x_4^{14}x_5$, & $\mathcal Y_{47,\,175}=x_1^{31}x_2x_4x_5^{14}$, & $\mathcal Y_{47,\,176}=x_1^{31}x_2x_4^{14}x_5$, \\
$\mathcal Y_{47,\,177}=x_1^{31}x_2x_3x_5^{14}$, & $\mathcal Y_{47,\,178}=x_1^{31}x_2x_3x_4^{14}$, & $\mathcal Y_{47,\,179}=x_1^{31}x_2x_3^{14}x_5$, & $\mathcal Y_{47,\,180}=x_1^{31}x_2x_3^{14}x_4$, \\
$\mathcal Y_{47,\,181}=x_2x_3x_4^{15}x_5^{30}$, & $\mathcal Y_{47,\,182}=x_2x_3x_4^{30}x_5^{15}$, & $\mathcal Y_{47,\,183}=x_2x_3^{15}x_4x_5^{30}$, & $\mathcal Y_{47,\,184}=x_2x_3^{15}x_4^{30}x_5$, \\
$\mathcal Y_{47,\,185}=x_2x_3^{30}x_4x_5^{15}$, & $\mathcal Y_{47,\,186}=x_2x_3^{30}x_4^{15}x_5$, & $\mathcal Y_{47,\,187}=x_2^{15}x_3x_4x_5^{30}$, & $\mathcal Y_{47,\,188}=x_2^{15}x_3x_4^{30}x_5$, \\
$\mathcal Y_{47,\,189}=x_1x_3x_4^{15}x_5^{30}$, & $\mathcal Y_{47,\,190}=x_1x_3x_4^{30}x_5^{15}$, & $\mathcal Y_{47,\,191}=x_1x_3^{15}x_4x_5^{30}$, & $\mathcal Y_{47,\,192}=x_1x_3^{15}x_4^{30}x_5$, \\
$\mathcal Y_{47,\,193}=x_1x_3^{30}x_4x_5^{15}$, & $\mathcal Y_{47,\,194}=x_1x_3^{30}x_4^{15}x_5$, & $\mathcal Y_{47,\,195}=x_1x_2x_4^{15}x_5^{30}$, & $\mathcal Y_{47,\,196}=x_1x_2x_4^{30}x_5^{15}$, \\
$\mathcal Y_{47,\,197}=x_1x_2x_3^{15}x_5^{30}$, & $\mathcal Y_{47,\,198}=x_1x_2x_3^{15}x_4^{30}$, & $\mathcal Y_{47,\,199}=x_1x_2x_3^{30}x_5^{15}$, & $\mathcal Y_{47,\,200}=x_1x_2x_3^{30}x_4^{15}$, \\
$\mathcal Y_{47,\,201}=x_1x_2^{15}x_4x_5^{30}$, & $\mathcal Y_{47,\,202}=x_1x_2^{15}x_4^{30}x_5$, & $\mathcal Y_{47,\,203}=x_1x_2^{15}x_3x_5^{30}$, & $\mathcal Y_{47,\,204}=x_1x_2^{15}x_3x_4^{30}$, \\
$\mathcal Y_{47,\,205}=x_1x_2^{15}x_3^{30}x_5$, & $\mathcal Y_{47,\,206}=x_1x_2^{15}x_3^{30}x_4$, & $\mathcal Y_{47,\,207}=x_1x_2^{30}x_4x_5^{15}$, & $\mathcal Y_{47,\,208}=x_1x_2^{30}x_4^{15}x_5$, \\
$\mathcal Y_{47,\,209}=x_1x_2^{30}x_3x_5^{15}$, & $\mathcal Y_{47,\,210}=x_1x_2^{30}x_3x_4^{15}$, & $\mathcal Y_{47,\,211}=x_1x_2^{30}x_3^{15}x_5$, & $\mathcal Y_{47,\,212}=x_1x_2^{30}x_3^{15}x_4$, \\
$\mathcal Y_{47,\,213}=x_1^{15}x_3x_4x_5^{30}$, & $\mathcal Y_{47,\,214}=x_1^{15}x_3x_4^{30}x_5$, & $\mathcal Y_{47,\,215}=x_1^{15}x_2x_4x_5^{30}$, & $\mathcal Y_{47,\,216}=x_1^{15}x_2x_4^{30}x_5$, \\
$\mathcal Y_{47,\,217}=x_1^{15}x_2x_3x_5^{30}$, & $\mathcal Y_{47,\,218}=x_1^{15}x_2x_3x_4^{30}$, & $\mathcal Y_{47,\,219}=x_1^{15}x_2x_3^{30}x_5$, & $\mathcal Y_{47,\,220}=x_1^{15}x_2x_3^{30}x_4$, \\
$\mathcal Y_{47,\,221}=x_2x_3^{2}x_4^{13}x_5^{31}$, & $\mathcal Y_{47,\,222}=x_2x_3^{2}x_4^{31}x_5^{13}$, & $\mathcal Y_{47,\,223}=x_2x_3^{31}x_4^{2}x_5^{13}$, & $\mathcal Y_{47,\,224}=x_2^{31}x_3x_4^{2}x_5^{13}$, \\
$\mathcal Y_{47,\,225}=x_1x_3^{2}x_4^{13}x_5^{31}$, & $\mathcal Y_{47,\,226}=x_1x_3^{2}x_4^{31}x_5^{13}$, & $\mathcal Y_{47,\,227}=x_1x_3^{31}x_4^{2}x_5^{13}$, & $\mathcal Y_{47,\,228}=x_1x_2^{2}x_4^{13}x_5^{31}$, \\
$\mathcal Y_{47,\,229}=x_1x_2^{2}x_4^{31}x_5^{13}$, & $\mathcal Y_{47,\,230}=x_1x_2^{2}x_3^{13}x_5^{31}$, & $\mathcal Y_{47,\,231}=x_1x_2^{2}x_3^{13}x_4^{31}$, & $\mathcal Y_{47,\,232}=x_1x_2^{2}x_3^{31}x_5^{13}$, \\
$\mathcal Y_{47,\,233}=x_1x_2^{2}x_3^{31}x_4^{13}$, & $\mathcal Y_{47,\,234}=x_1x_2^{31}x_4^{2}x_5^{13}$, & $\mathcal Y_{47,\,235}=x_1x_2^{31}x_3^{2}x_5^{13}$, & $\mathcal Y_{47,\,236}=x_1x_2^{31}x_3^{2}x_4^{13}$, \\
$\mathcal Y_{47,\,237}=x_1^{31}x_3x_4^{2}x_5^{13}$, & $\mathcal Y_{47,\,238}=x_1^{31}x_2x_4^{2}x_5^{13}$, & $\mathcal Y_{47,\,239}=x_1^{31}x_2x_3^{2}x_5^{13}$, & $\mathcal Y_{47,\,240}=x_1^{31}x_2x_3^{2}x_4^{13}$, \\
$\mathcal Y_{47,\,241}=x_2x_3^{2}x_4^{15}x_5^{29}$, & $\mathcal Y_{47,\,242}=x_2x_3^{2}x_4^{29}x_5^{15}$, & $\mathcal Y_{47,\,243}=x_2x_3^{15}x_4^{2}x_5^{29}$, & $\mathcal Y_{47,\,244}=x_2^{15}x_3x_4^{2}x_5^{29}$, \\
$\mathcal Y_{47,\,245}=x_1x_3^{2}x_4^{15}x_5^{29}$, & $\mathcal Y_{47,\,246}=x_1x_3^{2}x_4^{29}x_5^{15}$, & $\mathcal Y_{47,\,247}=x_1x_3^{15}x_4^{2}x_5^{29}$, & $\mathcal Y_{47,\,248}=x_1x_2^{2}x_4^{15}x_5^{29}$, \\
$\mathcal Y_{47,\,249}=x_1x_2^{2}x_4^{29}x_5^{15}$, & $\mathcal Y_{47,\,250}=x_1x_2^{2}x_3^{15}x_5^{29}$, & $\mathcal Y_{47,\,251}=x_1x_2^{2}x_3^{15}x_4^{29}$, & $\mathcal Y_{47,\,252}=x_1x_2^{2}x_3^{29}x_5^{15}$, \\
$\mathcal Y_{47,\,253}=x_1x_2^{2}x_3^{29}x_4^{15}$, & $\mathcal Y_{47,\,254}=x_1x_2^{15}x_4^{2}x_5^{29}$, & $\mathcal Y_{47,\,255}=x_1x_2^{15}x_3^{2}x_5^{29}$, & $\mathcal Y_{47,\,256}=x_1x_2^{15}x_3^{2}x_4^{29}$, \\
$\mathcal Y_{47,\,257}=x_1^{15}x_3x_4^{2}x_5^{29}$, & $\mathcal Y_{47,\,258}=x_1^{15}x_2x_4^{2}x_5^{29}$, & $\mathcal Y_{47,\,259}=x_1^{15}x_2x_3^{2}x_5^{29}$, & $\mathcal Y_{47,\,260}=x_1^{15}x_2x_3^{2}x_4^{29}$, \\
$\mathcal Y_{47,\,261}=x_2x_3^{3}x_4^{12}x_5^{31}$, & $\mathcal Y_{47,\,262}=x_2x_3^{3}x_4^{31}x_5^{12}$, & $\mathcal Y_{47,\,263}=x_2x_3^{31}x_4^{3}x_5^{12}$, & $\mathcal Y_{47,\,264}=x_2^{3}x_3x_4^{12}x_5^{31}$, \\
$\mathcal Y_{47,\,265}=x_2^{3}x_3x_4^{31}x_5^{12}$, & $\mathcal Y_{47,\,266}=x_2^{3}x_3^{31}x_4x_5^{12}$, & $\mathcal Y_{47,\,267}=x_2^{31}x_3x_4^{3}x_5^{12}$, & $\mathcal Y_{47,\,268}=x_2^{31}x_3^{3}x_4x_5^{12}$, \\
$\mathcal Y_{47,\,269}=x_1x_3^{3}x_4^{12}x_5^{31}$, & $\mathcal Y_{47,\,270}=x_1x_3^{3}x_4^{31}x_5^{12}$, & $\mathcal Y_{47,\,271}=x_1x_3^{31}x_4^{3}x_5^{12}$, & $\mathcal Y_{47,\,272}=x_1x_2^{3}x_4^{12}x_5^{31}$, \\
$\mathcal Y_{47,\,273}=x_1x_2^{3}x_4^{31}x_5^{12}$, & $\mathcal Y_{47,\,274}=x_1x_2^{3}x_3^{12}x_5^{31}$, & $\mathcal Y_{47,\,275}=x_1x_2^{3}x_3^{12}x_4^{31}$, & $\mathcal Y_{47,\,276}=x_1x_2^{3}x_3^{31}x_5^{12}$, \\
$\mathcal Y_{47,\,277}=x_1x_2^{3}x_3^{31}x_4^{12}$, & $\mathcal Y_{47,\,278}=x_1x_2^{31}x_4^{3}x_5^{12}$, & $\mathcal Y_{47,\,279}=x_1x_2^{31}x_3^{3}x_5^{12}$, & $\mathcal Y_{47,\,280}=x_1x_2^{31}x_3^{3}x_4^{12}$, \\
$\mathcal Y_{47,\,281}=x_1^{3}x_3x_4^{12}x_5^{31}$, & $\mathcal Y_{47,\,282}=x_1^{3}x_3x_4^{31}x_5^{12}$, & $\mathcal Y_{47,\,283}=x_1^{3}x_3^{31}x_4x_5^{12}$, & $\mathcal Y_{47,\,284}=x_1^{3}x_2x_4^{12}x_5^{31}$, \\
$\mathcal Y_{47,\,285}=x_1^{3}x_2x_4^{31}x_5^{12}$, & $\mathcal Y_{47,\,286}=x_1^{3}x_2x_3^{12}x_5^{31}$, & $\mathcal Y_{47,\,287}=x_1^{3}x_2x_3^{12}x_4^{31}$, & $\mathcal Y_{47,\,288}=x_1^{3}x_2x_3^{31}x_5^{12}$, \\
$\mathcal Y_{47,\,289}=x_1^{3}x_2x_3^{31}x_4^{12}$, & $\mathcal Y_{47,\,290}=x_1^{3}x_2^{31}x_4x_5^{12}$, & $\mathcal Y_{47,\,291}=x_1^{3}x_2^{31}x_3x_5^{12}$, & $\mathcal Y_{47,\,292}=x_1^{3}x_2^{31}x_3x_4^{12}$, \\
$\mathcal Y_{47,\,293}=x_1^{31}x_3x_4^{3}x_5^{12}$, & $\mathcal Y_{47,\,294}=x_1^{31}x_3^{3}x_4x_5^{12}$, & $\mathcal Y_{47,\,295}=x_1^{31}x_2x_4^{3}x_5^{12}$, & $\mathcal Y_{47,\,296}=x_1^{31}x_2x_3^{3}x_5^{12}$, \\
$\mathcal Y_{47,\,297}=x_1^{31}x_2x_3^{3}x_4^{12}$, & $\mathcal Y_{47,\,298}=x_1^{31}x_2^{3}x_4x_5^{12}$, & $\mathcal Y_{47,\,299}=x_1^{31}x_2^{3}x_3x_5^{12}$, & $\mathcal Y_{47,\,300}=x_1^{31}x_2^{3}x_3x_4^{12}$, \\
$\mathcal Y_{47,\,301}=x_2x_3^{3}x_4^{13}x_5^{30}$, & $\mathcal Y_{47,\,302}=x_2x_3^{3}x_4^{30}x_5^{13}$, & $\mathcal Y_{47,\,303}=x_2x_3^{30}x_4^{3}x_5^{13}$, & $\mathcal Y_{47,\,304}=x_2^{3}x_3x_4^{13}x_5^{30}$, \\
$\mathcal Y_{47,\,305}=x_2^{3}x_3x_4^{30}x_5^{13}$, & $\mathcal Y_{47,\,306}=x_2^{3}x_3^{13}x_4x_5^{30}$, & $\mathcal Y_{47,\,307}=x_2^{3}x_3^{13}x_4^{30}x_5$, & $\mathcal Y_{47,\,308}=x_1x_3^{3}x_4^{13}x_5^{30}$, \\
$\mathcal Y_{47,\,309}=x_1x_3^{3}x_4^{30}x_5^{13}$, & $\mathcal Y_{47,\,310}=x_1x_3^{30}x_4^{3}x_5^{13}$, & $\mathcal Y_{47,\,311}=x_1x_2^{3}x_4^{13}x_5^{30}$, & $\mathcal Y_{47,\,312}=x_1x_2^{3}x_4^{30}x_5^{13}$, \\
$\mathcal Y_{47,\,313}=x_1x_2^{3}x_3^{13}x_5^{30}$, & $\mathcal Y_{47,\,314}=x_1x_2^{3}x_3^{13}x_4^{30}$, & $\mathcal Y_{47,\,315}=x_1x_2^{3}x_3^{30}x_5^{13}$, & $\mathcal Y_{47,\,316}=x_1x_2^{3}x_3^{30}x_4^{13}$, \\
$\mathcal Y_{47,\,317}=x_1x_2^{30}x_4^{3}x_5^{13}$, & $\mathcal Y_{47,\,318}=x_1x_2^{30}x_3^{3}x_5^{13}$, & $\mathcal Y_{47,\,319}=x_1x_2^{30}x_3^{3}x_4^{13}$, & $\mathcal Y_{47,\,320}=x_1^{3}x_3x_4^{13}x_5^{30}$, \\
$\mathcal Y_{47,\,321}=x_1^{3}x_3x_4^{30}x_5^{13}$, & $\mathcal Y_{47,\,322}=x_1^{3}x_3^{13}x_4x_5^{30}$, & $\mathcal Y_{47,\,323}=x_1^{3}x_3^{13}x_4^{30}x_5$, & $\mathcal Y_{47,\,324}=x_1^{3}x_2x_4^{13}x_5^{30}$, \\
$\mathcal Y_{47,\,325}=x_1^{3}x_2x_4^{30}x_5^{13}$, & $\mathcal Y_{47,\,326}=x_1^{3}x_2x_3^{13}x_5^{30}$, & $\mathcal Y_{47,\,327}=x_1^{3}x_2x_3^{13}x_4^{30}$, & $\mathcal Y_{47,\,328}=x_1^{3}x_2x_3^{30}x_5^{13}$, \\
$\mathcal Y_{47,\,329}=x_1^{3}x_2x_3^{30}x_4^{13}$, & $\mathcal Y_{47,\,330}=x_1^{3}x_2^{13}x_4x_5^{30}$, & $\mathcal Y_{47,\,331}=x_1^{3}x_2^{13}x_4^{30}x_5$, & $\mathcal Y_{47,\,332}=x_1^{3}x_2^{13}x_3x_5^{30}$, \\
$\mathcal Y_{47,\,333}=x_1^{3}x_2^{13}x_3x_4^{30}$, & $\mathcal Y_{47,\,334}=x_1^{3}x_2^{13}x_3^{30}x_5$, & $\mathcal Y_{47,\,335}=x_1^{3}x_2^{13}x_3^{30}x_4$, & $\mathcal Y_{47,\,336}=x_2x_3^{3}x_4^{14}x_5^{29}$, \\
$\mathcal Y_{47,\,337}=x_2x_3^{3}x_4^{29}x_5^{14}$, & $\mathcal Y_{47,\,338}=x_2x_3^{14}x_4^{3}x_5^{29}$, & $\mathcal Y_{47,\,339}=x_2^{3}x_3x_4^{14}x_5^{29}$, & $\mathcal Y_{47,\,340}=x_2^{3}x_3x_4^{29}x_5^{14}$, \\
$\mathcal Y_{47,\,341}=x_2^{3}x_3^{29}x_4x_5^{14}$, & $\mathcal Y_{47,\,342}=x_2^{3}x_3^{29}x_4^{14}x_5$, & $\mathcal Y_{47,\,343}=x_1x_3^{3}x_4^{14}x_5^{29}$, & $\mathcal Y_{47,\,344}=x_1x_3^{3}x_4^{29}x_5^{14}$, \\
$\mathcal Y_{47,\,345}=x_1x_3^{14}x_4^{3}x_5^{29}$, & $\mathcal Y_{47,\,346}=x_1x_2^{3}x_4^{14}x_5^{29}$, & $\mathcal Y_{47,\,347}=x_1x_2^{3}x_4^{29}x_5^{14}$, & $\mathcal Y_{47,\,348}=x_1x_2^{3}x_3^{14}x_5^{29}$, \\
$\mathcal Y_{47,\,349}=x_1x_2^{3}x_3^{14}x_4^{29}$, & $\mathcal Y_{47,\,350}=x_1x_2^{3}x_3^{29}x_5^{14}$, & $\mathcal Y_{47,\,351}=x_1x_2^{3}x_3^{29}x_4^{14}$, & $\mathcal Y_{47,\,352}=x_1x_2^{14}x_4^{3}x_5^{29}$, \\
$\mathcal Y_{47,\,353}=x_1x_2^{14}x_3^{3}x_5^{29}$, & $\mathcal Y_{47,\,354}=x_1x_2^{14}x_3^{3}x_4^{29}$, & $\mathcal Y_{47,\,355}=x_1^{3}x_3x_4^{14}x_5^{29}$, & $\mathcal Y_{47,\,356}=x_1^{3}x_3x_4^{29}x_5^{14}$, \\
$\mathcal Y_{47,\,357}=x_1^{3}x_3^{29}x_4x_5^{14}$, & $\mathcal Y_{47,\,358}=x_1^{3}x_3^{29}x_4^{14}x_5$, & $\mathcal Y_{47,\,359}=x_1^{3}x_2x_4^{14}x_5^{29}$, & $\mathcal Y_{47,\,360}=x_1^{3}x_2x_4^{29}x_5^{14}$, \\
\end{tabular}%
\end{center}

\newpage
\begin{center}
\begin{tabular}{llll}
$\mathcal Y_{47,\,361}=x_1^{3}x_2x_3^{14}x_5^{29}$, & $\mathcal Y_{47,\,362}=x_1^{3}x_2x_3^{14}x_4^{29}$, & $\mathcal Y_{47,\,363}=x_1^{3}x_2x_3^{29}x_5^{14}$, & $\mathcal Y_{47,\,364}=x_1^{3}x_2x_3^{29}x_4^{14}$, \\
$\mathcal Y_{47,\,365}=x_1^{3}x_2^{29}x_4x_5^{14}$, & $\mathcal Y_{47,\,366}=x_1^{3}x_2^{29}x_4^{14}x_5$, & $\mathcal Y_{47,\,367}=x_1^{3}x_2^{29}x_3x_5^{14}$, & $\mathcal Y_{47,\,368}=x_1^{3}x_2^{29}x_3x_4^{14}$, \\
$\mathcal Y_{47,\,369}=x_1^{3}x_2^{29}x_3^{14}x_5$, & $\mathcal Y_{47,\,370}=x_1^{3}x_2^{29}x_3^{14}x_4$, & $\mathcal Y_{47,\,371}=x_2x_3^{3}x_4^{15}x_5^{28}$, & $\mathcal Y_{47,\,372}=x_2x_3^{3}x_4^{28}x_5^{15}$, \\
$\mathcal Y_{47,\,373}=x_2x_3^{15}x_4^{3}x_5^{28}$, & $\mathcal Y_{47,\,374}=x_2^{3}x_3x_4^{15}x_5^{28}$, & $\mathcal Y_{47,\,375}=x_2^{3}x_3x_4^{28}x_5^{15}$, & $\mathcal Y_{47,\,376}=x_2^{3}x_3^{15}x_4x_5^{28}$, \\
$\mathcal Y_{47,\,377}=x_2^{15}x_3x_4^{3}x_5^{28}$, & $\mathcal Y_{47,\,378}=x_2^{15}x_3^{3}x_4x_5^{28}$, & $\mathcal Y_{47,\,379}=x_1x_3^{3}x_4^{15}x_5^{28}$, & $\mathcal Y_{47,\,380}=x_1x_3^{3}x_4^{28}x_5^{15}$, \\
$\mathcal Y_{47,\,381}=x_1x_3^{15}x_4^{3}x_5^{28}$, & $\mathcal Y_{47,\,382}=x_1x_2^{3}x_4^{15}x_5^{28}$, & $\mathcal Y_{47,\,383}=x_1x_2^{3}x_4^{28}x_5^{15}$, & $\mathcal Y_{47,\,384}=x_1x_2^{3}x_3^{15}x_5^{28}$, \\
$\mathcal Y_{47,\,385}=x_1x_2^{3}x_3^{15}x_4^{28}$, & $\mathcal Y_{47,\,386}=x_1x_2^{3}x_3^{28}x_5^{15}$, & $\mathcal Y_{47,\,387}=x_1x_2^{3}x_3^{28}x_4^{15}$, & $\mathcal Y_{47,\,388}=x_1x_2^{15}x_4^{3}x_5^{28}$, \\
$\mathcal Y_{47,\,389}=x_1x_2^{15}x_3^{3}x_5^{28}$, & $\mathcal Y_{47,\,390}=x_1x_2^{15}x_3^{3}x_4^{28}$, & $\mathcal Y_{47,\,391}=x_1^{3}x_3x_4^{15}x_5^{28}$, & $\mathcal Y_{47,\,392}=x_1^{3}x_3x_4^{28}x_5^{15}$, \\
$\mathcal Y_{47,\,393}=x_1^{3}x_3^{15}x_4x_5^{28}$, & $\mathcal Y_{47,\,394}=x_1^{3}x_2x_4^{15}x_5^{28}$, & $\mathcal Y_{47,\,395}=x_1^{3}x_2x_4^{28}x_5^{15}$, & $\mathcal Y_{47,\,396}=x_1^{3}x_2x_3^{15}x_5^{28}$, \\
$\mathcal Y_{47,\,397}=x_1^{3}x_2x_3^{15}x_4^{28}$, & $\mathcal Y_{47,\,398}=x_1^{3}x_2x_3^{28}x_5^{15}$, & $\mathcal Y_{47,\,399}=x_1^{3}x_2x_3^{28}x_4^{15}$, & $\mathcal Y_{47,\,400}=x_1^{3}x_2^{15}x_4x_5^{28}$, \\
$\mathcal Y_{47,\,401}=x_1^{3}x_2^{15}x_3x_5^{28}$, & $\mathcal Y_{47,\,402}=x_1^{3}x_2^{15}x_3x_4^{28}$, & $\mathcal Y_{47,\,403}=x_1^{15}x_3x_4^{3}x_5^{28}$, & $\mathcal Y_{47,\,404}=x_1^{15}x_3^{3}x_4x_5^{28}$, \\
$\mathcal Y_{47,\,405}=x_1^{15}x_2x_4^{3}x_5^{28}$, & $\mathcal Y_{47,\,406}=x_1^{15}x_2x_3^{3}x_5^{28}$, & $\mathcal Y_{47,\,407}=x_1^{15}x_2x_3^{3}x_4^{28}$, & $\mathcal Y_{47,\,408}=x_1^{15}x_2^{3}x_4x_5^{28}$, \\
$\mathcal Y_{47,\,409}=x_1^{15}x_2^{3}x_3x_5^{28}$, & $\mathcal Y_{47,\,410}=x_1^{15}x_2^{3}x_3x_4^{28}$, & $\mathcal Y_{47,\,411}=x_2x_3^{6}x_4^{11}x_5^{29}$, & $\mathcal Y_{47,\,412}=x_1x_3^{6}x_4^{11}x_5^{29}$, \\
$\mathcal Y_{47,\,413}=x_1x_2^{6}x_4^{11}x_5^{29}$, & $\mathcal Y_{47,\,414}=x_1x_2^{6}x_3^{11}x_5^{29}$, & $\mathcal Y_{47,\,415}=x_1x_2^{6}x_3^{11}x_4^{29}$, & $\mathcal Y_{47,\,416}=x_2x_3^{6}x_4^{27}x_5^{13}$, \\
$\mathcal Y_{47,\,417}=x_1x_3^{6}x_4^{27}x_5^{13}$, & $\mathcal Y_{47,\,418}=x_1x_2^{6}x_4^{27}x_5^{13}$, & $\mathcal Y_{47,\,419}=x_1x_2^{6}x_3^{27}x_5^{13}$, & $\mathcal Y_{47,\,420}=x_1x_2^{6}x_3^{27}x_4^{13}$, \\
$\mathcal Y_{47,\,421}=x_2x_3^{7}x_4^{10}x_5^{29}$, & $\mathcal Y_{47,\,422}=x_2^{7}x_3x_4^{10}x_5^{29}$, & $\mathcal Y_{47,\,423}=x_1x_3^{7}x_4^{10}x_5^{29}$, & $\mathcal Y_{47,\,424}=x_1x_2^{7}x_4^{10}x_5^{29}$, \\
$\mathcal Y_{47,\,425}=x_1x_2^{7}x_3^{10}x_5^{29}$, & $\mathcal Y_{47,\,426}=x_1x_2^{7}x_3^{10}x_4^{29}$, & $\mathcal Y_{47,\,427}=x_1^{7}x_3x_4^{10}x_5^{29}$, & $\mathcal Y_{47,\,428}=x_1^{7}x_2x_4^{10}x_5^{29}$, \\
$\mathcal Y_{47,\,429}=x_1^{7}x_2x_3^{10}x_5^{29}$, & $\mathcal Y_{47,\,430}=x_1^{7}x_2x_3^{10}x_4^{29}$, & $\mathcal Y_{47,\,431}=x_2x_3^{7}x_4^{11}x_5^{28}$, & $\mathcal Y_{47,\,432}=x_2^{7}x_3x_4^{11}x_5^{28}$, \\
$\mathcal Y_{47,\,433}=x_2^{7}x_3^{11}x_4x_5^{28}$, & $\mathcal Y_{47,\,434}=x_1x_3^{7}x_4^{11}x_5^{28}$, & $\mathcal Y_{47,\,435}=x_1x_2^{7}x_4^{11}x_5^{28}$, & $\mathcal Y_{47,\,436}=x_1x_2^{7}x_3^{11}x_5^{28}$, \\
$\mathcal Y_{47,\,437}=x_1x_2^{7}x_3^{11}x_4^{28}$, & $\mathcal Y_{47,\,438}=x_1^{7}x_3x_4^{11}x_5^{28}$, & $\mathcal Y_{47,\,439}=x_1^{7}x_3^{11}x_4x_5^{28}$, & $\mathcal Y_{47,\,440}=x_1^{7}x_2x_4^{11}x_5^{28}$, \\
$\mathcal Y_{47,\,441}=x_1^{7}x_2x_3^{11}x_5^{28}$, & $\mathcal Y_{47,\,442}=x_1^{7}x_2x_3^{11}x_4^{28}$, & $\mathcal Y_{47,\,443}=x_1^{7}x_2^{11}x_4x_5^{28}$, & $\mathcal Y_{47,\,444}=x_1^{7}x_2^{11}x_3x_5^{28}$, \\
$\mathcal Y_{47,\,445}=x_1^{7}x_2^{11}x_3x_4^{28}$, & $\mathcal Y_{47,\,446}=x_2x_3^{7}x_4^{27}x_5^{12}$, & $\mathcal Y_{47,\,447}=x_2^{7}x_3x_4^{27}x_5^{12}$, & $\mathcal Y_{47,\,448}=x_2^{7}x_3^{27}x_4x_5^{12}$, \\
$\mathcal Y_{47,\,449}=x_1x_3^{7}x_4^{27}x_5^{12}$, & $\mathcal Y_{47,\,450}=x_1x_2^{7}x_4^{27}x_5^{12}$, & $\mathcal Y_{47,\,451}=x_1x_2^{7}x_3^{27}x_5^{12}$, & $\mathcal Y_{47,\,452}=x_1x_2^{7}x_3^{27}x_4^{12}$, \\
$\mathcal Y_{47,\,453}=x_1^{7}x_3x_4^{27}x_5^{12}$, & $\mathcal Y_{47,\,454}=x_1^{7}x_3^{27}x_4x_5^{12}$, & $\mathcal Y_{47,\,455}=x_1^{7}x_2x_4^{27}x_5^{12}$, & $\mathcal Y_{47,\,456}=x_1^{7}x_2x_3^{27}x_5^{12}$, \\
$\mathcal Y_{47,\,457}=x_1^{7}x_2x_3^{27}x_4^{12}$, & $\mathcal Y_{47,\,458}=x_1^{7}x_2^{27}x_4x_5^{12}$, & $\mathcal Y_{47,\,459}=x_1^{7}x_2^{27}x_3x_5^{12}$, & $\mathcal Y_{47,\,460}=x_1^{7}x_2^{27}x_3x_4^{12}$, \\
$\mathcal Y_{47,\,461}=x_2x_3^{7}x_4^{26}x_5^{13}$, & $\mathcal Y_{47,\,462}=x_2^{7}x_3x_4^{26}x_5^{13}$, & $\mathcal Y_{47,\,463}=x_1x_3^{7}x_4^{26}x_5^{13}$, & $\mathcal Y_{47,\,464}=x_1x_2^{7}x_4^{26}x_5^{13}$, \\
$\mathcal Y_{47,\,465}=x_1x_2^{7}x_3^{26}x_5^{13}$, & $\mathcal Y_{47,\,466}=x_1x_2^{7}x_3^{26}x_4^{13}$, & $\mathcal Y_{47,\,467}=x_1^{7}x_3x_4^{26}x_5^{13}$, & $\mathcal Y_{47,\,468}=x_1^{7}x_2x_4^{26}x_5^{13}$, \\
$\mathcal Y_{47,\,469}=x_1^{7}x_2x_3^{26}x_5^{13}$, & $\mathcal Y_{47,\,470}=x_1^{7}x_2x_3^{26}x_4^{13}$, & $\mathcal Y_{47,\,471}=x_2^{3}x_3^{13}x_4^{2}x_5^{29}$, & $\mathcal Y_{47,\,472}=x_2^{3}x_3^{29}x_4^{2}x_5^{13}$, \\
$\mathcal Y_{47,\,473}=x_1^{3}x_3^{13}x_4^{2}x_5^{29}$, & $\mathcal Y_{47,\,474}=x_1^{3}x_3^{29}x_4^{2}x_5^{13}$, & $\mathcal Y_{47,\,475}=x_1^{3}x_2^{13}x_4^{2}x_5^{29}$, & $\mathcal Y_{47,\,476}=x_1^{3}x_2^{13}x_3^{2}x_5^{29}$, \\
$\mathcal Y_{47,\,477}=x_1^{3}x_2^{13}x_3^{2}x_4^{29}$, & $\mathcal Y_{47,\,478}=x_1^{3}x_2^{29}x_4^{2}x_5^{13}$, & $\mathcal Y_{47,\,479}=x_1^{3}x_2^{29}x_3^{2}x_5^{13}$, & $\mathcal Y_{47,\,480}=x_1^{3}x_2^{29}x_3^{2}x_4^{13}$, \\
$\mathcal Y_{47,\,481}=x_2^{3}x_3^{3}x_4^{12}x_5^{29}$, & $\mathcal Y_{47,\,482}=x_2^{3}x_3^{3}x_4^{29}x_5^{12}$, & $\mathcal Y_{47,\,483}=x_2^{3}x_3^{29}x_4^{3}x_5^{12}$, & $\mathcal Y_{47,\,484}=x_1^{3}x_3^{3}x_4^{12}x_5^{29}$, \\
$\mathcal Y_{47,\,485}=x_1^{3}x_3^{3}x_4^{29}x_5^{12}$, & $\mathcal Y_{47,\,486}=x_1^{3}x_3^{29}x_4^{3}x_5^{12}$, & $\mathcal Y_{47,\,487}=x_1^{3}x_2^{3}x_4^{12}x_5^{29}$, & $\mathcal Y_{47,\,488}=x_1^{3}x_2^{3}x_4^{29}x_5^{12}$, \\
$\mathcal Y_{47,\,489}=x_1^{3}x_2^{3}x_3^{12}x_5^{29}$, & $\mathcal Y_{47,\,490}=x_1^{3}x_2^{3}x_3^{12}x_4^{29}$, & $\mathcal Y_{47,\,491}=x_1^{3}x_2^{3}x_3^{29}x_5^{12}$, & $\mathcal Y_{47,\,492}=x_1^{3}x_2^{3}x_3^{29}x_4^{12}$, \\
$\mathcal Y_{47,\,493}=x_1^{3}x_2^{29}x_4^{3}x_5^{12}$, & $\mathcal Y_{47,\,494}=x_1^{3}x_2^{29}x_3^{3}x_5^{12}$, & $\mathcal Y_{47,\,495}=x_1^{3}x_2^{29}x_3^{3}x_4^{12}$, & $\mathcal Y_{47,\,496}=x_2^{3}x_3^{3}x_4^{13}x_5^{28}$, \\
$\mathcal Y_{47,\,497}=x_2^{3}x_3^{3}x_4^{28}x_5^{13}$, & $\mathcal Y_{47,\,498}=x_2^{3}x_3^{13}x_4^{3}x_5^{28}$, & $\mathcal Y_{47,\,499}=x_1^{3}x_3^{3}x_4^{13}x_5^{28}$, & $\mathcal Y_{47,\,500}=x_1^{3}x_3^{3}x_4^{28}x_5^{13}$, \\
$\mathcal Y_{47,\,501}=x_1^{3}x_3^{13}x_4^{3}x_5^{28}$, & $\mathcal Y_{47,\,502}=x_1^{3}x_2^{3}x_4^{13}x_5^{28}$, & $\mathcal Y_{47,\,503}=x_1^{3}x_2^{3}x_4^{28}x_5^{13}$, & $\mathcal Y_{47,\,504}=x_1^{3}x_2^{3}x_3^{13}x_5^{28}$, \\
$\mathcal Y_{47,\,505}=x_1^{3}x_2^{3}x_3^{13}x_4^{28}$, & $\mathcal Y_{47,\,506}=x_1^{3}x_2^{3}x_3^{28}x_5^{13}$, & $\mathcal Y_{47,\,507}=x_1^{3}x_2^{3}x_3^{28}x_4^{13}$, & $\mathcal Y_{47,\,508}=x_1^{3}x_2^{13}x_4^{3}x_5^{28}$, \\
$\mathcal Y_{47,\,509}=x_1^{3}x_2^{13}x_3^{3}x_5^{28}$, & $\mathcal Y_{47,\,510}=x_1^{3}x_2^{13}x_3^{3}x_4^{28}$, & $\mathcal Y_{47,\,511}=x_2^{3}x_3^{4}x_4^{11}x_5^{29}$, & $\mathcal Y_{47,\,512}=x_1^{3}x_3^{4}x_4^{11}x_5^{29}$, \\
$\mathcal Y_{47,\,513}=x_1^{3}x_2^{4}x_4^{11}x_5^{29}$, & $\mathcal Y_{47,\,514}=x_1^{3}x_2^{4}x_3^{11}x_5^{29}$, & $\mathcal Y_{47,\,515}=x_1^{3}x_2^{4}x_3^{11}x_4^{29}$, & $\mathcal Y_{47,\,516}=x_2^{3}x_3^{4}x_4^{27}x_5^{13}$, \\
$\mathcal Y_{47,\,517}=x_1^{3}x_3^{4}x_4^{27}x_5^{13}$, & $\mathcal Y_{47,\,518}=x_1^{3}x_2^{4}x_4^{27}x_5^{13}$, & $\mathcal Y_{47,\,519}=x_1^{3}x_2^{4}x_3^{27}x_5^{13}$, & $\mathcal Y_{47,\,520}=x_1^{3}x_2^{4}x_3^{27}x_4^{13}$, \\
$\mathcal Y_{47,\,521}=x_2^{3}x_3^{5}x_4^{10}x_5^{29}$, & $\mathcal Y_{47,\,522}=x_1^{3}x_3^{5}x_4^{10}x_5^{29}$, & $\mathcal Y_{47,\,523}=x_1^{3}x_2^{5}x_4^{10}x_5^{29}$, & $\mathcal Y_{47,\,524}=x_1^{3}x_2^{5}x_3^{10}x_5^{29}$, \\
$\mathcal Y_{47,\,525}=x_1^{3}x_2^{5}x_3^{10}x_4^{29}$, & $\mathcal Y_{47,\,526}=x_2^{3}x_3^{5}x_4^{11}x_5^{28}$, & $\mathcal Y_{47,\,527}=x_1^{3}x_3^{5}x_4^{11}x_5^{28}$, & $\mathcal Y_{47,\,528}=x_1^{3}x_2^{5}x_4^{11}x_5^{28}$, \\
$\mathcal Y_{47,\,529}=x_1^{3}x_2^{5}x_3^{11}x_5^{28}$, & $\mathcal Y_{47,\,530}=x_1^{3}x_2^{5}x_3^{11}x_4^{28}$, & $\mathcal Y_{47,\,531}=x_2^{3}x_3^{5}x_4^{27}x_5^{12}$, & $\mathcal Y_{47,\,532}=x_1^{3}x_3^{5}x_4^{27}x_5^{12}$, \\
$\mathcal Y_{47,\,533}=x_1^{3}x_2^{5}x_4^{27}x_5^{12}$, & $\mathcal Y_{47,\,534}=x_1^{3}x_2^{5}x_3^{27}x_5^{12}$, & $\mathcal Y_{47,\,535}=x_1^{3}x_2^{5}x_3^{27}x_4^{12}$, & $\mathcal Y_{47,\,536}=x_2^{3}x_3^{5}x_4^{26}x_5^{13}$, \\
$\mathcal Y_{47,\,537}=x_1^{3}x_3^{5}x_4^{26}x_5^{13}$, & $\mathcal Y_{47,\,538}=x_1^{3}x_2^{5}x_4^{26}x_5^{13}$, & $\mathcal Y_{47,\,539}=x_1^{3}x_2^{5}x_3^{26}x_5^{13}$, & $\mathcal Y_{47,\,540}=x_1^{3}x_2^{5}x_3^{26}x_4^{13}$, \\
$\mathcal Y_{47,\,541}=x_2^{3}x_3^{7}x_4^{9}x_5^{28}$, & $\mathcal Y_{47,\,542}=x_2^{7}x_3^{3}x_4^{9}x_5^{28}$, & $\mathcal Y_{47,\,543}=x_1^{3}x_3^{7}x_4^{9}x_5^{28}$, & $\mathcal Y_{47,\,544}=x_1^{3}x_2^{7}x_4^{9}x_5^{28}$, \\
$\mathcal Y_{47,\,545}=x_1^{3}x_2^{7}x_3^{9}x_5^{28}$, & $\mathcal Y_{47,\,546}=x_1^{3}x_2^{7}x_3^{9}x_4^{28}$, & $\mathcal Y_{47,\,547}=x_1^{7}x_3^{3}x_4^{9}x_5^{28}$, & $\mathcal Y_{47,\,548}=x_1^{7}x_2^{3}x_4^{9}x_5^{28}$, \\
$\mathcal Y_{47,\,549}=x_1^{7}x_2^{3}x_3^{9}x_5^{28}$, & $\mathcal Y_{47,\,550}=x_1^{7}x_2^{3}x_3^{9}x_4^{28}$, & $\mathcal Y_{47,\,551}=x_2^{3}x_3^{7}x_4^{25}x_5^{12}$, & $\mathcal Y_{47,\,552}=x_2^{7}x_3^{3}x_4^{25}x_5^{12}$, \\
$\mathcal Y_{47,\,553}=x_1^{3}x_3^{7}x_4^{25}x_5^{12}$, & $\mathcal Y_{47,\,554}=x_1^{3}x_2^{7}x_4^{25}x_5^{12}$, & $\mathcal Y_{47,\,555}=x_1^{3}x_2^{7}x_3^{25}x_5^{12}$, & $\mathcal Y_{47,\,556}=x_1^{3}x_2^{7}x_3^{25}x_4^{12}$, \\
$\mathcal Y_{47,\,557}=x_1^{7}x_3^{3}x_4^{25}x_5^{12}$, & $\mathcal Y_{47,\,558}=x_1^{7}x_2^{3}x_4^{25}x_5^{12}$, & $\mathcal Y_{47,\,559}=x_1^{7}x_2^{3}x_3^{25}x_5^{12}$, & $\mathcal Y_{47,\,560}=x_1^{7}x_2^{3}x_3^{25}x_4^{12}$.
\end{tabular}%
\end{center}

\subsection{$\mathcal A_2$-generators for $\mathscr P_5^+$ in degree $47$}\label{s5.7}

Recall that Kameko's squaring operation $(\widetilde {Sq_*^0})_{(5,47)}: (Q\mathscr P_5)_{47}\to (Q\mathscr P_5)_{21}$ is an epimorphism of $\mathbb Z/2GL_5$-modules. Therefore, we get
$$ \mathscr B_5(47) = \mathscr B_5^0(47)\bigcup \varphi(\mathscr B_5(21))\bigcup \big(\mathscr B_5^+(47)\bigcap {\rm Ker}((\widetilde {Sq_*^0})_{(5,47)})\big),$$ where $|\mathscr B_5^0(47)|  = 560,\ |\varphi(\mathscr B_5(21))| = 840$ with $ \varphi: \mathscr P_5\to \mathscr P_5,\ \varphi(u) = X_{(\emptyset,\, 5)}u^2,\ \forall u\in \mathscr P_5.$
From the results in Sect.\ref{s3.2.2}, we have
$$ \mathscr B_5^+(47)\,\bigcap\, {\rm Ker}((\widetilde {Sq_*^0})_{(5,47)}) = \mathscr B^+_5(\overline{\omega}_{(1)})\,\bigcup\, \mathscr B^+_5(\overline{\omega}_{(4)})\bigcup \mathscr B^+_5(\overline{\omega}_{(5)}),$$
where $\overline{\omega}_{(1)} = (3,2,2,2,1),\ \overline{\omega}_{(4)} = (3,4,3,1,1),$ and $\overline{\omega}_{(5)} = (3,4,3,3).$ 

\medskip

$\mathscr B^+_5(\overline{\omega}_{(1)})$ is the set of $196$ admissible monomials: $\mathcal Y_{47,\, t},\ 1\leq t\leq 370$

\begin{center}
\begin{tabular}{lrr}
$\mathcal Y_{47,\,1}=x_1x_2x_3x_4^{14}x_5^{30}$, & \multicolumn{1}{l}{$\mathcal Y_{47,\,2}=x_1x_2x_3x_4^{30}x_5^{14}$,} & \multicolumn{1}{l}{$\mathcal Y_{47,\,3}=x_1x_2x_3^{14}x_4x_5^{30}$,} \\
$\mathcal Y_{47,\,4}=x_1x_2x_3^{14}x_4^{30}x_5$, & \multicolumn{1}{l}{$\mathcal Y_{47,\,5}=x_1x_2x_3^{30}x_4x_5^{14}$,} & \multicolumn{1}{l}{$\mathcal Y_{47,\,6}=x_1x_2x_3^{30}x_4^{14}x_5$,} \\
$\mathcal Y_{47,\,7}=x_1x_2^{14}x_3x_4x_5^{30}$, & \multicolumn{1}{l}{$\mathcal Y_{47,\,8}=x_1x_2^{14}x_3x_4^{30}x_5$,} & \multicolumn{1}{l}{$\mathcal Y_{47,\,9}=x_1x_2^{30}x_3x_4x_5^{14}$,} \\
$\mathcal Y_{47,\,10}=x_1x_2^{30}x_3x_4^{14}x_5$, & \multicolumn{1}{l}{$\mathcal Y_{47,\,11}=x_1x_2x_3^{2}x_4^{12}x_5^{31}$,} & \multicolumn{1}{l}{$\mathcal Y_{47,\,12}=x_1x_2x_3^{2}x_4^{31}x_5^{12}$,} \\
$\mathcal Y_{47,\,13}=x_1x_2x_3^{31}x_4^{2}x_5^{12}$, & \multicolumn{1}{l}{$\mathcal Y_{47,\,14}=x_1x_2^{2}x_3x_4^{12}x_5^{31}$,} & \multicolumn{1}{l}{$\mathcal Y_{47,\,15}=x_1x_2^{2}x_3x_4^{31}x_5^{12}$,} \\
$\mathcal Y_{47,\,16}=x_1x_2^{2}x_3^{31}x_4x_5^{12}$, & \multicolumn{1}{l}{$\mathcal Y_{47,\,17}=x_1x_2^{31}x_3x_4^{2}x_5^{12}$,} & \multicolumn{1}{l}{$\mathcal Y_{47,\,18}=x_1x_2^{31}x_3^{2}x_4x_5^{12}$,} \\
$\mathcal Y_{47,\,19}=x_1^{31}x_2x_3x_4^{2}x_5^{12}$, & \multicolumn{1}{l}{$\mathcal Y_{47,\,20}=x_1^{31}x_2x_3^{2}x_4x_5^{12}$,} & \multicolumn{1}{l}{$\mathcal Y_{47,\,21}=x_1x_2^{2}x_3^{12}x_4x_5^{31}$,} \\
$\mathcal Y_{47,\,22}=x_1x_2^{2}x_3^{12}x_4^{31}x_5$, & \multicolumn{1}{l}{$\mathcal Y_{47,\,23}=x_1x_2^{2}x_3^{31}x_4^{12}x_5$,} & \multicolumn{1}{l}{$\mathcal Y_{47,\,24}=x_1x_2^{31}x_3^{2}x_4^{12}x_5$,} \\
\end{tabular}%
\end{center}

\newpage
\begin{center}
\begin{tabular}{llll}
$\mathcal Y_{47,\,103}=x_1x_2^{7}x_3^{10}x_4x_5^{28}$, & \multicolumn{1}{l}{$\mathcal Y_{47,\,104}=x_1^{7}x_2x_3x_4^{10}x_5^{28}$,} & \multicolumn{1}{l}{$\mathcal Y_{47,\,105}=x_1^{7}x_2x_3^{10}x_4x_5^{28}$,} \\
$\mathcal Y_{47,\,106}=x_1x_2x_3^{7}x_4^{26}x_5^{12}$, & \multicolumn{1}{l}{$\mathcal Y_{47,\,107}=x_1x_2^{7}x_3x_4^{26}x_5^{12}$,} & \multicolumn{1}{l}{$\mathcal Y_{47,\,108}=x_1x_2^{7}x_3^{26}x_4x_5^{12}$,} \\
$\mathcal Y_{47,\,109}=x_1^{7}x_2x_3x_4^{26}x_5^{12}$, & \multicolumn{1}{l}{$\mathcal Y_{47,\,110}=x_1^{7}x_2x_3^{26}x_4x_5^{12}$,} & \multicolumn{1}{l}{$\mathcal Y_{47,\,111}=x_1x_2^{2}x_3^{13}x_4^{2}x_5^{29}$,} \\
$\mathcal Y_{47,\,112}=x_1x_2^{2}x_3^{29}x_4^{2}x_5^{13}$, & \multicolumn{1}{l}{$\mathcal Y_{47,\,113}=x_1x_2^{2}x_3^{3}x_4^{12}x_5^{29}$,} & \multicolumn{1}{l}{$\mathcal Y_{47,\,114}=x_1x_2^{2}x_3^{3}x_4^{29}x_5^{12}$,} \\
$\mathcal Y_{47,\,115}=x_1x_2^{2}x_3^{29}x_4^{3}x_5^{12}$, & \multicolumn{1}{l}{$\mathcal Y_{47,\,116}=x_1x_2^{3}x_3^{2}x_4^{12}x_5^{29}$,} & \multicolumn{1}{l}{$\mathcal Y_{47,\,117}=x_1x_2^{3}x_3^{2}x_4^{29}x_5^{12}$,} \\
$\mathcal Y_{47,\,118}=x_1x_2^{3}x_3^{12}x_4^{2}x_5^{29}$, & \multicolumn{1}{l}{$\mathcal Y_{47,\,119}=x_1x_2^{3}x_3^{29}x_4^{2}x_5^{12}$,} & \multicolumn{1}{l}{$\mathcal Y_{47,\,120}=x_1^{3}x_2x_3^{2}x_4^{12}x_5^{29}$,} \\
$\mathcal Y_{47,\,121}=x_1^{3}x_2x_3^{2}x_4^{29}x_5^{12}$, & \multicolumn{1}{l}{$\mathcal Y_{47,\,122}=x_1^{3}x_2x_3^{12}x_4^{2}x_5^{29}$,} & \multicolumn{1}{l}{$\mathcal Y_{47,\,123}=x_1^{3}x_2x_3^{29}x_4^{2}x_5^{12}$,} \\
$\mathcal Y_{47,\,124}=x_1^{3}x_2^{29}x_3x_4^{2}x_5^{12}$, & \multicolumn{1}{l}{$\mathcal Y_{47,\,125}=x_1^{3}x_2^{29}x_3^{2}x_4x_5^{12}$,} & \multicolumn{1}{l}{$\mathcal Y_{47,\,126}=x_1^{3}x_2^{29}x_3^{2}x_4^{12}x_5$,} \\
\end{tabular}%
\end{center}

\newpage
\begin{center}
\begin{tabular}{lrr}
$\mathcal Y_{47,\,253}=x_1^{3}x_2x_3^{5}x_4^{8}x_5^{30}$, & \multicolumn{1}{l}{$\mathcal Y_{47,\,254}=x_1^{3}x_2x_3^{5}x_4^{30}x_5^{8}$,} & \multicolumn{1}{l}{$\mathcal Y_{47,\,255}=x_1x_2^{3}x_3^{5}x_4^{10}x_5^{28}$,} \\
$\mathcal Y_{47,\,256}=x_1^{3}x_2x_3^{5}x_4^{10}x_5^{28}$, & \multicolumn{1}{l}{$\mathcal Y_{47,\,257}=x_1^{3}x_2^{5}x_3x_4^{10}x_5^{28}$,} & \multicolumn{1}{l}{$\mathcal Y_{47,\,258}=x_1^{3}x_2^{5}x_3^{10}x_4x_5^{28}$,} \\
$\mathcal Y_{47,\,259}=x_1^{3}x_2^{5}x_3^{10}x_4^{28}x_5$, & \multicolumn{1}{l}{$\mathcal Y_{47,\,260}=x_1x_2^{3}x_3^{5}x_4^{26}x_5^{12}$,} & \multicolumn{1}{l}{$\mathcal Y_{47,\,261}=x_1^{3}x_2x_3^{5}x_4^{26}x_5^{12}$,} \\
$\mathcal Y_{47,\,262}=x_1^{3}x_2^{5}x_3x_4^{26}x_5^{12}$, & \multicolumn{1}{l}{$\mathcal Y_{47,\,263}=x_1^{3}x_2^{5}x_3^{26}x_4x_5^{12}$,} & \multicolumn{1}{l}{$\mathcal Y_{47,\,264}=x_1^{3}x_2^{5}x_3^{26}x_4^{12}x_5$,} \\
$\mathcal Y_{47,\,265}=x_1x_2^{3}x_3^{14}x_4^{5}x_5^{24}$, & \multicolumn{1}{l}{$\mathcal Y_{47,\,266}=x_1x_2^{3}x_3^{5}x_4^{14}x_5^{24}$,} & \multicolumn{1}{l}{$\mathcal Y_{47,\,267}=x_1x_2^{3}x_3^{5}x_4^{24}x_5^{14}$,} \\
$\mathcal Y_{47,\,268}=x_1^{3}x_2x_3^{14}x_4^{5}x_5^{24}$, & \multicolumn{1}{l}{$\mathcal Y_{47,\,269}=x_1^{3}x_2x_3^{5}x_4^{14}x_5^{24}$,} & \multicolumn{1}{l}{$\mathcal Y_{47,\,270}=x_1^{3}x_2x_3^{5}x_4^{24}x_5^{14}$,} \\
$\mathcal Y_{47,\,271}=x_1x_2^{3}x_3^{6}x_4^{8}x_5^{29}$, & \multicolumn{1}{l}{$\mathcal Y_{47,\,272}=x_1x_2^{3}x_3^{6}x_4^{29}x_5^{8}$,} & \multicolumn{1}{l}{$\mathcal Y_{47,\,273}=x_1^{3}x_2x_3^{6}x_4^{8}x_5^{29}$,} \\
$\mathcal Y_{47,\,274}=x_1^{3}x_2x_3^{6}x_4^{29}x_5^{8}$, & \multicolumn{1}{l}{$\mathcal Y_{47,\,275}=x_1x_2^{3}x_3^{6}x_4^{9}x_5^{28}$,} & \multicolumn{1}{l}{$\mathcal Y_{47,\,276}=x_1x_2^{6}x_3^{3}x_4^{9}x_5^{28}$,} \\
\end{tabular}%
\end{center}

$\mathscr B^+_5(\overline{\omega}_{(4)})$ is the set of $109$ admissible monomials: $\mathcal Y_{47,\, t},\ 371\leq t\leq 479$

\begin{center}
\begin{tabular}{lrr}
$\mathcal Y_{47,\,371}=x_1x_2^{3}x_3^{6}x_4^{6}x_5^{31}$, & \multicolumn{1}{l}{$\mathcal Y_{47,\,372}=x_1x_2^{3}x_3^{6}x_4^{31}x_5^{6}$,} & \multicolumn{1}{l}{$\mathcal Y_{47,\,373}=x_1x_2^{3}x_3^{31}x_4^{6}x_5^{6}$,} \\
$\mathcal Y_{47,\,374}=x_1x_2^{31}x_3^{3}x_4^{6}x_5^{6}$, & \multicolumn{1}{l}{$\mathcal Y_{47,\,375}=x_1^{3}x_2x_3^{6}x_4^{6}x_5^{31}$,} & \multicolumn{1}{l}{$\mathcal Y_{47,\,376}=x_1^{3}x_2x_3^{6}x_4^{31}x_5^{6}$,} \\
$\mathcal Y_{47,\,377}=x_1^{3}x_2x_3^{31}x_4^{6}x_5^{6}$, & \multicolumn{1}{l}{$\mathcal Y_{47,\,378}=x_1^{3}x_2^{31}x_3x_4^{6}x_5^{6}$,} & \multicolumn{1}{l}{$\mathcal Y_{47,\,379}=x_1^{31}x_2x_3^{3}x_4^{6}x_5^{6}$,} \\
$\mathcal Y_{47,\,380}=x_1^{31}x_2^{3}x_3x_4^{6}x_5^{6}$, & \multicolumn{1}{l}{$\mathcal Y_{47,\,381}=x_1x_2^{3}x_3^{6}x_4^{7}x_5^{30}$,} & \multicolumn{1}{l}{$\mathcal Y_{47,\,382}=x_1x_2^{3}x_3^{7}x_4^{6}x_5^{30}$,} \\
$\mathcal Y_{47,\,383}=x_1x_2^{3}x_3^{7}x_4^{30}x_5^{6}$, & \multicolumn{1}{l}{$\mathcal Y_{47,\,384}=x_1x_2^{7}x_3^{3}x_4^{6}x_5^{30}$,} & \multicolumn{1}{l}{$\mathcal Y_{47,\,385}=x_1x_2^{7}x_3^{3}x_4^{30}x_5^{6}$,} \\
$\mathcal Y_{47,\,386}=x_1^{3}x_2x_3^{6}x_4^{7}x_5^{30}$, & \multicolumn{1}{l}{$\mathcal Y_{47,\,387}=x_1^{3}x_2x_3^{7}x_4^{6}x_5^{30}$,} & \multicolumn{1}{l}{$\mathcal Y_{47,\,388}=x_1^{3}x_2x_3^{7}x_4^{30}x_5^{6}$,} \\
$\mathcal Y_{47,\,389}=x_1^{3}x_2^{7}x_3x_4^{6}x_5^{30}$, & \multicolumn{1}{l}{$\mathcal Y_{47,\,390}=x_1^{3}x_2^{7}x_3x_4^{30}x_5^{6}$,} & \multicolumn{1}{l}{$\mathcal Y_{47,\,391}=x_1^{7}x_2x_3^{3}x_4^{6}x_5^{30}$,} \\
$\mathcal Y_{47,\,392}=x_1^{7}x_2x_3^{3}x_4^{30}x_5^{6}$, & \multicolumn{1}{l}{$\mathcal Y_{47,\,393}=x_1^{7}x_2^{3}x_3x_4^{6}x_5^{30}$,} & \multicolumn{1}{l}{$\mathcal Y_{47,\,394}=x_1^{7}x_2^{3}x_3x_4^{30}x_5^{6}$,} \\
\end{tabular}%
\end{center}

\newpage
\begin{center}
\begin{tabular}{lrr}
$\mathcal Y_{47,\,398}=x_1^{3}x_2^{7}x_3x_4^{14}x_5^{22}$, & \multicolumn{1}{l}{$\mathcal Y_{47,\,399}=x_1^{7}x_2x_3^{3}x_4^{14}x_5^{22}$,} & \multicolumn{1}{l}{$\mathcal Y_{47,\,400}=x_1^{7}x_2^{3}x_3x_4^{14}x_5^{22}$,} \\
$\mathcal Y_{47,\,401}=x_1x_2^{7}x_3^{27}x_4^{6}x_5^{6}$, & \multicolumn{1}{l}{$\mathcal Y_{47,\,402}=x_1^{7}x_2x_3^{27}x_4^{6}x_5^{6}$,} & \multicolumn{1}{l}{$\mathcal Y_{47,\,403}=x_1^{7}x_2^{27}x_3x_4^{6}x_5^{6}$,} \\
$\mathcal Y_{47,\,404}=x_1x_2^{7}x_3^{11}x_4^{6}x_5^{22}$, & \multicolumn{1}{l}{$\mathcal Y_{47,\,405}=x_1x_2^{7}x_3^{11}x_4^{22}x_5^{6}$,} & \multicolumn{1}{l}{$\mathcal Y_{47,\,406}=x_1^{7}x_2x_3^{11}x_4^{6}x_5^{22}$,} \\
$\mathcal Y_{47,\,407}=x_1^{7}x_2x_3^{11}x_4^{22}x_5^{6}$, & \multicolumn{1}{l}{$\mathcal Y_{47,\,408}=x_1^{7}x_2^{11}x_3x_4^{6}x_5^{22}$,} & \multicolumn{1}{l}{$\mathcal Y_{47,\,409}=x_1^{7}x_2^{11}x_3x_4^{22}x_5^{6}$,} \\
$\mathcal Y_{47,\,410}=x_1^{3}x_2^{5}x_3^{2}x_4^{6}x_5^{31}$, & \multicolumn{1}{l}{$\mathcal Y_{47,\,411}=x_1^{3}x_2^{5}x_3^{2}x_4^{31}x_5^{6}$,} & \multicolumn{1}{l}{$\mathcal Y_{47,\,412}=x_1^{3}x_2^{5}x_3^{6}x_4^{2}x_5^{31}$,} \\
$\mathcal Y_{47,\,413}=x_1^{3}x_2^{5}x_3^{6}x_4^{31}x_5^{2}$, & \multicolumn{1}{l}{$\mathcal Y_{47,\,414}=x_1^{3}x_2^{5}x_3^{31}x_4^{2}x_5^{6}$,} & \multicolumn{1}{l}{$\mathcal Y_{47,\,415}=x_1^{3}x_2^{5}x_3^{31}x_4^{6}x_5^{2}$,} \\
$\mathcal Y_{47,\,416}=x_1^{3}x_2^{31}x_3^{5}x_4^{2}x_5^{6}$, & \multicolumn{1}{l}{$\mathcal Y_{47,\,417}=x_1^{3}x_2^{31}x_3^{5}x_4^{6}x_5^{2}$,} & \multicolumn{1}{l}{$\mathcal Y_{47,\,418}=x_1^{31}x_2^{3}x_3^{5}x_4^{2}x_5^{6}$,} \\
$\mathcal Y_{47,\,419}=x_1^{31}x_2^{3}x_3^{5}x_4^{6}x_5^{2}$, & \multicolumn{1}{l}{$\mathcal Y_{47,\,420}=x_1^{3}x_2^{5}x_3^{2}x_4^{7}x_5^{30}$,} & \multicolumn{1}{l}{$\mathcal Y_{47,\,421}=x_1^{3}x_2^{5}x_3^{7}x_4^{2}x_5^{30}$,} \\
\end{tabular}%
\end{center}

\medskip

$\mathscr B^+_5(\overline{\omega}_{(5)})$ is the set of $15$ admissible monomials: $\mathcal Y_{47,\, t},\ 480\leq t\leq 494$

\begin{center}
\begin{tabular}{lll}
$\mathcal Y_{47,\,480}=x_1x_2^{7}x_3^{11}x_4^{14}x_5^{14}$, & $\mathcal Y_{47,\,481}=x_1^{7}x_2x_3^{11}x_4^{14}x_5^{14}$, & $\mathcal Y_{47,\,482}=x_1^{7}x_2^{11}x_3x_4^{14}x_5^{14}$, \\
$\mathcal Y_{47,\,483}=x_1^{7}x_2^{11}x_3^{13}x_4^{2}x_5^{14}$, & $\mathcal Y_{47,\,484}=x_1^{7}x_2^{11}x_3^{13}x_4^{14}x_5^{2}$, & 
$\mathcal Y_{47,\,485}=x_1^{3}x_2^{7}x_3^{9}x_4^{14}x_5^{14}$, \\
 $\mathcal Y_{47,\,486}=x_1^{7}x_2^{3}x_3^{9}x_4^{14}x_5^{14}$,& $\mathcal Y_{47,\,487}=x_1^{3}x_2^{7}x_3^{13}x_4^{10}x_5^{14}$, & $\mathcal Y_{47,\,488}=x_1^{3}x_2^{7}x_3^{13}x_4^{14}x_5^{10}$, \\
 $\mathcal Y_{47,\,489}=x_1^{7}x_2^{3}x_3^{13}x_4^{10}x_5^{14}$, & $\mathcal Y_{47,\,490}=x_1^{7}x_2^{3}x_3^{13}x_4^{14}x_5^{10}$, &
$\mathcal Y_{47,\,491}=x_1^{7}x_2^{11}x_3^{5}x_4^{10}x_5^{14}$, \\
 $\mathcal Y_{47,\,492}=x_1^{7}x_2^{11}x_3^{5}x_4^{14}x_5^{10}$, & $\mathcal Y_{47,\,493}=x_1^{7}x_2^{11}x_3^{13}x_4^{6}x_5^{10},$ & $\mathcal Y_{47,\,494}=x_1^{3}x_2^{5}x_3^{11}x_4^{14}x_5^{14}.$
\end{tabular}%
\end{center}

\end{document}